\documentclass[a4paper,12pt]{report} %classe del documento e grandezze di scrittura 
\usepackage[latin1]{inputenc} %scrive in latin e si puo' scrivere direttamente accentato 
\usepackage[T1]{fontenc} 
\usepackage[english]{babel} 
\usepackage{indentfirst} %fa iniziare i paragrafi allineati 
\usepackage{amssymb}
\usepackage{amsmath}
\usepackage{amsthm}
\usepackage{graphicx}
\usepackage{srcltx}
\usepackage{epsfig}
\usepackage{url}

\usepackage{amsthm}

\theoremstyle{plain}
\newtheorem{thm}{Theorem}[section]
\newtheorem{cor}[thm]{Corollary}
\newtheorem{lem}[thm]{Lemma}
\newtheorem{prop}[thm]{Proposition}
\newtheorem{defn}[thm]{Definition}

\theoremstyle{remark}

\newcommand{\U}{\mathcal{U}}
\newcommand{\N}{\mathbb{N}}
\newcommand{\Z}{\mathbb{Z}}
\newcommand{\Q}{\mathbb{Q}}
\newcommand{\R}{\mathbb{R}}
\newcommand{\V}{\mathcal{V}}
\newcommand{\uZ}{\mathcal{Z}}
\newcommand{\W}{\mathcal{W}}
\newcommand{\h}{\heartsuit}
\newcommand{\dia}{\diamondsuit}
\newcommand{\bN}{\beta\mathbb{N}}

\newcommand{\fN}{\mathtt{Fun}(\N,\N)}

\begin{document}
\begin{titlepage}

\begin{center}
{\LARGE {\scshape UNIVERSIT\`{A} DEGLI STUDI DI SIENA}}\\\vspace{0.3cm}

{\scshape  Scuola di Dottorato di Ricerca in Logica Matematica, Informatica e Bioinformatica\\
Sezione di Logica Matematica ed Informatica\\
Ciclo XXIV}
\end{center}

{\Large \begin{center} {\scshape Dissertazione di Dottorato}
\end{center}}

{\LARGE \begin{center}
\textbf{HYPERINTEGERS AND NONSTANDARD TECHNIQUES IN COMBINATORICS OF NUMBERS}
\end{center}} \vspace{0.5cm}

\begin{center}
\begin{tabular}{lcl}
{\Large {\scshape Candidato}} & {\Large{\scshape Relatori}}&\\
{\Large{\itshape Lorenzo Luperi Baglini}} & {\Large{\itshape Prof. Mauro Di Nasso}}\\
 & {\Large{\itshape Prof. Franco Montagna}}\\
\end{tabular}
\end{center}
\vspace{1.2cm}
\begin{center}
{\Large{\scshape Anno Accademico 2011/2012}}
\end{center}
\end{titlepage}

\tableofcontents
\newpage

\chapter*{Introduction}
\addcontentsline{toc}{chapter}{Introduction} 

In literature, many important combinatorial properties of subsets of $\N$ have been studied both with nonstandard techniques (see e.g. [JI01], [Hir88]) and from the point of view of $\bN$ (see e.g. [HS98], [Hin11]). The idea behind the researches presented in this thesis is to mix these two different approaches in a technique that, at the same time, incorporates nonstandard tools and ultrafilters. The "monads" of ultrafilters are the basis of this technique: given an hyperextension $^{*}\N$ of $\N$ with the $\mathfrak{c}^{+}$-enlarging property, and an ultrafilter $\U$ on $\N$, the monad of $\U$ is the set

\begin{center} $\mu(\U)=\{\alpha\in$$^{*}\N\mid \U=\mathfrak{U}_{\alpha}\}$, \end{center}

where 

\begin{center} $\mathfrak{U}_{\alpha}=\{A\in\wp(\N)\mid \alpha\in$$^{*}A\}$. \end{center}

Theorem 2.2.9 is the result that generated our researches: it states that particular combinatorial properties of ultrafilters can be seen as "generated" by the elements in their monads (that, for this reason, we call "generators"). So, in a way, the mix between nonstandard tools and ultrafilters is realized by watching the ultrafilters as nonstandard points of specifical hyperextensions of $\N$.\\
Tensor products of ultrafilters are a central concept in our studies. While the problem of how the monad of $\U\otimes\V$ and the monads of $\U,\V$ are related was answered by Puritz in [Pu72, Theorem 3.4], an effective precedure to produce a generator of the tensor product $\U\otimes\V$ starting with a generator of $\U$ and a generator of $\V$ is, at least as far as we know, unknown. By this problem we are led to consider a particular hyperextension of $\N$, that we call $\omega$-hyperextension and denote by $^{\bullet}\N$. Its particularity is that in $^{\bullet}\N$ we can iterate the star map. This gives rise to the desired procedure to construct generators of tensor products: as we prove in Theorem 2.5.16, if $\alpha$ is a generator of $\U$ and $\beta$ is a generator of $\V$ then the pair $(\alpha,$$^{*}\beta)$ is a generator of $\U\otimes\V$.\\
The combination of Theorem 2.2.9 and Theorem 2.5.16 gives a nonstandard technique to study particular combinatorial properties on $\N$. To investigate the potentialities of this technique we re-prove some well-known results in Ramsey Theory, e.g. Schur's Theorem and Folkman's Theorem. One of the results that we discuss is a particular case of Rado's Theorem, concerning with linear polynomials in $\Z[\mathbf{X}]$ with sum of coefficients zero. In this case we show that the problem of the partition regularity of such polynomials can be solved considering particular linear combinations of idempotent ultrafilters.\\
This suggests to face the problem of the partition regularity of nonlinear polynomials. While we have not a complete characterization (in the style of Rado's Theorem), we prove some general result that ensures the partition regularity of suitably constructed polynomials. The most general result we prove in this context is Theorem 3.5.13: given a natural number $n$ and distinct variables $y_{1},...,y_{n}$, for every $F\subseteq\{1,...,n\}$ we pose

\begin{center} $Q_{F}(y_{1},...,y_{n})=\prod_{j\in F} y_{j}$ \end{center}

($Q_{F}(y_{1},...,y_{n})=1$ if $F=\emptyset$). Theorem 3.5.13 states the following:

\begin{thm} Let $k\geq 3$ be a natural number, $P(x_{1},...,x_{k})=\sum_{i=1}^{k} a_{i}x_{i}$ an injectively partition regular polynomial, and $n$ a positive natural number. Then, for every $F_{1},...,F_{k}\subseteq\{1,..,n\}$, the polynomial

\begin{center} $R(x_{1},...,x_{k},y_{1},...,y_{n})=\sum_{i=1}^{k} a_{i}x_{i}Q_{F_{i}}(y_{1},...,y_{n})$ \end{center}

is injectively partition regular. \end{thm}

We also study some general properties of the set of partition regular polynomials. An interesting result is that, as a consequence of Theorem 3.6.3, the research on partition regular polynomials can be restricted to consider only irreducible polynomials.\\
The last topic we face are particular (pre--)orders defined on $\bN$. This research is motivated by the notion of "finite embeddability" (see [DN12]): given two subsets $A,B$ of $\N$, we say that $A$ is finitely embeddable in $B$ (notation $A\leq_{fe} B$) if for every finite subset $F$ of $A$ there is a natural number $n$ such that $n+F\subseteq B$. This notion can be extended to ultrafilters: we say that an ultrafilter $\U$ is finitely embeddable in $\V$ (notation $\U\trianglelefteq_{fe}\V$) if for every set $B$ in $\V$ there is a set $A\in\U$ such that $A\leq_{fe} B$.\\
This notion has some interesting combinatorial properties. In particular, we prove that the relation $\trianglelefteq_{fe}$ is a filtered pre--order with maximal elements. An interesting result is the following: let $\mathcal{M}_{fe}$ denotes the set of maximal ultrafilters in $(\bN,\trianglelefteq_{fe})$, and consider the minimal bilateral ideal $K(\bN,\oplus)$ of $(\bN,\oplus)$. Then

\begin{center} $\mathcal{M}_{fe}=\overline{K(\bN,\oplus)}.$ \end{center} 

Motivated by these nice features of the finite embeddability, we consider a generalization that is motivated by the following observation: the finite embeddability is related with the class $\mathbb{T}$ of translations on $\N$. In fact, $A\leq_{fe} B$ if and only if for every finite subset $F$ of $A$ there is a translation $t_{n}\in\mathbb{T}$ such that $t_{n}(F)\subseteq B$. This leads to consider the following notion: given a set $\mathcal{F}$ of functions in $\N^{\N}$, and two subsets $A,B$ of $\N$, we say that $A$ is $\mathcal{F}$-finitely mappable in $B$ (notation $A\leq_{\mathcal{F}} B$) if for every finite subset $F$ of $A$ there is a function $f$ in $\mathcal{F}$ such that $f(F)\subseteq B$. Similarly, given two ultrafilters $\U,\V$, we say that $\U$ is $\mathcal{F}$-finitely mappable in $\V$ (notation $\U\trianglelefteq_{\mathcal{F}}\V$) if for every set $B$ in $\V$ there is a set $A$ in $\U$ such that $A\leq_{\mathcal{F}} B$. The properties of these notions are, under particular assumptions on the family $\mathcal{F}$, similar to that of finite embeddability. In particular, when the pre-orders $\trianglelefteq_{\mathcal{F}}$ have maximal elements, the maximal elements have important combinatorial properties.\\

The thesis is structured in four chapters.\\
Chapter One contains a short introduction to the "non-nonstandard" background of the thesis: we shortly recall the definitions and the basic properties of ultrafilters and of the Stone-\v{C}ech compactification $\bN$ of $\N$, we present the relations between ultrafilters on a set $S$ and partition regularity of families of subsets of $S$, and we recall (and prove) some of the most known results in Ramsey Theory on $\N$.\\
In Chapter Two we present, and construct, the nonstandard tools that we need. We recall the definition of nonstandard model, as well as some important properties of the hyperextensions. Then we concentrate on the sets of generators of ultrafilters and, led by considerations on tensor products of ultrafilter, we introduce the $\omega$-hyperextension $^{\bullet}\N$ of $\N$.\\
In Chapter Three we test our nonstandard technique by re-proving some well-known result in Ramsey Theorem; then we present our research on the partition regularity of polynomials, indicating some possible future developments.\\
In Chapter Four we start presenting the properties of the finite embeddability for sets and ultrafilters. Then we concentrate on the generalization to arbitrary families of functions, showing under what conditions the properties of finite embeddability can be generalized in this more general context, and which combinatorial properties are satisfied by the maximal elements in the pre-ordered structures $(\bN,\trianglelefteq_{\mathcal{F}})$ for appropriate classes of functions $\mathcal{F}$.\\

\newpage

\chapter*{Acknowledgements}
\addcontentsline{toc}{chapter}{Acknowledgements} 

I would like to express my gratitude to Professor Mauro Di Nasso, that introduced me to Nonstandard Analysis and Ramsey Theory, not only for countless discussions and continuous advices regarding the topics of this thesis, but also for guiding me throughout all my studies at university, both as an undergraduate student and as a PhD student.\\
I would like also to thank Professor Franco Montagna for his high permissiveness with respect to all my requests, and for letting me organize the cycle of seminars of PhD students.\\
Finally, I would like to thank my high school Professor Pietro Ribas, who taught me that mathematics can be elegant, and to my elementary school teacher Piero Franceschi, who taught me that mathematics can be fun.

\newpage
\chapter{Ultrafilters, Partition Regularity and some Infinite Combinatorics}

In this chapter we introduce the "non-nonstandard" background of the thesis. In the first part of the chapter we recall some known facts about the space $\bN$ of ultrafilters on $\N$ (for a comprehensive introduction to ultrafilters see $\cite{rif21}$). Then we recall the notion of partition regularity for a family of subsets of a given set $S$, and we show how this notion is strictly related to ultrafilters. In the last section we present some well-known results in Ramsey Theory (in particular Schur's, Folkman's, Van der Waerden's and Rado's Theorems), and we show proofs of these results done from the point of view of combinatorics and from the point of view of ultrafilters.

\section{Ultrafilters: a Short Introduction}

\subsection{Basic definitions and properties of ultrafilters}

A central notion in this thesis is that of ultrafilter on a set:

\begin{defn} If $S$ is a nonempty set, a filter is a collection $\mathcal{F}$ of subsets of $S$ such that:

\begin{itemize}
	\item $S\in\mathcal{F}$, $\emptyset\notin\mathcal{F}$;
	\item If $A,B\in\mathcal{F}$ then $A\cap B\in\mathcal{F}$;
	\item If $A\in\mathcal{F}$ and $A\subseteq B\subseteq S$ then $B\in\mathcal{F}$.
\end{itemize}

An \itshape{ultrafilter} $\U$ is a filter with the following additional property:

\begin{itemize}
  \item For every subset $A$ of $S$, if $A\notin \U$ then $A^{c}\in\U$. 
\end{itemize}

\end{defn}

An important fact is that every family that is closed by finite intersection is included in a filter:

\begin{defn} A nonempty family $\mathcal{G}$ of subsets of a set $S$ has the {\bfseries finite intersection property} if, for every natural number $n$, for every sets $G_{1},...,G_{n}$ in $\mathcal{G}$, the intersection 

\begin{center} $G=\bigcap_{i=1}^{n} G_{i}$ \end{center}

is nonempty. \end{defn}

By definition, every filter has the finite intersection property. Conversely, every family with the finite intersection property is included in a filter:

\begin{prop} Let $S$ be a set. For every nonempty family $\mathcal{G}$ of subsets of $S$ with the finite intersection property there is a filter $\mathcal{F}$ on $S$ that extends $\mathcal{G}$: $\mathcal{G}\subseteq\mathcal{F}$. \end{prop}

\begin{proof} Given the family $\mathcal{G}$ with the finite intersection property, consider

\begin{center} $\mathcal{F}=\{F\subseteq S\mid$ there is a natural number $n$ and elements $G_{1},..,G_{n}$ in $\mathcal{G}$ with $G_{1}\cap..\cap G_{n}\subseteq F\}$. \end{center}

{\bfseries Claim:} $\mathcal{F}$ is a filter that extends $\mathcal{G}$.\\

The proof of the claim is straightforward; that $\mathcal{G}$ has the finite intersection property is used to prove that the empty set is not an element of $\mathcal{F}$ and that $\mathcal{F}$ is closed by intersection.\\ \end{proof}

\begin{defn} A filter $\mathcal{F}$ on $S$ is {\bfseries maximal} if, for every filter $\mathcal{F}^{\prime}$ on $S$ such that $\mathcal{F}\subseteq\mathcal{F}^{\prime}$, $\mathcal{F}^{\prime}=\mathcal{F}$. \end{defn}

\begin{prop} Let $S$ be a set and $\mathcal{F}$ a filter on $S$. The following two conditions are equivalent:
\begin{enumerate}
	\item $\mathcal{F}$ is an ultrafilter;
	\item $\mathcal{F}$ is maximal.
\end{enumerate}
\end{prop}

\begin{proof} $(1)\Rightarrow (2)$ Let $\mathcal{F}$ be an ultrafilter. If $\mathcal{F}$ is properly contained in a filter $\mathcal{F}^{\prime}$, and $I\subseteq S$ is an element in $\mathcal{F}^{\prime}\setminus\mathcal{F}$, by definifition of ultrafilter the set $S\setminus I$ is in $\mathcal{F}$. Since $\mathcal{F}\subseteq\mathcal{F}^{\prime}$, in $\mathcal{F}^{\prime}$ there are $I$ and $S\setminus I$, and this is absurd as $I\cap (S\setminus I)=\emptyset$.\\
$(2)\Rightarrow (1)$ Let $\mathcal{F}$ be a maximal filter, and suppose that $\mathcal{F}$ is not an ultrafilter. Then there is a subset $I$ of $S$ such that $I, S\setminus I\notin\mathcal{F}$.\\

{\bfseries Claim:} $\mathcal{F}\cup \{I\}$ has the finite intersection property.\\

Since $\mathcal{F}$ has the finite intersection property, to prove the claim it is enough to prove that, given any element $F$ of $\mathcal{F}$, $F\cap I\neq\emptyset$. And this holds: in fact, if $F\cap I=\emptyset$ then $F\subseteq S\setminus I$ and, as $\mathcal{F}$ is closed under supersets, from this it follows that $S\setminus I\in\mathcal{F}$, absurd.\\
By Proposition 1.1.3 it follows that there is a filter $\mathcal{F}'$ on $S$ that includes $\mathcal{F}\cup I$, and this is in contrast with the maximality of $\mathcal{F}$. So $\mathcal{F}$ is an ultrafilter.\\ \end{proof}

A question that naturally arises is if every filter can be extended to an ultrafilter: the answer is affirmative. The following result is due to Alfred Tarski (see $\cite{rif60}$):

\begin{thm}[Tarski] Given a set $S$, every filter $\mathcal{F}$ on $S$ can be extended to an ultrafilter on $S$. \end{thm}

\begin{proof} Let $\mathcal{F}$ be a filter on $S$, and consider the set $Fil(\mathcal{F})_{S}$ of filters on $S$ extending $\mathcal{F}$, ordered by inclusion. We observe that every chain in $(Fil(\mathcal{F})_{S},\subseteq)$ has an upper bound: in fact, if $\langle \mathcal{F}_{i}\mid  i\in I\rangle$ is a chain of filters containing $\mathcal{F}$ (this means that $I$ is a linearly ordered set and, for every $i,j\in I$ with $i<j$, $\mathcal{F}_{i}\subseteq\mathcal{F}_{j}$), then 

\begin{center} $\mathcal{F}=\bigcup_{i\in I}\mathcal{F}_{i}$ \end{center} 

is a filter containing $\mathcal{F}$.\\
Since every chain in $(Fil(\mathcal{F})_{S},\subseteq)$ has an upper bound, by Zorn's Lemma it follows that there are maximal elements, and in Proposition 1.1.5 we proved that every maximal filter is an ultrafilter: this proves that there are ultrafilters on $S$ that extend $\mathcal{F}$.\\\end{proof}

For every set $S$, the simplest ultrafilters on $S$ are the principals:

\begin{defn} An ultrafilter $\U$ on a set $S$ is {\bfseries principal} if there is an element $s$ in $S$ such that 

\begin{center} $\U=\{I\subseteq S\mid s\in I\}$. \end{center}

An ultrafilter is {\bfseries nonprincipal} if it is not a principal ultrafilter.

\end{defn}

If $S$ is finite every ultrafilter on $S$ is principal: in fact, if $S=\{s_{1},...,s_{n}\}$, since $S=\{s_{1}\}\cup....\cup\{s_{n}\}$ then every ultrafilter on $S$ contains exactly one of the sets $\{s_{1}\}, \{s_{2}\},...,\{s_{n}\}$, so it is principal.\\
Talking of filters on $S$, the simplest is $\mathcal{F}=\{S\}$. An important filter, in case $S$ is infinite, is:

\begin{center} $Fr=\{I\subseteq S\mid |S\setminus I|<\aleph_{0}\}.$ \end{center}

The above filter is called Fréchet's filter. This is not an ultrafilter, as every infinite set $S$ contains many infinite subsets $I$ with both $I$ and $S\setminus I$ infinite.\\
Nevertheless, this filter is related with nonprincipal ultrafilters:

\begin{prop} Let $S$ be an infinite set, and $\U$ an ultrafilter on $S$. The following two conditions are equivalent:
\begin{enumerate}
	\item $\U$ is nonprincipal;
	\item $\U$ extends the Fréchet's filter on $S$.
\end{enumerate}
\end{prop}

\begin{proof} $(1)\Rightarrow (2)$ Suppose that the ultrafilter $\U$ is nonprincipal. Then, for every element $s$ in $S$, the set $S\setminus\{s\}$ is in $\U$. Since every element of the Fréchet's filter is obtained as a finite intersection of sets in the form $S\setminus\{s\}$, and $\U$ is closed under finite intersection, the Fréchet's filter is a subset of $\U$.\\
$(2)\Rightarrow (1)$ Suppose that $\U$ is a principal ultrafilter, and let $s$ be the element of $S$ such that $\{s\}\in\U$. Since $\U$ extends the Fréchet's filter, $S\setminus\{s\}\in\U$, so $\{s\}\cap (S\setminus\{s\})\in\U$, and this is absurd.\\\end{proof}

As a corollary, if $S$ is an infinite set then there are both principal and nonprincipal ultrafilters on $S$. Observe that, since every ultrafilter on a set $S$ is an element of $\wp(\wp(S))$ then, if $\kappa$ is the cardinality of $S$, there are at most $2^{2^{\kappa}}$ ultrafilters on $S$ (while there are exactly $\kappa$ principal ultrafilters).\\
In 1937, B. Posp\'i\v{s}il proved that $2^{2^{\kappa}}$ is exactly the number of ultrafilters on $S$ whenever $S$ is infinite (see $\cite{rif27})$. His result is even stronger: call an ultrafilter $\U$ on $S$ regular if, whenever $A\in\U$, $|A|=|S|$.

\begin{thm}[Posp\'i\v{s}il] If $S$ is an infinite set of cardinality $\kappa$, there are $2^{2^{\kappa}}$ regular ultrafilters on $S$. \end{thm}

\subsection{The space $\bN$}

From now forwards, except when explicitally stated otherwise, we concentrate on the case $S=\N$. As a corollary of Posp\'i\v{s}il's Theorem, in the space of ultrafilters on $\N$ there are $\aleph_{0}$ principal and $2^{2^{\aleph_{0}}}$ nonprincipal ultrafilters: our aim in this section is to present some important properties of this space.

\begin{defn} The {\bfseries set of ultrafilters on $\N$} $($denoted by $\bN)$ is the set 

\begin{center} $\bN=\{\U\subseteq \wp(\wp(\N))\mid \U$ is an ultrafilter$\}$. \end{center}

$\bN$ is endowed with the {\bfseries Stone topology}, that is the topology generated by the family of open sets $\{\Theta_{A}\}_{A\in\wp(\N)}$, where

\begin{center} $\Theta_{A}=\{\U\in\bN\mid A\in\U\}$. \end{center}

\end{defn}

Some observations: first of all since, for every subset $A$ of $\N$,

\begin{center} $\Theta_{A}^{c}=\Theta_{A^{c}}$ \end{center}

(because, for every ultrafilter $\U$, exactly one between $A$ and $A^{c}$ is in $\U$), every element $\Theta_{A}$ of the topological base is both closed and open: this entails that the space $\bN$ is totally disconnected.\\
The second observation is that, via the identification of every natural number $n$ with the principal ultrafilter $\mathfrak{U}_{n}$,

\begin{center}$(\dagger)$  $n\leftrightarrow \mathfrak{U}_{n}=\{A\in\N\mid n\in A\}$, \end{center}

$\N$ can be identified with a subset of $\bN$ and, most importantly, via this identification:

\begin{prop} $\N$ is a dense subset of $\bN$. \end{prop}

\begin{proof} Let $A$ be a nonempty subset of $\N$, and consider the base open set $\Theta_{A}$. If $n$ is an element of $A$, the principal ultrafilter $\mathfrak{U}_{n}$ is in $\Theta_{A}$, as $A$ is in $\mathfrak{U}_{n}$. This proves that in every base open set there is a principal ultrafilter, so the set of principal ultrafilters, that via the identification $(\dagger)$ is $\N$, is a dense subset of $\bN$.\\ \end{proof}

The Stone topology has the following two important features:

\begin{prop} $\bN$, endowed with the Stone topology, is a compact Hausdorff space. \end{prop}

\begin{proof} $\bN$ is compact: we use this equivalent formulation for compactness: a topological space is compact if and only if for every family $\mathcal{F}$ of closed subsets with the finite intersection property the intersection $\bigcap_{F\in\mathcal{F}} F$ is nonemtpy.\\
Let 

\begin{center} $\mathcal{F}=\{\Theta_{A_{i}}\mid A_{i}\subseteq\N, i\in I\}$ \end{center}

be a family of (base) closed subsets of $\bN$, and assume that $\mathcal{F}$ has the finite intersection property. From this it follows that 

\begin{center} $\mathcal{G}=\{A_{i}\mid i\in I\}$ \end{center}

has the finite intersection property. In fact, let $A_{i_{1}},...,A_{i_{k}}$ be sets in $\mathcal{G}$ and consider $A_{i_{1}}\cap...\cap A_{i_{k}}$. Since $\Theta_{A_{i_{1}}}\cap...\cap\Theta_{A_{i_{k}}}\neq\emptyset$, there is an ultrafilter $\U$ in this intersection; in particular, for every index $1\leq j\leq k$, since $A_{i_{j}}\in\U$, the intersection $\bigcap_{j=1}^{k} A_{i_{j}}\in \U$; this entails that $\bigcap_{j=1}^{k} A_{i_{j}}\neq \emptyset$.\\
As $\mathcal{G}$ has the finite intersection property, in can be extended to a filter, so by Theorem 1.1.6 it follows that there is an ultrafilter $\U$ on $\N$ such that $\mathcal{G}\subseteq\U$. Observe that, for every index $i$, $A_{i}\in\mathcal{G}\subseteq\U$, so $\U\in\Theta_{A_{i}}$ for every closed set $\Theta_{A_{i}}$ in $\mathcal{F}$. In particular, $\U\in\bigcap\mathcal{F}$.\\
$\bN$ is an Hausdorff space: Let $\U\neq\V$ be ultrafilters on $\N$. As $\U\neq\V$, there is a set $A$ in $\wp(\N)$ such that $A\in\U$ and $A^{c}\in\V$. So $\U\in\Theta_{A}$ and $\V\in\Theta_{A^{c}}$, and these are two disjoint open sets in the Stone topology. \\ \end{proof}

Observe that, as $\N=\bigcup_{n\in\N} \Theta_{\{n\}}$, $\N$ is an open subset of $\bN$, so $\bN\setminus\N$ is closed (and, since it is a closed subset of a compact space, it is also compact).\\
Another key aspect of $\bN$ is the following universal property:

\begin{thm} Every continuous map $f:\N\rightarrow K$, where $K$ is a compact Hausdorff space, admits exactly one continuous extension $\overline{f}:\bN\rightarrow K$. \end{thm}

\begin{proof} Given a function $f:\N\rightarrow K$, where $K$ is a topological compact Hausdorff space, and an ultrafilter $\U$ in $\bN$, we define the $\U$ limit of $f$ as

\begin{center} $\U-\lim_{n\in\N} f(n)= k$ if and only if for every neighborhood $I$ of $k$ the set $f^{-1}(I)=\{n\in\N\mid f(n)\in I\}\in\U$. \end{center}

This definition is well-posed: this limit always exists and it is unique. The existence follows by the compactness of $K$: by converse, suppose that the $\U$ limit of $f$ does not exist. Then, for every element $k$ in $K$ there is an open neighborhood $O_{k}$ of $k$ such that $f^{-1}(O_{k})\notin\U$. $\{O_{k}\}_{k\in K}$ is an open cover of $K$, which is compact, so we can extract a finite subcover $K=\bigcup_{i=1}^{n} O_{k_{i}}$. But 

\begin{center} $\N=f^{-1}(K)=\bigcup_{i=1}^{n} f^{-1}(O_{k_{i}})$ \end{center} 

so, since $\U$ is an ultrafilter, for some index $i\leq n$ the set $f^{-1}(k_{i})$ is in $\U$, and this is absurd: the $\U$-limit of $f$ always exists.\\
The unicity of the $\U$-limit follows since $K$ is Hausdorff because, if $k_{1}\neq k_{2}$ are two different $\U$ limits of $f$, there are neighboroods $I_{1}$ of $k_{1}$, $I_{2}$ of $k_{2}$ with $I_{1}\cap I_{2}=\emptyset$, so $f^{-1}(I_{1})\cap f^{-1}(I_{2})=\emptyset$ while, as $f^{-1}(I_{1})$ and $f^{-1}(I_{2})$ are sets in $\U$, their intersection should be nonempty.\\
This ensures that $\overline{f}:\bN\rightarrow K$, defined as

\begin{center} $\overline{f}(\U)=\U-\lim_{n\in\N} f(n)$ for every $\U\in\bN$ \end{center}

is a function in $K^{\bN}$.\\

{\bfseries Claim:} $\overline{f}$ is the unique continuous extenction of $f$ to $\bN$.\\

1) $\overline{f}$ is an extension of $f$: via the identification of every natural number $m$ with the principal ultrafilter $\mathfrak{U}_{m}$, 
\begin{center} $\overline{f}(m)=\overline{f}(\mathfrak{U}_{m})=\mathfrak{U}_{m}-\lim_{n\in\N} f(n)$,\end{center}

and $\mathfrak{U}_{m}-\lim_{n\in\N} f(n)= f(m)$: by definition, $k=\mathfrak{U}_{m}-\lim_{n\in\N} f(n)$ if and only if for every neighborood $I$ of $k$, $f^{-1}(I)\in\mathfrak{U}_{m}$ if and only if for every neighborood $I$ of $k$, $m\in f^{-1}(I)$, if and only if for every neighborood $I$ of $k$, $f(m)\in I$, and as $K$ is Hausdorff it follows that $k=f(m)$.\\
2) $\overline{f}$ is continuous: if $I$ is any open set in $K$,

\begin{center} $\overline{f}^{-1}(I)=\{\U\in \bN\mid \overline{f}(\U)\in I\}=$\\\vspace{0.3cm}$=\{\U\in\bN\mid \{n\in\N\mid f(n)\in I\}\in\U\}=\Theta_{\{n\in\N\mid f(n)\in I\}}$,\end{center}

and $\Theta_{\{n\in\N\mid f(n)\in I\}}$ is an open set.\\
3) $\overline{f}$ is the unique continuous extension of $f$: if $g$ is any other continuous extension of $f$, since $\overline{f}(n)=g(n)$ for every natural number $n$, and $\N$ is a dense subset of $\bN$, which is Hausdorff, then $\overline{f}(\U)=g(\U)$ for every ultrafilter $\U\in\bN$, so $\overline{f}=g$. \\ \end{proof}

$\bN$, endowed with the Stone topology, is a compact, Hausdorff space that satisfies the universal property expressed in Theorem 1.1.13: in the literature, this is called Stone-\v{C}ech compactification:

\begin{defn} The {\bfseries Stone-\v{C}ech compactification} of a discrete space $X$ is a compact Hausdorff space $\beta X$ in which the original space forms a dense subset, such that any continuous function from $X$ to a compact Hausdorff space $K$ has a unique continuous extension to $\beta X$.\end{defn}

It can be proved that the Stone-\v{C}ech compactification of a discrete space is unique up to homeomorphism; so $\bN$ is (up to homeomorphism) the Stone-\v{C}ech compactification of $\N$.\\
Sometimes, instead of ultrafilters on $\N$ we  need ultrafilters on $\N^{k}$:

\begin{defn} Let $k$ be a natural number $\geq 1$. The {\bfseries set of ultrafilters on $\N^{k}$} $($denoted by $\beta(\N^{k}))$ is the set 

\begin{center} $\beta(\N^{k})=\{\U\subseteq \wp(\wp(\N^{k}))\mid \U$ is an ultrafilter$\}$. \end{center}

$\beta(\N^{k})$ is endowed with the {\bfseries Stone topology}, that is the topology generated by the family of open sets $\{\Theta_{A}\}_{A\in\wp(\N^{k})}$, where

\begin{center} $\Theta_{A}=\{\U\in\beta(\N^{k})\mid A\in\U\}$. \end{center}

\end{defn}

A particularly important subset of $\beta(\N^{k})$ is the set of ultrafilters that are tensor products of elements in $\bN$:

\begin{defn} Given ultrafilters $\U_{1},...,\U_{k}$ on $\N$, the {\bfseries tensor product} of $\U_{1},...,\U_{k}$ $($denoted by $\U_{1}\otimes\U_{1}\otimes...\otimes\U_{k})$ is the ultrafilter on $\N^{k}$ defined by this condition: for every subset $A$ of $\N^{k}$,

\begin{center} $A\in\U_{1}\otimes\U_{2}\otimes...\otimes\U_{k}\Leftrightarrow$\\\vspace{0.3cm}$\Leftrightarrow\{n_{1}\in \N\mid\{n_{2}\in\N\mid...\{n_{k}\in\N\mid (n_{1},...,n_{k})\in A\}\in \U_{k}\}...\}\in\U_{2}\}\in\U_{1}$. \end{center}

\end{defn}

We present some basic properties of tensor products in the form $\U\otimes\V$, where $\U,\V$ are ultrafilters on $\N$. In this proposition, with $\Delta^{+}$ we denote this subset of $\N^{2}$:

\begin{center} $\Delta^{+}=\{(n,m)\in\N^{2}\mid n<m\}$, \end{center}

which is also called the upper diagonal of $\N^{2}$, by obvious geometrical considerations.

\begin{prop} For every ultrafilters $\U,\V$ in $\bN$, $\W$ in $\beta\N^{2}$, the following properties hold:
\begin{enumerate}
	\item if $\V$ is nonprincipal then $\Delta^{+}\in\U\otimes\V$;
	\item if $\U$ is nonprincipal and $A$ is an set in $\U\otimes\U$ then there is an infinite subset $B$ of $\N$ such that $\{(b_{1},b_{2})\mid b_{1}<b_{2}, b_{1}, b_{2}\in B\}\subseteq A$;
	\item $\W$ is the principal ultrafilter on $(n,m)$ if and only if $\W=\mathfrak{U}_{n}\otimes\mathfrak{U}_{m}$;
	\item if $\U\neq\V$ then $\U\otimes\V\neq\V\otimes\U$.
\end{enumerate}
\end{prop}

\begin{proof} 1) By definition, $\Delta^{+}\in\U\otimes\V$ if and only if $\{n\in\N\mid\{m\in\N\mid (n,m)\in\Delta^{+}\}\in\V\}\in\U$.\\
We observe that, for every natural number $n$, 

\begin{center}$\{m\in\N\mid (n,m)\in\Delta^{+}\}=\{m\in\N\mid n<m\}\in\V$\end{center}

since this set is cofinite and $\V$ is nonprincipal.\\ 
This entails that 

\begin{center} $\{n\in\N\mid \{m\in\N\mid (n,m)\in\Delta^{+}\}\in\V\}=\N$,\end{center}

and $\N$ in an element of every ultrafilter $\U$ in $\bN$. This proves that $\Delta^{+}\in\U\otimes\V$.\\
2) Given a set $A$ in $\U\otimes\U$, consider

\begin{center} $\pi_{1}(A)=\{n\in\N\mid \{m\in\N\mid (n,m)\in A\}\in\U\}$ \end{center} 

and, for every $n\in\pi_{1}(A)$, let $A_{n}$ be the set 

\begin{center} $A_{n}=\{m\in\N\mid (n,m)\in A\}$; \end{center}

by definition, for every natural number $n$ in $\pi_{1}(A)$, $A_{n}\in\U$.\\
We inductively construct the set $B$: pose 

\begin{center} $b_{0}=\min \pi_{1}(A)$, $B_{0}=\pi_{1}(A)\cap A_{b_{0}}$. \end{center}

For $n=m+1$, pose 

\begin{center} $b_{m+1}=\min B_{m}$, $B_{m+1}=B_{m}\cap A_{b_{m+1}}$.\end{center}

Observe that, for every natural number $n$, $B_{n}$ is a set in $\U$: by induction, $B_{0}\in\U$ since both $\pi_{1}(A)$ and $A_{b_{0}}$ are in $\U$, and $B_{m+1}\in\U$ since $B_{n}\in\U$ by inductive hypothesis and $A_{b_{n+1}}$ is in $\U$ by construction; this entails, in particular, that every set $B_{n}$ is nonempty, so it is always possible to define $b_{n}$. Also, for every $y\in B_{n}$, $(b_{n},y)\in A$ since $\{b_{n}\}\times B_{n}\subseteq A_{n}$.\\
Consider

\begin{center} $B=\{b_{n}\mid n\in\N\}$. \end{center}

If $b_{i}<b_{j}$ are two elements of $B$, by construction $i<j$, so in particular $b_{j}\in B_{i}$ and $(b_{i},b_{j})\in A$.\\
3) $A\in\mathfrak{U}_{n}\otimes\mathfrak{U}_{m}$ if and only if $\{a\in\N\mid \{b\in\N\mid (a,b)\in A\}\in\mathfrak{U}_{m}\}\in\mathfrak{U}_{n}$ if and only if $\{a\in\N\mid (a,m)\in A\}\in\mathfrak{U}_{n}$ if and only if $(n,m)\in A$ if and only if $A\in \mathfrak{U}_{(n,m)}$.\\
4) Let $A\subseteq\N$ be a set in $\U$ such that $A^{c}\in\V$. Then $(A\times A^{c})\in\U\otimes\V$ and $(A^{c}\times A)\in\V\otimes\U$ and, since $(A\times A^{c})\cap (A^{c}\times A)=\emptyset$, it follows that $\U\otimes\V\neq\V\otimes\U$.\\ \end{proof}

\subsection{Semigroup structure and idempotents of $\bN$}

$\bN$, similarly to $\N$, is endowed with two operations: a sum and a product.

\begin{defn} Let $\U,\V$ be ultrafilters $\bN$. The {\bfseries sum} of $\U$ and $\V$ $($notation $\U\oplus\V)$ is the ultrafilter:

\begin{center} $\U\oplus\V=\{A\subseteq\N\mid\{n\in\N\mid\{m\in\N\mid m+n\in A\}\in\V\}\in\U\}$; \end{center}

and the {\bfseries product} of $\U$ and $\V$ $($notation $\U\odot\V)$ is the ultrafilter:

\begin{center} $\U\odot\V=\{A\subseteq\N\mid\{n\in\N\mid\{m\in\N\mid m\cdot n\in A\}\in\V\}\in\U\}$. \end{center}

\end{defn}

The operations of sum and product can also be seen as the extensions to $\beta\N^{2}$ of the functions $S: \N^{2}\rightarrow \bN$ such that, for every pair of natural numbers $(n,m)$ in $\N^{2}$, 

\begin{center} $S(n,m)=\U_{n+m}$ \end{center} 

and $P:\N^{2}\rightarrow\bN$ such that, for every pair of natural numbers $(n,m)$ in $\N^{2}$,

\begin{center}  $P(n,m)=\U_{n\cdot m}$ \end{center}

restricted to the subspace of $\bN^{2}$ consisting of tensor products:

\begin{center} $\U\oplus\V=\overline{S}(\U\otimes\V)$; $\U\odot\V=\overline{P}(\U\otimes\V)$. \end{center}

A fact that has many important consequences both from an algebraical, a topological and a combinatorial point of view is that ($\bN,\oplus$) and ($\bN,\odot$), endowed with the Stone topology, are right topological semigroups:

\begin{defn} A {\bfseries right topological semigroup} is a pair $(G;\star )$ where 
\begin{enumerate}
	\item $G$ is topological space;
	\item $\star$ is a binary associative operation on $G$;
	\item for every element $x$ in $G$, the function $\star_{x}:G\rightarrow G$ that maps every element $y$ of $G$ in $y\star x$ is continuous.
\end{enumerate}
\end{defn}

Condition number three is usually rephrased saying that the operation $\star$ is right continuous.

\begin{prop} $(\bN,\oplus)$ and $(\bN,\odot)$ are right topological semigroups. \end{prop}

\begin{proof} We prove that $(\bN,\oplus)$ is a right topological semigroup; the proof for $(\bN,\odot)$ is analogue.\\
$\bN$ is a topological space, as it is endowed with the Stone topology.\\
That the operation $\oplus$ is associative follows by this chain of equivalences: for every subset $A$ of $\N$,

\begin{center}

$A\in (\U\oplus\V)\oplus\W\Leftrightarrow\{n\in\N\mid\{m\in\N\mid n+m\in A\}\in\W\}\in (\U\oplus\V)\Leftrightarrow$\\\vspace{0.3cm}$\{a\in\N\mid\{b\in\N\mid\{m\in\N\mid a+b+m\in A\}\in \W\}\in\V\}\in\U\Leftrightarrow$\\\vspace{0.3cm}$\{a\in\N\mid\{x\in\N\mid a+x\in A\}\in (\V\oplus\W)\}\in\U\Leftrightarrow A\in\U\oplus(\V\oplus\W)$.

\end{center}

The operation $\oplus$ is right continuous: let $\U$ be an ultrafilter in $\bN$, and denote with $\varphi_{\U}$ the function such that, for every ultrafilter $\V$ in $\bN$, 

\begin{center} $\varphi_{\U}(\V)=\V\oplus\U$. \end{center}

To prove that $\varphi_{\U}$ is continuous we observe that, for every subset $A$ of $\N$:

\begin{center} $\varphi_{\U}^{-1}(\Theta_{A})=\{\V\in\bN\mid A\in \V\oplus\U\}=$\\\vspace{0.3cm}$=\{\V\in\bN\mid \{n\in\N\mid A-n \in\U\}\in\V\}=\Theta_{B}$, \end{center}

where $B=\{n\in\N\mid A-n\in\U\}$ and $A-n=\{m\in\N\mid n+m\in A\}$. This proves that $\varphi_{\U}$ is continuous.\\ \end{proof}

\begin{defn} Given a right topological semigroup $(G,\star)$, an {\bfseries idempotent} of $(G\star)$ is an element $x$ such that $x\star x=x$. \end{defn}

The key result, when talking about idempotent in right topological semigroups, is due to Robert Ellis (see $\cite{rif17}$):

\begin{thm}[Ellis] Every compact right topological semigroup $(G,\star)$ contains an idempotent element. \end{thm}

\begin{proof} First of all, consider the set 

\begin{center}$\mathcal{C}=\{C\subseteq G\mid$ $C$ is a closed subset of $G$ and $C\star C\subseteq C\}$. \end{center}

$\mathcal{C}$ is nonempty (since $G\in\mathcal{C}$). We can order this set by reverse inclusion: 

\begin{center} $C\leq C'$ if and only if $C'\subseteq C$. \end{center}

Given a chain $\{C_{i}\}$ of elements in $\mathcal{C}$, an upper bound for this chain is $\bigcap_{i\in I} C_{i}$. By Zorn's Lemma it follows that there are maximal elements in $(\mathcal{C},\leq)$ so, as $\leq$ is the reverse inclusion, there are minimal elements in $(\mathcal{C},\subseteq)$.\\
Let $C$ be a minimal element in $(\mathcal{C},\subseteq)$, $x$ an element in $C$, and consider 

\begin{center} $\varphi_{x}(C)=\{c\star x\mid c\in C\}$.\end{center}

Since $\varphi_{x}$ is continuos and $C$ is closed (so, compact), $\varphi_{x}(C)$ is a compact subset of $G$. So it is closed, and $\varphi_{x}(C)\star \varphi_{x}(C)\subseteq C$. By minimality of $C$ it follows that $\varphi_{x}(C)=C$.\\
Then $\varphi_{x}^{-1}(x)\neq\emptyset$ is a closed subset of $C$ and, most importantly, $\varphi_{x}^{-1}(x)\star \varphi_{x}^{-1}(x)\subseteq\varphi_{x}^{-1}(x)$ since, given any $y,z$ in $\varphi_{x}^{-1}(x)$, $x\star (y\star z)=(x\star y)\star z=x\star z=x$. So $\varphi_{x}^{-1}(x)=C$ (again by minimality); in particular

\begin{center} $x\star x=x$: \end{center}

$x$ is an idempotent element in $(G,\star)$.\\ \end{proof}

Applying Ellis's Theorem to $\bN$, we get a very important corollary:

\begin{cor}[Galvin] There are nonprincipal idempotents both in $(\bN,\oplus)$ and in $(\bN,\odot)$. \end{cor}

\begin{proof} We have just to apply Ellis Theorem to the compact Hausdorff semigroups $(\bN\setminus\N,\oplus)$ and $(\bN\setminus\N,\odot)$. \\ \end{proof}

We just stretch that, for every ultrafilter $\U$ in $\bN$, 

\begin{center} if $\U\oplus\U = \U\odot\U$ then $\U=\mathfrak{U}_{0}$ or $\U=\mathfrak{U}_{2}$, \end{center}

so the only ultrafilter which is both additively and multiplicatively idempotent is the principal ultrafilter $\mathfrak{U}_{0}$. For the principal ultrafilters, the assertion is trivial; for the nonprincipal, it is proved in $\cite[\mbox{Theorem 10.25}]{rif39}$.

\subsection{$K(\bN,\oplus), K(\bN,\odot)$}

In this section we present the concepts of right, left, bilateral and minimal ideal in $(\bN,\oplus)$ (resp. $(\bN,\odot)$). These are important concepts for the applications of the algebra of $\bN$ to combinatorics.

\begin{defn} A subset $I$ of $\bN$ is a {\bfseries right ideal} $($resp. left ideal$)$ in $(\bN,\oplus)$ if, for every ultrafilters $\U$ in $I$, $\V$ in $\bN$, $\U\oplus\V$ $($resp. $\V\oplus\U)$ is in $I$.\\
$I$ is a {\bfseries bilateral ideal} in $(\bN,\oplus)$ if it is both a left and a right ideal.\\
$I$ is a {\bfseries minimal right ideal} $($resp. minimal left ideal$)$ in $(\bN,\oplus)$ if, whenever $J\subseteq I$ is a right $($resp. left$)$ ideal in $(\bN,\oplus)$, $J=I$.\\
The notions of right, left, bilateral and minimal ideal in $(\bN,\odot)$ are defined similarly. \end{defn}

In this context, a very important result on compact topological semigroups is the following theorem:

\begin{thm} In every compact right topological semigroup $(G,\star )$ there is a smallest bilateral ideal $K(G,\star)$, which can be described as the union of all the minimal left ideals or, also, as the union of all the minimal right ideals of $(G,\star)$. \end{thm}

A proof of this result can be found in $\cite{rif21}$. As a consequence, since $(\bN,\oplus)$ and $(\bN,\odot)$ are compact right topological semigroups, we have:

\begin{cor} $(\bN,\oplus)$ $($resp. $(\bN,\odot))$ has a minimal bilateral ideal $K(\bN,\oplus)$ $($resp. $K(\bN,\odot))$ which is the union of all its minimal left ideals and, also, the union of all its minimal right ideals:

\begin{center} $K(\bN,\oplus)=\bigcup\{ R\mid R$ minimal right ideal in $(\bN,\oplus)\}=\bigcup\{L\mid L$ minimal left ideal in $(\bN,\oplus)\}$;\\ 
$K(\bN,\odot)=\bigcup\{ R\mid R$ minimal right ideal in $(\bN,\odot)\}=\bigcup\{L\mid L$ minimal left ideal in $(\bN,\odot)\}$. \end{center}

\end{cor}

Observe that, if $I_{1}$ and $I_{2}$ are two minimal right (resp. left) ideals in $(\bN,\oplus)$, $I_{1}=I_{2}$ or $I=I_{1}\cap I_{2}=\emptyset$, otherwise $I$ would be a right (resp. left) ideal strictly included in $I_{1}$ and $I_{2}$, which is absurd.\\

\begin{prop} Every right and every left topological ideal in $(\bN,\oplus)$ $($respect. $(\bN,\odot))$ contains an idempotent. \end{prop}

\begin{proof} This result is proven in $\cite[\mbox{Corollary 2.6 and Theorem 2.7}]{rif21}$.\\ \end{proof}

In $\cite{rif36}$, Yevhen Zelenyuk proved that there is plenty of minimal right ideals whenever we consider the Stone-\v{C}ech compactification of infinite discrete abelian groups:

\begin{thm}[Zelenyuk] Let $G$ be an infinite discrete abelian group with $|G|=\kappa$. Then $\beta G$ contains $2^{2^{\kappa}}$ minimal right ideals.\end{thm}

As a consequence, both in $(\bN,\oplus)$ and $(\bN,\odot)$ there are $2^{2^{\aleph_{0}}}$ idempotent ultrafilters.

\subsection{Piecewise syndetic sets and $K(\bN,\oplus)$}

The ultrafilters in the closure of $K(\bN,\oplus)$ have an interesting characterization in terms of a notion called "piecewise syndeticity":

\begin{defn} A subset $A$ of $\N$ is {\bfseries thick} if it contains arbitrarily long intervals; it is {\bfseries piecewise syndetic} if there is a natural number $n$ such that 

\begin{center} $A-[0,n]=\{m\in\N\mid\exists i\leq n$ with $m+i\in A$\}\end{center}

is thick.
\end{defn}

By definition every thick set is piecewise syndetic. In this section we want to prove a well-known result: there is a close connection between piecewise syndetic sets and $\overline{K(\bN,\oplus)}$.

\begin{lem} A subset $A$ of $\N$ is thick if and only if there is a minimal left ideal $L$ included in $\Theta_{A}$. \end{lem}

\begin{proof} This result is $\cite[\mbox{Theorem 2.9(c)}]{rif40}$. \end{proof}

\begin{thm} A set $A$ is piecewise syndetic if and only if $K(\bN,\oplus)\cap \Theta_{A}\neq\emptyset$. \end{thm}

\begin{proof} Suppose that $A$ is piecewise syndetic, and let $n$ be a natural number such that $T=A-[0,n]=\bigcup_{i\leq n}(A-i)$ is thick.\\
Let $L$ be a minimal left ideal included in $\Theta_{T}$ and $\U$ an ultrafilter in $L$. By construction, $\U\in\Theta_{T}$, so 

\begin{center} $T=\bigcup_{i\leq n}(A-i)\in \U$. \end{center}

In particular, there is an index $i\leq n$ such that $(A-i)\in\U$. This means that $A\in\U\oplus i=i\oplus\U$, and $i\oplus\U\in L$, so $\Theta_{A}\cap L\neq\emptyset$, and the thesis follows since $L\subseteq K(\bN,\oplus)$.\\
Conversely, let $\U$ be an ultrafilter in $K(\bN,\oplus)$ with $A\in\U$. Let $L=\bN\oplus \U$ be the minimal left ideal containing $\U$. Pose

\begin{center} $T_{A}=\bigcup_{n\in\N}(A-n)$. \end{center}

{\bfseries Claim:} $L\subseteq\Theta_{T_{A}}$. \\

We prove the claim: let $\V$ be an element of $L$; $L=\bN\oplus\V=\overline{\{n\oplus \V\mid n\in\N\}}$, $\U\in L$ and $\Theta_{A}$ is a neighbourhood of $\U$ since $\U\in\Theta_{A}$: this entails that there is a natural number $n$ such that $A\in n\oplus\V$, so $\V\in\Theta_{A-n}$, and this proves that $L\subseteq\Theta_{T_{A}}$.\\
In particular

\begin{center} $\{\Theta_{A-n}\mid n\in\N\}$ \end{center}

is an open cover of $L$; but $L$ is compact (since $L$ is the image of $\bN$, which is compact, respect to the function $\V\rightarrow\V\oplus\U$, which is continuous), so there exists a natural number $n$ such that $\{\Theta_{A-i}\mid i\leq n\}$ covers $L$, and this is equivalent to say that, denoting with $T$ the set $\bigcup_{i\leq n}(A-i)$, $L\subseteq\bigcup_{i\leq n}\Theta_{A-i}=\Theta_{T}$.\\
By Lemma 1.1.33 it follows that $\bigcup_{i\leq n}(A-i)$ is a thick set, so $A$ is piecewise syndetic. \\\end{proof}

\begin{cor} An ultrafilter $\U$ is in $\overline{K(\bN,\oplus)}$ if and only if every element $A$ of $\U$ is piecewise syndetic. \end{cor}

\begin{proof} We use this property of Stone topology: for every subset $S$ of $\bN$, its topological closure satisfies this condition:

\begin{center} For every ultrafilter $\U$, $\U\in\overline{S}$ if and only if, for every $A\in\U$, $\Theta_{A}\cap S\neq\emptyset$. \end{center}

So, if $\U\in\overline{K(\bN,\oplus)}$ and $A\in\U$, there is an ultrafilter $\V$ in $K(\bN,\oplus)$ with $A\in\V$, so by Theorem 1.1.34 $A$ is piecewise syndetic. Conversely, if $\U$ is an ultrafilter such that every element $A$ of $\U$ is piecewise syndetic, since for every piecewise syndetic set $A$ there is some ultrafilter $\V_{A}$ in $K(\bN,\oplus)\cap \Theta_{A}$, it follows that $\U\in\overline{K(\bN,\oplus)}$. \\ \end{proof}

\subsection{$\U$-limits}

In this section we introduce the operation of limit in $\bN$, and we show that, for a subset of $\bN$, to be closed in the Stone topology is equivalent to be closed under $\U$-limits.

\begin{defn} Given a nonempty set $I$, a sequence $\mathcal{F}=\langle \U_{i}\mid i\in I\rangle$ of elements in $\bN$ and an ultrafilter $\V$ on $I$, the {\bfseries $\V$-limit of the sequence $\mathcal{F}$} $($notation $\V-\lim_{I}\U_{i})$ is the ultrafilter:

\begin{center} $\V-\lim_{I} \U_{i}=\{A\subseteq\N\mid \{i\in I\mid A\in \U_{i}\}\in \V\}$. \end{center}

\end{defn}

Let us verify that $\U=\V-\lim_{I}\U_{i}$ is actually an ultrafilter on $\N$. First of all, $\U$ is a filter: $\N$ is in $\U$ since the set of indexes $i$ such that $\N$ is in $\U_{i}$ is $I$, which is in $\V$; $\U$ is closed under intersection since, if $A\in\U$ and $B\in \U$, if $I_{A}=\{i\in I\mid A\in\U_{i}\}$ and $I_{B}=\{i\in I\mid B\in\U_{i}\}$, then both $I_{A}$ and $I_{B}$ are in $\V$, so $I_{A}\cap I_{B}$ is in $\V$ and $I_{A\cap B}=\{i\in I\mid A\cap B\in \U_{i}\}=I_{A}\cap I_{B}$.\\
$\U$ is an ultrafilter: for every subset $A$ of $\N$, for every index $i\in I$, $A\in\U_{i}$ or $A^{c}\in\U_{i}$, so $I=I_{A}\cup I_{A^{c}}$, and this entails that exactly one between $I_{A}$ and $I_{A^{c}}$ is in $\V$, so exactly one between $A$ and $A^{c}$ is in $\U$.\\
The operation of limit-ultrafilter is important from a topological point of view. In fact, the closed subsets of $\bN$ can be characterized in terms of $\U$-limits:

\begin{thm} Let $X$ be a subset of $\bN$. The following two conditions are equivalent:
\begin{enumerate}
	\item $X$ is closed in the Stone topology;
	\item for every set $I$, for every sequence $\langle\U_{i}\mid i\in I\rangle$ of elements in $X$, for every ultrafilter $\V$ on $I$, the ultrafilter $\U=\V-\lim_{i\in I} \U_{i}$ is in $X$.
\end{enumerate} \end{thm}

\begin{proof} $(1)\Rightarrow (2)$ If $X$ is a closed base set $\Theta_{A}$ then for every set $I$, for every sequence of ultrafilters $\langle \U_{i}\mid i\in I\rangle$ on $\N$, for every ultrafilter $\V$ on $I$, $\V-\lim_{I} \U_{i}$ is in $\Theta_{A}$, i.e. $A\in \V-\lim_{I} \U_{i}$: indeed, by definition, $A\in \V-\lim_{I} \U_{i}$ if and only if the set $I_{A}=\{i\in I\mid A\in \U_{i}\}$ is in $\V$ and, as $X=\Theta_{A}$, $I_{A}=I$, which is in $\V$.\\
If $X$ is an intersection of closed base sets, $X=\bigcap_{j}\Theta_{A_{j}}$, the conclusion follows immediately from the previous case.\\
$(2)\Rightarrow(1)$ Suppose, conversely, that $X^{c}$ is not an open set. Then there is an ultrafilter $\U$ in $X^{c}$ such that, for every base open set $\Theta_{A}$ that includes $\U$, there is an ultrafilter $\U_{A}$ in $X$ with $A\in \U_{A}$.\\
We use this fact to produce a sequence $\langle \U_{B}\mid B\in\wp(\N)\rangle$ of ultrafilters indexed by $\wp(\N)$ and an ultrafilter $\V$ on $\wp(\N)$ such that $\U$ is the $\V-\lim$ of this family, and this concludes, since by hypothesis this entails that $\U\in X$, which is a contradiction.\\
Step 1: For every set $B$ in $\wp(\N)$, pick $\U_{B}\in X$ such that if $B\in\U$, then $B\in\U_{B}$ (if $B$ is not in $\U$, pick $\U_{B}$ randomly).\\
Step 2: For every set $A\in\U$, define

\begin{center} $\Gamma_{A}=\{B\in\wp(\N)\mid B\subseteq A$ and $B\in\U$\}. \end{center}

Observe that the family $\{\Gamma_{A}\}_{A\in\U}$ has the finite intersection property, as $\Gamma_{A_{1}}\cap \Gamma_{A_{2}}=\Gamma_{A_{1}\cap A_{2}}$. Let $\V$ be an ultrafilter on $\wp(\N)$ with $\{\Gamma_{A}\}_{A\in\U}\subseteq \V$.\\
We claim that 

\begin{center} $\U=\V-\lim_{B\in\wp(\N)}\U_{B}$.\end{center}

In fact, if $A$ is any element in $\U$, the set $\Omega_{A}=\{B\in\wp(\N)\mid A\in \U_{B}\}$ includes $\Gamma_{A}$ since, by definition, if $B\in\Gamma_{A}$ then $B\in \U$ (so $B\in\U_{B}$) and $B\subseteq A$ (so $A\in \U_{B}$).\\
But $\Gamma_{A}$ is an element of $\V$, so also $\Omega_{A}$ is in $\V$, and this entails that $A\in\V-\lim \U_{B}$.\\
Since $\U$ and $\V-\lim \U_{B}$ are ultrafilters, this proves that $\U=\V-\lim_{B\in\wp(\N)}\U_{B}$.\\\end{proof}

\section{Partition Regularity}

We introduce an argument that is strictly related to ultrafilters: the partition regularity of a family of subsets of a set $S$.

\begin{defn} Let $S$ be a set, $n$ a natural number, and $\{A_{1},....,A_{n}\}$ a subset of $\wp(\wp(\N))$. $\{A_{1},...,A_{n}\}$ is a {\bfseries partition} of $S$ if the following three conditions are satisfied:
\begin{enumerate}
	\item $S=\bigcup_{i=1}^{n} A_{i}$;
	\item $A_{i}\neq\emptyset$ for every index $i\leq n$;
	\item $A_{i}\cap A_{j}=\emptyset$ for every indexes $i\neq j$, $i,j\leq n$.
\end{enumerate}
\end{defn}

{\bfseries Convention:} Whenever, given a set $S$, we write 

\begin{center} $S=A_{1}\cup...\cup A_{n}$ \end{center}

it is intended that $\{A_{1},...,A_{n}\}$ is a partition of $S$.

\begin{defn}[] Let $\mathcal{F}$ be a family, closed under superset, of nonempty subsets of a set $S$. $\mathcal{F}$ is {\bfseries weakly partition regular} if, whenever $S=A_{1}\cup...\cup A_{n}$, there exists an index $i\leq n$ such that $A_{i}\in \mathcal{F}$.\\
$\mathcal{F}$ is {\bfseries strongly partition regular} if, for every set $A$ in $\mathcal{F}$, if $A=A_{1}\cup...\cup A_{n}$ then there exists an index $i\leq n$ such that $A_{i}\in\mathcal{F}$.
\end{defn}

Trivially, every strongly partition regular family of sets is also weakly partition regular.\\ 
Partition regular families of subsets of a set $S$ are related with ultrafilters on $S$:

\begin{thm}[] Let $S$ be a set, and $\mathcal{F}$ a family closed under supersets of nonempty subsets of $S$. Then the following equivalences hold:
\begin{enumerate}
	\item $\mathcal{F}$ is weakly partition regular if and only if there exists an ultrafilter $\U$ on $S$ such that $\U\subseteq\mathcal{F}$;
	\item $\mathcal{F}$ is strongly partition regular if and only if $\mathcal{F}$ is an union of ultrafilters.
\end{enumerate}

\end{thm}

\begin{proof} 1) Suppose that $\mathcal{F}$ is weakly partition regular and, for every subset $A$ of $S$, consider the partition $S=A\cup A^{c}$. Since $\mathcal{F}$ is weakly partition regular, at least one between $A$ and $A^{c}$ is in $\mathcal{F}$.\\ Consider the subfamily $\mathcal{G}$ of $\mathcal{F}$ such that

\begin{center} $\mathcal{G}=\{A\in\mathcal{F}\setminus\{\emptyset\}\mid A^{c}\notin\mathcal{F}\setminus\{\emptyset\}\}$.\end{center}

{\bfseries Claim:} $\mathcal{G}$ is a filter.\\

In fact, $\mathcal{G}$ does not contain the empty set and it contains $S$; it is closed under superset because, if $A^{c}\notin\mathcal{F}$, then no subset of $A^{c}$ is in $\mathcal{F}$ since this family is closed under superset; it is closed under intersection because, if $A,B\in\mathcal{G}$, then 

\begin{center} $S=(A\cap B)\cup A^{c}\cup (B^{c}\setminus A^{c})$,\end{center}

and $B^{c}\setminus A^{c}$ and $A^{c}$ are not elements of $\mathcal{F}$, so $A\cap B\in \mathcal{F}$.\\
We can extend the filter $\mathcal{G}$ to an ultrafilter $\U$, and $\U\subseteq\mathcal{F}$: extending $\mathcal{G}$, we include in $\U$ only elements of $\mathcal{F}$, because the only elements that are not in $\mathcal{F}$ are complements of something that is in $\mathcal{G}$, so they cannot be in $\U$. So $\U$ is an ultrafilters included in the weakly partition regular family $\mathcal{F}$.\\
Conversely, suppose that $\U$ is an ultrafilter included in $\mathcal{F}$, and let $S=A_{1}\cup...\cup A_{n}$ be a partition of $S$. Then, by definition of ultrafilter, one of the sets $A_{i}$ is in $\U$ so, in particular, it is in $\mathcal{F}$, and this proves that $\mathcal{F}$ is weakly partition regular.\\
2) Suppose that $\mathcal{F}$ is strongly partition regular, and for every set $A$ in $\mathcal{F}$ consider

\begin{center} $\mathcal{F}_{A}=\{B\in\mathcal{F}\mid B\subseteq A\}$. \end{center}

$\mathcal{F}_{A}$ is a weakly partition regular family of subsets of $A$, so there is an ultrafilter $\U_{A}$ on $A$ with $\U_{A}\subseteq \mathcal{F}_{A}$.\\
$\U_{A}$ can be extended to an ultrafilter on $S$ closing under supersets, and this is an internal operation for $\mathcal{F}$: if $\U$ is the ultrafilter on $S$ obtained extending the ultrafilter $\U_{A}$, then $\U\subseteq\mathcal{F}$. This proves that every set $A$ in $\mathcal{F}$ is included in an ultrafilter $\U$ with $\U\subseteq\mathcal{F}$ so, in particular, $\mathcal{F}$ is an union of ultrafilters.\\
Conversely, suppose that $\mathcal{F}$ is an union of ultrafilters, and let $A$ be an element of $\mathcal{F}$ and $\U_{A}$ an ultrafilter included in $\mathcal{F}$ such that $A\in\U_{A}$. Then, if $A=A_{1}\cup...\cup A_{n}$, there is an index $i$ such that $A_{i}$ is in $\U_{A}$, so $A_{i}$ is in $\mathcal{F}$.\\
\end{proof}

From this moment on, we consider fixed a first order language of arithmetic $\mathcal{L}$, and when we talk about first order formulas is intended that the language used is $\mathcal{L}$. For an introduction to first order logic, see e.g. $\cite{rif7}$.

\begin{defn} We say that a first order sentence $\varphi$ is {\bfseries weakly partition regular} or {\bfseries strongly partition regular} on a set $S$ if the related family $\mathcal{F}(S,\varphi)$ has the corresponding property, where 

\begin{center} $\mathcal{F}(S,\varphi)=\{A\subseteq S\mid A\models\varphi\}$. \end{center}

\end{defn}

E.g.: let $\varphi:$ be the sentence "`$\exists x$ $x=7$", and $S=\N$. Then $\mathcal{F}(\N,\varphi)$ is the principal ultrafilter $\mathfrak{U}_{7}$, which is a strongly partition regular family, so $\varphi$ is strongly partition regular.\\
As partition regular families are related with ultrafilters, we introduce the important concept of $\varphi$-ultrafilter:

\begin{defn} Let $\varphi$ be a first order sentence, and $\U$ an ultrafilter on $\N$. $\U$ is a {\bfseries $\varphi$-ultrafilter} if, for every element $A$ of $\U$, $A$ satisfies $\varphi$. \end{defn}

As a corollary of Theorem 1.2.3, given a first order sentence $\varphi$, there is a $\varphi$-ultrafilter on $S$ if and only if the sentence $\varphi$ is weakly partition regular on $S$; similarly, for every subset $A$ of $S$ that satisfies $\varphi$ there is a $\varphi$-ultrafilter that contains $A$ if and only if $\varphi$ is strongly partition regular. This proves the following theorem:

\begin{thm} Let $\varphi$ be a first order sentence, and $S$ a set. The following two equivalences hold:
\begin{enumerate}
	\item $\varphi$ is weakly partition regular on $S$ if and only if there exists a $\varphi$-ultrafilter on $S$;
	\item $\varphi$ is strongly partition regular on $S$ if and only if for every subset $A$ of $S$ that satisfies $\varphi$ there is a $\varphi$-ultrafilter that contains $A$.
\end{enumerate}
\end{thm}

\section{Some Results in Ramsey Theory} 

In this section, our goal is to recall some basic definitions and to present some classical and widely known theorems in Ramsey Theory on $\N$. These results are proved both combinatorially and with the use of ultrafilters. This is done since, in Chapter Three, we will reprove these results with nonstandard techniques, and the proofs given here are used as a yardstick to outline advantages and disadvantages of these nonstandard techniques.\\
The results in Ramsey Theory are usually presented in terms of colorations:

\begin{defn} Given a set $S$ and a natural number $n\geq 1$, a {\bfseries coloration} $c$ of $S$ with $n$ colors is a map $c: S\rightarrow\{1,...,n\}$. \end{defn}

It is usually intended that, for every natural number $i\leq n$, $c^{-1}(i)\neq \emptyset$.\\
In a precise sense, colorations and partitions are the same: if $P=\{A_{1},...,A_{n}\}$ is a partition of $S$ in $n$ pieces, the function $c: S\rightarrow \{1,..,n\}$ such that

\begin{center} for every element $s$ in $S$, $c(s)=i$ if and only if $s\in A_{i}$ \end{center}

is a coloration of $S$ with $n$ colors and, if $c$ is a coloration of $S$ with $n$ colors, the family $P=\{A_{1},...,A_{n}\}$, where

\begin{center} for every $i\leq n$, $A_{i}=c^{-1}(i)$ \end{center}

is a partition of $S$ in $n$ pieces.\\
The result that gives the name to this branch of mathematic is Ramsey's Theorem (proved by Frank Plumpton Ramsey in $\cite{rif32}$). Before stating and proving Ramsey Theorem, we have to introduce the following definition:

\begin{defn} Given a set $S$ and a natural number $n$, the set of subsets of $S$ with cardinality $n$ is denoted by $[S]^{n}$:

\begin{center} $[S]^{n}=\{ A\subseteq S\mid |S|=n\}$. \end{center} \end{defn}

\begin{thm}[Ramsey] Given a natural number $n$, if $[\N]^{n}$ is finitely colored then there exists an infinite subset $S$ of $\N$ such that $[S]^{n}$ is monochromatic. \end{thm}

\begin{proof} Combinatorial Proof: We give the proof for the case $n=2$. The general case follows by induction on $n$.\\
Let $c$ be the finite coloration of $[\N]^{2}$, and let $r$ be the number of colors of the coloration $c$. We inductively define infinite sets $A_{i}$ and natural numbers $x_{i}$ in this way:

\begin{enumerate}
	\item $A_{1}=\N$;
	\item $x_{1}$ is any element of $\N$;
	\item Fixed $x_{i}$ in $A_{i}$ define, for $1\leq j\leq r$, $T^{j}_{i}=\{y\in A_{i}\mid c(\{x_{i},y\})=j\}$.
\end{enumerate}

Now, by definition, we have

\begin{center} $A_{i}=\bigcup_{j=1}^{r} T^{j}_{i}$, \end{center}

and this forms a finite partition of the infinite set $A_{i}\setminus\{x_{i}\}$. So, there is at least one index $j$ with $T^{j}_{i}$ infinite. Fix that index $j$ and pose $A_{i+1}=T^{j}_{i}$.\\
Finally, pose 

\begin{center} $A=\{x_{i}\mid i\geq 1\}$.\end{center}

We observe that, given natural numbers $i,j,k$ with $i<j$ and $i<k$, since $A_{j}\subseteq A_{i}$ and $A_{k}\subseteq A_{i}$, by definitions we have $c(\{x_{i},x_{j}\})=c(\{x_{i},x_{k}\})$.\\
We construct this new coloration $c'$ with $r$ colors for the set $A$: given $x_{i}$ in $A$, we pose $c'(x_{i})=j$ if and only if for every $i<k$ we have $c(\{x_{i},x_{k}\})=j$.\\
$c'$ is a finite coloration of an infinite set, so there is a monochromatic infinite subset $S$ of $A$ (with color $j$ respect $c^{\prime}$). And, by construction, for every $\{x,y\}\in [S]^{2}$, $c(\{x,y\})=j$, so $[S]^{2}$ is monochromatic.\\ \end{proof}

\begin{proof} Proof with Ultrafilters: Let $\U$ be a nonprincipal ultrafilter on $\N$ and let $c$ be a coloration of $[\N]^{2}$ with $r$ colors. Identify $[\N]^{2}$ with the set \begin{center} $\Delta^{+}=\{(n,m)\in\N^{2}\mid n<m\}$,\end{center}

and consider the coloration $c'$ of $\Delta^{+}$:

\begin{center} For every $(n,m)\in \N^{2}$, $c'((n,m))=c(\{n,m\})$. \end{center}

Put, for $1\leq i\leq r$, $C_{i}=\{(n,m)\in\Delta^{+}\mid c'((n,m))=i\}$. Then 

\begin{center}$\Delta^{+}=\bigcup_{i=1}^{r} C_{i}$,\end{center} and this is a partition of $\Delta^{+}$.\\
As we proved in Proposition 1.1.17, $\Delta^{+}$ is a set in $\U\otimes\U$ so there is an index $i\leq r$ such that $C_{i}\in\U\otimes\U$. By point two of Proposition 1.1.17 it follows that there is a subset $B$ of $\N$ with 

\begin{center}$\{(b_{1},b_{2})\mid b_{1}<b_{2}, b_{1},b_{2}\in B\}\subseteq C_{i}$.\end{center}

By construction $B$ is an infinite subset of $\N$ monochromatic respect $c$: for every $\{b_{1},b_{2}\}\in [B]^{2}$, $c(b_{1},b_{2})=i$.\\\end{proof}

One other gem in the early development of Ramsey Theory is Van Der Waerden's Theorem (see $\cite{rif35}$):

\begin{thm}[Van Der Waerden] In every finite coloration of $\N$ there are arbitrarily long monochromatic arithmetic progressions. \end{thm}

Combinatorial proofs of this result can be found in $\cite{rif19}$, and two proofs with the use of ultrafilters are given in $\cite{rif41}$.\\
Here, we present this particular case:

\begin{thm} In every $2$-coloration of $\N$ there is a monochromatic arithmetic progression of lenght three. \end{thm}

\begin{proof} To prove the result it is sufficient to show that there is a natural number $n$ such that, for every 2-coloration of $\{1,...,n\}=A_{1}\cup A_{2}$, $A_{1}$ or $A_{2}$ contains an arithmetic progression of lenght three. Following the proof in $\cite{rif19}$, we show that $n=325$ is sufficient for our pourposes.\\
Divide $325$ in $65$ blocks $B_{i}$, $1\leq i\leq 65$, of lenght $5$, where 

\begin{center}$B_{i}=\{5\cdot (i-1), 5\cdot (i-1) +1,...,5\cdot(i-1) +4\}$.\end{center}

There are only $2^{5}=32$ ways to 2-color a block of lenght 5 so, by the pigehonhole principle, there are two indexes $i,j$ with $1\leq i<j\leq 33$ such that the block $B_{i}$ and the block $B_{j}$ have the same coloration.\\
Consider the block $B_{i}$, and its elements $5i+1,5i+2,5i+3$. If they have the same coloration, we found a monochromatic arithmetic progression of lenght three. Otherwise, only two of them have the same color; observe that, if those two numbers are $5i+a, 5i+b$, also $5i+b+(b-a)\in B_{i}$ (that is why we choose blocks of lenght 5). If $b+(b-a)$ has the same coloration of $b$ and $a$, we have a monochromatic arithmetic progression of lenght three; otherwise, consider $B_{i}, B_{j}, B_{j+(j-i)}$ (since $1\leq i<j\leq 33$, $j+(j-i)\leq 65$, and that is why we choose $n=65\cdot 5=325$). In $B_{j+(j-i)}$ consider the element $5(j+(j-i))+b+(b-a)$. If its color is the same as $5i+a$, then 

\begin{center} $5i+a, 5j+b, 5(j+(j-i))+b+(b-a)$\end{center} is a monochromatic arithmetic progression of lenght three with common difference $(j-i)+(b-a)$; if the color of $5(j+(j-i))+b+(b-a)$ is not the same as the color of $5i+a$ then, by construction, it must be the same as the color of $5i+b+(b-a)$, and in this case 

\begin{center} $5i+b+(b-a), 5j+b+(b-a),5(j+(j-i))+b+(b-a)$ \end{center}

is a monochromatic arithmetic progression of lenght three of common difference $(j-i)$.\\\end{proof}

\begin{proof} Proof with ultrafilters: The result is a straightforward consequence of this claim:\\

{\bfseries Claim:} If $\U$ is an idempotent ultrafilter in $(\bN,\oplus)$ and $A$ is a set in $2\U\oplus\U$, in $A$ there is an arithmetic progression of lenght three.\\

We prove the claim: the property that we use is that, for every idempotent ultrafilter $\U$ in $\bN$, for every set $S$ in $\U$, the set 

\begin{center} $\{n\in\N\mid S-n\in\U\}$ \end{center}

is in $\U$, where $S-n=\{a\in\N\mid a+n\in S\}$.\\
Consider the set $A$. By definition of sum of ultrafilters, \begin{center} $A\in 2\U\oplus\U\Leftrightarrow \{n\in\N\mid\{m\in\N\mid n+m\in A\}\in 2\U\}\in\U$.\end{center} Consider the set

\begin{center} $B=\{n\in\N\mid \{m\in\N\mid n+m\in A\}\in 2\U\}$. \end{center}
As $A\in 2\U\oplus\U$, for every natural number $n$ in $B$ the set 

\begin{center} $B_{n}=\{m\in\N\mid n+2m\in A\}$\end{center}
is in $\U$.\\
Let $x$ be an element in $B\cap \{n\in\N\mid B-n\in\U\}$ and $y$ an element in $(B-x)\cap \{n\in\N\mid B_{x}- n\in\U\}$. Observe that, by construction, $x+y\in B$ and $B_{x}-y\in\U$.\\
Let $z$ be an element in $B_{x}\cap B_{x+y}\cap (B_{x}-y)$. Observe that, by construction, $z+y\in B_{x}$. Then
\begin{enumerate}
	\item $x+2z\in A$, as $x\in B$ and $z\in B_{x}$;
	\item $x+y+2z\in A$, as $x+y\in B$ and $z\in B_{x+y}$;
	\item $x+2y+2z\in A$, as $x\in B$ and $z+y\in B_{x}$.
\end{enumerate}

As $x+2z, x+y+2z, x+2y+2z$ is am arithmetic progression of lenght three in $A$, this proves the claim.
\\\end{proof}

We point out that Van der Waerden's theorem does not entail that in every finite coloration of $\N$ there are infinite monochromatic arithmetic progressions: e.g. consider this 2-coloration of $\N$: 

\begin{center} For every natural number $n$, $c(n)=1$ if and only if there is an odd natural number $m$ with $\frac{m(m+1)}{2}\leq n <\frac{(m+1)(m+2)}{2}$. \end{center}

In this coloration, there are arbitrarily long arithmetic progression both with color 1 and color 2, but there are not infinite monochromatic arithmetic progression, since both $c^{-1}(1)$ and $c^{-1}(2)$ are subsets of $\N$ with arbitrarily long gaps.\\

Another important result in Ramsey Theory on $\N$ is Schur's Theorem (see $\cite{rif34}$), which is the oldest of the results presented in this section (it has been proved in 1916). This theorem ensures a weak condition of closure under sum for some piece of any finite partition:

\begin{thm}[Shur] For every finite coloration $c$ of $\N$ then are positive natural numbers $n,m$ such that $c(n)=c(m)=c(n+m)$. \end{thm}

\begin{proof} Combinatorial Proof: Suppose that $c:\N\rightarrow\{1,...,r\}$ is an $r$-coloration. We use $c$ to induce an $r$ coloration $c^{\prime}$ on $[\N]^{2}$ defined in this way:

\begin{center} $c^{\prime}((i,j))=c(|i-j|)$. \end{center}

As a consequence of Ramsey Theorem, there is an infinite subset $S$ on $\N$ with $[S]^{2}$ monochromatic. Let $i<j<k$ be elements of $S$ (observe that $c^{\prime}((i,j))=c^{\prime}((i,k))=c^{\prime}((j,k)) )$, and pose $n=j-i$ and $m=k-j$.\\
Then 
\begin{enumerate}
	\item $c(n)=c(j-i)=c^{\prime}((i,j))$
	\item $c(m)=c(k-j)=c^{\prime}((j,k))$
	\item $c(n+m)=c((k-j)+(j-i))=c(k-i)=c^{\prime}((i,k))$
\end{enumerate}
so $n,m$ and $n+m$ are monochromatic.

\end{proof}

\begin{proof} Proof with Ultrafilters: Schur's Theorem is a straightforward consequence of this claim:\\

{\bfseries Claim:} For every idempotent ultrafilter $\U$ in $(\bN,\oplus)$, for every set $A$ in $\U$, there are elements $n,m$ in $A$ such that $n+m\in A$.\\

We prove the claim: let $\U$ be an additively idempotent ultrafilter, and $A$ a set in $\U$; since $\U$ is idempotent, by definition

\begin{center} $\{n\in\N\mid\{m\in\N\mid n+m\in A\}\in\U\}\in\U$ \end{center}

Let $B$ be the set:

\begin{center} $B=\{n\in\N\mid\{m\in\N\mid n+m\in A\}\in\U\}$ \end{center} and, for every natural number $n$ in $A\cap B$ (that is nonempty, since it is in $\U$), let $B_{n}$ be the set

\begin{center} $B_{n}=\{m\in\N\mid n+m\in A\}$. \end{center}

If $m$ is an element in $B_{n}\cap A$ (that is nonempty since it is in $\U$), by construction $n,m,n+m\in A$, and this proves the claim.\\
If $\U$ is an additively idempotent ultrafilter, and $c:\N\rightarrow\{1,...,n\}$ a finite coloration of $\N$, one of the monochromatic sets $c^{-1}(i)$ is in $\U$, so it contains a monochromatic lenght three arithmetic progression.\\ \end{proof}

Schur's Theorem can be generalized: 

\begin{defn} Let $S=\{s_{1},...,s_{n}\}$ be a finite subset of $\N$ with cardinality $n$. The {\bfseries set of finite sums of elements in $S$} $($notation $FS(S))$ is the set

\begin{center} $FS(S)=\{\sum_{i\in J} s_{i}\mid \emptyset\neq J\subseteq \{1,...,n\} \}$. \end{center}

Similarly, if $S=\{s_{n}\mid n\in\N\}$ is an infinite subset of $\N$, then

\begin{center} $FS(S)=\{\sum_{i\in J} s_{i}\mid J\in\wp_{fin}(\N)\setminus\emptyset\}$, \end{center}

where $\wp_{fin}(\N)$ denotes the set of finite subsets of $\N$. \end{defn}

An important generalization of Schur's Theorem is Folkman's Theorem (this theorem has been proved independently by many mathematicians, we call it Folkman's Theorem following $\cite{rif19}$):

\begin{thm}[Folkman] For every finite coloration of $\N$, for every natural number $k$, there is a set $S_{k}$ of cardinality $k$ such that $FS(S_{k})$ is monochromatic. \end{thm}

Combinatorial Proof: We follow the proof presented in $\cite{rif19}$. To prove this result, we need a lemma that involves the notion of "weakly monochromaticity" for a set in the form $FS(S)$:

\begin{defn} Given a nonempty finite subset $S=\{s_{1},...,s_{k}\}$ of $\N$, $FS(S)$ is {\bfseries weakly monochromatic} if, for every nonempty subset $I$ of $\{s_{1},...,s_{k}\}$, the color of $\sum_{i\in I} s_{i}$ is the color of $s_{\max(I)}$. \end{defn}

The combinatorial proof of Folkman's Theorem is based on this lemma:

\begin{lem} For every natural numbers $m,k\geq 1$ there is natural number $n$ such that if $\{1,...,n\}$ is $m$-colored then there are $x_{1},...,x_{k}$ in $\{1,...,n\}$ such that $FS(\{x_{1},...,x_{k}\})$ is weakly monochromatic. \end{lem}

\begin{proof} Let $m$ be the number of colors of the colorations. We proceed by induction on $k$. If $k=1$, there is nothing to prove.\\
Suppose to have proved the result for $k$, and consider $k+1$. By inductive hypothesis, there is a natural number $n$ such that in any $m$-coloration of $\{1,..,n\}$ there are $x_{1},..,x_{k}$ with $FS(\{x_{1},...,x_{k}\})$ weakly monochromatic.\\
By Van der Waerden's Theorem, there is a natural number $N$ such that in any $m$-coloration of $\{1,...,N\}$ there is a monochromatic arithmetic progression $A=\{a,a+b,...,a+nb\}$.\\
Consider $B=\{b,...,nb\}$. By inductive hypothesis (here there is a little abuse: the inductive hypothesis involves the set $\{1,...,n\}$, and here we have $\{b,...,bn\}$, but it should be clear that we can return to the precise inductive hypothesis by division for $b$), there are $x_{1},...,x_{k}$ in $B$ with $FS(\{x_{1},...,x_{k}\})$ weakly monochromatic. Pose $x_{k+1}=a$. Then

\begin{center} $FS(\{x_{1},...,x_{k},x_{k+1}\})$ \end{center}

is weakly monochromatic: consider any two elements $y_{1},y_{2}$ in the form $y_{1}=\sum_{i\in I_{1}} x_{i}, y_{2}=\sum_{i\in I_{2}} x_{i}$, where $I_{1},I_{2}$ are nonempty subset of $\{1,...,k+1\}$ with $\max (I_{1})=\max (I_{2})= M$. If $M\leq k$, then $y_{1}$ and $y_{2}$ are monochromatic since they are in $FS(\{x_{1},...,x_{k}\})$. If $M=k+1$, $y_{1},y_{2}$ are monochromatic since they are two terms of the arithmetic progression $A$.\\
This proves that in every $m$-coloration of $\{1,...,N\}$ there are $k+1$ elements $x_{1},..,x_{k+1}$ with $FS(\{x_{1},..,x_{k+1}\})$ weakly monochromatic, and the theorem follows by induction.\\ \end{proof}

We can now prove Folkman's Theorem:

\begin{proof} Combinatorial Proof: The equivalent finitary formulation of Folkman's Theorem is this: for every natural numbers $r,k$ there is a natural number $n$ such that, for every $r$ coloration of $\{1,...,n\}$ there are $x_{1},...,x_{k}$ with $FS(\{x_{1},...,x_{k}\})$ monochromatic.\\
By Lemma 1.3.10 there is a natural number $n$ such that in any coloration of $\{1,...,n\}$ there are $y_{1},...,y_{rk}$ with $FS(\{y_{1},...,y_{rk}\})$ weakly monochromatic.\\
$y_{1},...,y_{rk}$ are colored with $r$ colors, so there are at least $k$ of them (say, $y_{j_{1}},...,y_{j_{k}}$) that are monochromatic. Pose

\begin{center} $x_{i}=y_{j_{i}}$. \end{center}

By construction, $\{x_{1},...,x_{k}\}$ is such that $FS(\{x_{1},..,x_{k}\})$ is monochromatic (say of color $j$): in fact, for any nonempty $I\subseteq \{1,...,k\}$, the color of $\sum_{i\in I}x_{i}$ depends only on the color of $x_{\max(I)}$, which is $j$, so all the elements in $FS(\{x_{1},..,x_{k}\})$ are monochromatic. \\ \end{proof}   

\begin{proof} Proof with Ultrafilters: The result follows by this claim:\\

{\bfseries Claim:} if $\U$ is a idempotent ultrafilter in $(\bN,\oplus)$ and $A$ is an element of $\U$, for every natural number $k$ there is a subset $S_{k}$ of $A$ such that $|S_{k}|=k$ and $FS(S_{k})\subseteq A$.\\

We prove the claim: let $\U$ be an additively idempotent ultrafilter, and $A$ an element of $\U$. Since $A\in\U$, and $\U$ is idempotent, the set 

\begin{center} $B_{1}=\{n\in\N\mid A-n\in \U\}$\end{center}

is in $\U$. Take an element $n_{1}$ in $A_{1}\cap B_{1}=A_{2}$.\\
Since $A_{2}\in\U$, the set $B_{2}=\{n\in\N\mid A_{1}-n\in\U\}\in\U$. Take $n_{2}\in A_{2}\cap B_{2}=A_{3}$.\\
Similarly, suppose to have defined $A_{1},...,A_{k}, B_{1},...,B_{k}$, $n_{1},...,n_{k}$. Pose 

\begin{center} $A_{k+1}=A_{k}\cap B_{k}$; \end{center}

since $A_{k+1}\in\U$, the set $B_{k+1}=\{n\in\N\mid A_{k+1}-n\in\U\}$ is in $\U$. Take $n_{k+1}\in A_{k+1}\cap B_{k+1}=A_{k+2}$.\\
In this way, we construct a set $X=\{n_{1},n_{2},....\}\subseteq A$, infinite, with 

\begin{center} $FS(X)\subseteq A$, \end{center}

and, for every natural number $k$, for every subset $S_{k}$ of $X$ with $k$ elements, $S_{k}$ is a subset of $A$ with cardinality $k$ such that $FS(S_{k})\subseteq A$. This proves the claim, and the claim entails the thesis because, if $c:\N\rightarrow\{1,...,m\}$ is an $m$-coloration of $\N$, since $\N=\bigcup_{i=1}^{m} c^{-1}(i)$ one of the monochomatic sets $c^{-1}(i)$ is in $\U$, so it contains arbitrarily large subsets $S$ such that $FS(S)\subseteq c^{-1}(i)$.\\
 \end{proof}

{\bfseries Fact:} The proof with ultrafilters of Folkman's Theorem is a proof of a stronger result, known as Hindman's Theorem:

\begin{thm}[Hindman] Whenever $\N$ is finitely colored there is an infinite subset $S=\{x_{1}, x_{2}, x_{3},....\}$ of $\N$ such that $FS(S)$ is monocromatic.\end{thm}

Hindman's Theorem (see e.g. $\cite{rif22}$, $\cite[\mbox{Corollary 2.10}]{rif21}$) has both a great combinatorial and a great historical importance, since it is an example of theorem where the ultrafilter proof is undoubtedly simpler than the combinatorial one. In fact, the ultrafilter proof uses only few properties of idempotent ultrafilters while, as for the combinatorial proof, quoting the words of Neil Hindman:

\begin{center} {\itshape Anyone with a very masochistic bent is invited to wade through the original combinatorial proof.} \end{center}

Both Schur's and Folkman's Theorems can be seen as dealing with partition regularity for homogeneous linear polynomials:

\begin{defn} An homogeneous linear polynomial with coefficients in $\Z$ \begin{center} $P(x_{1},...,x_{n}): \sum_{i=1}^{n} c_{i} x_{i}$ \end{center} is {\bfseries weakly partition regular} on $\N$ if in every finite coloration of $\N\setminus\{0\}$ there is a monochromatic solution to the equation $P(x_{1},...,x_{n})=0.$ \\
Similarly, a matrix $A$ with coefficients in $\Z$ 
\begin{center}

$\begin{pmatrix}
a_{1,1}& a_{1,2} &  ...& a_{1,n}\\
a_{2,1}&  a_{2,2} &  ...& a_{2,n} \\
 ...&...&...&...\\
a_{m,1}& a_{m,2} &...& a_{m,n} \end{pmatrix}
$

\end{center}

is {\bfseries weakly partition regular} on $\N$ if in every finite coloration of $\N\setminus\{0\}$ there is a monochromatic solution to the linear system:
\begin{center}

$\left\{\begin{array}{rcrcccrcc}
a_{1,1}x_{1}&+& a_{1,2}x_{2} & + & ...& + &a_{1,n}x_{n}& = & 0 \\
a_{2,1}x_{1}& + & a_{2,2}x_{2}& + & ...& + &a_{2,n}x_{n}& =& 0  \\
& & & & ... \\
a_{m,1}x_{1}& + & a_{m,2}x_{2} & + &...& + &a_{m,n}x_{n}& = &0 \end{array}\right.
$

\end{center}

\end{defn}

For example, Schur's Theorem states that the polynomial

\begin{center} $P(x,y,z): x+y-z$ \end{center}

is weakly partition regular on $\N$, while Folkman's Theorem is equivalent to state that, for every natural number $n$, the $(2^{n}-1)\times (n+2^{n}-1)$ matrix $A$

\begin{center}

$\begin{pmatrix}
r_{1} \\
r_{2} \\
 ...\\
r_{2^{n}-1} \end{pmatrix}
$

\end{center}

is weakly partition regular on $\N$, where the elements $a_{i.j}$ in the row $r_{i}$ are constructed in this way: let $f$ be a bijection between $2^{n}-1$ and the set of nonempty subsets of $\{1,...,n\}$. Then
\begin{itemize}
	\item if $1\leq j\leq n$ and $j\in f(i)$ then $a_{i,j}=1$;
	\item if $1\leq j\leq n$ and $j\notin f(i)$ then $a_{i,j}=0$;
	\item if $n<j\leq n+2^{n}-1$ and $j-n=i$ then $a_{i,j}=-1$;
	\item if $n<j\leq n+2^{n}-1$ and $j-n\neq i$ then $a_{i,j}=0$.
\end{itemize}
 
E.g., for $n=3$, the matrix $A$ can be choosen in this way 

\begin{center}
$\begin{pmatrix}
1&0&0&-1&0&0&0&0&0&0\\
0&1&0&0&-1&0&0&0&0&0\\
0&0&1&0&0&-1&0&0&0&0\\
1&1&0&0&0&0&-1&0&0&0\\
1&0&1&0&0&0&0&-1&0&0\\
0&1&1&0&0&0&0&0&-1&0\\
1&1&1&0&0&0&0&0&0&-1\\
\end{pmatrix}$
\end{center}
\vspace{0.3cm}

In 1933 Richard Rado, a student of Issai Schur, characterized the weakly partition regular matrices on $\N$ in terms of a condition on the system that he called "column condition" (see $\cite{rif30}$):

\begin{defn} A matrix 

\begin{center}

$\begin{pmatrix}

a_{1,1} & a_{1,2} &...& a_{1,n} \\
a_{2,1} & a_{2,2} & ... & a_{2,n} \\
... & ... & ... & ...\\
a_{m,1} & a_{m,2} & ... & a_{m,n}
\end{pmatrix}$

\end{center}

with coefficients in $\Z$ satisfies the {\bfseries columns condition} if it is possible to order the column vectors $c_{1},...,c_{n}$ and to find integers $i_{0}=0$ and $i_{1},...,i_{k}$ with $1\leq i_{1}<i_{2}<...<i_{k}=n$ such that, if we pose

\begin{center} $A_{j} = \sum_{(s=i_{j-1}+1)}^{i_{j}} c_{s}$, \end{center}

then  
\begin{enumerate}
	\item $A_{1}= 0$;
	\item for $1<j\leq k$, $A_{j}$ is a linear combination of $c_{1},...,c_{i_{j-1}}$.
\end{enumerate}

\end{defn}

\begin{thm}[Rado] Let $A$ be a matrix with coefficients in $\Z$. Then the following two conditions are equivalent:
\begin{enumerate}
	\item $A$ is weakly partition regular on $\N$;
	\item $A$ satisfies the columns condition.
\end{enumerate}
\end{thm}

The proof of Rado's Theorem can be found, e.g., in $\cite{rif19}$; here we provide a direct combinatorial proof of this corollary, that characterizes the weakly partition regular linear homogeneous polynomials with coefficients in $\Z$:

\begin{thm} Let $P(x_{1},...,x_{n}): \sum_{i=1}^{n}c_{i}x_{i}$ be an homogeneous linear polynomial with nonzero coefficients in $\Z$. The following conditions are equivalent:

\begin{enumerate}
	\item $P(x_{1},...,x_{n})$ is weakly partition regular on $\N$;
	\item there is a nonempy subset $J$ of $\{1,...,n\}$ such that $\sum_{j\in J}c_{j}=0$.
\end{enumerate}
\end{thm}

The "only if" part of the theorem is based on this lemma, that strenghtens Van der Waerden's Theorem:

\begin{lem} For every finite coloration of $\N$, for every natural numbers $n, m$, there are natural numbers $a,b$ such $a, a+b,...,a+nb, a-b,...,a-nb, mb$ are monochromatic. \end{lem}

This lemma is proved in $\cite[\mbox{pag 55}]{rif19}$.\\

The "if" part of the theorem uses the so-called smod($p$) colorations (where smod stand for "super modulo"):

\begin{defn} Let $p$ be a prime number in $\N$. The $smod(p)$ coloration of $\N$ is the $(p-1)$-coloration such that, for every natural number $n$, if $n=a p^{k}$, where $k\geq 0$ and gcd$(a,p)$=1, then 

\begin{center} $smod(p)(n)$= $a\mod p$. \end{center} \end{defn}

This given, we can prove Theorem 1.3.15:

\begin{proof} $(2)\Rightarrow (1)$ Reordering, if necessary, the coefficients of $P(x_{1},...,x_{n})$, we can assume that the sum of the first $k$ coefficients is 0:

\begin{center} $c_{1}+...+c_{k}=0$. \end{center}
The idea is to describe a parametric solution $S_{m,\{z_{i}\}}$ of the equation $P(x_{1},...,x_{n})=0$, in the form 

\begin{equation*} s_{i}=\begin{cases} a+z_{i} b, & 1\leq i \leq k;\\ mb, & k<i\leq n,\end{cases}\end{equation*}

where $\{z_{i}\mid i\leq k\}, m$ are given and $a,b$ are parameters, and then to apply Lemma 1.3.16 to find such a structure in one of the colors, obtaining a monochromatic solution.\\
To get the desired parametrization we need to find $s, \{z_{i}\mid i\leq k\}$. There are two cases: if $k=n$, there is nothing to prove, since if $\sum_{i=1}^{n} c_{i}=0$ then for every natural number $m$ the constant solution $x_{i}=m$ solves the equation; so, we suppose that $k<n$ and we consider the natural numbers $c,d,m$ such that

\begin{center} $c=gcd(c_{1},...,c_{k})$;\\\vspace{0.3cm} $d=\sum_{i=k+1}^{n} c_{k}$; \\\vspace{0.3cm} $m=\frac{c}{gcd(c,d)}$. \end{center}

By construction, $dm$ is an integer multiple of $c$, so there is an integer $z$ in $\Z$ with \begin{center} $cz+dm=0$ \end{center}

and, by Bézout's Identity, since $c=gcd(c_{1},...,c_{k})$, there are integers $z_{1},...,z_{k}$ such that

\begin{center} $c_{1}z_{1}+...+c_{k}z_{k}=cz$. \end{center}

With this choice of $m, \{z_{i}\mid i\in I\}$, $S_{m,\{z_{i}\}}$ is a parametric solution of the equation:

\begin{center}$\sum_{i=1}^{n} c_{i}s_{i}=\sum_{i=1}^{k}c_{i}(a+z_{i}b) + \sum_{i=k+1}^{n} c_{i}mb=$\\\vspace{0.3cm}$= a\sum_{i=1}^{k} c_{i} + \sum_{i=1}^{k} c_{i} z_{i} b + \sum_{i=k+1}^{n} c_{i} m b= 0 + bcz + bdm= 0+ b(cz+dm)=0$\end{center}

and Lemma 1.3.16 provides that, if $Z=\max\{z_{i}\mid i\leq k\}$, there are natural numbers $a,b$ in $\N$ such that $a, a+b, ..., a+Zb, a-b,...,a-Zb, mb$ monochromatic, and the set $\{a, a+b,...,a+Zb, a-b,...,a-Zb,mb\}$ contains a parametric solution in the form $S_{m,\{z_{i}\}}$.\\
$(1)\Rightarrow (2)$ This implication follows by this claim:\\

{\bfseries Claim:} If there is a prime number $p$ such that $p$ does not divide the sum of any nonempty subset of $\{c_{i}\mid i\leq n\}$, then the equation $\sum_{i=1}^{n} c_{i}x_{i}=0$ has no monochromatic solution respect the smod($p$) coloration.\\

The thesis is a consequence of the claim: if every subset of the set of coefficients has sum different from 0, since the set of all the possible sums of the coefficients of $P(x_{1},...,x_{n})$ is finite, there is a prime number $p$ that does not divide any of these sums, and $P(x_{1},...,x_{n})$ is not weakly partition regular, since it has not a monochromatic solution respect the smod($p$) coloration.\\
Proof of the claim: Let $p$ be a prime number such that $p$ does not divide any sum of the coefficients of $P(x_{1},...,x_{n})$, and assume that $(a_{1},...,a_{n})$ is a monochromatic solution of the equation $P(x_{1},...,x_{n})=0$. We suppose that 

\begin{center} $N=gcd(a_{1},...,a_{n})=1$ \end{center}

since, by definition of smod($p$)-coloration, if $(a_{1},...,a_{n})$ is monochromatic then $(\frac{a_{1}}{N},...,\frac{a_{n}}{N})$ is monochromatic.\\
Reordering if necessary, we suppose that $p$ does not divide $x_{1},...,x_{k}$ and $p$ divides $x_{k+1},...,x_{n}$: since $gcd(x_{1},...,x_{n})=1$, $k$ is necessarily $\geq 1$.\\
We reduce the equation $P(a_{1},...,a_{n})=0$ modulo $p$:

\begin{center} $\sum_{i=1}^{n} c_{i}a_{i}\equiv 0\mod p$. \end{center}

By construction, this is equivalent to say that

\begin{center} $\sum_{i=1}^{k} c_{i}a_{i}\equiv 0\mod p$, \end{center}

since we assumed that $c_{j}\equiv 0 \mod p$ for every index $j>k$. By construction 

\begin{center} smod($p$)($a_{1}$)=smod($p$)($a_{2}$)=...=smod($p$)($a_{k}$)=$y$ \end{center} for some $1\leq y <p$ and, since $p$ does not divide any element $a_{i}$, 

\begin{center} smod($p$)($a_{i}$)= ($a_{i}\mod p$). \end{center}

From these observations it follows that

\begin{center} $\sum_{i=1}^{k} c_{i}a_{i}\equiv y\cdot \sum_{i=1}^{k} c_{i}\equiv 0\mod p$, \end{center}

that, since $y\neq 0 \mod p$, holds if and only if

\begin{center} $\sum_{i=1}^{k} c_{i}\equiv 0 \mod p$: \end{center}

$\{c_{1},...,c_{k}\}$ is a finite subset of the set of coefficients, and $\sum_{i=1}^{k}c_{i}$ is divided by $p$, and this is absurd. This proves the claim.\\ \end{proof}

In next chapter we introduce a nonstandard technique that, in Chapter Three, will be used to re-prove the results presented in this section from a different point of view.

\newpage

\chapter{Nonstandard Tools}

In this chapter we expose the nonstandard tools that we need in the rest of the thesis.\\
After recalling some general facts about nonstandard methods, we start to study the monads of ultrafilters. These objects are defined in the following way: given an ultrafilter $\U$, and a hyperextension $^{*}\N$ of $\N$ satisfying a particular additional property, the monad of $\U$ is

\begin{center} $G_{\U}=\{\alpha\in$$^{*}\N\mid \alpha\in\bigcap_{A\in\U}$$^{*}A\}$. \end{center}

These structures have already been studied in the literature (e.g. $\cite{rif28}$ and $\cite{rif29}$). In Section Two we outline some of their known properties and, in Section Three, we present an original result, called Bridge Theorem, which shows that particular combinatorial properties of ultrafilters can be seen as generated by properties of their monads. Motivated by this observation, we decide to rename the elements in the monad of $\U$ as generators of $\U$.\\
The study of the sets of generators leads us to consider the tensor products of ultrafilters. A problem that we outline is that, even if the set of generators of a tensor product $\U\otimes\V$ can be characterized in terms of $G_{\U}$ and $G_{\V}$ (this has been done by Christian Puritz in $\cite[\mbox{Theorem 3.4}]{rif29}$), this characterization does not give a procedure to construct, given generators $\alpha$ of $\U$ and $\beta$ of $\V$, a generator of $\U\otimes\V$.\\
This leads, by observations explicitated in Section Five, to consider a particular hyperextension of $\N$, that we call $\omega$-hyperextension and denote by $^{\bullet}\N$. Its particularity is that in $^{\bullet}\N$ we can iterate the star map. This property turns out to be particularly usefull to deal with tensor products of ultrafilters: in fact, the possibility to iterate the star gives the desired procedure to construct generators of tensor products $\U\otimes\V$ starting with generators of $\U$ and $\V$.\\
In last section we apply this procedure to study, given ultrafilters $\U,\V$ on $\N$, the sets of generators of $\U\oplus\V$ and $\U\odot\V$ in $^{\bullet}\N$.

\section{Nonstandard Methods}

Nonstandard Analysis was ideated by Abraham Robinson in the late 1950's. Its original intent was to give a rigorous formalization to the notion of infinitesimal element, and to apply such a notion to mathematical analysis (see $\cite{rif33}$).\\
In the following fifty years, Nonstandard Analysis has grown both as a branch of mathematical logic, where its foundations are studied, and in terms of applications to a wide variety of problems. We suggest to the interested reader the books $\cite{rif1}$ and $\cite{rif18}$, where the methodology of nonstandard analysis, as well as many of its applications, are presented. In this section, we are concerned with its logic foundations. Since this does not concern mathematical analysis, we prefer to talk about nonstandard methods.\\
We observe that there are many ways in which nonstandard methods have been formalized. In this thesis, in order to formalize and use nonstandard methods we adopt the framework of superstructures. For a comprehensive tractation of this approach, see e.g. $\cite[\mbox{Section 4.4}]{rif7}$. Among the alternative approaches, we recall the classical $\cite{rif26}$, where nonstandard methods are presented from the point of view of the so-called Internal Set Theory, and $\cite{rif4}$, where the autors give an introduction to the hyper-methods of nonstandard analysis and present eight different approaches to nonstandard methods.\\
We begin our short introduction to nonstandard methods recalling the definition of superstructure on a set:

\begin{defn} Let $X$ be an infinite set. The {\bfseries superstructure on $X$} is $\mathbb{V}(X)=\bigcup_{n\in\N}\mathbb{V}_{n}(X)$, where

\begin{center} $\mathbb{V}_{0}(X)=X$;\vspace{0.3cm}\\ $\mathbb{V}_{n+1}(X)=\mathbb{V}_{n}(X)\cup\wp(\mathbb{V}_{n})$. \end{center}\end{defn}

We make the extra assumption that $X$ is a "base set": this means that its elements behave as atoms within $\mathbb{V}(X)$ (formally, $\emptyset\notin X$ and, for every $x$ in $X$, $x\cap\mathbb{V}(X)=\emptyset$). This makes it possible to define the "individuals" relative to $\mathbb{V}(X)$ as the elements of $X$, and the "`sets"' relative to $\mathbb{V}(X)$ as the elements in $\mathbb{V}(X)\setminus X$.\\
Probably, there are two reasons for the diffusion of superstructures in nonstandard methods: the first is that superstructures satisfy many nice closure properties with respect to set-theoretic operations; the second is that superstructures provide an easy way to formalize the naive concept of "mathematical object" as "element of a superstructure". E.g., any real number is a mathematical object, as well as any subset of real numbers, the set $\R^{2}$, the relation of order in $\R$, any real function, and so on. Within the Zermelo-Fraenkel framework, all of these concepts can be realized as elements of a superstructure $\mathbb{V}(\R)$: e.g., any subset of the real numbers is a set in $\mathbb{V}_{1}(\R)$.\\
Superstructures, as "universes of mathematical objects", are the first ingredient in the construction of nonstandard methods. The other two are the star map and the transfer principle.

\begin{defn} Given two superstructures $\mathbb{V}(X)$, $\mathbb{V}(Y)$, a {\bfseries star map} is a map $*:\mathbb{V}(X)\rightarrow\mathbb{V}(Y)$ such that $Y=$$^{*}X$. \end{defn}

Usually, it is assumed that $*$ is proper:

\begin{defn} A star map $*$ is {\bfseries proper} if for every infinite set $A$ relative to $\mathbb{V}(X)$ the inclusion $^{\sigma}A\subseteq$$^{*}A$ is proper, where

\begin{center} $^{\sigma}A=\{$$^{*}a\mid a\in A\}$.\end{center}\end{defn}

The last notion we have to introduce to talk about nonstandard methods is the transfer principle.\\
We assume that the reader knows the basics of first order logic, in particolar the notions of first order formula, free and bounded variables, open formula and sentence. We fix some notations and conventions:
\begin{itemize}
	\item With $\mathcal{L}$ we denote a first order logical language containing the simbol of membership $\in$ (e.g., the language of set theory);
	\item The formulas that we consider are constructed in the language $\mathcal{L}$;
	\item We reserve the letters $x,y,z,x_{1},x_{2},...$ to denote variables;
	\item When writing a formula $\varphi(x_{1},...,x_{n},p_{1},....,p_{m})$ we shall mean that $\varphi(x_{1},...,x_{n},p_{1},...,p_{m})$ is a first order formula, that its free variables are exactly $x_{1},...,x_{n}$ and its parameters are exactly $p_{1},...,p_{m}$; 
	\item Except when strictly necessary, we do not indicate the parameters; in particular, when we denote a formula as $\varphi$, it is intended that it may have parameters, but that it has not free variables. 
\end{itemize}

\begin{defn} Let $\varphi(x,x_{1},....,x_{k})$ be a formula of $\mathcal{L}$, $x$ a free variable in $\varphi$ and $y$ a variable that is not bounded in $\varphi$. The abbreviation

\begin{center}$(\forall x \in y) \varphi(x,x_{1},....,x_{k})$\end{center}

means $\forall x (x\in y)\Rightarrow \varphi(x,x_{1},....,x_{k})$. Similarly,

\begin{center} $(\exists x\in y) \varphi(x,x_{1},....,x_{k})$\end{center}

means $\exists x (x\in y)\Rightarrow\varphi(x,x_{1},....,x_{k})$.\\
The quantifiers $\forall x\in y$ and $\exists x\in y$ are called {\bfseries bounded quantifiers}.\\
A {\bfseries bounded quantifier formula} is obtained from atomic formulas using only connectives and bounded quantifiers.\end{defn}

An other kind of formulas that we will need later are the elementary formulas:

\begin{defn} A formula $\varphi(x_{1},...,x_{n})$ is {\bfseries elementary} if it is a bounded quantifier formula and its only parameters are
\begin{itemize}
	\item elements of $\N^{k}$,
	\item subsets of $\N^{k}$,
	\item functions $f:\N^{k}\rightarrow\N^{h}$,
	\item relations on $\N^{k}$,
\end{itemize}

where $h,k$ are two positive natural numbers.\end{defn}

The last ingedient needed to talk about nonstandard methods is the transfer principle:

\begin{defn} The star map $*:\mathbb{V}(X)\rightarrow\mathbb{V}(Y)$ satisfies the {\bfseries transfer principle} if for every bounded quantifier formula $\varphi(x_{1},...,x_{n})$ and for every $a_{1},...,a_{n}\in\mathbb{V}(X)$ 

\begin{center}$\mathbb{V}(X)\models\varphi(a_{1},...,a_{n})$ if and only if $\mathbb{V}(Y)\models$$^{*}\varphi($$^{*}a_{1},...,$$^{*}a_{n})$. \end{center}

\end{defn}

In this definition, if $\varphi(x_{1},...,x_{n})$ is a formula with parameters $p_{1},...,p_{k}$, $^{*}\varphi(x_{1},...,x_{n})$ is the formula obtained substituting, in $\varphi(x_{1},...,x_{n})$, each parameter $p_{i}$ with $^{*}p_{i}$: e.g., if $\varphi(x)$ is the formula "$x\in\N$", then $^{*}\varphi(x)$ is the formula "$x\in$$^{*}\N$".\\
The transfer principle can be equivalently reformulated saying that $*$ is a bounded elementary embedding (with a terminology mutuated from model theory, see e.g. $\cite[\mbox{pp.266--267}]{rif7}$).

\begin{defn} A {\bfseries superstructure model of nonstandard methods} is a triple $\langle \mathbb{V}(X), \mathbb{V}(Y), *\rangle$ where 
\begin{enumerate}
	\item a copy of $\N$ is included in $X$ and in $Y$;
	\item $\mathbb{V}(X)$ and $\mathbb{V}(Y)$ are superstructures on the infinite sets $X$, $Y$ respectively;
	\item $*$ is a proper star map from $\mathbb{V}(X)$ to $\mathbb{V}(Y)$ that satisfies the transfer property.
\end{enumerate}

\end{defn}

It is assumed that, for every natural number $n$, $^{*}n=n$. Observe that, since $Y=$$^{*}X$, by transfer it follows that $x\in X$ if and only if $^{*}x\in Y$: so the star map send individuals relative to $\mathbb{V}(X)$ to individuals relative to $\mathbb{V}(Y)$ and sets relative to $\mathbb{V}(X)$ to sets relative to $\mathbb{V}(Y)$.

\begin{defn} If $y\in\mathbb{V}(Y)$, and there exists $x\in \mathbb{V}(X)$ such that $y=$$^{*}x$, then $y$ is called {\bfseries hyper-image of $x$}, and we say that $y$ is an {\bfseries hyper-image}. \\If $x$ is an infinite set in $\mathbb{V}(X)$ and $y=$$^{*}x$, $y$ is also called the {\bfseries hyperextension} of $x$.\\
In particular, if $x=\N$, then $^{*}\N$ is called the set of {\bfseries hypernatural numbers}.\\
An element $y$ of $\mathbb{V}(Y)$ is {\bfseries internal} if there exists $x\in X$ such that $y\in$$^{*}x$.\\
An element $y$ of $\mathbb{V}(Y)$ is {\bfseries external} if it is not internal. \end{defn}

Observe that every hyper-image is internal since, if $y=$$^{*}x$ then $y\in\{$$^{*}x\}$=$^{*}\{x\}$.\\
The importance of the internal elements is clear when applying transfer: since the transfer applies to bounded quantifier formulas, we get that properties of subsets of a given set $Z\in\mathbb{V}(X)$ transfer to the internal subsets of $^{*}Z$ in $\mathbb{V}(Y)$.\\
The Internal Definition Principle characterizes many internal objects in $\mathbb{V}(Y)$:

\begin{thm}[Internal Definition Principle] Let $\varphi(x_{1},...,x_{n},y)$ be a bounded quantifier formula. If $A_{1},...,A_{n},B$ are internal, then the set

\begin{center} $\{b\in B\mid \mathbb{V}(Y)\models\varphi(A_{1},...,A_{m},b)\}$\end{center}

is internal. \end{thm}

\begin{proof} This is Proposition 4.4.14 in $\cite{rif7}$. \\\end{proof}

One other result of great importance is the Overspill Principle. Before stating this principle, we need this definition:

\begin{defn} An element $\eta\in$$^{*}\N$ is {\bfseries infinite} if, for every natural number $n$, $\eta>n$. \end{defn}

\begin{thm}[Overspill Principle] Let $A$ be a nonempty internal subset of $^{*}\N$ that contains arbitrary large finite elements. Then $A$ contains an
infinite element.\end{thm}

For a comprehensive tractation of this principle, see $\cite[\mbox{section 11.4}]{rif18}$. We just warn the reader that, in Goldblatt's book, the Overspill principle is called Overflow principle.\\
Observe that, as a consequence of the Overspill Principle, in $^{*}\N$ there are infinite natural numbers.\\
Since hyperextensions of functions are particularly important in this chapter, we point out that functions $f:\N^{k}\rightarrow\N^{h}$ are extended to functions $^{*}f:$$^{*}\N^{k}\rightarrow$$^{*}\N^{h}$ that satisfy this condition: if $\Gamma_{f}$ is the graph of $f$,

\begin{center} $\Gamma_{f}=\{(x,y)\in\N^{k}\times\N^{h}\mid f(x)=y$\}, \end{center}

the graph of $^{*}f$ is $^{*}\Gamma_{f}$.\\
One other important aspect of nonstandard methods in the study of $\N$ is this:

\begin{prop} $\N$ is a bounded elementary submodel of $^{*}\N$. \end{prop}

\begin{proof} $\N$ is a submodel of $^{*}\N$, and the star map is a bounded elementary embedding. \\ \end{proof}

We usually work in nonstandard settings that satisfy some additional condition: in the definition below, we recall that a family $\mathcal{F}$ of susets of a given set $S$ has the finite intersection property if for every natural number $n$, for every $F_{1},...,F_{n}$ elements of $\mathcal{F}$, $F_{1}\cap...\cap F_{n}\neq\emptyset$.

\begin{defn} Let $\kappa$ be an infinite cardinal number. We say that the nonstandard model $\langle \mathbb{V}(X),\mathbb{V}(Y),*\rangle$ has the {\bfseries $\kappa$-enlarging property} $($resp. $\kappa^{+}$-enlarging property$)$ if, for every set $x$ in $\mathbb{V}(X)$, for every family $\mathcal{F}$ of subsets of $x$ with $|\mathcal{F}|<\kappa$ $($resp. $|\mathcal{F}|\leq\kappa)$ and with the finite intersection property, the intersection $\bigcap_{F\in\mathcal{F}}$$^{*}F$ is nonempty. \end{defn}

Enlarging is a weaker form of a model-theoretic property called saturation:

\begin{defn} Let $\kappa$ be an infinite cardinal number. We say that the nonstandard model $\langle \mathbb{V}(X),\mathbb{V}(Y),*\rangle$ has the {\bfseries $\kappa$-saturation property} $($resp. $\kappa^{+}$-saturation property$)$ if for every family $\mathcal{F}$ of internal subsets of $\mathbb{V}(Y)$ with $|\mathcal{F}|<\kappa$ $($resp. $|\mathcal{F}|\leq\kappa)$ and with the finite intersection property, the intersection $\bigcap_{F\in\mathcal{F}}$$F$ is nonempty.
\end{defn}

We remark that $\kappa$-saturation trivially implies $\kappa$-enlarging, since every hyper-image is an internal object in $\mathbb{V}(Y)$.\\
Nonstandard models satisfying these properties have particular features: e.g., if the nonstandard model has the $\mathfrak{c}^{+}$-enlarging property, and $\mathcal{F}$ is a filter on $\N$, then

\begin{center} $\bigcap_{F\in\mathcal{F}}$$^{*}F\neq\emptyset$, \end{center}

since $\mathcal{F}$ has the finite intersection property and its cardinality is, at most, $\mathfrak{c}$; if the nonstandard model has the $\mathfrak{c}^{+}$-saturation property, the cofinality of $^{*}\N$ is at least $\mathfrak{c}^{+}$: by contrast, if $S=\langle \alpha_{i}\mid i<\mathfrak{c}\rangle$ is an unbounded sequence in $^{*}\N$, and for every index $i\leq \mathfrak{c}$ we consider the set 

\begin{center} $I_{i}=\{\eta\in$$^{*}\N\mid\alpha_{i}<\eta\}$, \end{center}

the family $\langle I_{i}\mid i<\mathfrak{c}\rangle$ is a family of internal sets with the finite intersection property and cardinality $\mathfrak{c}$ so, by $\mathfrak{c}^{+}$-saturation, the intersection

\begin{center} $\bigcap_{i\leq\mathfrak{c}} I_{i}$\end{center}

is nonempty, and this is in contrast with $S$ being unbounded. So the cofinality of $^{*}\N$ is at least $\mathfrak{c}^{+}$.

\section{The Bridge Theorem}

\subsection{The Bridge Map}

In this section we present an important nexus between ultrafilters and hyperextensions, that gives a nonstandard characterization of ultrafilters that we will use throghout the thesis.\\
The correspondence is this:

\begin{prop} $(1)$ Let $^{*}\N$ be a hyperextension of $\N$. For every hypernatural number $\alpha$ in $^{*}\N$, the set 

\begin{center} $\mathfrak{U}_{\alpha}=\{A\in\N\mid \alpha\in$$^{*}A\}$ \end{center} is an ultrafilter on $\N$.\\
$(2)$ Let $^{*}\N$ be a hyperextension of $\N$ with the $\mathfrak{c}^{+}$-enlarging property. For every ultrafilter $\U$ on $\N$ there exists an element $\alpha$ in $^{*}\N$ such that $\U=\mathfrak{U}_{\alpha}$.\end{prop}

\begin{proof} (1) Let $\alpha$ be an hypernatural number in $^{*}\N$, and consider $\mathfrak{U}_{\alpha}$. $\mathfrak{U}_{\alpha}$ is not empty because it contains $\N$; moreover, it is easily seen that $\mathfrak{U}_{\alpha}$ is closed under supersets, under intersections and that it does not contain the empty set, so $\mathfrak{U}_{\alpha}$ is a proper filter on $\N$. It is an ultrafilter because, for every subset $A$ of $\N$, $A\in\mathfrak{U}_{\alpha}$ or $A^{c}\in\mathfrak{U}_{\alpha}$. In fact, we have the property 

\begin{center} "for every natural number $n\in\N$, either $n\in A$ or $n\in A^{c}$",\end{center}

and by transfer it follows that 

\begin{center}"for every hypernatural number $\alpha \in$$^{*}\N$, either $\alpha\in$$^{*}A$ or $\alpha\in$$^{*}A^{c}$".\end{center}
(2) Let $^{*}\N$ be a hyperextension of $\N$ with the $\mathfrak{c}^{+}$-enlarging property and $\U$ an ultrafilter on $\N$. The family 

\begin{center} $\{A\}_{A\in\U}$ \end{center}

has the finite intersection property; by $\mathfrak{c}^{+}$-enlarging property it follows that 

\begin{center} $\bigcap_{A\in\U}$$^{*}A\neq\emptyset$,\end{center}

as $|\U|\leq \mathfrak{c}$. If $\alpha$ is any element in this intersection, by construction $\U=\mathfrak{U}_{\alpha}$.\\
\end{proof}

An important fact is that, if the hyperextension $^{*}\N$ does not satisfy the $\mathfrak{c}^{+}$-enlarging property then there may be ultrafilters $\U$ on $\N$ such that, for every element $\alpha\in$$^{*}\N$, $\mathfrak{U}_{\alpha}\neq\U$. E.g., let $^{*}\N$ be a hyperextension of $\N$ that does not satisfy the $\mathfrak{c}^{+}$-enlarging property and $\mathcal{F}$ a family of subsets of $\N$ with the finite intersection property, and suppose that $\bigcap_{F\in\mathcal{F}}$$^{*}F=\emptyset$. Then, for every ultrafilter $\U$ that extends $\mathcal{F}$ (that such ultrafilters exist has been proved in Chapter One), the set $G_{\U}$ is empty.\\
Since we want to avoid this fact, throughout this chapter the hyperextensions that we consider satisfy (at least) the $\mathfrak{c}^{+}$-enlarging property.

\begin{defn}[] The {\bfseries bridge map} is the function $\psi:$$^{*}\N\rightarrow\bN$ defined by putting for every hypernatural number $\alpha$ in $^{*}\N$

\begin{center} $\psi(\alpha)=\mathfrak{U}_{\alpha}$. \end{center} 

\end{defn}

\begin{defn} We say that two hypernatural numbers $\alpha,\beta$ in $^{*}\N$ are {\bfseries $\U$-equivalent} $($notation $\alpha\sim_{u}\beta)$ when $\mathfrak{U}_{\alpha}=\mathfrak{U}_{\beta}$, i.e. if $\psi(\alpha)=\psi(\beta)$.\\
Given an ultrafilter $\U$ in $\bN$, the {\bfseries set of generators of $\U$} is the set

\begin{center} $G_{\U}=\{\alpha\in $$^{*}\N\mid \U=\U_{\alpha}\}$. \end{center}

\end{defn}

A warning is in order: in literature, the set $G_{\U}$ is called the monad of the ultrafilter $\U$. Here, we prefer to rename it as "set of generators" to emphasize the fact that vary properties of the ultrafilter $\U$ can be seen as actually "generated" by the elements in $G_{\U}$. This, in our point of view, is the core of our method of studying combinatorial properties of ultrafilters.\\
Some of the results about sets of generators that we present in this chapter are adaptations, in our context, of results presented in $\cite{rif29}$.\\
We observe that the bridge map associate finite hypernatural numbers with principal ultrafilters, and infinite hypernatural numbers with nonprincipal ultrafilters: e.g., let $\U$ be the principal ultrafilter on $n$. Not surprisingly, $G_{\U}=\{n\}$: in fact, by definition, an hypernatural number $\alpha$ is in $G_{\U}$ if and only if $\alpha\in$$^{*}A$ for every set $A$ in $\U$; since $\U$ is principal, the set $A=\{n\}$ is in $\U$, so by taking $^{*}A=^{*}\{n\}=\{n\}$ it follows that $\alpha=n$.\\
Conversely, if $\alpha$ is infinite, and $A$ is a subset of $\N$ such that $\alpha\in$$^{*}A$, then $A$ is necessarily infinite, so $\mathfrak{U}_{\alpha}$ is nonprincipal. Next proposition shows that, in this case, $G_{\U}$ is huge:

\begin{prop} The bridge map $\psi$ is not 1-1: in fact, for every non principal ultrafilter $\U$, the set $G_{\U}=\psi^{-1}(\U)$ has $|$$^{*}\N|$-many elements.\end{prop}

\begin{proof} Let $\U$ be a non principal ultrafilter. For every set $A$ in $\U$, consider the set

\begin{center} $\Gamma_{A}=\{f:\N\rightarrow A\mid$ $f$ is $1-1$\}. \end{center}

The family $\{\Gamma_{A}\}_{A\in \U}$ has the finite intersection property, so by $\mathfrak{c}^{+}$-enlarging there is an element $\varphi\in\bigcap$$^{*}\Gamma_{A}$.\\

{\bfseries Claim:} For every hypernatural number $\alpha$ in $^{*}\N$, $\varphi(\alpha)\in G_{\U}$.\\

In fact, for every set $A$ in $\U$, as $\varphi\in\Gamma_{A}$ the range of $\varphi$ is included in $^{*}A$ so, by construction, the range of $\varphi$ is included in $G_{\U}$, and this proves the claim.\\
As $\varphi$ is $1-1$, it follows that |$^{*}\N|\leq |G_{\U}|$, and this concludes the proof.\\\end{proof}

The following is an easy, but important, property of $G_{\U}$:

\begin{prop} Let $\U$ be a nonprincipal ultrafilter on $\N$, and $\alpha,\beta$ distinct generators of $\U$. Then $|\alpha-\beta|\in$$^{*}\N\setminus\N$.  \end{prop}

\begin{proof} Let $k$ be a positive natural number, and let $\{A_{1}, A_{2},...,A_{k}\}$ be the partition of $\N$ induced by the Euclidean division:

\begin{center} $A_{i}=\{n\in\N\mid n\equiv i\mod k\}$. \end{center}
Since $\U$ is an ultrafilter, there is exactly one index $i$ with $A_{i}$ in $\U$. As 

\begin{center} $^{*}A_{i}=\{\eta\in$$^{*}\N\mid \eta\equiv i \mod k\}$\end{center}

and $\alpha,\beta$ are in $^{*}A_{i}$, it follows that $|\alpha-\beta|\equiv 0 \mod k$. Since this is true for every natural number $k$, the only possibilities are that $\alpha=\beta$, which has been excluded in the hypothesis, or $|\alpha-\beta|\in$$^{*}\N\setminus\N$.\\\end{proof}

When the hyperextension $^{*}\N$ is $\mathfrak{c}^{+}$-saturated, the sets of generators of nonprincipal ultrafilters are coinitial and cofinal in $^{*}\N\setminus\N$:

\begin{prop} Let $^{*}\N$ be a $\mathfrak{c}^{+}$-saturated hyperextension of $\N$ and $\U$ a nonprincipal ultrafilter on $\N$. Then:
\begin{enumerate}
	\item $G_{\U}$ is left unbounded in $^{*}\N\setminus\N$;
  \item for every infinite hypernatural number $\eta$, $G_{\U}\cap [0,\eta)$ contains at least $\mathfrak{c}^{+}$ elements;
  \item the coinitiality of $G_{\U}$ in $^{*}\N\setminus\N$ is greater than $\mathfrak{c}$;
  \item $G_{\U}$ is right unbounded in $^{*}\N$;
	\item for every $\eta$ in $^{*}\N\setminus\N$, $G_{\U}\cap (\eta,+\infty)$ contains $|$$^{*}\N|$-many elements;
	\item the cofinality of $G_{\U}$ is greater than $\mathfrak{c}$.
	
\end{enumerate}
\end{prop}

\begin{proof} 1) Let $\eta$ be an hypernatural number in $^{*}\N\setminus\N$ and pose, for every $A\in\U$, 

\begin{center} $A_{\eta}=\{\alpha\in$$^{*}A\mid \alpha<\eta\}$. \end{center}

$A_{\eta}$ is internal, nonempty (as $A\subseteq A_{\eta}$) and the family 

\begin{center}$\mathcal{F}=\{A_{\eta}\}_{A\in\U}$ \end{center}

has the finite intersection property (as $A_{\eta}\cap B_{\eta}= (A\cap B)_{\eta}$) and cardinality equal to $\mathfrak{c}$. By $\mathfrak{c}^{+}$-saturation, the intersection

\begin{center} $\bigcap_{A\in\U} A_{\eta}$\end{center}

is nonempty; if $\alpha$ is an element in this intersection then $\alpha$ is infinite, less than $\eta$ and in $G_{\U}$: this proves that $G_{\U}$ is left unbounded in $^{*}\N\setminus\N$.\\
2) Let $\eta$ be an element in $^{*}\N\setminus\N$. Since $G_{\U}$ is left unbounded in $^{*}\N\setminus\N$, for every natural number $k$ and for every hypernatural number $\mu<\eta$ the set 

\begin{center} $(k,\mu)=\{\alpha\in$$^{*}\N\mid k<\alpha<\mu\}$\end{center}

is a nonempty internal set. Then the family

\begin{center}$\mathcal{F}=\{(k,\mu)\mid k\in\N$ and $\mu\in [0,\eta)\cap G_{\U}\}$ \end{center}

is a family of nonempty internal sets with the finite intersection property. Observe that the intersection

\begin{center} $\bigcap_{(k,\mu)\in\mathcal{F}}(k,\mu)$ \end{center}

is empty. As $\mathfrak{c}^{+}$-saturation holds, the only possibility is that $|\mathcal{F}|\geq \mathfrak{c}^{+}$. By construction, $|\mathcal{F}|=|G_{\U}\cap [0,\eta)|$, so $|G_{\U}\cap [0,\eta)|\geq \mathfrak{c}^{+}$: this entails that $G_{\U}\cap [0,\eta)$ contains at least $\mathfrak{c}^{+}$-many elements.\\
3) Suppose that $S=\langle \alpha_{i}\mid i<\mathfrak{c}\rangle$ is a left unbounded sequence in $G_{\U}$, with $\alpha_{i}>\alpha_{j}$ whenever $i<j$. Consider the family

\begin{center} $\mathcal{F}=\{(k,\alpha)\mid k\in\N, \alpha\in S\}$. \end{center}

$\mathcal{F}$ has the same cardinality as $S$ and it is a family of internal sets, so by $\mathfrak{c}^{+}$-saturation there should be an hypernatural number $\beta$ such that $\beta\in (k,\alpha)$ for every $(k,\alpha)\in\mathcal{F}$, and this is absurd: $\beta$ can not be finite, otherwise if $\alpha$ is any element in $G_{\U}$ and $k=\beta+1$ then $\beta\notin (k,\alpha)$; $\beta$ can not be infinite otherwise, since $G_{\U}$ is left unbounded in $^{*}\N\setminus\N$, there is an element $\alpha\in S$ with $\alpha<\beta$, so $\beta\notin (0,\alpha)$.\\
We found an absurd, so such a sequence $S$ can not exist: the coinitiality of $G_{\U}$ is greater than $\mathfrak{c}$.\\
4)-5) Let $\eta$ be an hypernatural number in $^{*}\N\setminus\N$. As $\U$ is non principal, for every set $A$ in $\U$, for every natural number $k$, there is an increasing function $f:\N\rightarrow A\cap (k,+\infty)$ that is 1-1. By transfer, this entails that the set 

\begin{center} $F_{A}=\{f:$$^{*}\N\rightarrow$$^{*}A\cap (\eta,+\infty)\mid f$ is internal, increasing and 1-1$\}$ \end{center}

is not empty. The family 

\begin{center} $\mathcal{F}=\{F_{A}\}_{A\in\U}$ \end{center}

is a family of internal sets with the finite intersection property (since $F_{A}\cap F_{B}=F_{A\cap B}$ for every $A,B\in\U$), and it has cardinality $\mathfrak{c}$. By $\mathfrak{c}^{+}$-saturation, the intersection

\begin{center} $\bigcap_{A\in\U} F_{A}$\end{center}

is nonempty. Let $\varphi$ be a function in this intersection. By construction, $\varphi:$$^{*}\N\rightarrow G_{\U}\cap (\eta,+\infty)$, it is increasing and it is 1-1: this entails that $|$$^{*}\N|=|G_{\U}\cap (\eta,\infty)|$, and that $G_{\U}$ is right unbounded in $^{*}\N$.\\
6) Suppose that $S=\langle \alpha_{i}\mid i < \mathfrak{c}\rangle$ is an increasing unbounded sequence in $G_{\U}$. Then the family

\begin{center} $\mathcal{F}=\{(\alpha,+\infty)\mid \alpha\in S\}$ \end{center}

has empty intersection, and this is absurd, since $|\mathcal{F}|=|S|=\mathfrak{c}$, $\mathcal{F}$ has the finite intersection property and its elements are internal.\\\end{proof}

As a corollary we get:

\begin{cor} For every nonprincipal ultrafilter $\U$ the set of generators of $\U$ is an external subset of $^{*}\N$. \end{cor}

\begin{proof} Every internal subset of $^{*}\N$ has a least element, while $G_{\U}$ is left unbounded. \\ \end{proof}

\subsection{The Bridge Theorem}

Throughout this section we suppose that $^{*}\N$ is a hyperextension of $\N$ with the $\mathfrak{c}^{+}$-enlarging property. We also adopt the conventions about logical formulas introduced in Section One.

\begin{defn} Let $\phi(x_{1},...,x_{n})$ be a first order formula. The {\bfseries existential closure} of $\phi(x_{1},...,x_{n})$ is the sentence

\begin{center} $E(\phi(x_{1},...,x_{n})):$ $\exists x_{1}....\exists x_{n}\phi(x_{1},...,x_{n})$. \end{center}

The {\bfseries universal closure} of $\phi(x_{1},...,x_{n})$ is the sentence

\begin{center} $U(\phi(x_{1},...,x_{n})):$ $\forall x_{1},...,\forall x_{n}\phi(x_{1},...,x_{n})$. \end{center}

A first order formula is {\bfseries existential} $($resp. {\bfseries universal}$)$ if it is the existential $($resp. universal$)$ closure of a first order formula. \end{defn}
The Bridge Theorem concerns $\varphi$-ultrafilters: we recall that, given a sentence $\varphi$, an ultrafilter $\U$ is a $\varphi$-ultrafilter if every set $A$ in $\U$ satisfies $\varphi$. The Bridge Theorem states that, whenever $\varphi$ is a first order existential sentence, to check if an ultrafilter $\U$ is a $\varphi$-ultrafilter it is enough to study its set of generators:

\begin{thm}[Bridge Theorem] Let $\varphi=E(\phi(x_{1},...,x_{n}))$ be a first order existential sentence and $\U$ an ultrafilter in $\bN$. Then the following conditions are equivalent:
\begin{enumerate}
	\item $\U$ is a $\varphi$-ultrafilter;
	\item there are elements $\alpha_{1},...,\alpha_{n}$ in $G_{\U}$ such that $^{*}\phi(\alpha_{1},...,\alpha_{n})$ holds.
\end{enumerate}
\end{thm}

\begin{proof} $(1)\Rightarrow (2)$: Let $\U$ be a $\varphi$-ultrafilter. Given a set $A$ in $\U$, consider

\begin{center} $\Phi_{A}=\{(a_{1},...,a_{n})\in A^{n}\mid \phi(a_{1},...,a_{n})\}$. \end{center}

Observe that, since $\U$ is a $\varphi$-ultrafilter, $\Phi_{A}$ is nonempty for every set $A$ in $\U$, and that the family \{$\Phi_{A}\}_{A\in \U}$ has the finite intersection property. In fact, if $A_{1},...,A_{m}$ are elements in $\U$, then

\begin{center} $\Phi_{A_{1}}\cap...\cap \Phi_{A_{m}}=\Phi_{A_{1}\cap...\cap A_{m}}\neq\emptyset$.\end{center}

By $\mathfrak{c}^{+}$-enlarging property, the intersection

\begin{center} $\Theta=\bigcap_{A\in\U}$$^{*}\Phi_{A}$\end{center}

is nonempty. Since, by construction, 

\begin{center} "`for every $(a_{1},....,a_{n})\in \Phi_{A}$ $\phi(a_{1},...,a_{n})$"',\end{center}

by transfer it follows 

\begin{center} "`for every $(\alpha_{1},...,\alpha_{n})\in$$^{*}\Phi_{A}$ $^{*}\phi(\alpha_{1},...,\alpha_{n})$".\end{center}

Let $(\alpha_{1},...,\alpha_{n})$ be an element of $\Theta$. As observed, $^{*}\phi(\alpha_{1},...,\alpha_{n})$ holds and, by construction, $\alpha_{1},...,\alpha_{n}\in G_{\U}$ since, for every index $i\leq n$, for every set $A$ in $\U$, $\alpha_{i}\in$$^{*}A$.\\
$(2)\Rightarrow(1)$ Suppose that $\U$ is not a $\varphi$-ultrafilter, and let $A$ be an element of $\U$ such that, for every $a_{1},...,a_{n}$ in $A$, $\neg\phi(a_{1},....,a_{n})$ holds.\\
Then by transfer it follows that, for every $\xi_{1},...,\xi_{n}$ in $^{*}A$, $\neg$$^{*}\phi(\xi_{1},...,\xi_{n})$ holds; in particular, as $G_{\U}\subseteq$$^{*}A$, for every $\xi_{1},...,\xi_{n}$ in $G_{\U}$, $\neg$$^{*}\phi(\xi_{1},...,\xi_{n})$ holds, and this is absurd. So $\U$ is a $\varphi$-ultrafilter.\\
\end{proof}

\begin{cor} An existential formula $\varphi=E(\phi(x_{1},...,x_{n}))$ is weakly partition regular if and only if there are $n$ hypernatural numbers $\alpha_{1},...,\alpha_{n}$ such that $\alpha_{1}\sim_{u}\alpha_{2}\sim_{u}....\sim_{u}\alpha_{n}$ and $^{*}\phi(\alpha_{1},...,\alpha_{n})$ holds. \end{cor}

\begin{proof} By Theorem 1.2.6, $\varphi$ is weakly partition regular if and only if there is a $\varphi$-ultrafilter, and the thesis follows by Theorem 2.2.9.\\ \end{proof}

E.g, let $\phi(x,y,z)$ be the open formula

\begin{center} $(x,y,z>0)\wedge (x+y=z)$. \end{center}

As a consequence of Theorem 2.2.9, an ultrafilter $\U$ is an $E(\phi(x,y,z))$-ultrafilter if and only if there are positive $\alpha,\beta,\gamma\in G_{\U}$ such that $\alpha+\beta=\gamma$. \\
What can we say about universal formulas?

\begin{thm} Let $\phi(x_{1},...,x_{n})$ be a first order formula, $\varphi$ its universal closure and $\U$ an ultrafilter in $\bN$. Then the following conditions are equivalent:
\begin{enumerate}
	\item there is a set $A$ in $\U$ that satisfies $\varphi$;
	\item for every $\alpha_{1},...,\alpha_{n}$ in $G_{\U}$ $^{*}\phi(\alpha_{1},...,\alpha_{n})$ holds.
\end{enumerate}
\end{thm}

\begin{proof} (1)$\Rightarrow(2)$ Let $A$ be a set in $\U$ such that for every $a_{1},...,a_{n}$ in $A$ $\phi(a_{1},...,a_{n})$ holds; by transfer it follows that, for every $\alpha_{1},...,\alpha_{n}$ in $^{*}A$, $^{*}\phi(\alpha_{1},...,\alpha_{n})$ holds. In particular, as $G_{\U}$ is a nonempty subset of $^{*}A$, $^{*}\phi(\alpha_{1},...,\alpha_{n})$ holds for every $\alpha_{1},...,\alpha_{n}$ in $G_{\U}$.\\
$(2)\Rightarrow (1)$ Suppose that for every set $A$ in $\U$ there are $a_{1},...,a_{n}$ in $A$ such that $\neg\phi(a_{1},...,a_{n})$ holds. Then $\U$ is an $E(\neg(\phi(x_{1},...,x_{n})))$-ultrafilter so, by Theorem 2.2.9, in $G_{\U}$ there are elements $\alpha_{1},...,\alpha_{n}$ such that $\neg(^{*}\phi(\alpha_{1},...,\alpha_{n}))$ holds, and this is absurd.\\\end{proof}

If $^{*}\N$ is $\mathfrak{c}^{+}$-saturated, we can prove two results similar to Theorem 2.2.9 and Theorem 2.2.11:

\begin{lem} Let $\phi(x_{1},...,x_{n},y_{1},...,y_{m})$ be a first order formula, $\alpha_{1},...,\alpha_{n}$ elements in $^{*}\N$ and $\U$ an ultrafilter on $\N$. The following two conditions are equivalent:
\begin{enumerate}
	\item $(\exists B\in\U)(\forall\beta_{1},...,\beta_{m}\in$$^{*}B)$ $^{*}\phi(\alpha_{1},...,\alpha_{n},\beta_{1},...,\beta_{m})$;
	\item $\forall\beta_{1},...,\beta_{n}\in G_{\U}$ $^{*}\phi(\alpha_{1},...,\alpha_{n},\beta_{1},...,\beta_{m})$.
\end{enumerate}
\end{lem}

\begin{proof} $(1)\Rightarrow (2)$ If $B$ is a set in $\U$ such that, for every $\beta_{1},...,\beta_{m}$ in $^{*}B$, $^{*}\phi(\alpha_{1},...,\alpha_{n},\beta_{1},...,\beta_{m})$ holds then the thesis follows as $G_{\U}\subseteq$$^{*}B$.\\
$(2)\Rightarrow(1)$ Suppose that, for every set $B$ in $\U$ there are $\beta_{1},...,\beta_{m}$ in $^{*}B$ such that $\neg$$^{*}\phi(\alpha_{1},...,\alpha_{n},\beta_{1},...,\beta_{m})$ holds. Let

\begin{center} $\Gamma_{B}=\{(\beta_{1},...,\beta_{m})\in$$^{*}B^{m}\mid \neg$$^{*}\phi(\alpha_{1},...,\alpha_{n},\beta_{1},...,\beta_{m})\}$. \end{center}

By the internal definition principle, for every set $B$ in $\U$ the set $\Gamma_{B}$ is internal. So the family $\{\Gamma_{B}\}_{B\in\U}$ is a family of nonempty internal subsets of $^{*}\N$ with the finite intersection property: in fact, if $B_{1},...,B_{k}$ are elements in $\U$, then the intersection

\begin{center} $\bigcap_{i=1}^{k}\Gamma_{B_{i}}=\Gamma_{\bigcap_{i=1}^{k}B_{i}}\neq\emptyset.$ \end{center}

By $\mathfrak{c}^{+}$-saturation property, the intersection

\begin{center} $\Theta=\bigcap_{B\in\U} \Gamma_{B}$ \end{center}

is nonempty. Let $(\beta_{1},...,\beta_{m})$ be an element in $\Theta$. Then, for every index $i\leq n$, $\beta_{i}\in G_{\U}$ and $\neg$$^{*}\phi(\alpha_{1},...,\alpha_{n},\beta_{1},...,\beta_{m})$ holds, and this is absurd.\\\end{proof}

\begin{thm} Let $\phi(x_{1},...,x_{n},y_{1},...,y_{m})$ be a first order formula, and $\U,\V$ ultrafilters on $\N$. The following conditions are equivalent:

\begin{enumerate}
	\item $(\exists \alpha_{1},...,\exists\alpha_{n}\in G_{\U})(\forall\beta_{1},...,\forall\beta_{m}\in G_{\V})$ $^{*}\phi(\alpha_{1},...,\alpha_{n},\beta_{1},...,\beta_{m})$;
	\item $(\exists \alpha_{1},...,\exists\alpha_{n}\in G_{\U})(\exists B\in\V)(\forall\beta_{1},...,\forall\beta_{m}\in$$^{*}B)$ $^{*}\phi(\alpha_{1},...,\alpha_{n},\beta_{1},...,\beta_{m})$;
	\item $(\exists B\in\V)(\forall A\in\U)(\exists a_{1},...,\exists a_{n}\in A)(\forall b_{1},...,\forall b_{m}\in B)$ $\phi(a_{1},...,a_{n},b_{1},...,b_{m})$. 
\end{enumerate}

\end{thm}

\begin{proof} $(1)\Rightarrow (2)$ Let $\alpha_{1},...\alpha_{n}$ be elements in $G_{\U}$ as in the hypothesis. Then, by Lemma 2.2.12 it follows that it exists a set $B$ in $\V$ such that, for every $\beta_{1},...,\beta_{m}$ in $^{*}B$, $^{*}\phi(\alpha_{1},...,\alpha_{n},\beta_{1},...,\beta_{m})$ holds, and this entails the thesis.\\
$(2)\Rightarrow (3)$: Let $\alpha_{1},...,\alpha_{n},B$ be as in the hypothesis. Suppose that, by contrast, there is a set $A$ in $\U$ such that, for every $a_{1},...,a_{n}$ in $A$, there are $b_{1},...,b_{m}$ in $B$ such that $\neg\phi(a_{1},...,a_{n},b_{1},...,b_{m})$ holds. Then by transfer property it follows that 

\begin{center} $\forall \xi_{1},...,\forall\xi_{n}\in$$^{*}A \exists \beta_{1},...,\exists\beta_{m}\in$$^{*}B$ $\neg$$^{*}\phi(\xi_{1},...,\xi_{n},\beta_{1},...,\beta_{m})$ \end{center}

and this is absurd since $\alpha_{1},...,\alpha_{n}\in G_{\U}\subseteq$$^{*}A$ and $^{*}\phi(\alpha_{1},...,\alpha_{n},\beta_{1},...,\beta_{m})$ holds for every $\beta_{1},...,\beta_{m}$ in $^{*}B$.\\
$(3)\Rightarrow (1)$ Let $B$ be a set in $\V$ as in the hypothesis and, for every set $A$ in $\U$, let 

\begin{center} $\Phi_{A}=\{(a_{1},...,a_{n})\in A^{n}\mid \forall (b_{1},...,b_{m})\in B^{m}$ $\phi(a_{1},...,a_{n},b_{1},...,b_{m})$ holds$\}$. \end{center}

The family $\{\Phi_{A}\}_{A\in\U}$ is a family of nonempty subsets of $\N^{n}$ with the finite intersection property as, if $A_{1},...,A_{k}$ are elements of $\U$, the intersection

\begin{center} $\Phi_{A_{1}}\cap...\cap\Phi_{A_{k}}=\Phi_{A_{1}\cap....\cap A_{k}}$ \end{center}

is nonempty. By $\mathfrak{c}^{+}$-enlarging property, the intersection

\begin{center} $\Theta=\bigcap_{A\in\U}$$^{*}\Gamma_{A}$\end{center}

is nonempty. Let $(\alpha_{1},...,\alpha_{n})$ be an element in this intersection. By construction, $\alpha_{1},...,\alpha_{n}$ are elements in $G_{\U}$ and, since for every set $A$ in $\U$

\begin{center} $^{*}\Theta_{A}=\{(\xi_{1},...,\xi_{n})\in$$^{*}A_{n}\mid \forall(\beta_{1},...,\beta_{m})\in B^{m}$ $^{*}\phi(\xi_{1},...,\xi_{n},\beta_{1},...,\beta_{m})$ holds$\}$, \end{center}

then $^{*}\phi(\alpha_{1},...,\alpha_{n},\beta_{1},...,\beta_{m})$ holds for every $\beta_{1},...,\beta_{m}$ in $^{*}B$. Since $G_{\V}\subseteq$$^{*}B$, we get the thesis.\\ \end{proof}

The previous theorem has the following three interesting corollaries:

\begin{cor} Let $\phi(x_{1},...,x_{n},y_{1},...,y_{m})$ be a first order formula, and $\U,\V$ ultrafilters on $\N$. The following conditions are equivalent:

\begin{enumerate}
	\item $(\forall \alpha_{1},...,\forall\alpha_{n}\in G_{\U})(\exists\beta_{1},...,\exists\beta_{m}\in G_{\V})$ $^{*}\phi(\alpha_{1},...,\alpha_{n},\beta_{1},...,\beta_{m})$;
	\item $(\forall \alpha_{1},...,\forall\alpha_{n}\in G_{\U})(\forall B\in\V)(\exists\beta_{1},...,\exists\beta_{m}\in$$^{*}B)$ $^{*}\phi(\alpha_{1},...,\alpha_{n},\beta_{1},...,\beta_{m})$;
	\item $(\forall B\in\V)(\exists A\in\U)(\forall a_{1},...,\forall a_{n}\in A)(\exists b_{1},...,\exists b_{m}\in B)$ $\phi(a_{1},...,a_{n},b_{1},...,b_{m})$.
\end{enumerate}

\end{cor}
\begin{proof} Each one of the conditions (1)-(2)-(3) is the contronominal of the corrispective condition in Theorem 2.2.13. \\ \end{proof}

\begin{cor} Let $\phi(x_{1},...,x_{n},y_{1},...,y_{m})$ be a first order formula, and $\U$ an ultrafilter on $\N$. The following conditions are equivalent:

\begin{enumerate}
	\item $(\exists \alpha_{1},...,\exists\alpha_{n}\in G_{\U})(\forall\beta_{1},...,\forall\beta_{m}\in G_{\U})$ $^{*}\phi(\alpha_{1},...,\alpha_{n},\beta_{1},...,\beta_{m})$;
	\item $(\exists \alpha_{1},...,\exists\alpha_{n}\in G_{\U})(\exists B\in\U)(\forall\beta_{1},...,\forall\beta_{m}\in$$^{*}B)$ $^{*}\phi(\alpha_{1},...,\alpha_{n},\beta_{1},...,\beta_{m})$;
	\item $(\exists B\in\U)(\forall A\in\U)(\exists a_{1},...,\exists a_{n}\in A)(\forall b_{1},...,\forall b_{m}\in B)$ $\phi(a_{1},...,a_{n},b_{1},...,b_{m})$.
\end{enumerate}

\end{cor}

\begin{proof} This follows by Theorem 2.2.13 by putting $\U=\V$. \\\end{proof}

\begin{cor} Let $\phi(x_{1},...,x_{n},y_{1},...,y_{m})$ be a first order formula, and $\U$ an ultrafilter on $\N$. The following conditions are equivalent:

\begin{enumerate}
	\item $(\forall \alpha_{1},...,\forall\alpha_{n}\in G_{\U})(\exists\beta_{1},...,\exists\beta_{m}\in G_{\U})$ $^{*}\phi(\alpha_{1},...,\alpha_{n},\beta_{1},...,\beta_{m})$;
	\item $(\forall \alpha_{1},...,\forall\alpha_{n}\in G_{\U})(\forall B\in\V)(\exists\beta_{1},...,\exists\beta_{m}\in$$^{*}B)$ $^{*}\phi(\alpha_{1},...,\alpha_{n},\beta_{1},...,\beta_{m})$;
	\item $(\forall B\in\U)(\exists A\in\U)(\forall a_{1},...,\forall a_{n}\in A)(\exists b_{1},...,\exists b_{m}\in B)$ $\phi(a_{1},...,a_{n},b_{1},...,b_{m})$.
\end{enumerate}
\end{cor}

\begin{proof} This follows by Corollary 2.2.13 by putting $\U=\V$. \\\end{proof}

\section{Extension of Functions to Ultrafilters}

This section consists of two parts: in the first one we study, given a function $f:\N^{k}\rightarrow\N$ and an ultrafilter $\U\in\beta(\N^{k})$, how $G_{\U}$ and $G_{\overline{f}(\U)}$ are related. In the second part, we study which functions defined on $^{*}\N$ can be naturarly associated to functions defined on $\bN$.\\
Since in this section we deal with ultrafilters on $\N^{k}$, where $k$ is any positive natural number, and with sets of functions, we introduce the following two definitions:

\begin{defn} If $k$ is a positive natural number, and $\U$ an ultrafilter on $\N^{k}$, the {\bfseries set of generators of $\U$} is

\begin{center} $G_{\U}=\{(\alpha_{1},...,\alpha_{n})\in$$^{*}\N\mid A\in\U\Leftrightarrow(\alpha_{1},...,\alpha_{n})\in$$^{*}A\}$.\end{center}

\end{defn}

Since we want that every ultrafilter has generators, thoughout this section we still assume that the hyperextension $^{*}\N$ that we consider satisfies the $\mathfrak{c}^{+}$-enlarging property.

\begin{defn} Let $A,B$ be sets. We denote by $\mathbf{\mathtt{Fun}(A,B)}$ the set of functions with domain $A$ and range included in $B$. \end{defn}

\subsection{Extension of functions in $\mathtt{Fun}(\N^{k},\N)$ to functions in $\mathtt{Fun}(\beta(\N^{k}),\bN)$}

It is well-known, as a consequence of $\bN$ being the Stone-\v{C}ech compactification of $\N$, that given any natural number $k$ and any function $f:\N^{k}\rightarrow\N$, $f$ induces a map $\overline{f}:\beta(\N^{k})\rightarrow\bN$:

\begin{defn} If $f$ is a function in $\mathtt{Fun}(\N^{k},\N)$, we denote by $\overline{f}$ the unique continuous extension of $f$ in $\mathtt{Fun}(\beta(\N^{k}),\bN)$. \end{defn}

$\overline{f}$ is the function such that, for every subset $A$ of $\N$, for every ultrafilter $\U$ in $\beta(\N^{k})$, 

\begin{center} $A\in\overline{f}(\U)\Leftrightarrow f^{-1}(A)\in\U$. \end{center}

Observe that the application 

\begin{center} $\overline{\cdot}:\mathtt{Fun}(\N^{k},\N)\rightarrow \mathtt{Fun}(\beta(\N^{k}),\bN)$ \end{center}

that associate to every function $f$ in $\mathtt{Fun}(\bN^{k},\bN)$ its unique continuous extension is 1-1 but is not surjective, as its range is included in the subset of $\mathtt{Fun}(\beta(\N^{k}),\N)$ of continuous functions (and not every function in $\mathtt{Fun}(\beta(\N^{k}),\bN)$ is continuous).\\
Also, the above application is not surjective onto the subset of continuous functions in $\mathtt{Fun}(\beta(N^{k}),\bN)$: e.g., if $k=1$, $\V$ is a non principal ultrafilter on $\N$ and $g$ is the function in $\mathtt{Fun}(\bN,\bN)$ such that, for every ultrafilter $\U$ in $\bN$, $g(\U)=\V$, then $g$ is a continuous function which is not the extension of any element in $\mathtt{Fun}(\N,\N)$.\\
We present a few properties of the map $\overline{\cdot}$ respect to the notion of $\U$-equivalence (that have been presented also in $\cite{rif29}$ and $\cite{rif15}$). A known fact that we use is the following theorem:

\begin{thm} For every function $f$ in $\mathtt{Fun}(\N,\N)$ there is three-coloration $\N=A_{1}\cup A_{2}\cup A_{3}$ such that, for every natural number $n$, if $f(n)\neq n$ then $n$ and $f(n)$ have two different colors.\end{thm}

The proof of the above result can be found, e.g., in $\cite{rif15}$.

\begin{thm} Let $^{*}\N$ be a hyperextension of $\N$ with the $\mathfrak{c}^{+}$-enlarging property, $f,g$ be two functions in $\mathtt{Fun}(\N^{k},\N)$, $\alpha$ an element in $^{*}\N^{k}$, and $\U$ an ultrafilter on $\N^{k}$. Then the following properties holds:
\begin{enumerate}
	\item If $\alpha\in G_{\U}$ then $^{*}f(\alpha)\in G_{\overline{f}(\U)}$;
	\item If $k=1$ and $\alpha,$$^{*}f(\alpha)\in G_{\U}$ then $\alpha=$$^{*}f(\alpha)$;
	\item If $k=1$, $f$ is 1-1 and $^{*}f(\alpha),$$^{*}g(\alpha)\in G_{\U}$ then $^{*}f(\alpha)=$$^{*}g(\alpha)$.
\end{enumerate}
\end{thm}

\begin{proof} (1) Suppose that $\U=\mathfrak{U}_{\alpha}$, and consider $\overline{f}(\U)$. By definition, a subset $A$ of $\N$ is in $\overline{f}(\U)$ if and only if $f^{-1}(A)\in\U$. As $\U=\mathfrak{U}_{\alpha}$, it follows that $A\in\overline{f}(\U)$ if and only if $\alpha\in$$^{*}(f^{-1}(A))$ and this happens if and only if $^{*}f(\alpha)\in$$^{*}A$. So, if $\U=\mathfrak{U}_{\alpha}$ then $\overline{f}(\U)=\mathfrak{U}_{^{*}f(\alpha)}$.\\
(2) This is a consequence of Theorem 2.3.4: let $A_{1},A_{2},A_{3}$ be subsets of $\N$ such that if $n\neq f(n)$ then there is an index $i\leq 3$ such that $n\in A_{i}$ and $f(n)\notin A_{i}$. Suppose that $^{*}f(\alpha)\neq \alpha$; then, by transfer, there is an index $i\leq 3$ such that $\alpha\in$$^{*}A_{i}$ and $^{*}f(\alpha)\notin$$^{*}A_{i}$; in particular, $A_{i}\in\mathfrak{U}_{\alpha}$ and $A_{i}\notin\mathfrak{U}_{^{*}f(\alpha)}$, absurd.\\
(3) Let $A$ be an infinite subset of $\N$ with $\alpha\in$$^{*}A$ and $A^{c}$ infinite. $f$ is 1-1, so there exists a bijection $\varphi:\N\rightarrow\N$ such that $f$ and $\varphi$ coincide on $A$. Since 

\begin{center}$\{n\in\N \mid f(n)=\varphi(n)\}$\end{center}

includes $A$, then $^{*}f(\alpha)=$$^{*}\varphi(\alpha)$, so by hypothesis $^{*}g(\alpha)\sim_{u}$$^{*}\varphi(\alpha)$. By (1) it follows that

\begin{center} $^{*}\varphi^{-1}($$^{*}g(\alpha))\sim_{u}$$^{*}\varphi^{-1}($$^{*}\varphi(\alpha))=\alpha$  \end{center}

so by (2) it follows that 

\begin{center} $^{*}\varphi^{-1}($$^{*}g(\alpha))=\alpha$ \end{center}

which, as $\varphi$ is bijective, holds if and only if $^{*}g(\alpha)=$$^{*}\varphi(\alpha)$, and $^{*}\varphi(\alpha)$ is by construction equal to $^{*}f(\alpha)$. \\ \end{proof}

When the star map satisfies the $\mathfrak{c}^{+}$-saturation property, the previous result can be strengthened:

\begin{thm} Let $^{*}\N$ be a hyperextension of $\N$ with the $\mathfrak{c}^{+}$-saturation property, $f$ a function in $\mathtt{Fun}(\N^{k},\N)$, $\alpha,\beta$ elements in $^{*}\N^{k}$, and $\U$ an ultrafilter on $\N^{k}$. Then the following properties holds:
\begin{enumerate}
	
	\item If $^{*}f(\alpha)\sim_{u}\beta$ then there is an element $\gamma\sim_{u}\alpha$ such that $\beta=$$^{*}f(\gamma)$;
	\item $^{*}f[G_{\U}]=G_{\overline{f}(\U)}$;
	
\end{enumerate}
\end{thm}

\begin{proof} (1) Assume that $^{*}f(\alpha)\sim_{u}\beta$, and let $A$ be a set in $\mathfrak{U}_{\alpha}$. By hypothesis, $\beta\in$$^{*}f[$$^{*}A]$, so $^{*}f^{-1}(\beta)\cap$$^{*}A$ is not empty, and the family of internal sets 

\begin{center} $\Gamma_{A}=\{$$^{*}f^{-1}(\beta)\cap$$^{*}A\mid A\in\mathfrak{U}_{\alpha}\}$\end{center}

has the finite intersection property. By $\mathfrak{c}^{+}$-saturation, there is an element $\gamma$ in $\bigcap_{A\in\U}\Gamma_{A}$: by construction, $\gamma\sim_{u}\alpha$ and $^{*}f(\gamma)=\beta$.\\ 
(2) By the result (1) in Theorem 2.3.5, $^{*}f[G_{\U}]\subseteq G_{\overline{f}(\U)}$; by (1), for every $\beta$ in $G_{\overline{f}(\U)}$, there is an element $\alpha$ in $G_{\U}$ such that $\beta=$$^{*}f(\alpha)$; so $^{*}f[G_{\U}]=G_{\overline{f}(\U)}$.\\\end{proof}

\subsection{The $\sim_{u}$-preserving functions in $\mathtt{Fun}($$^{*}\N^{k},$$^{*}\N)$}

In this section we still assume that $^{*}\N$ is a hyperextension of $\N$ that satisfies the $\mathfrak{c}^{+}$-enlarging property, and we study which functions in $\mathtt{Fun}($$^{*}\N^{k},$$^{*}\N)$ can be naturally associated to functions in $\mathtt{Fun}(\beta(\N^{k}),\bN)$.\\
To explain the underlying idea, we pose $k=1$. Given a function $\varphi:$$^{*}\N\rightarrow$$^{*}\N$, we want to use the bridge map to construct a function $\widehat{\varphi}:\bN\rightarrow\bN$ such that, for every hypernatural number $\alpha\in$$^{*}\N$,

\begin{center} $\widehat{\varphi}(\mathfrak{U}_{\alpha})=\mathfrak{U}_{\varphi(\alpha)}$. \end{center}

Equivalently, 

\begin{center} $\hat{\varphi}=\psi\circ \varphi$, \end{center}

where $\psi$, as usual, denotes the bridge map.\\
We observe that $\hat{\varphi}$ is not well-defined for every function $\varphi\in\mathtt{Fun}($$^{*}\N,$$^{*}\N)$; the only functions such that the above construction can be done are the $\sim_{u}$-preserving:

\begin{defn} A function $\varphi$ in $\mathtt{Fun}($$^{*}\N^{k},$$^{*}\N)$ is {\bfseries $\sim_{u}$-preserving} if, for every $(\alpha_{1},...,\alpha_{k})\sim_{u}(\beta_{1},...,\beta_{k})$ in $^{*}\N^{k}$, $\varphi((\alpha_{1},...,\alpha_{k}))\sim_{u} \varphi((\beta_{1},...,\beta_{k}))$.\\
We denote by $\mathbb{P}_{k}$ the set of $\sim_{u}$-preserving functions in $\mathtt{Fun}($$^{*}\N^{k},$$^{*}\N)$. \end{defn}

\begin{thm} The following diagram commutes:
\begin{center} 

     \begin{picture}(80,50)(0,10) 

      \put(-25,50){\makebox(0,0){$\mathtt{Fun}(\N^{k},\N)$}}
      
      \put(155,50){\makebox(0,0){$\mathtt{Fun}(\beta(\N^{k}),\bN)$}}
      
      \put(50,5){\makebox(0,0){$\mathtt{Fun}($$^{*}\N^{k},$$^{*}\N)$}}   

      \put(23,28){\makebox(0,0){$*$}} 

      \put(90,28){\makebox(0,0){$\widehat{\cdot}$}} 

      \put(55,53){\makebox(0,0){$\overline{\cdot}$}}

      \put(5,47){\vector(3,-2){49}}

      \put(60,15){\vector(3,2){50}} 

      \put(5,50){\vector(1,0){105}} 

    \end{picture} 

      \end{center}

Equivalently, for every function $f$ in $\mathtt{Fun}(\N^{k},\N)$, $^{*}f$ is $\sim_{u}$-preserving and $\widehat{^{*}f}=\overline{f}$. \end{thm}

\begin{proof} Let $f$ be a function in $\mathtt{Fun}(\N^{k},\N)$. That $^{*}f$ is $\sim_{u}$ preserving, and $\overline{f}(\mathfrak{U}_{(\alpha_{1},...,\alpha_{k})})=\mathfrak{U}_{^{*}f(\alpha_{1},...,\alpha_{k})}$, is the content of point number one of Theorem 2.3.5.\\
As, for every $(\alpha_{1},...,\alpha_{k})\in$$^{*}\N^{k}$,

\begin{center} $\widehat{^{*}f}(\mathfrak{U}_{((\alpha_{1},...,\alpha_{k}))})=\mathfrak{U}_{^{*}f((\alpha_{1},...,\alpha_{k}))}=\overline{f}(\mathfrak{U}_{(\alpha_{1},...,\alpha_{k})})$, \end{center}

we have the thesis.\\\end{proof}

This theorem shows that the set $\{$$^{*}f\mid f\in\mathtt{Fun}(\N^{k},\N)\}$ of hyper-images of functions in $\mathtt{Fun}(\N^{k},\N)$ is included in $\mathbb{P}_{k}$. The reverse inclusione is false: e.g., let $\alpha$ be an infinite element in $^{*}\N$, and $\varphi$ the function such that, for every element $\beta$ in $^{*}\N$, $\varphi(\beta)=\alpha$. Then $\varphi$ is $\sim_{u}$-preserving, but it is not the hyper-image of a function in $\mathtt{Fun}(\N^{k},\N)$.\\
The association $\widehat{\cdot}$ is not 1-1. An interesting question is: given a function $f$ in $\mathtt{Fun}(\N^{k},\N)$, which functions $\varphi$ in $\mathtt{Fun}($$^{*}\N^{k},$$^{*}\N)$ satisfy the condition $\overline{f}=\widehat{\varphi}$?

\begin{defn} Let $f$ be a function in $\mathtt{Fun}(\N^{k},\N)$. A function $\varphi$ in $\mathtt{Fun}($$^{*}\N^{k},$$^{*}\N)$ is {\bfseries $\sim_{u}$-equal to $f$} if
\begin{enumerate}
	\item $\varphi\in\mathbb{P}_{k}$;
	\item $\widehat{\varphi}=\overline{f}$.
\end{enumerate}
\end{defn}

\begin{thm} Let $f$ be a function in $\mathtt{Fun}(\N^{k},\N)$ and $\varphi$ a function in $\mathtt{Fun}($$^{*}\N^{k},$$^{*}\N)$. The following two conditions are equivalent:
\begin{enumerate}
	\item $\varphi$ is $\sim_{u}$-equal to $f$;
	\item $\varphi^{-1}($$^{*}A)=$$^{*}(f^{-1}(A))$ for every subset $A$ of $\N$. 
\end{enumerate}
\end{thm}

\begin{proof} $(1)\Rightarrow (2)$ Suppose that $\varphi$ is $\sim_{u}$-equal to $f$, and let $A$ be a subset of $\N$. Since $\widehat{\varphi}=\overline{f}$, it follows that $\widehat{\varphi}^{-1}(\Theta_{A})=\overline{f}^{-1}(\Theta_{A})$.\\
In particular, given an element $(\alpha_{1},...,\alpha_{k})$ in $^{*}\N^{k}$, if $\U=\mathfrak{U}_{(\alpha_{1},...,\alpha_{k})}$ then $\U\in \widehat{\varphi}^{-1}(\Theta_{A})$ if and only if $\U\in \overline{f}^{-1}(\Theta_{A})$; this entails that $\widehat{\varphi}(\U)\in\Theta_{A}$ if and only if $\overline{f}(\U)\in\Theta_{A}$. Since $\widehat{\varphi}(\U)=\mathfrak{U}_{\varphi((\alpha_{1},...,\alpha_{k}))}$ and $\overline{f}(\U)=\mathfrak{U}_{^{*}f((\alpha_{1},...,\alpha_{k}))}$, then $\varphi((\alpha_{1},...,\alpha_{k}))\in$$^{*}A$ if and only if $^{*}f((\alpha_{1},...,\alpha_{k}))\in$$^{*}A$, so $(\alpha_{1},...,\alpha_{k})\in \varphi^{-1}($$^{*}A)$ if and only if $(\alpha_{1},...,\alpha_{k})\in$$^{*}f^{-1}($$^{*}A)=$$^{*}(f^{-1}(A))$. As this holds for every $(\alpha_{1},...,\alpha_{k})$ in $^{*}\N^{k}$ and every subset $A$ of $\N$, we get the thesis.\\
$(2)\Rightarrow(1)$ Suppose that $\varphi^{-1}($$^{*}A)=$$^{*}(f^{-1}(A))$ for every subset $A$ of $\N$.\\

{\bfseries Claim 1:} $\varphi$ is in $\sim_{u}$-preserving.\\

By contrast, suppose that $\varphi$ is not in $\mathbb{P}_{k}$. Then there are elements $(\alpha_{1},...,\alpha_{k})\sim_{u}(\beta_{1},...,\beta_{k})$ in $^{*}\N^{k}$ and a subset $A$ of $\N$ such that $\varphi((\alpha_{1},...,\alpha_{k}))\in$$^{*}A$ and $\varphi((\beta_{1},...,\beta_{k}))\notin$$^{*}A$. So 

\begin{center} $(\alpha_{1},...,\alpha_{k})\in \varphi^{-1}($$^{*}A)=$$^{*}(f^{-1}(A))$ and $(\beta_{1},...,\beta_{k})\notin$$^{*}(f^{-1}(A))$, \end{center} 

which is absurd since $(\alpha_{1},...,\alpha_{k})\sim_{u}(\beta_{1},...,\beta_{k})$ .\\

{\bfseries Claim 2:} $\widehat{\varphi}=\overline{f}$.\\

In fact, let $\U$ be an ultrafilter in $\beta(\N^{k})$, and $(\alpha_{1},...,\alpha_{k})$ a generator of $\U$. By hypothesis, for every subset $A$ of $\N$, $\varphi((\alpha_{1},...,\alpha_{n}))\in$$^{*}A$ if and only if $^{*}f((\alpha_{1},...,\alpha_{n}))\in$$^{*}A$, so $\varphi((\alpha_{1},...,\alpha_{n}))\sim_{u}$$^{*}{f}((\alpha_{1},...,\alpha_{n}))$, and since $\widehat{\varphi}(\U)=\mathfrak{U}_{\varphi((\alpha_{1},...,\alpha_{n}))}$ and $\overline{f}(\U)=\mathfrak{U}_{^{*}f((\alpha_{1},...,\alpha_{n}))}$, it follows that $\widehat{\varphi}=\overline{f}$.\\\end{proof}

\section{Tensor $k$-tuples}

In this section, we suppose that the hyperextension $^{*}\N$ satisfies the $\mathfrak{c}^{+}$-saturation property.\\
Usually, when working in $\bN$, we are interested in functions in $\mathtt{Fun}((\bN)^{k},\bN)$ rather than in functions in $\mathtt{Fun}(\beta(\N^{k}),\bN)$. E.g., we are interested to study the property of the sum $\oplus\in\mathtt{Fun}((\bN)^{2},\bN)$ of ultrafilters rather than in studying the unique continuos extension in $\mathtt{Fun}(\beta(\N^{2}),\bN)$ of the usual addition of natural numbers.\\
$\beta(\N^{k})$ and $(\bN)^{k}$ (that from now on we denote by $\bN^{k}$) are two different spaces. Nevertheless, $\bN^{k}$ can be identified with an important subset of $\beta(\N^{k})$: the subset of tensor products. We recall that, given ultrafilters $\U_{1},...,\U_{k}$ in $\bN$, $\U_{1}\otimes...\otimes\U_{k}$ is the ultrafilter on $\N^{k}$ defined by this condition: for every subset $A$ of $\N^{k}$,

\begin{center} $A\in\U_{1}\otimes...\otimes\U_{k}\Leftrightarrow$\\\vspace{0.3cm}$\Leftrightarrow\{n_{1}\in\N\mid\{n_{2}\in\N\mid...\mid\{n_{k}\in\N\mid (n_{1},...,n_{k})\in A\}\in\U_{k}\}....\in\U_{2}\}\in\U_{1}$.\end{center}

\begin{defn} We denote by $T_{k}$ the subset of $\beta(\N^{k})$ having as elements the tensor products of ultrafilters in $\bN$:

\begin{center} $T_{k}=\{\U\in\beta(\N^{k})\mid \exists \U_{1},...,\exists\U_{k}\in\bN$ with $\U=\U_{1}\otimes\U_{2}\otimes...\otimes\U_{k}\}$. \end{center}

\end{defn}

\begin{prop} If $k>1$ then $T_{k}$ is properly included in $\beta(\N^{k})$. \end{prop}

\begin{proof} In Proposition 1.1.17 we proved that the upper diagonal of $\N^{2}$ is an element of every tensor product $\U\otimes\V$ of ultrafilters. Similarly, it could be proved that the set 

\begin{center}$A=\{(n_{1},...,n_{k})\in\N^{k}\mid n_{i}\leq n_{k}$ for every $i\leq k\}$\end{center}

is an element of every tensor product in $\beta(\N^{k})$. So $T_{k}\subseteq \Theta_{A}$. As $\Theta_{A^{c}}\neq\emptyset$, it follows the thesis. \\ \end{proof}

\begin{thm} The map $\otimes_{k}:\bN^{k}\rightarrow T_{k}$ such that, $\forall (\U_{1},...,\U_{k})\in \bN^{k}$,

\begin{center} $\otimes_{k}((\U_{1},...,\U_{k}))=\U_{1}\otimes\U_{2}\otimes...\otimes\U_{k}$\end{center} is a bijection.\end{thm}

\begin{proof} The map is clearly surjective. To prove that $\otimes_{k}$ is 1-1 we observe that, given any subset $A$ of $\N$, the set 

\begin{center} $A_{i}=\{(n_{1},...,n_{k})\in\N^{k}\mid n_{i}\in A\}$\end{center}

is in $\U_{1}\otimes...\otimes\U_{k}$ if and only if $A\in\U_{i}$; observe also that $(A^{c})_{i}=(A_{i})^{c}$.\\
Suppose that $(\U_{1},...,\U_{k})$ and $(\V_{1},...,\V_{k})$ are two different elements in $\bN^{k}$ with $\U_{1}\otimes...\otimes\U_{k}=\V_{1}\otimes...\otimes\V_{k}$. For every subset $A$ of $\N$, for every index $i$, by hypothesis $A_{i}\in\U_{1}\otimes...\otimes\U_{k}$ if and only if $A_{i}\in\V_{1}\otimes...\otimes\V_{k}$, and this entails that $A\in \U_{i}$ if and only if $A\in \V_{i}$. Since this is true for every index $i$ and for every subset $A$ of $\N$, then $\U_{i}=\V_{i}$ for every index $i\leq k$. This proves that the map $\otimes$ is 1-1.\\  \end{proof}

To use the nonstandard techniques that we are introducing we need a characterization of the set of generators of $\U_{1}\otimes...\otimes\U_{k}$ in terms of $G_{\U_{1}},...,G_{\U_{k}}$. To semplify the tractation, and clarify the basic ideas, we pose $k=2$; at the end of the section, we give the formulation for a generical natural number $k$.\\
Let $\U,\V$ be ultrafilters on $\N$, and consider $\U\otimes\V$.

\begin{prop} For every $\U,\V$ ultrafilters on $\N$, $G_{\U\otimes\V}$ is subset of $G_{\U}\times G_{\V}$. In particular, if both $\U$ and $\V$ are nonprincipal then $G_{\U\otimes\V}$ is a proper subset of $G_{\U}\times G_{\V}$. \end{prop}

\begin{proof} Let $\U,\V$ be ultrafilters on $\N$. For every sets $A\in\U$, $B\in\V$, the cartesian product $A\times B\in\U\otimes\V$, so 

\begin{center}$\bigcap_{S\in\U\otimes\V}$$^{*}S\subseteq\bigcap_{A\in\U, B\in\V}($$^{*}A\times$$^{*}B)\subseteq G_{\U}\times G_{\V}$.\end{center}

When $\U$ and $\V$ are nonprincipal the inclusion is proper since, as we proved in Proposition 1.1.17, the upper diagonal of $\N^{2}$ is in every such tensor product; in particular, we get that for every $(\alpha,\beta)$ in $G_{\U\otimes\V}$, $\alpha\leq\beta$. But $G_{\V}$, as we proved in Proposition 2.2.6, is left unbounded: if $\alpha$ is in $G_{\U}$, in $G_{\V}$ there is an element $\beta<\alpha$, and $(\alpha,\beta)\in G_{\U}\times G_{\V}\setminus G_{\U\otimes\V}$.\\ \end{proof}

We give a special name to the pairs in $^{*}\N^{2}$ that generate tensor products: 

\begin{defn} Let $\alpha,\beta$ be two hypernatural numbers in $^{*}\N$. $(\alpha,\beta)$ is a {\bfseries tensor pair} if $\mathfrak{U}_{(\alpha,\beta)}=\mathfrak{U}_{\alpha}\otimes\mathfrak{U}_{\beta}$. \end{defn}

\begin{prop} For every natural number $n$ and every hypernatural number $\alpha$, $(\alpha,n)$ and $(n,\alpha)$ are tensor pairs.\end{prop}

\begin{proof} By definition, $A\in\mathfrak{U}_{(\alpha,n)}$ if and only if $(\alpha,n)\in$$^{*}A$ if and only if $\alpha\in\{\beta\in$$^{*}\N\mid (\beta,n)\in$$^{*}A\}=$$^{*}\{a\in\N\mid (a,n)\in A\}$ if and only if $\{a\in \N\mid \{b\in\N\mid (a,b)\in A\}\in\mathfrak{U}_{n}\}\in\mathfrak{U}_{\alpha}$ if and only if $A\in\mathfrak{U}_{\alpha}\otimes\mathfrak{U}_{n}$. Similarly for $(n,\alpha)$.\\\end{proof}

The problem becomes: given two nonprincipal ultrafilters $\U,\V$, how can we characterize the tensor pairs $(\alpha,\beta)$ that generate $\U\otimes\V$?

This problem has been solved by Christian W. Puritz in $\cite[\mbox{Theorem 3.4}]{rif29}$:

\begin{thm}[Puritz] Let $^{*}\N$ be a hyperextension of $\N$ with the $\mathfrak{c}^{+}$-enlarging property. For every ultrafilters $\U,\V$ on $\N$, 

\begin{center} $G_{\U\otimes\V}$=$\{(\alpha,\beta)\in$$^{*}\N^{2}\mid \alpha\in G_{\U}, \beta\in G_{\V}, \alpha<er(\beta)\}$, \end{center} 

where
 
\begin{center} $er(\beta)=\{$$^{*}f(\beta)\mid f\in \mathtt{Fun}(\N,\N),$$^{*}f(\beta)\in$$^{*}\N\setminus\N\}$. \end{center}
\end{thm}

We just warn the reader that we have presented this theorem with the modern notation $\U\otimes\V$ for tensor products, while Puritz uses the notation $\U\times\V$, which is not only graphically different from ours, but has also a different meaning: given ultrafilters $\U$, $\V$, by definition $\U\times\V=\V\otimes\U$.\\ 
For tensor pairs, this theorem tells that $(\alpha,\beta)$ is a tensor pair if and only if $\alpha<er(\beta)$. The problem is that, given two generical hypernatural numbers $\alpha,\beta$, it can be very difficult to decide if $\alpha<er(\beta)$ or not; it would be useful to have conditions equivalent to state that $(\alpha,\beta)$ is a tensor pair: the equivalences below are exposed in $\cite{rif15}$.

\begin{thm} If $\alpha,\beta$ are infinite hypernatural numbers, the following conditions are equivalent:
\begin{enumerate}
	\item $(\alpha,\beta)$ is a tensor pair;
	\item For every $A\subseteq \N^{2}$, if $(\alpha,\beta)\in$$^{*}A$ then there is a natural number $n$ with $(n,\beta)\in$$^{*}A$;
	\item For every $A\subseteq\N^{2}$, if $(n,\beta)\in$$^{*}A$ for every $n\in\N$, then $(\alpha,\beta)\in$$^{*}A$;
	\item $\alpha<er(\beta)$.
\end{enumerate}
\end{thm}

\begin{proof} First of all, that (4) is equivalent to (1) is the content of Puritz's Theorem; observe also that (2) is equivalent to (3), since the one is the contrapositive of the other applyed to $A^{c}$.\\
(1)$\Rightarrow$(3) Let $A$ be a subset of $\N^{2}$. By definition, given any natural number $n$, $(n,\beta)\in$$^{*}A$ if and only if $\beta\in\{\eta\in$ $^{*}\N\mid (n,\eta)\in$$^{*}A\}$ if and only if $\beta\in$$^{*}\{m\in\N\mid (n,m)\in A\}$ if and only if $\{m\in\N\mid (n,m)\in A\}\in \mathfrak{U}_{\beta}$.\\
Now suppose that, for every natural number $n$, $(n,\beta)\in$$^{*}A$. As we observed, this entails that for every natural number $n$ the set $\{m\in\N\mid (n,m)\in A\}$ is in $\mathfrak{U}_{\beta}$ so, in particular, the set 

\begin{center} $\{n\in\N\mid \{m\in\N\mid (n,m)\in A\}\in\mathfrak{U}_{\beta}\}$\end{center}

is $\N$, which is in $\mathfrak{U}_{\alpha}$, and this by definition implies that $A\in\mathfrak{U}_{\alpha}\otimes\mathfrak{U}_{\beta}$. But $(\alpha,\beta)$ is, by hypothesis, a tensor pair, so $\mathfrak{U}_{\alpha}\otimes\mathfrak{U}_{\beta}=\mathfrak{U}_{(\alpha,\beta)}$, and $(\alpha,\beta)\in$$^{*}A$.\\
(3)$\Rightarrow$(4) Let $f$ be a function in $\mathtt{Fun}(\N,\N)$ such that $^{*}f(\beta)$ infinite. Consider 

\begin{center} $A=\{(n,m)\in\N^{2}\mid n<f(m)\}$. \end{center}

Observe that, since $^{*}f(\beta)$ is infinite, for every natural number $n$ the pair $(n,\beta)$ is in $^{*}A=\{(\eta,\mu)\in$$^{*}\N^{2}\mid \eta<$$^{*}f(\mu)\}$. By hypothesis, we get that $(\alpha,\beta)\in$$^{*}A$, so $\alpha<$$^{*}f(\mu)$. Since this is true for every $f$ with $^{*}f(\beta)$ infinite, we get that $\alpha<er(\beta)$. \\ \end{proof}

We can generalize Puritz's result and Theorem 2.4.8 to $k$-tuples:

\begin{defn} An element $(\alpha_{1},...,\alpha_{k})$ in $^{*}\N^{k}$ is a {\bfseries tensor $k$-tuple} if $\mathfrak{U}_{(\alpha_{1},...,\alpha_{k})}=\mathfrak{U}_{\alpha_{1}}\otimes...\otimes\mathfrak{U}_{\alpha_{k}}$. \end{defn}

\begin{prop} For every natural number $k>1$, for every $\alpha_{1},...,\alpha_{k+1}$ in $^{*}\N$, the following two conditions are equivalent:
\begin{enumerate}
	\item $(\alpha_{1},....,\alpha_{k+1})$ is a tensor $(k+1)$-tuple;
	\item $(\alpha_{1},...,\alpha_{k})$ is a tensor $k$-tuple and $((\alpha_{1},...,\alpha_{k}),\alpha_{k+1})$ is a tensor pair.
\end{enumerate}
\end{prop}

\begin{proof} Observe that, via the function $f:\N^{k+1}\rightarrow\N^{k}\times\N$ that maps $(a_{1},...,a_{k},a_{k+1})$ in $((a_{1},...,a_{k}),a_{k+1})$, $\mathfrak{U}_{(\alpha_{1},...,\alpha_{k},\alpha_{k+1})}$ can be identified with $\mathfrak{U}_{((\alpha_{1},...,\alpha_{k}),\alpha_{k+1})}$.\\
So $(\alpha_{1},...,\alpha_{k},\alpha_{k+1})$ is a tensor $(k+1)$-tuple if and only if $\mathfrak{U}_{(\alpha_{1},...,\alpha_{k},\alpha_{k+1})}=\mathfrak{U}_{\alpha_{1}}\otimes...\otimes\mathfrak{U}_{\alpha_{k}}\otimes\mathfrak{U}_{\alpha_{k+1}}$ if and only if $\mathfrak{U}_{(\alpha_{1},...,\alpha_{k},\alpha_{k+1})}=\mathfrak{U}_{((\alpha_{1},....,\alpha_{k}),\alpha_{k+1})}=(\mathfrak{U}_{\alpha_{1}}\otimes...\otimes\mathfrak{U}_{\alpha_{k}})\otimes\mathfrak{U}_{\alpha_{k+1}}$ if and only if $(\alpha_{1},...,\alpha_{k})$ is a tensor $k$-tuple and $((\alpha_{1},...,\alpha_{k}),\alpha_{k+1})$ is a tensor pair. \\\end{proof}

As a consequence, a characterization of tensor pairs in the form $((\alpha_{1},...,\alpha_{k}),\alpha_{k+1})$ would give an inductive procedure to test if $(\alpha_{1},...,\alpha_{k+1})$ is a tensor $(k+1)$-tuple. Such a characterization can be obtained by generalizing Theorem 2.4.8:

\begin{thm} Given elements $\alpha_{1},...,\alpha_{k}$ in $^{*}\N$, $\alpha_{k+1}$ in $^{*}\N\setminus\N$, the following conditions are equivalent:

\begin{enumerate}
	\item $((\alpha_{1},...,\alpha_{k}),\alpha_{k+1})$ is a tensor pair;
	\item For every subset $A\subseteq\N^{k}\times \N$, if $((\alpha_{1},...,\alpha_{k}),\alpha_{k+1})\in$$^{*}A$ then there exists $(n_{1},...,n_{k})\in\N^{k}$ with $((n_{1},...,n_{k}),\alpha_{k+1})\in$$^{*}A$;
	\item For every subset $A\subseteq\N^{k}\times\N$, if for every natural numbers $n_{1},...,n_{k}$ $((n_{1},...,n_{k}),\alpha_{k+1})\in$$^{*}A$ then $((\alpha_{1},...,\alpha_{k}),\alpha_{k+1})\in$$^{*}A$;
	\item $\alpha_{i}\leq er(\alpha_{k+1})$ for every index $i\leq k$.
\end{enumerate}

\end{thm}

\begin{proof} Similarly to Theorem 2.4.8, we have that conditions (2) and (3) are equivalent, since the one is the contrapositive of the other applyed to $A^{c}$.\\
(1)$\Rightarrow$(3): Let $A$ be a subset of $\N^{k}\times\N$. By definition, given natural numbers $n_{1},...,n_{k}$, $((n_{1},...,n_{k}),\alpha_{k+1})\in$$^{*}A$ if and only if $\alpha_{k+1}\in\{\beta\in$ $^{*}\N\mid ((n_{1},...,n_{k}),\beta)\in$$^{*}A\}$ if and only if $\alpha_{k+1}\in$$^{*}\{m\in\N\mid ((n_{1},...,n_{k}),m)\in A\}$ if and only if $\{m\in\N\mid ((n_{1},...,n_{k}),m)\in A\}\in \mathfrak{U}_{\alpha_{k+1}}$.\\
Suppose that, for every natural numbers $n_{1},...,n_{k}$, $((n_{1},...,n_{k}),\alpha_{k+1})\in$$^{*}A$. As we observed, this entails that for every natural numbers $n_{1},...,n_{k}$ the set $\{m\in\N\mid ((n_{1},...,n_{k}),m)\in A\}$ is in $\mathfrak{U}_{\alpha_{k+1}}$ so, in particular, 

\begin{center} $\{(n_{1},...,n_{k})\in\N^{k}\mid \{m\in\N\mid ((n_{1},...,n_{k}),m)\in A\}\in\mathfrak{U}_{\alpha_{k+1}}\}=\N^{k}$,\end{center}

so this set is in $\mathfrak{U}_{(\alpha_{1},...,\alpha_{k})}$, and this by definition implies that

\begin{center} $A\in\mathfrak{U}_{(\alpha_{1},...,\alpha_{k})}\otimes\mathfrak{U}_{\alpha_{k+1}}$. \end{center}

But $((\alpha_{1},...,\alpha_{k}),\alpha_{k+1})$ is, by hypothesis, a tensor pair, so

\begin{center} $\mathfrak{U}_{(\alpha_{1},...,\alpha_{k})}\otimes\mathfrak{U}_{\alpha_{k+1}}=\mathfrak{U}_{((\alpha_{1},...,\alpha_{k}),\alpha_{k+1})}$, \end{center}

and $((\alpha_{1},...,\alpha_{k}),\alpha_{k+1})\in$$^{*}A$.\\
(3)$\Rightarrow$(4) Let $f$ be a function in $\mathtt{Fun}(\N^{k},\N)$ such that $^{*}f(\alpha_{k+1})$ is infinite. Consider 

\begin{center} $A=\{((n_{1},...,n_{k}),m)\in\N^{k}\times\N\mid n_{i}<f(m)$ for every $i\leq k\}$. \end{center}

Observe that, as $^{*}f(\alpha_{k+1})$ is infinite, for every natural numbers $n_{1},...,n_{k}$ the pair $((n_{1},...,n_{k}),\alpha_{k+1})$ is in

\begin{center} $^{*}A=\{((\eta_{1},...,\eta_{k}),\mu)\in$$^{*}\N^{k}\times$$^{*}\N\mid \eta_{i}<$$^{*}f(\mu)$ for every $i\leq k\}$. \end{center}

By hypothesis, $((\alpha_{1},...,\alpha_{k}),\alpha_{k+1})\in$$^{*}A$, so $\alpha_{i}<$$^{*}f(\alpha_{k+1})$ for every $i\leq k$. Since this holds for every function $f$ in $\mathtt{Fun}(\N^{k},\N)$ with $^{*}f(\alpha_{k+1})$ infinite, it follows that $\alpha_{i}<er(\alpha_{k+1})$ for every index $i\leq k$.\\
(4)$\Rightarrow$(1): This proof follows the original ideas of Puritz.\\
Consider $\U=\mathfrak{U}_{(\alpha_{1},...,\alpha_{k})}\otimes\mathfrak{U}_{\alpha_{k+1}}$, and let $A$ be any set in $\U$. For every $n_{1},...,n_{k}$ in $\N$, define

\begin{center} $A_{(n_{1},...,n_{k})}=\{m\in\N\mid ((n_{1},...,n_{k}),m)\in A\}$ \end{center}

and 

\begin{center}$A^{\prime}_{(n_{1},...,n_{k})}=A_{(n_{1},...,n_{k})}\setminus [0,\max\{n_{1},...,n_{k}\}]$. \end{center}

Consider the function $f:\N\rightarrow\N$ such that, for every natural number $m$,

\begin{center} $f(m)=\min\{n\in\N\mid\exists n_{1},...,n_{k}$ with $n_{i}\leq n$ for every $i\leq k$, $A_{(n_{1},...,n_{k})}\in\mathfrak{U}_{\alpha_{k+1}}$ and $m\notin A_{(n_{1},..,n_{k})}\setminus[0,n]\}$. \end{center}

The definition is well-posed because, since $A$ is in $\U$, there are arbitrarily large natural numbers $n$ (in particular, $n>m$) such that $A_{(n_{1},...,n_{k})}\in\mathfrak{U}_{\alpha_{k+1}}$ for some $n_{1},...n_{k}$ with $n_{i}\leq n$ for every index $i\leq k$.\\
Observe that $^{*}f(\alpha_{k+1})$ is infinite: in fact, suppose that $^{*}f(\alpha_{k+1})=n$ for some finite natural number $n$; by definition, there are $n_{1},...,n_{k}$ with $n_{i}\leq n$ for every index $i\leq k$ and $A_{(n_{1},...,n_{k})}\in\mathfrak{U}_{\alpha_{k+1}}$, which entails $\alpha_{k+1}$ in $^{*}A_{(n_{1},...,n_{k})}$ and, as $\alpha_{k+1}$ is infinite, $\alpha_{k+1}\in$$^{*}A_{(n_{1},...,n_{k})}\setminus[0,n]$, so $^{*}f(\alpha_{k+1})\neq n$, absurd.\\
By hypothesis, $\alpha_{i}<$$^{*}f(\alpha_{k+1})$ for every index $i\leq k$. Pose

\begin{center} $A_{0}=\{(n_{1},...,n_{k})\in\N^{k}\mid A_{(n_{1},...,n_{k})}\in\mathfrak{U}_{\alpha_{k+1}}\}$. \end{center}

As $A$ is a set in $\U$, $A_{0}\in A_{(\alpha_{1},...,\alpha_{k})}$ so $(\alpha_{1},...,\alpha_{k})\in$$^{*}A_{0}=\{(\beta_{1},...,\beta_{k})\in$$^{*}\N^{k}\mid A_{(\beta_{1},...,\beta_{k})}\in$$^{*}\mathfrak{U}_{\alpha_{k+1}}\}$. So 

\begin{center} $A_{(\alpha_{1},...,\alpha_{k})}=\{\mu\in$$^{*}\N\mid ((\alpha_{1},...,\alpha_{k}),\mu)\in$$^{*}A\}\in$$^{*}\mathfrak{U}_{\alpha_{k+1}}$. \end{center}

For every index $i\leq k$, $\alpha_{i}<$$^{*}f(\alpha_{k+1})$, and 

\begin{center} $^{*}f(\alpha_{k+1})=\min\{\eta\mid\exists \rho_{1},...,\rho_{k}$ with $\rho_{i}\leq \eta$ for every $i\leq k$, $A_{(\rho_{1},...,\rho_{k})}\in$$^{*}\mathfrak{U}_{\alpha_{k+1}}$ and $\alpha_{k+1}\notin A_{\rho_{1},...,\rho_{k}}\setminus [0,\eta]\}$.\end{center}

Since, as we showed, $A_{(\alpha_{1},...,\alpha_{k})}\in$$^{*}\mathfrak{U}_{\alpha_{k+1}}$, necessarily 

\begin{center} $\alpha_{k+1}\in A_{(\alpha_{1},...,\alpha_{k})}\setminus [0,\max\{\alpha_{1},...,\alpha_{k}\}]$,\end{center}

so $((\alpha_{1},...,\alpha_{k}),\alpha_{k+1})\in$$^{*}A$. Since this is true for every set $A$ in $\U$, it follows that $\mathfrak{U}_{(\alpha_{1},...,\alpha_{k})}\otimes\mathfrak{U}_{\alpha_{k+1}}=\mathfrak{U}_{((\alpha_{1},...,\alpha_{k}),\alpha_{k+1})}$, so $((\alpha_{1},...,\alpha_{k}),\alpha_{k+1})$ is a tensor pair.\\ \end{proof}

\begin{cor} For every natural number $k\geq 1$, for every $\alpha_{1},...,\alpha_{k+1}$ in $^{*}\N$, the following two conditions are equivalent:

\begin{enumerate}
	\item $(\alpha_{1},...,\alpha_{k+1})$ is a tensor $(k+1)$-tuple;
	\item $\alpha_{i} < er (\alpha_{i+1})$ for every index $i\leq k$.
\end{enumerate}
\end{cor}

\begin{proof} We proceed by induction on $k$. If $k=1$, this is Puritz's Theorem.\\
(1)$\Rightarrow(2)$ Pose $k=h+1$, and suppose that $(\alpha_{1},...,\alpha_{k+1})$ is a tensor $(k+1)$-tuple. By Proposition 2.4.10 it follows that $(\alpha_{1},....,\alpha_{k})$ is a tensor $k$-tuple and $((\alpha_{1},...,\alpha_{k}),\alpha_{k+1})$ is a tensor pair. By inductive hypothesis, $\alpha_{i}< er(\alpha_{i+1})$ for every index $i\leq k-1$, and by Theorem 2.4.11 it follows that $\alpha_{i}< er(\alpha_{k+1})$ for every index $i\leq k$; in particular $\alpha_{k} < er(\alpha_{k+1})$.\\
(2)$\Rightarrow(1)$ Suppose that $\alpha_{i}<er(\alpha_{i+1})$ for every index $i\leq k$. By Theorem 2.4.11 it follows that $((\alpha_{1},...,\alpha_{k}),\alpha_{k+1})$ is a tensor pair and, by inductive hypothesis, that $(\alpha_{1},...,\alpha_{k})$ is a tensor $k$-tuple. So, by Proposition 2.4.10, $(\alpha_{1},...,\alpha_{k+1})$ is a tensor $(k+1)$-tuple. \\ \end{proof}

The above theorem, similarly to Theorem 2.4.8, gives four equivalent characterizations for the tensor $(k+1)$-tuples but, in some sense, it has a weakness: even with these new characterizations, it is difficult, given $k+1$ elements $\alpha_{1},....,\alpha_{k+1}$ in $^{*}\N$, to decide if $(\alpha_{1},...,\alpha_{k+1})$ is a tensor $(k+1)$-tuple or not.\\
We are led by this observation to ask if it is possible to find a relation $R$ on $^{*}\N^{k+1}$ with these two properties:
\begin{enumerate}
	\item if $(\alpha_{1},...,\alpha_{k+1})\in R$ then $(\alpha_{1},...,\alpha_{k+1})$ is a tensor $(k+1)$-tuple, and
	\item given $\alpha_{1},....,\alpha_{k+1}$ it is simple to decide if $(\alpha_{1},...,\alpha_{k+1})\in R$ or not.
\end{enumerate}
This question can be solved by considering a hyperextension of $\N$ that has a particular property: it allows the iteration of the star map.

\section{The Nonstandard Structure $^{\bullet}\N$}

Let $\alpha,\beta$ be infinite hypernatural numbers, and consider the ultrafilter $\U=\mathfrak{U}_{\alpha}\otimes\mathfrak{U}_{\beta}$. By definition, a subset $A$ of $\N^{2}$ is in $\U$ if and only if the set 

\begin{center} $\{n\in\N\mid \{m\in\N\mid (n,m)\in A\}\in\mathfrak{U}_{\beta}\}$ \end{center}

is in $\mathfrak{U}_{\alpha}$. Since $\alpha,\beta$ are generators of $\mathfrak{U}_{\alpha},\mathfrak{U}_{\beta}$, this condition holds if and only if 

\begin{center} $(\dagger)$ $\alpha\in$$^{*}\{n\in\N\mid \beta\in$$^{*}\{m\in\N\mid (n,m)\in A\}\}$.\end{center}

As, by transfer, $^{*}\{m\in\N\mid (n,m)\in A\}=\{\eta\in$$^{*}\N\mid (n,\eta)\in$$^{*}A\}$, $(\dagger)$ can be rewritten in this way:

\begin{center} $\alpha\in$$^{*}\{n\in\N\mid (n,\beta)\in$$^{*}A\}$. \end{center}

Here, we are tempted to continue in this way: by transfer

\begin{center} $^{*}\{n\in\N\mid (n,\beta)\in$$^{*}A\}=\{\eta\in$$^{*}\N\mid (\eta,$$^{*}\beta)\in$$^{**}A\}$ \end{center}

so, as $\alpha$ is an element of this set, it follows that

\begin{center}$A\in\mathfrak{U}_{\alpha}\otimes\mathfrak{U}_{\beta}$ if and only if $(\alpha,$$^{*}\beta)\in$$^{**}A$. \end{center}

There are two problems: if $\beta\in$$^{*}\N\setminus \N$, what is the meaning of $^{*}\beta$? And what is $^{**}A$? Recall that, in the superstructure approach that we adopt, the star map $*$ goes from a superstructure $\mathbb{V}(X)$ to one other superstructure $\mathbb{V}(Y)$: $\alpha,\beta$ are elements in $\mathbb{V}(Y)$, and we cannot apply the star map $*$ to them. This would be possible if $\mathbb{V}(X)=\mathbb{V}(Y)$: we would like to work with a superstructure model $\langle \mathbb{V}(X), \mathbb{V}(X), *\rangle$ of nonstandard methods where the star map goes from a superstructure $\mathbb{V}(X)$ to itself.\\
A natural question arises: is such a construction possible? The answer is affirmative; an example is given in the article $\cite{rif3}$ by Vieri Benci, with a construction that was motivated by the Alpha Theory; another possibility is given by the nonstandard set theory {\itshape $^{*}$ZFC} by Mauro di Nasso (see $\cite{rif13}$), where the enlarging map $*$ is defined for every set of the universe. The theory $^{*}$ZFC is equiconsistent with ZFC.\\
In this section, we fix a superstructure model of nonstandard methods in the form $\langle \mathbb{V}(X), \mathbb{V}(X), *\rangle$, and we study a few of its properties, with a particular interest for the relations between this kind of superstructures and the bridge map.

\subsection{Star iterations}

Throughout this section, we fix a single superstructure model of nonstandard methods 

\begin{center}$\langle \mathbb{V}(X), \mathbb{V}(X), *\rangle$.\end{center}

Since $*$ is, by definition, a function that maps $\mathbb{V}(X)$ into $\mathbb{V}(X)$, it is natural to ask what happen if one iterates this function.

\begin{defn} We define inductively the family $\langle S_{n}\mid n\in\N\rangle$ of functions $S_{n}:\mathbb{V}(X)\rightarrow\mathbb{V}(X)$ posing

\begin{center} $S_{0}=id;$ \end{center}

and, for $n\geq 0$, 

\begin{center} $S_{n+1}=*\circ S_{n}$. \end{center}
\end{defn}

We make this convention: if $y$ is any object in $\mathbb{V}(X)$, for every natural number $n$ the notation

\begin{center}$^{\overset{n}{\overbrace{*...*}}}y$ \end{center}

is equivalent to $S_{n}(y)$, and it is used when the natural number $n$ is "small": e.g., if $\alpha$ is an hypernatural number, we usually denote $S_{1}(\alpha), S_{2}(\alpha)$ by $^{*}\alpha$, $^{**}\alpha$, respectively.\\
In the following proposition, if $\varphi(x_{1},...,x_{n})$ is a first order formula with parameters $p_{1},...,p_{k}$, for every natural number $n$ the formula $\varphi_{n}(x_{1},...,x_{n})$ is obtained by $\varphi(x_{1},...,x_{n})$ by substituting each parameter $p_{i}$ with $S_{n}(p_{i})$.

\begin{thm} For every positive natural number $n$, $\langle \mathbb{V}(X),\mathbb{V}(X),S_{n}\rangle$ is a superstructure model of nonstandard methods.\end{thm}

\begin{proof} By induction on $n$; the case $n=1$ holds, as $\langle \mathbb{V}(X), \mathbb{V}(X), *\rangle$ is a superstructure model of nonstandard methods.\\
Suppose $n=m+1$. By definition of superstructure model of nonstandard methods, we have to prove that $S_{n}(X)=X$ and that $S_{n}$ is a proper star map with the transfer property.\\
1) $S_{n}(X)=X$: by induction we know that $S_{m}(X)=X$ so, by transfer property we get that $^{*}$$(S_{m}(X))=$$^{*}X$. As $^{*}$$(S_{m}(X))=$$S_{n}(X)$ by definition, and $^{*}X=X$ since $\langle \mathbb{V}(X),\mathbb{V}(X),*\rangle$ is a superstructure model of nonstandard methods, it follows that $S_{n}(X)=X$.\\
2) $S_{n}$ is a proper star map: let $A$ be an infinite subset of $X$. Consider 

\begin{center} $^{\sigma_{n}}A=\{S_{n}(a)\mid a\in A\}$. \end{center}

Pose $B=$$^{\sigma_{m}}A$. By construction, $^{\sigma_{n}}A=$$^{\sigma}B$ and, since $*$ is proper, $^{\sigma}B$ is a proper subset of $^{*}B$. Also, by inductive hypothesis, $B=$$^{\sigma_{m}}A$ is a proper subset of $S_{m}(A)$ so, by transfer property (of $*$), it follows that $^{*}B$ is a proper subset of $S_{m}(A)$.\\
3) $S_{n}$ satisfies the transfer property: let $\varphi(x_{1},...,x_{n})$ be a bounded quantifier formula. Since $*$ satisfies the transfer property, we know that, for every $a_{1},...,a_{k}$ in $X$, 

\begin{center} $\mathbb{V}(X)\models\varphi(a_{1},...,a_{k})\Leftrightarrow\mathbb{V}(X)\models$$^{*}\varphi($$^{*}a_{1},...,$$^{*}a_{k})$. (1) \end{center}

By inductive hypothesis, $*_{m}$ satisfies the transfer property, so for every $a_{1},...,a_{k}$ in $X$,

\begin{center} $\mathbb{V}(X)\models\varphi(a_{1},...,a_{k})\Leftrightarrow\mathbb{V}(X)\models\varphi_{m}(S_{m}(a_{1}),...,S_{m}(a_{k}))$. (2) \end{center}

Now, take any $b_{1},...,b_{k}$ in $X$. By (1) it follows that $\mathbb{V}(X)\models\varphi(b_{1},...,b_{k})\Leftrightarrow\mathbb{V}(X)\models$$^{*}\varphi($$^{*}b_{1},...,$$^{*}b_{k})$. By (2), were we take $a_{i}=$$^{*}b_{i}$ for every $i\leq k$, we get that $\mathbb{V}(X)\models$$^{*}\varphi($$^{*}b_{1},...,$$^{*}b_{k})\Leftrightarrow\mathbb{V}(X)\models\varphi_{(m+1)}($$S_{m}($$^{*}b_{1}),...,S_{m}($$^{*}b_{k}))$, so we conclude that

\begin{center} $\mathbb{V}(X)\models\varphi(b_{1},...,b_{k})\Leftrightarrow\mathbb{V}(X)\models\varphi_{n}($$S_{n}(b_{1}),...,S_{n}(b_{k}))$. \end{center}
\end{proof}

From now on, we focus on some particular properties of $\N$ and of the hyperextensions $S_{n}(\N)$. Of course since, as we proved, $\langle \mathbb{V}(X),\mathbb{V}(X),S_{n}\rangle$ is a superstructure model of nonstandard methods, $S_{n}(\N)$ has all the generical properties of the hyperextensions of $\N$. The particularity, in this context, are the relations between different hyperextensions $S_{n}(\N), S_{m}(\N)$ for different natural numbers $n,m$.

\begin{prop} Let $n\leq m$ be natural numbers, and $A$ a subset of $\N$. Then
\begin{enumerate}
  \item For every natural number $a$, $S_{n}(a)=a$;
  \item $S_{n}(A)\subseteq S_{m}(A)$, and the inclusion is proper if and only if $A$ is infinite;
  \item If $\alpha$ is an hypernatural number in $S_{n}(A)$ then, for every natural number $k$, $S_{k}(\alpha)\in S_{(n+k)}(A)$;
  \item $S_{n}(A)=S_{m}(A)\cap S_{n}(\N)$;
	\item $S_{n+1}(\N)$ is an end extension of $S_{n}(\N)$: $\forall\alpha\in S_{n}(\N), \forall\beta\in S_{n+1}(\N)\setminus S_{n}(\N)$, $\alpha<\beta$.
\end{enumerate}

\end{prop}

\begin{proof} 1) This property holds in every superstructure model of nonstandard methods.\\
2) To prove this result we show that it holds if $m=n+1$, as this clearly entails the thesis.\\
Suppose that $m=n+1$. That $S_{n}(A)\subseteq S_{n+1}(A)$ holds since the map $S_{n}$ satisfies the transfer property: in fact, for every subset $A$ on $\N$, $A\subseteq$$^{*}A$ so, by transfer property (of $S_{n}$), $S_{n}(A)\subseteq S_{n}($$^{*}A)=S_{n+1}(A)$.\\
Observe that, as a consequence of (1), if $A$ is finite then $S_{n}(A)=A$, so if the inclusion $S_{n}(A)\subset S_{n+1}(A)$ is proper then $A$ is infinite. Conversely, if $A$ is infinite, then the inclusion $A\subset$$^{*}A$ is proper and, by transfer property of $S_{n}$, it follows that the inclusion $S_{n}(A)\subset S_{n}($$^{*}A)=S_{n+1}(A)$ is proper. \\
3) This follows since the map $S_{k}$ satisfies the transfer property: in fact, by hypothesis $\alpha\in S_{n}(A)$ so, by transfer, $S_{k}(\alpha)\in S_{k}(S_{n}(A))$, and $S_{k}(S_{n}(A))=S_{(n+k)}(A)$.\\
4) $S_{n}(A)$ is a subset of $S_{m}(A)$ by (2), and it is a subset of $S_{n}(\N)$ by transfer (of $S_{n}$), as $A\subseteq\N$. So $S_{n}(A)\subseteq S_{m}(A)\cap S_{n}(\N)$. As for the reverse inclusion, if $\alpha$ is an element in $S_{m}(A)\cap S_{n}(\N)$ and $\alpha\in S_{n}(A^{c})$, then by (3) it follows that

\begin{center} $S_{(m-n)}(\alpha)\in S_{(m-n)}(S_{n}(A^{c}))=S_{m}(A^{c})$, \end{center}

and this is absurd.\\
5) This follows by transfer property (of $S_{n})$: we know that

\begin{center} $\forall n\in\N, \forall\eta\in$$^{*}\N\setminus\N$, $n<\eta$, \end{center}

so by transfer property it follows that 

\begin{center} $\forall \alpha\in S_{n}(\N),\forall \beta\in S_{n}($$^{*}\N\setminus\N)$, $\alpha<\beta$, \end{center}

and the conclusion follows as $S_{n}($$^{*}\N\setminus\N)=S_{n+1}(\N)\setminus S_{n}(\N)$.\\\end{proof}

We observe that, if $A$ is not a subset of $\N$ but it is a generical subset of $S_{n}(\N)$ for some natural number $n\geq 1$, the properties of Proposition 2.5.3 do not necessarily hold: in fact, e.g., if $\alpha$ is an infinite hypernatural number in $^{*}\N$, and $A=\{\alpha\}$, then $^{*}A=\{$$^{*}\alpha\}$, so $A$ is not a subset of $^{*}A$ (as $^{*}\alpha\in$$^{**}\N\setminus$$^{*}\N$, so $\alpha<$$^{*}\alpha$).

\begin{defn} Let $\langle\mathbb{V}(X),\mathbb{V}(X),*\rangle$ be a superstructure model of nonstandard methods. We call {\bfseries $\omega$-hyperextension} of $\N$, and denote by $^{\bullet}\N$, the union of all hyperextensions $S_{n}(\N)$:

\begin{center} $^{\bullet}\N=\bigcup_{n\in\mathbb{N}} S_{n}(\N)$. \end{center}

\end{defn}

In particular, for every natural number $n$, $S_{n}(\N)$ is a subset of $^{\bullet}\N$. Actually, since 

\begin{center} $\langle S_{n}(\N)\mid n<\omega\rangle$ \end{center}

is an elementary chain of models, it follows that:

\begin{thm} For every natural number $n$, $S_{n}(\N)$ is an elementary submodel of $^{\bullet}\N$. \end{thm}

This theorem is a particular case of the Elementary Chain Theorem, see e.g. $\cite[\mbox{Theorem 3.1.9}]{rif7}$. As a consequence, it follows that $^{\bullet}\N$ is an hyperextension of $\N$.\\
Observe that, by definition of $^{\bullet}\N$, Proposition 2.5.3 can be reformulated in this way:

\begin{prop} Let $n$ be a natural number, and $A$ a subset of $\N$. Then
\begin{enumerate}
  \item For every natural number $a$, $^{\bullet}a=a$;
  \item $S_{n}(A)\subseteq $$^{\bullet}A$, and the inclusion is proper if and only if $A$ is infinite;
  \item For every hypernatural number $\alpha$ in $^{\bullet}\N$, $\alpha\in$$^{\bullet}A\Leftrightarrow S_{n}(\alpha)\in$$^{\bullet}A$
  \item $S_{n}(A)=$$^{\bullet}A\cap S_{n}(\N)$;
	\item $^{\bullet}\N$ is an end extension of $S_{n}(\N)$: $\forall\alpha\in S_{n}(\N), \forall\beta\in$$^{\bullet}\N\setminus S_{n}(\N)$, $\alpha<\beta$;
	\end{enumerate}

\end{prop}

Our aim is to study the bridge map $\psi:$$^{\bullet}\N\rightarrow\bN$. As we saw in Section 2.2.1, the definition of the bridge map requires that the associated hyperextension has, at least, the $\mathfrak{c}^{+}$-enlarging property. The question is: does $^{\bullet}\N$ satisfy the $\mathfrak{c}^{+}$-enlarging property?

\begin{prop} For every natural number $n\geq 1$ the implications $(1)\Rightarrow (2)\Rightarrow (3)$ hold, where
\begin{enumerate}
	\item $^{*}\N$ has the $\mathfrak{c}^{+}$-enlarging property;
	\item $S_{n}(\N)$, seen as a hyperextension of $\N$, has the $\mathfrak{c}^{+}$-enlarging property;
	\item $^{\bullet}\N$, seen as a hyperextension of $\N$, has the $\mathfrak{c}^{+}$-enlarging property.
\end{enumerate}
\end{prop}

\begin{proof} In the proof, $\mathcal{F}$ denotes a family of subsets of $\N$ with the finite intersection property.\\
 $(1)\Rightarrow (2)$: For every set $F\in\mathcal{F}$, as we proved in Proposition 2.5.3, since $F\subseteq\N$ then $^{*}F\subseteq$$S_{n}(F)$. In particular

\begin{center} $\bigcap_{F\in\mathcal{F}}$$^{*}F\subseteq\bigcap_{F\in\mathcal{F}}$$S_{n}(F)$.\end{center}

Since $\bigcap_{F\in\mathcal{F}}$$^{*}F$ is nonempty by hypothesis, it follows that $\bigcap_{F\in\mathcal{F}}$$S_{n}(F)$ is nonempty, so $S_{n}(\N)$ has the $\mathfrak{c}^{+}$-enlarging property.\\
$(2)\Rightarrow (3)$:  For every set $F\in\mathcal{F}$, as $F\subseteq\N$ by Proposition 2.5.6 it follows that $S_{n}(F)\subseteq$$^{\bullet}F$, so

\begin{center} $\bigcap_{F\in\mathcal{F}}S_{n}(F)\subseteq\bigcap_{F\in\mathcal{F}}$$^{\bullet}F$. \end{center}
Since $\bigcap_{F\in\mathcal{F}} S_{n}(F)$ is nonempty by hypothesis, it follows that $\bigcap_{F\in\mathcal{F}}$$^{\bullet}(F)$ is nonempty, so $^{\bullet}\N$ has the $\mathfrak{c}^{+}$-enlarging property.
\\ \end{proof}

The specification "seen as a hyperextension of $\N$" has been pointed out since, e.g., $^{**}\N$ is a hyperextension of $\N$ and a hyperextension of $^{*}\N$, and asking if $^{**}\N$ has the $\mathfrak{c}^{+}$-enlarging property with respect to families of subsets of $\N$ is different to ask if $^{**}\N$ has the $\mathfrak{c}^{+}$-enlarging property with respect to families of subsets of $^{*}\N$.\\
Also, we observe that this result does not hold for the $\mathfrak{c}^{+}$-saturation property, as a consequence of the following fact:\\

{\bfseries Fact:} $^{\bullet}\N$ has cofinality $\aleph_{0}$.\\

In fact, a countable right unbounded sequence in $^{\bullet}\N$ can be constructed choosing, for every natural number $n$, an hypernatural number $\alpha_{n}$ in $S_{n+1}(\N)\setminus S_{n}(\N)$. Since $^{\bullet}\N$ has cofinality $\aleph_{0}$, it can not be $\mathfrak{c}^{+}$-saturated: if $\langle\alpha_{n}\mid n\in\N\rangle$ is the countable sequence previously introduced, and for every natural number $n$ we pose

\begin{center} $I_{n}=\{\eta\in$$^{\bullet}\N\mid \eta\geq\alpha_{n}\}$, \end{center}

the family

\begin{center} $\langle I_{n}\mid n\in\N\rangle$ \end{center}

is a countable family of internal subsets of $^{\bullet}\N$ and $\bigcap_{n\in\N} I_{n}=\emptyset$. In particular, $^{\bullet}\N$ is not $\mathfrak{c}^{+}$-saturated.\\
As for the coinitiality of $^{\bullet}\N\setminus\N$, we have:

\begin{prop} If the map $*$ satisfies the $\mathfrak{c}^{+}$-saturation property, the coinitiality of $^{\bullet}\N\setminus\N$ is at least $\mathfrak{c}^{+}$. \end{prop}

\begin{proof} Observe that, by construction, $^{*}\N\setminus\N$ is an initial segment of $^{\bullet}\N\setminus\N$, so the coinitiality of $^{\bullet}\N\setminus\N$ is equal to the coinitiality of $^{*}\N\setminus\N$ which, if $*$ satisfies the $\mathfrak{c}^{+}$-saturation property, is at least $\mathfrak{c}^{+}$. \\ \end{proof}

The structure of $^{\bullet}\N$ leads to introduce the following concept:

\begin{defn}[] Let $n\geq 1$ be a natural number, and $(\alpha_{1},....,\alpha_{n})$ be an element of $^{\bullet}\N^{n}\setminus \N^{n}$. The {\bfseries height} of $(\alpha_{1},....,\alpha_{n})$ $($notation $h((\alpha_{1},...,\alpha_{n})) )$ is the least natural number $m$ such that $(\alpha_{1},....,\alpha_{n})\in S_{m}(\N^{n})$.\end{defn}

Observe that, by definition, 

\begin{center} $h((\alpha_{1},...,\alpha_{n}))=m\Leftrightarrow (\alpha_{1},...,\alpha_{n})\in S_{m}(\N^{n})\setminus S_{m-1}(\N^{n})$. \end{center}

\begin{prop} For every natural number $n$, for every hypernatural numbers $\alpha,\alpha_{1},...,\alpha_{n}$ in $^{\bullet}\N\setminus \N$, for every subset $A$ of $\N$, for every function $f$ in $\mathtt{Fun}(\N,\N)$ the following properties hold:
\begin{enumerate}
	\item $h($$^{*}\alpha)=h(\alpha)+1$;
	\item $h(S_{n}(\alpha))=h(\alpha)+n$;
	\item $h(\alpha\cdot\beta)=h(\alpha+\beta)=h(\alpha^{\beta})=\max\{h(\alpha),h(\beta)\}$;
	\item $h((\alpha_{1},\alpha_{2},...,\alpha_{n}))=\max\{h(\alpha_{i})\mid i\leq n\}$;
	\item $h($$^{\bullet}f(\alpha))\leq h(\alpha)$;
	\item $\alpha\in$$^{\bullet}A\Leftrightarrow \alpha\in S_{h(\alpha)}(A)$.
\end{enumerate}

\end{prop}

\begin{proof} 1) Observe that $m=h(\alpha)\Leftrightarrow \alpha\in S_{m}(\N)\setminus S_{m-1}(\N) \Leftrightarrow $$^{*}\alpha\in S_{m+1}(\N)\setminus S_{m}(\N)\Leftrightarrow h($$^{*}\alpha)=m+1$.\\
2) It trivially follows by induction by (1).\\
3) This follows by observing that, for every $\alpha,\beta$ in $^{\bullet}\N$, for every natural number $n$, $\alpha+\beta\in S_{n}(\N)$ if and only if $\alpha\cdot\beta\in S_{n}(\N)$ if and only if $\alpha^{\beta}\in S_{n}(\N)$ if and only if both $\alpha,\beta$ are in $S_{n}(\N)$.\\
4) This follows by observing that, for every natural number $m\geq 1$, $(\alpha_{1},...,\alpha_{n})\in S_{m}(\N^{n})$ if and only if, for every index $i\leq n$, $\alpha_{i}\in S_{m}(\N)$. \\
5) For every natural number $m$, by definition 

\begin{center} $^{\bullet}f_{\upharpoonright_{S_{m}(\N)}}=S_{m}(f)$.\end{center}
And $S_{m}(f)\in\mathtt{Fun}(S_{m}(\N),S_{m}(\N))$ so, if $h(\alpha)=m$, then $h($$^{\bullet}f(\alpha))=h(S_{m}(f)(\alpha))\leq m$.\\
6) If $\alpha\in S_{h(\alpha)}(A)$ then, as $S_{h(\alpha)}(A)\subseteq$$^{\bullet}A$, $\alpha\in$$^{\bullet}A$. Conversely, since $\alpha\in S_{h(\alpha)}(\N)$, if $\alpha\in S_{h(\alpha)}(A^{c})$ then $\alpha\in$$^{\bullet}A^{c}$, and this is absurd.\\\end{proof}

From now on, we concentrate on the hyperextension $^{\bullet}\N$ of $\N$, constructed as to satisfy the $\mathfrak{c}^{+}$-enlarging property, and we study the property of the bridge map $\psi:$$^{\bullet}\N\rightarrow \N$.

\subsection{Sets of generators in $^{\bullet}\N$}

In this section we study, in some detail, the properties of the sets of generators of ultrafilters in $^{\bullet}\N$. We recall that, given an ultrafilter $\U$, the set of generators of $\U$ in $^{\bullet}\N$ is 

\begin{center} $G_{\U}=\{\alpha\in$$^{\bullet}\N\mid \U=\mathfrak{U}_{\alpha}\}$ \end{center}

where, for every hypernatural number $\alpha$ in $^{\bullet}\N$, 

\begin{center}$\mathfrak{U}_{\alpha}=\{A\subseteq\N\mid \alpha\in$$^{\bullet} A\}$.\end{center}

\begin{prop} For every ultrafilter $\U$ in $\bN$, for every hypernatural number $\alpha$ in $^{\bullet}\N$, for every natural number $n$, the following two conditions are equivalent:
\begin{enumerate}
	\item $\alpha$ is a generator if $\U$;
	\item $S_{n}(\alpha)$ is a generator of $\U$.
\end{enumerate}
\end{prop}

\begin{proof} By point four of Proposition 2.5.6 it follows that, for every subset $A$ of $\N$, $\alpha\in$$^{\bullet}A$ if and only if $S_{n}(\alpha)\in$$^{\bullet}A$. From this follows that, for every subset $A$ of $\N$,

\begin{center} $A\in\mathfrak{U}_{\alpha}\Leftrightarrow A\in\mathfrak{U}_{S_{n}(\alpha)}$, \end{center}

so $\U=\mathfrak{U}_{\alpha}$ if and only if $\U=\mathfrak{U}_{S_{n}(\alpha)}$.\\\end{proof}

\begin{defn} Given an ultrafilter $\U$ and a natural number $n\geq 1$, $G^{n}_{\U}$ denotes the set of generators of $\U$ with height at most $n$: 

\begin{center}$G^{n}_{\U}=\{\alpha\in G_{\U}\mid h(\alpha)\leq n\}=G_{\U}\cap S_{n}(\N)$. \end{center}

\end{defn}

\begin{prop} Let $\U$ be any nonprincipal ultrafilter on $\N$, and $n$ any natural number. Then:
\begin{enumerate}
  \item $|G^{n}_{\U}|=|S_{n}(\N)|$; $|G_{\U}|=|$$^{\bullet}\N|$;
	\item $^{*}G^{n}_{\U}\subseteq G^{n+1}_{\U}$;
	\item The cofinality of $G_{\U}$ is $\aleph_{0}$.
\end{enumerate}
	
\end{prop}

\begin{proof} 1) These two assertions follow by Proposition 2.2.4, since both $S_{n}(\N)$ and $^{\bullet}\N$ are hyperextensions of $\N$ that satisfy the $\mathfrak{c}^{+}$-enlarging property by Proposition 2.5.7.\\
2) Observe that, for every set $A$ in $\U$, for every natural number $n$, by definition of sets of generators 

\begin{center} $G^{n}_{\U}\subseteq S_{n}(A)$. \end{center}

By transfer it follows that 

\begin{center} $^{*}G^{n}_{\U}\subseteq$$^{*}S_{n}(A)=S_{n+1}(A)$. \end{center}

Since this holds for every set $A$ in $\U$, then 

\begin{center}$^{*}G^{n}_{\U}\subseteq\bigcap_{A\in\U}S_{n+1}(A)=G^{n+1}_{\U}$.\end{center}

3) Since $G^{n}_{\U}$, as $\U$ is nonprincipal, is infinite, by point (2) it follows that, for every natural number $n$, since $*$ is a proper star map then $G^{n+1}_{\U}\setminus G^{n}_{\U}\neq\emptyset.$ So, for every natural number $n$, there is an element $\alpha_{n}$ in $G^{n+1}_{\U}\setminus G^{n}_{\U}$. The sequence $\langle \alpha_{n}\mid n\in\N\rangle$ is a right unbounded sequence of elements in $G_{\U}$, so the cofinality of $G_{\U}$ is $\aleph_{0}$.\\\end{proof}

When the $*$ map satisfies the $\mathfrak{c}^{+}$-saturation property, the sets of generators satisfy the following additional properties: 

\begin{prop} Let $*$ be a star map with the $\mathfrak{c}^{+}$-saturation property, $\U$ an ultrafilter on $\N$ and $n\geq 1$ a natural number. Then
\begin{enumerate}
	\item $G^{n+1}_{\U}\setminus G^{n}_{\U}$ is left unbounded in $S_{n+1}(\N)\setminus S_{n}(\N)$; 
	\item $G_{\U}$ has coinitiality greater than $\mathfrak{c}$, and it is left unbounded in $^{\bullet}\N\setminus\N$.
	
\end{enumerate}
\end{prop}

\begin{proof} 1) We proceed by induction on $n$. Suppose $n=0$. Let $\eta$ be an infinite hypernatural number in $^{*}\N\setminus\N$ and pose, for every set $A$ in $\U$,

\begin{center} $A_{\eta}=\{\alpha\in$$^{*}A\mid \alpha<\eta\}$. \end{center}

These sets are internal, nonempty (as $A\subseteq A_{\eta}$) and the family $\{A_{\eta}\}_{A\in\U}$ has the finite intersection property and cardinality $\leq\mathfrak{c}$. By $\mathfrak{c}^{+}$-saturation property, 

\begin{center} $\bigcap_{A\in\U} A_{\eta}\neq \emptyset$; \end{center}

if $\alpha$ is an element in this intersection then $\alpha$ is infinite, $\alpha<\eta$ and $\alpha\in G^{1}_{\U}$: this proves that $G^{1}_{\U}$ is left unbounded in $^{*}\N\setminus\N$.\\
By induction, suppose to have proved the property for every $n\leq k$, and consider $n=k+1$. By inductive hypothesis, we know that

\begin{center} For every $\eta$ in $S_{k+1}(\N)\setminus S_{k}(\N)$ it exists $\alpha$ in $G^{k+1}_{\U}\setminus G^{k}_{\U}$ such that $\alpha < \eta$. \end{center}

Since $G^{k}_{\U}=G^{k+1}_{\U}\cap S_{k}(\N)$, we can substitute $G^{k+1}_{\U}\setminus G^{k}_{\U}$ with $G^{k+1}_{\U}\setminus S_{k}(\N)$; by transfer, it follows that

\begin{center} For every $\eta$ in $S_{k+2}(\N)\setminus S_{k+1}(\N)$ it exists $\alpha$ in $^{*}G^{k+1}_{\U}\setminus S_{k+1}(\N)$ such that $\alpha < \eta$ \end{center}

and we conclude observing that, since $^{*}G^{k+1}_{\U}\subseteq G^{k+2}_{\U}$, it follows

\begin{center} $^{*}G^{k+1}_{\U}\setminus S_{k+1}(\N)\subseteq$$G^{k+2}_{\U}\setminus S_{k+1}(\N)=G^{k+2}_{\U}\setminus G^{k+1}_{\U}$. \end{center}

2) By construction, since $G^{1}_{\U}$ is an initial segment of $G_{\U}$, the coinitiality of $G_{\U}$ is equal to that of $G^{1}_{\U}$, which is greater than $\mathfrak{c}$ by $\mathfrak{c}^{+}$-saturation.\\\end{proof}

In next section is showed an important feature of the sets of generators in $^{\bullet}\N$: in this context, where the iteration of the star map is allowed, there are particularly simple rules that, given generators $\alpha_{1},...,\alpha_{n}$ of ultrafilters $\U_{1},...,\U_{n}$, produce generators of the tensor product $\U_{1}\otimes\U_{2}\otimes...\otimes\U_{n}$.

\subsection{Tensor $k$-tuples in $^{\bullet}\N$}

As we observed in Section 2.4, given two hypernatural numbers $\alpha,\beta$ in a generical extension $^{*}\N$ of $\N$ (that satisfies the $\mathfrak{c}^{+}$-enlarging property) it is usually complicated to decide if $(\alpha,\beta)$ is, or is not, a tensor pair. In this section, we consider the hyperextension $^{\bullet}\N$ of $\N$, constructed starting with a superstructure model of nonstandard methods $\langle \mathbb{V}(X),\mathbb{V}(X),*\rangle$ with the star map $*$ that satisfies the $\mathfrak{c}^{+}$-enlarging property.\\
What we search is a binary relation $R$ over $^{\bullet}\N$ that satisfies the following two properties:
\begin{enumerate}
	\item given two hypernatural numbers $\alpha,\beta$ it is simple to decide if the pair $(\alpha,\beta)$ is in $R$ or not;
	\item every pair $(\alpha,\beta)$ in $R$ is a tensor pair.
\end{enumerate}

The star iteration provides such a relation:

\begin{defn} The binary relation $R$ on $^{\bullet}\N$ is the relation such that, for every $\alpha,\beta$ in $^{\bullet}\N$:

\begin{center} $(\alpha,\beta)\in R\Leftrightarrow \exists k\in\N,\exists\gamma\in$$^{\bullet}\N$ such that $\beta=S_{(h(\alpha)+k)}(\gamma)$. \end{center}

\end{defn}

In this definition, we just observe that $\beta\sim_{u}\gamma$, as a consequence of Proposition 2.5.11. This relation satisfies the property (1); the important fact is that $R$ satisfies also the second property, as it is proved in the theorem below:

\begin{thm} For every hypernatural numbers $\alpha,\beta$ in $^{\bullet}\N$, if $(\alpha,\beta)\in R$ then $(\alpha,\beta)$ is a tensor pair. \end{thm}

\begin{proof} As $(\alpha,\beta)\in R$, there are a natural number $k$ and an hypernatural number $\gamma\in$$^{\bullet}\N$ such that $\beta=S_{(h(\alpha)+k)}(\gamma).$ To prove that $(\alpha,\beta)$ is a tensor pair we have to show that, for every subset $A$ of $\N^{2}$, $(\alpha,\beta)\in$$^{\bullet}A$ if and only if $A\in\mathfrak{U}_{\alpha}\otimes\mathfrak{U}_{\beta}$.\\
Let $A$ be a subset of $\N^{2}$. By definition, 

\begin{center} $A\in \mathfrak{U}_{\alpha}\otimes\mathfrak{U}_{\beta}\Leftrightarrow \{n\in\N\mid\{m\in\N\mid (n,m)\in A\}\in\mathfrak{U}_{\beta}\}\in \mathfrak{U}_{\alpha}$.\end{center}

Since, as observed, $\mathfrak{U}_{\beta}=\mathfrak{U}_{\gamma}$, it follows that 

\begin{center}$A\in \mathfrak{U}_{\alpha}\otimes\mathfrak{U}_{\beta}\Leftrightarrow\{n\in\N\mid\{m\in\N\mid (n,m)\in A\}\in\mathfrak{U}_{\gamma}\}\in \mathfrak{U}_{\alpha}$.\end{center}

By definition of generated ultrafilter, 

\begin{center}$\{n\in\N\mid\{m\in\N\mid (n,m)\in A\}\in\mathfrak{U}_{\gamma}\}\in \mathfrak{U}_{\alpha}\Leftrightarrow$\\\vspace{0.3cm}$\Leftrightarrow\alpha\in S_{h(\alpha)}(\{n\in\N\mid \gamma\in S_{h(\gamma)}(\{m\in\N\mid (n,m)\in A\})\})$.\end{center}
By transfer property, 

\begin{center}$\alpha\in S_{h(\alpha)}(\{n\in\N\mid \gamma\in S_{h(\gamma)}(\{m\in\N\mid (n,m)\in A\})\})\Leftrightarrow$\\\vspace{0.3cm}$\Leftrightarrow \alpha\in S_{h(\alpha)}(\{n\in\N\mid S_{k}(\gamma)\in S_{(h(\gamma)+k)}(\{m\in\N\mid (n,m)\in A\})\})\Leftrightarrow $\\\vspace{0.3cm}$\Leftrightarrow (\alpha,S_{(h(\alpha)+k)}(\gamma)) \in S_{(h(\alpha)+h(\gamma)+k)}(A)$.\end{center}

Since $S_{(h(\alpha)+k)}(\gamma)=\beta$ and $h((\alpha,\beta))=h(\beta)=h(\alpha)+h(\gamma)+k$, by Proposition 2.5.6 it follows that 

\begin{center}$(\alpha,S_{(h(\alpha)+k)}(\gamma)) \in S_{(h(\alpha)+h(\gamma)+k)}(A)\Leftrightarrow (\alpha,\beta)\in$$^{\bullet}A$.\end{center}
This proves that, for every subset $A$ of $\N^{2}$, 

\begin{center}$A\in \mathfrak{U}_{\alpha}\otimes\mathfrak{U}_{\beta}\Leftrightarrow (\alpha,\beta)\in$$^{\bullet}A$,\end{center}

so $(\alpha,\beta)$ is a tensor pair. \\\end{proof}

\begin{cor} For every hypernatural numbers $\alpha,\beta$ in $^{\bullet}\N$, $(\alpha,S_{h(\alpha)}(\beta))$ is a tensor pair. \end{cor}

\begin{proof} Just observe that, for every $\alpha,\beta$ in $^{\bullet}\N$, $(\alpha,S_{h(\alpha)}(\beta))\in R$. \\\end{proof}

In $^{\bullet}\N$, tensor pairs have the following equivalent characterization:

\begin{prop} Let $\alpha,\beta$ be two hypernatural numbers in $^{\bullet}\N$. The following conditions are equivalent:
\begin{enumerate}
	\item $(\alpha,\beta)$ is a tensor pair;
	\item $(\alpha,\beta)\sim_{u}(\alpha,S_{h(\alpha)}(\beta))$.
\end{enumerate}
\end{prop}

\begin{proof} Observe that, given $\alpha,\beta\in$$^{\bullet}\N$, $(\alpha,S_{h(\alpha)}(\beta))$ is a tensor pair and, as $\beta\sim_{u} S_{h(\alpha)}(\beta)$, if follows that $\mathfrak{U}_{(\alpha,S_{h(\alpha)}(\beta))}=\mathfrak{U}_{\alpha}\otimes\mathfrak{U}_{\beta}$.\\
$(1)\Rightarrow (2)$ If $(\alpha,\beta)$ is a tensor pair, $\mathfrak{U}_{(\alpha,\beta)}=\mathfrak{U}_{\alpha}\otimes\mathfrak{U}_{\beta}$ by definition, so by the above observation we get that $(\alpha,\beta)\sim_{u}(\alpha,S_{h(\alpha)}(\beta))$.\\
$(2)\Rightarrow (1)$ If $\mathfrak{U}_{(\alpha,\beta)}=\mathfrak{U}_{(\alpha,S_{h(\alpha)}(\beta))}$ then, by the previous observation, $\mathfrak{U}_{(\alpha,\beta)}=\mathfrak{U}_{\alpha}\otimes\mathfrak{U}_{\beta}$ so $(\alpha,\beta)$ is a tensor pair.\\ \end{proof}

We remark that the result of Theorem 2.5.16 could be derived combining Puritz's Theorem with the following fact:

\begin{prop} For every hypernatural number $\alpha$ in $^{\bullet}\N$, for every natural number $n\geq 1$, for every function $f\in\mathtt{Fun}(\N,\N)$, or $^{\bullet}f(S_{n}(\alpha))\in\N$ or $h($$^{\bullet}f(S_{n}(\alpha)))\geq n+1$. \end{prop}

\begin{proof} The result follows from this claim:\\

{\bfseries Claim:} $^{\bullet}f(S_{n}(\alpha))=S_{n}($$^{\bullet}f(\alpha))$.\\

We prove the claim: if $\Gamma_{^{\bullet} f}$ is the graph of $^{\bullet}f$, for every $x,y$ in $^{\bullet}\N$, for every $n\geq 1\in\N$,

\begin{center} $(x,y)\in \Gamma_{^{\bullet}f}\Leftrightarrow (S_{n}(x),S_{n}(y))\in\Gamma_{^{\bullet}f}$.\end{center}
In particular, if $x= \alpha, y=$$^{\bullet}f(\alpha)$ the claim is proved.\\
So, if $^{\bullet}f(S_{n}(\alpha))\notin\N$, then $h($$^{\bullet}f(S_{n}(\alpha)))=h(S_{n}($$^{\bullet}f(\alpha)))=n+h($$^{\bullet}f(\alpha))\geq n+1.$ \\ \end{proof}

\begin{cor} Theorem 2.5.16. \end{cor}

\begin{proof} Let $\alpha,\beta$ be hypernatural numbers in $^{\bullet}\N$ such that $(\alpha,\beta)\in R$, and let $n,\gamma$ be such that $\beta=S_{(h(\alpha)+n)}(\gamma)$. Let $f$ be a function in $\mathtt{Fun}(\N,\N)$.\\
Then or $^{\bullet}f(\beta)\in\N$, or $h($$^{\bullet}f(\beta))\geq h(\alpha)+n+1$; in this second case, since $h($$^{\bullet}f(\beta))> h(\alpha)$, it follows that $\alpha< $$^{\bullet}f(\beta)$.\\
Since this happens for every function $f$ in $\mathtt{Fun}(\N,\N)$ with $^{\bullet}f(\beta)$ infinite, by Puritz's Theorem it follows that $(\alpha,\beta)$ is a tensor pair. \\\end{proof}

\begin{cor} Let $k\geq 2$ be a positive natural number and $\alpha_{1},...,\alpha_{k}$ hypernatural numbers in $^{\bullet}\N$. The following two conditions are equivalent:
\begin{enumerate}
	\item $(\alpha_{1},...,\alpha_{k})$ is a tensor $k$-tuple;
	\item $($$^{*}\alpha_{1},...,$$^{*}\alpha_{k})$ is a tensor $k$-tuple.
\end{enumerate}
\end{cor}

\begin{proof} Observe that, for every $\alpha,\beta$ in $^{\bullet}\N$, $\alpha<er(\beta)$ if and only if $^{*}\alpha<er($$^{*}\beta)$ (since, for every function $f$ in $\mathtt{Fun}(\N,\N)$ such that $^{\bullet}f(\beta)\notin\N$, $\alpha\leq$$^{\bullet}f(\beta)\Leftrightarrow$$^{*}\alpha\leq$$^{*}($$^{\bullet}f(\beta))$ and $^{*}($$^{\bullet}f(\beta))=$$^{\bullet}f($$^{*}\beta)$). This, combined with Theorem 2.4.11, gives the equivalence between (1) and (2). \\ \end{proof}

We still have the problem, given generical hypernatural numbers $\alpha_{1},...,\alpha_{k}$ in $^{\bullet}\N$, to decide if $(\alpha_{1},...,\alpha_{k})$ is a tensor $k$-tuple. Similarly to the case $k=2$, we get that an appropriate use of star iteration gives a procedure to bypass this problem:

\begin{defn} Let $k\geq 2$ be a natural number and $\alpha_{1},...,\alpha_{k}$ hypernatural numbers in $^{\bullet}\N$. The {\bfseries tensorized of $(\alpha_{1},...,\alpha_{k})$} $($notation: $T(\alpha_{1},...,\alpha_{k}))$ is the $k$-tuple

\begin{center} $T(\alpha_{1},...,\alpha_{k})=($$S_{h_{1}}(\alpha_{1}),S_{h_{2}}(\alpha_{2}),...,S_{h_{k}}(\alpha_{k}))$, \end{center}

where $h_{i}=\sum_{j< i} h(\alpha_{j})$ for every index $i$ in $\{1,...,k\}$.\end{defn}

Observe that $h_{1}=0$ (we included $h_{1}$ in the definition because this inclusion gives an uniform formulation to this notion).\\
E.g., if $\alpha,\beta,\gamma$ are three hypernatural numbers in $^{*}\N$, then 

\begin{center} $T(\alpha,\beta,\gamma)=(\alpha,$$^{*}\beta,$$^{**}\gamma)$. \end{center}

\begin{prop} For every hypernatural numbers $\alpha_{1},...,\alpha_{k}$ in $^{\bullet}\N$ the tensorized $T(\alpha_{1},...,\alpha_{k})$ of $(\alpha_{1},...,\alpha_{k})$ is a tensor $k$-tuple. \end{prop}

\begin{proof} We just observe that $h(S_{h_{i}}(\alpha_{i}))=h_{i}+h(\alpha_{i})$, so $er(S_{h_{i+1}}(\alpha_{i+1}))>S_{h_{i}}(\alpha_{i})$; by Theorem 2.4.11 it follows that $T(\alpha_{1},...,\alpha_{k})$ is a tensor $k$-tuple.\\ \end{proof}

As a corollary we get that, for every hypernatural numbers $\alpha_{1},...\alpha_{k}$ in $^{\bullet}\N$, 

\begin{center} $\mathfrak{U}_{\alpha_{1}}\otimes...\otimes\mathfrak{U}_{\alpha_{k}}=\mathfrak{U}_{T(\alpha_{1},...,\alpha_{k})}$. \end{center}

This gives a procedure to study, given a function $f$ $\sim_{u}$-preserving in $\mathtt{Fun}($$^{\bullet}\N^{k},$$^{\bullet}\N)$, the restriction of $\hat{f}$ to $T_{k}$: in fact, for every $\alpha_{1},...,\alpha_{k}$ in $^{\bullet}\N$ 

\begin{center} $\hat{f}(\mathfrak{U}_{\alpha_{1}}\otimes...\otimes\mathfrak{U}_{\alpha_{k}})=\mathfrak{U}_{f(T(\alpha_{1},...,\alpha_{k}))}$. \end{center}

This gives the possibility to study the functions in $\mathtt{Fun}(\bN^{k},\bN)$ in nonstandard terms. Next section is dedicated to two important particular cases: the functions $\oplus$ and $\odot$.\\
We conclude this section with a remark. Puritz's Theorem, in some sense, tells that a pair $(\alpha,\beta)$ is a tensor pair if $\beta$ is "much larger" than $\alpha$. In $^{\bullet}\N$, whenever $\alpha,\beta$ are hypernatural numbers such that $h(\alpha)<h(\beta)$, we can surely state that $\beta$ is much larger than $\alpha$. One could imagine, then, that the following property holds:

\begin{center} $\forall\alpha,\beta\in$$^{\bullet}\N\setminus\N$, if $h(\alpha)<h(\beta)$ then $(\alpha,\beta)$ is a tensor pair. \end{center}

This is false: consider the function $f\in\mathtt{Fun}(\N,\N)$ such that

\begin{center} $f(n)=\begin{cases} p, & \mbox{if there are a prime number } p\\ &\mbox{and a natural number } k \mbox{ such that } n=p^{k};\\ n, & \mbox{otherwise} .\end{cases}$ \end{center}

Let $\eta$ be a prime number in $^{*}\N\setminus\N$, $\xi$ an hypernatural number in $^{*}\N$, and consider $\beta=\eta^{^{*}\xi}$. By construction, $^{\bullet}f(\beta)=\eta\in$$^{*}\N\setminus\N$. If $\alpha$ is any hypernatural number in $^{*}\N$ with $\alpha>\eta$, since $\alpha$ is not smaller than $er(\beta)$, by Puritz's Theorem it follows that the pair $(\alpha,\beta)$ is not a tensor pair.

\subsection{Sets of generators of sums and products of ultrafilters}

We want to apply the results of Section 2.5.3 to study in nonstandard terms two of the most important operations on $\bN$, the sum $\oplus$ and the product $\odot$.\\
As we already observed, the sum $\U\oplus\V$ (resp. the product $\U\odot\V$) of two ultrafilters can be seen as the image, respect the continuous extension $S:\bN^{2}\rightarrow\bN$ of the sum $+:\N^{2}\rightarrow\N$ (resp. the continuous extension $P:\bN^{2}\rightarrow\bN$ of the product $\cdot:\N^{2}\rightarrow\N$), of the ultrafilter $\U\otimes\V$. This, combined to various results proven in this chapter, has the following consequence:

\begin{prop} For every ultrafilters $\U,\V$ on $\N$

\begin{center} $G_{\U\oplus\V}=\{\alpha+\beta\mid (\alpha,\beta)\in G_{\U\otimes\V}\}$ \end{center}

and 

\begin{center} $G_{\U\odot\V}=\{\alpha\cdot\beta\mid (\alpha,\beta)\in G_{\U\otimes\V}\}$. \end{center} \end{prop}

In particular:

\begin{prop} For every ultrafilters $\U,\V,\W$ in $\bN$, $\U\oplus\V=\W$ if and only if there are $\alpha\in G_{\U}, \beta\in G_{\V}$ such that $(\alpha,\beta)$ is a tensor pair and $\alpha+\beta\in G_{\W}$. \end{prop}

The Proposition 2.5.23 provides an easy way to find generators for tensor products of ultrafilters. This gives a method to explicitally construct generators for sums and products of ultrafilters:

\begin{defn} For every hypernatural numbers $\alpha,\beta$ in $^{\bullet}\N$, we pose 

\begin{center}$\alpha\h\beta=\alpha+$$S_{h(\alpha)}(\beta)$,\end{center}

and 
 
\begin{center} $\alpha\dia\beta=\alpha\cdot S_{h(\alpha)}(\beta)$.\end{center}\end{defn}

In next theorem, we denote by $\psi$ the bridge map with domain $^{\bullet}\N$.

\begin{thm} For every hypernatural numbers $\alpha,\beta$ in $^{\bullet}\N$,  

\begin{center}$\psi(\alpha\h\beta)=\psi(\alpha)\oplus\psi(\beta)$\end{center}

and

\begin{center} $\psi(\alpha\dia\beta)=\psi(\alpha)\odot\psi(\beta)$.\end{center}\end{thm}

\begin{proof} Simply observe that, since $(\alpha,S_{h(\alpha)}(\beta))$ is a tensor pair, $\alpha+S_{h(\alpha)}(\beta)$ is a generator of $\U\oplus\V$, as a consequence of Theorem 2.3.5. Similarly with the product.\\ \end{proof}

Note that the above property is false if we consider $+,\cdot$ in place of $\h,\dia$. Next proposition contains a list of easy properties of $\h, \dia$:

\begin{prop}[] For every $\alpha,\beta,\gamma\in$$^{\bullet}\N$ we have the following relations:
\begin{enumerate}
	\item For all $n\in\N$, $\alpha\h n=n\h\alpha =\alpha+n$;
	\item For all $n\in\N$ $\alpha\dia n=n\dia \alpha= \alpha\cdot n$;
	\item $\alpha\h(\beta\h\gamma)=(\alpha\h\beta)\h\gamma$;
	\item $\alpha\dia(\beta\dia\gamma)=(\alpha\dia\beta)\dia\gamma$;
	\item $(\alpha\h\beta)\dia\gamma=\alpha\dia S_{h(\beta)}(\gamma)+S_{h(\alpha)}(\beta\dia\gamma)$;
	\item $\gamma\dia(\alpha\h\beta)=\gamma\dia\alpha+\gamma\dia S_{h(\alpha)}(\beta)$;
	\item $\gamma\dia(\alpha+\beta)=\gamma\dia\alpha+\gamma\dia\beta$;
	\item $^{*}\alpha\h\beta=$$^{*}(\alpha\h\beta)$;
	\item $^{*}\alpha\dia\beta=$$^{*}(\alpha\dia\beta)$;
	\item In general, $\alpha\h\beta\neq\beta\h\alpha$ e $\alpha\dia\beta\neq\beta\dia\alpha$;
	\item For all $n\in\N$, $\alpha\h\beta=(\alpha+n)\h(\beta-n)=(\alpha-n)\h(\beta+n)$;
	\item $h(\alpha\h\beta)=h(\alpha)+h(\beta)$;
	\item $h(\alpha\dia\beta)=h(\alpha)+h(\beta)$;
	\item $h(\alpha)=h(\beta)=h(\alpha+\beta)\Rightarrow (\alpha+\beta)\dia\gamma=\alpha\dia\gamma+\beta\dia\gamma$;
	\item Let $I$ be a finite set. If, for all $i\in I$, $h(\sum_{i\in I} \alpha_{i})=h(\alpha_{i})$, then $(\sum_{i\in I} \alpha_{i})\h(\sum_{i\in I}\beta_{i})=\sum_{i\in I} (\alpha_{i}\h\beta_{i})$;
	\item Let $I$ be a finite set. If, for all $i\in I$, $h(\sum_{i\in I} \alpha_{i})=h(\alpha_{i})$, then $(\prod_{i\in I} \alpha_{i})\dia(\prod_{i\in I}\beta_{i})=\prod_{i\in I} (\alpha_{i}\dia\beta_{i})$;
	\item $(\alpha\h\beta)\dia\gamma\sim_{u}(\alpha\dia\gamma)\h(\beta\dia\gamma)$.
	
\end{enumerate}
\end{prop}

\begin{proof} 1) and 2) are obtained since the height of every natural number $n$ is 0, and $^{*}n=n$.\\
3) $\alpha\h(\beta\h\gamma)=\alpha\h(\beta+S_{h(\beta)}(\gamma))=\alpha+S_{h(\alpha)}(\beta)+S_{(h(\alpha)+h(\beta))}(\gamma)=(\alpha\h\beta)\h\gamma$. The same for 4.\\
5)-6)-7) are simple calculations.\\
8) $^{*}\alpha\h\beta=$$^{*}\alpha+$$S_{(h(\alpha)+1)}(\beta)=$$^{*}(\alpha+S_{h(\alpha)}(\beta))=$$^{*}(\alpha\h\beta)$. Same calculation for 9).\\
10) We can say more: we have $\alpha\h\beta=\beta\h\alpha \Leftrightarrow (\alpha\in\N)\vee(\beta\in\N)$, and the same for $\dia$.\\
11) $(\alpha+n)\h(\beta-n)=\alpha+n+$$S_{h(\alpha)}(\beta)-n=\alpha\h\beta$, and the same with $(\alpha-n)\h(\beta+n)$.\\
12)-13) $\alpha\h\beta$=$\alpha+S_{h(\alpha)}(\beta)$. In this sum the maximum height is that of $S_{h(\alpha)}(\beta)$, which is $h(\alpha)+h(\beta)$. So $h(\alpha\h\beta)=h(\alpha)+h(\beta)$. The same for $\dia$.\\
14)-15)-16) are similar (13 is a particular case of 14). Call $n$ the height common to all the $\alpha_{i}$'s and their sum (product). Then $(\sum_{i\in I} \alpha_{i})\h(\sum_{i\in I}\beta_{i})=(\sum_{i\in I}\alpha_{i})+$$S_{n}(\sum_{i\in I}\beta_{i})=\sum_{i\in I}(\alpha_{i}+$$S_{n}(\beta_{i}))=\sum_{i\in I}\alpha_{i}\h\beta_{i}$, and similar for the product.\\
17) We know from $\beta\N$ that this is true, because $\oplus$ and $\odot$ on $\beta\N$ are distributive.\\ \end{proof}

We conclude this section showing that the quotient space $^{\bullet}\N_{/_{\sim_{u}}}$, endowed with the operation $\h$ (resp. $\dia$), is algebraically and topologically equivalent to $(\bN,\oplus)$ (resp. $(\bN,\odot)$):

\begin{defn} Two topological semigroups are {\bfseries algebraically and topologically equivalent} if there is a mapping from the one to the other which is both a homeomorphism and an algebraic isomorphism.\end{defn}

The terminology is mutuated from $\cite{rif21}$. We fix some notations: whenever $[\alpha]_{\sim_{u}}, [\beta]_{\sim_{u}}$ are two equivalence classes in $^{\bullet}\N_{/_{\sim_{u}}}$, then $[\alpha]_{\sim_{u}}\h [\beta]_{\sim_{u}}=[\alpha\h\beta]_{\sim_{u}}$ and $[\alpha]_{\sim_{u}}\dia [\beta]_{\sim_{u}}=[\alpha\dia\beta]_{\sim_{u}}$. These definitions, as a consequence of Theorem 2.5.27, are well-posed. \\
The topology that we consider on $^{\bullet}\N$ is the so-called S-Topology (see e.g. $\cite[\mbox{Section 3}]{rif16}$). This topology, which can be similarly defined in every hyperextension of $\N$, is generated by taking, as base of open sets, the family of hyperextensions of subsets of $\N$:

\begin{center} $\mathcal{B}=\{$$^{\bullet}A\mid A\subseteq\N\}$. \end{center}

\begin{lem} $($$^{\bullet}\N_{/_{\sim_{u}}},\h)$ $($resp. $($$^{\bullet}\N_{/_{\sim_{u}}},\dia))$, enodowed with the quotient star topology, is a right topological semigroup. \end{lem}

\begin{proof} To prove the thesis we have to show that $\h$ is associative and that if is right continuous.\\
That $\h$ is associative is proved in Proposition 2.5.28.\\
To prove that $\h$ is right continuous, let $\beta$ be an hypernatural number in $^{\bullet}\N$, and consider the function $\varphi_{\beta}$ such that, for every $\alpha\in$$^{\bullet}\N$, 

\begin{center} $\varphi_{\beta}(\alpha)=\alpha+S_{h(\alpha)}(\beta)$. \end{center}

The map $\varphi_{\beta}$ is continuous in the S-Topology: in fact, for every subset $A$ of $\N$, 

\begin{center}$\varphi^{-1}_{\beta}(A)=\{\alpha\in$$^{\bullet}\N\mid \alpha+S_{h(\alpha)}(\beta)\in$$^{\bullet}A\}=$$^{\bullet}\{n\in\N\mid n+\beta\in S_{h(\beta)}(A)\}$\end{center}

since for every hypernatural number $\alpha$, $\alpha+S_{h(\alpha)}(\beta)\in$$^{\bullet}A$ if and only if $\alpha+S_{h(\alpha)}(\beta)\in S_{h(\alpha)+h(\beta)}(A)$ if and only if $\alpha\in S_{h(\alpha)}(\{n\in\N\mid n+\beta\in S_{h(\beta)}(A)\})$ if and only if $\alpha\in$$^{\bullet}\{n\in\N\mid n+\beta\in S_{h(\beta)}(A)\}$.\\
If we pose 

\begin{center} $B_{A}= \{n\in\N\mid n+\beta\in S_{h(\beta)}(A)\}$, \end{center}

it follows that $\varphi_{\beta}^{-1}($$^{\bullet}A)=$$^{\bullet}B_{A}$, so $\varphi_{\beta}$ is continuous in the star topology.\\
That $\dia$ is continuous can be proved similarly.\\ \end{proof}

\begin{thm} $(\bN,\oplus)$ $($resp $(\bN,\odot))$ is algebraically and topologically equivalent to $($$^{\bullet}\N_{/_{\sim_{u}}},\h)$ $($resp. $($$^{\bullet}\N_{/_{\sim_{u}}},\dia))$. \end{thm}

\begin{proof} The map to consider is the bridge map $\psi$. That it is an algebraical isomorphism has been proved in Theorem 2.5.27. It is also an homeomorphism: in fact, $\psi$ is clearly bijective. It is continuous and open since, for every subset $A$ of $\N$, $\psi($$^{\bullet}A_{/_{\sim_{u}}})=\Theta_{A}$ and $\psi^{-1}(\Theta_{A})=$$^{\bullet}A_{/_{\sim_{u}}}$. Since every bijective continuous open function is an homeomorphism, we have the thesis.\\\end{proof}

\section{Further Studies}

Among the possible future studies, we want to point out a question that concerns the $\omega$-hyperextension of $\N$. Since $^{\bullet}\N$ is included in $\mathbb{V}(X)$, we can apply the star map to $^{\bullet}\N$, and consider $^{*}($$^{\bullet}\N)$, $^{**}($$^{\bullet}\N)$ and so on. If $\alpha$ is an ordinal number, one could consider the $\alpha$-hyperextension $S_{\alpha}(\N)$ of $\N$, that can be inductively defined as follows: if $\alpha=\beta+1$, then 

\begin{center} $S_{\alpha}(\N)=$$^{*}(S_{\beta}(\N))$; \end{center}

if $\alpha$ is a limit ordinal, then 

\begin{center} $S_{\alpha}=\bigcup_{\gamma<\alpha}S_{\gamma}(\N)$.\end{center}

The questions that arise regard the structure of $S_{\alpha}(\N)$, as well as the relations between $S_{\alpha}(\N)$ and $S_{\beta}(\N)$ for different ordinal numbers $\alpha,\beta$. We just outline a fact: in Section 2.5.1 we proved that $^{\bullet}\N$ is not $\mathfrak{c}^{+}$-saturated. The proof was based on the fact that $^{\bullet}$ is the union of an $\omega$-chain of end extensions of $\N$. For a generical ordinal number $\alpha$, this proof does not necessarily work. So, maybe, there are $\alpha$-extensions that are $\mathfrak{c}^{+}$-saturated: \\

{\bfseries Question 2:} Let $\kappa$ be an infinite cardinal number. Does it exists an ordinal number $\alpha$ such that $S_{\alpha}(\N)$ is $\kappa$-saturated? Does it exists an ordinal number $\alpha$ such that $S_{\alpha}(\N)$ is $\kappa^{+}$-saturated?

\newpage

\chapter{Proofs in Combinatorics by Mean of Star Iterations}

The basic concepts that we use in this chapter have been introduced in Chapters One an Two. We recall that, given a logical sentence $\varphi$, an ultrafilter $\U$ is a $\varphi$-ultrafilter if and only if every set $A$ in $\U$ satisfies $\varphi$. We work in the $\omega$-hyperextension $^{\bullet}\N$ of $\N$, constructed starting with a hyperextension $^{*}\N$ that satisfies the $\mathfrak{c}^{+}$-enlarging property. Finally, we recall that, given two hypernatural numbers $\alpha,\beta\in$$^{*}\N$, $\alpha\heartsuit\beta$ denotes the hypernatural number $\alpha+$$^{*}\beta\in$$^{\bullet}\N$.\\
The tools introduced so far are used, in this chapter, to study some topics in infinite combinatorics. In Section One, we present two known examples of application of nonstandard methods to combinatorics. Then, in Section Two, we test our nonstandard technique re-proving some well-known result in Ramsey Theory, e.g. Schur's Theorem and Folkman's Theorem. The results proved in Sections Three and Four show that, under certain assumptions on the first order sentence $\varphi$, there are $\varphi$-ultrafilters that are additively or multiplicatively idempotent. This is used, in Section Five, to study the partition regularity of polynomials. This topic is faced also in Section Six, where we concentrate on the closure of the set of partition regular polynomials under certain operations. Finally, is Section Seven, we indicate three possible future developments of the researches presented in this chapter.

\section{Applications of Nonstandard Methods to Ramsey Theory: Two Examples}

The idea of applying nonstandard methods in infinite combinatorics is not new. In this section we expose two known examples of such applications.\\
The first one concerns the particular case of Ramsey Theorem about colorings of subsets of $\N$ with cardinality two (a treatment of Ramsey Theorem both from the combinatorial and the ultrafilter points of view has been done in Chapter One). In the literature, there are nonstandard proofs of this result (see e.g. $\cite{rif24}$); the one we present here uses star iteration (see $\cite{rif15}$).\\
We recall that, for every subset $A$ of $\N$, $[A]^{2}=\{B\subseteq A\mid |B|=2\}$.

\begin{thm}[Ramsey] For every finite partition $C_{1}\cup...\cup C_{k}$ of $[\N]^{2}$ there is an infinite subset $H$ of $\N$ with $[H]^{2}\subseteq C_{i}$ for some index $i$. \end{thm}

\begin{proof} First of all, observe that, by transfer, $^{**}C_{1}\cup....\cup$$^{**}C_{k}$ is a partition of $[$$^{**}\N]^{2}$. Let $\alpha$ be any infinite number in $^{*}\N$, and let $i$ be the index such that $\{\alpha,$$^{*}\alpha\}\in$$^{**}C_{i}$. We now construct inductively a sequence $B_{0},B_{1},...$ of subsets of $\N$ and a sequence $h_{0},h_{1},...$ of natural numbers, with $h_{n}< h_{n+1}$ for every $n$ and such that $[H]^{2}\subseteq C_{i}$, where $H$ is the infinite set 

\begin{center} $H=\{h_{n}\mid n\in\N\}$.\end{center}

Step 0: Observe that, by transfer, $\{\alpha,$$^{*}\alpha\}\in$$^{**}C_{i}$ if and only if

\begin{center} $\alpha\in$$^{*}\{n\in\N\mid \{n,\alpha\}\in$$^{*}C_{i}\}$.\end{center}

Put

\begin{center} $A=\{n\in\N\mid\{n,\alpha\}\in$$^{*}C_{i}\}$ \end{center}

and let $h_{0}$ be any element in $A$. By construction, $\{h_{0},\alpha\}\in$$^{*}C_{i}$ so, by transfer, $\alpha\in$$^{*}B_{0}$, where

\begin{center} $B_{0}=\{m\in\N\mid \{h_{0}, m\}\in C_{i}\}$. \end{center}

As $\alpha\in$$^{*}A\cap$$^{*}B_{0}$, the set $A\cap B_{0}$ is nonempty and unbounded. Let $h_{1}>h_{0}$ be an element in this set. Observe that, by construction, $\{h_{0},h_{1}\}\in C_{i}$.\\
Step $n+1$: Suppose we have constructed $B_{0},B_{1},...,B_{n}$ and we have taken elements $h_{0}<h_{1}<...<h_{n}$ such that $h_{i}\in A\cap B_{0}\cap...\cap B_{i}$ for every index $i\leq n-1$, and $[\{h_{1},...,h_{n}\}]^{2}\subseteq C_{i}$. By construction, $\{h_{n},\alpha\}\in$$^{*}C_{i}$ so, by transfer, $\alpha\in$$^{*}\{m\in\N\mid \{h_{n},m\}\}\in C_{i}\}$. Pose

\begin{center} $B_{n}=\{m\in\N\mid \{h_{n},m\}\in C_{i}\}$. \end{center}

As $\alpha\in$$^{*}A\cap$$^{*}B_{0}\cap...\cap$$^{*}B_{n}$, the set $A\cap B_{0}\cap...\cap B_{n}$ is nonempty and unbounded; take $h_{n+1}>h_{n}$ in this set. Observe that, by construction, $[\{h_{0},...,h_{n+1}\}]^{2}\subseteq C_{i}$.\\
By construction, the infinite set $H=\{h_{n}\mid n\in\N\}$ is such that $[H]^{2}\subseteq C_{i}$ because, if $h_{n}<h_{m}$ are elements in $H$ then, as $h_{m}\in B_{n}$, $\{h_{n},h_{m}\}\in C_{i}$. \\ \end{proof}

We observe that it is not coincidental that in this proof the set $C_{i}$ is chosen to have $\{\alpha,$$^{*}\alpha\}\in$$^{**}C_{i}$: in fact, as we observed in Chapter Two, for every element $\alpha$ in $^{*}\N$, $(\alpha,$$^{*}\alpha)$ is a tensor pair, so this proof is involving the tensor product $\mathfrak{U}_{\alpha}\otimes\U_{\alpha}$, whit $\mathfrak{U}_{\alpha}$ a non principal ultrafilter, exactly as the ultrafilter proof of Ramsey Theorem given in Chapter One does.\\
The second well-known application of nonstandard methods to infinite combinatorics that we present is a well-known theorem of Renling Jin:

\begin{thm}[Jin] Let $A, B$ be subsets of $\N$ with positive Banach density. Then $A+B=\{a+b\mid a\in A, b\in B\}$ is piecewise syndetic. \end{thm}

The notion of piecewise syndetic set has been introduced in Chapter One. We recall the Banach density of a subset of $\Z$:

\begin{defn} Given a subset $A$ of $\Z$, its {\bfseries Banach density} $BD(A)$ is

\begin{center} $BD(A)=\lim_{n\to +\infty} (\sup_{a-b=n} \frac{|A\cap [a,b]|}{n})$. \end{center}

\end{defn}

Banach density can be characterized in nonstandard terms:

\begin{center} Given a subset $A$ of $\N$ and a number $x\in [0,1]$, $BD(A)>x$ if and only if there are $\alpha\in$$^{*}\N$, $\beta\in$$^{*}\N\setminus\N$ such that $st(\frac{|^{*}A\cap [\alpha,\alpha+\beta)|}{\beta})>x$. \end{center}

The proof of Theorem 3.1.2 is based on a result proved by Jin himself in $\cite{rif25}$. To state his result we have to define two notions:

\begin{defn} An infinite initial segment $C$ of $^{*}\N$ is a {\bfseries cut} if it is closed under sums, i.e. if $C+C=\{a+b\mid a,b\in C\}\subseteq C$.\\
Let $\eta\in$$^{*}\N\setminus\N$ be given. If $C\subseteq [0,\eta]$ is a cut, and $A$ is a subset of $[0,\eta]$, then $A$ is {\bfseries $C$-nowhere dense} if for every interval $I=[a,b]$ in $[0,\eta]$ such that $b-a>C$ $($i.e. $b-a>c$ for every $c\in C)$ there is a subinterval $[c,d]\subseteq [a,b]\setminus A$ such that $d-c> C$. \end{defn}

\begin{thm}[Jin] Let $\eta$ be an infinite hypernatural number and let $C\subseteq [0,\eta]$ be a cut. If $A, B\subseteq [0,\eta]$ are two internal sets such that $st(\frac{|A|}{\eta})>0$ and $st(\frac{|B|}{\eta})>0$, then $A\oplus_{\eta} B$ is not $C$-nowhere dense, where $\oplus_{\eta}$ is the addition mod $\eta$ on $^{*}\N$. \end{thm}

The proof of the above thorem, as well as five important corollaries (the third corollary is Theorem 3.1.2), can be found in $\cite{rif25}$.\\
Given the above result, we can prove Theorem 3.1.2:

\begin{proof} By the nonstandard characterization of the Banach density, since $BD(A)>0$ and $BD(B)>0$ there are elements $\alpha,\mu$ in $^{*}\N$, $\beta,\eta$ in $^{*}\N\setminus\N$ with $st(\frac{|^{*}A\cap [\alpha,\alpha+\mu)|}{\mu})>0$ and $st(\frac{|^{*}B\cap [\beta,\beta+\eta)|}{\eta})>0$. The nonstandard characterization of the Banach density ensures that, if necessary, we can clip both $A$ and $B$ and assume that $\mu=\eta$. Consider the internal sets $A^{\prime}$ and $B^{\prime}$, where

\begin{center} $A^{\prime}=($$^{*}A\cap [\alpha,\alpha+\eta))-\alpha$ and $B^{\prime}=($$^{*}B\cap [\beta,\beta+\eta))-\beta$, \end{center}

and the hypernatural number $2\eta$. By construction, $A^{\prime}$ and $B^{\prime}$ are subsets of $[0,\eta]$, and both $st(\frac{|A^{\prime}|}{\eta})$ and $st(\frac{|B^{\prime}|}{\eta})$ are greater than 0 (since these quantities equal the Banach densities of $A$ and $B$, respectively). By Theorem 3.1.5, if $C$ is any cut included in $[0,2\eta]$, $A^{\prime}\oplus_{2\eta} B^{\prime}$ (which, by construction, is equal to $A^{\prime}+B^{\prime}$) is not $C$-nowhere dense.\\
Let $C=\N$. The fact that $A^{\prime}+B^{\prime}$ is not $\N$-nowhere dense entails the existence of an interval $I=[a,b]$, with $b-a$ infinite, such that $A^{\prime}+B^{\prime}$ has no gaps of infinite lenght in $I$. But, as $A^{\prime}+B^{\prime}$ is internal, by overspill it follows that there are not arbitrarily long finite gaps in $A^{\prime}+B^{\prime}$, so there is a natural number $n$ such that in $A^{\prime}+B^{\prime}$ there are no gaps of length greater than $n$. So $A^{\prime}+B^{\prime}+\alpha+\beta=$$^{*}A\cap [\alpha,\alpha+\eta)+$$^{*}B\cap [\beta, \beta+\eta)$ has no gaps of length greater than $n$ in the interval $I+\alpha+\beta$; in particular, this entails that $^{*}A+$$^{*}B=$$^{*}(A+B)$ has no gaps of length greater than $n$ in $I+\alpha+\beta$. By transfer it follows that $A+B$ is piecewise syndetic.\\ \end{proof}

We choosed this theorem as an example of nonstandard methods applied to infinite combinatorics for two reasons. The first one is that it is connected with the arguments that we expose in Chapter Four; the second reason is that, in our opinion, the underlying philosophy between Jin's and our approach is similar. Quoting Jin's words from the article $\cite{rif50}$: 

\begin{center} \begin{quote} {\itshape Nonstandard methods are used here to reduce the complexity of the mathematical
objects that one needs in a proof}. [...] {\itshape This complexity reduction from second order to first order enables us to see the path
towards solutions more clearly with a better understanding, hence produce a shorter
proof with greater efficiency;}\end{quote} \end{center}

and 

\begin{center}\begin{quote} {\itshape Nonstandard methods offer a better intuition.} \end{quote} \end{center}

In our opinion, these are exactly the two advantages that the star iterations and the Bridge Theorem present when dealing with certain combinatorial problems.

\section{Some New Proofs of Old Results}

Most of the results in this section are not new: they are either well-known in literature or straight consequences of Rado's Theorem for linear equations.\\
However, we present new proofs of these results by the technique of star iterations (and the application of Bridge Theorem) so as to outline, in a well-known setting, the potentialites of this technique.\\
Throughout this chapter, we consider given a superstructure model of nonstandard methods $\langle\mathbb{V}(X),\mathbb{V}(X),*\rangle$ such that the star map satisfies the $\mathfrak{c}^{+}$-enlarging property, and we work in the hyperextension $^{\bullet}\N$ of $\N$ (we recall that, since the star map $*$ satisfies the $\mathfrak{c}^{+}$-enlarging property, also the map $\bullet$ satisfies the $\mathfrak{c}^{+}$-enlarging property).\\
The notations we use have been introduced in Chapter Two. We recall that, given a first order formula $\phi(x_{1},...,x_{n},p_{1},...,p_{k})$, its existential closure is

\begin{center} $E(\phi(x_{1},...,x_{n},p_{1},...,p_{k})):$ $\exists x_{1},...,x_{n} \phi(x_{1},...,x_{n},p_{1},...,p_{k})$, \end{center}

where $x_{1},...,x_{k}$ are the only free variables and $p_{1},...,p_{k}$ the only parameters of $\phi$; to simplify notations, we do not explicitally mention the parameters in $\phi(x_{1},...,x_{n})$, except if necessary. A first order sentence is existential if it is the existential closure of a first order formula.\\
Whenever $\varphi(x_{1},...,x_{n})$ is a first order formula with parameters $p_{1},...,p_{k}$, and $m$ is a natural number, $S_{m}(\varphi(x_{1},....,x_{n}))$ is the formula obtained by replacing each parameter $p_{i}$ in $\varphi(x_{1},...,x_{n})$ with $S_{m}(p_{i})$; similarly, $^{\bullet}\varphi(x_{1},...,x_{n})$ is the formula obtained by replacing each parameter $p_{i}$ in $\varphi(x_{1},...,x_{n})$ with $^{\bullet}p_{i}$.\\
Finally, we recall that a first order formula $\varphi(x_{1},...,x_{n})$ is elementary if its only parameters are elements in $\N^{k}$, subsets of $\N^{k}$, functions in $\mathtt{Fun}(\N^{k},\N^{h})$ or relations on $\N^{k}$, where $k,h$ are positive natural numbers.\\
The following result will be used to semplify most of the proofs in this chapter:

\begin{prop} Let $\varphi=E(\phi(x_{1},...,x_{n}))$ be an existential sentence; if $\U$ is a $\varphi$-ultrafilter then there are $\alpha_{1},...,\alpha_{n}\in$$^{*}\N$ generators of $\U$ such that $\phi(\alpha_{1},...,\alpha_{n})$ holds. \end{prop}

\begin{proof} This is just the formulation of the Bridge Theorem applied to $^{*}\N$.\\\end{proof}

Idempotent ultrafilters have been widely used in the study of partition regularity. We recall that an ultrafilter $\U$ is additively (resp. multiplicatively) idempotent if $\U=\U\oplus\U$ (resp. if $\U=\U\odot\U$).\\
Given their prominence in this context, it is natural and necessary to look for their characterizations in terms of generators:

\begin{prop} Let $\alpha$ be an hypernatural number in $^{\bullet}\N$ with height $n$. The following properties are equivalent:
\begin{enumerate}
	\item $\mathfrak{U}_{\alpha}$ is additively idempotent;
	\item $\alpha\sim_{u}\alpha\h\alpha$;
	\item There is an element $\beta$ in $^{\bullet}\N$ with $\alpha\sim_{u}\beta\sim_{u}\alpha\h\beta$;
	\item For every $\beta\sim_{u}\alpha$, $\alpha\h\beta\sim_{u}\alpha$;	
  \item There is an element $\beta$ in $^{\bullet}\N$ such that $(\alpha,\beta)$ is a tensor pair and $\mathfrak{U}_{\alpha}=\mathfrak{U}_{\beta}=\mathfrak{U}_{\alpha+\beta}$;
	\item For every subset $A$ of $\N$, if $\alpha\in S_{n}(A)$ then there exists a subset $B$ of $A$ such that $\alpha\in S_{n}(B)$ and $\alpha+B\subseteq S_{n}(B)$;
	\item For every subset $A$ of $\N$, if $\alpha\in S_{n}(A)$ then there is an element $a$ in $A$ such that $\alpha+a\in S_{n}(A)$.
\end{enumerate}
\end{prop}

\begin{proof} (1)$\Rightarrow$(2): Suppose that $\mathfrak{U}_{\alpha}$ is additively idempotent. As we proved in Chapter Two, $(\alpha,S_{n}(\alpha))$ is a tensor pair, so $\mathfrak{U}_{\alpha+S_{n}(\alpha)}=\mathfrak{U}_{\alpha}\oplus\mathfrak{U}_{\alpha}=\mathfrak{U}_{\alpha}$, and this implies that $\alpha\sim_{u}\alpha\h\alpha$.\\
(2)$\Rightarrow$(3): Just put $\beta=\alpha$.\\
(3)$\Rightarrow$(4): This is a consequence of Theorem 2.5.27: fix $\beta$ as in the hypothesis. If $\gamma$ is any other element in $G_{\mathfrak{U}_{\alpha}}$, since $\beta\sim_{u}\gamma$ then $\alpha\h\beta\sim_{u}\alpha\h\gamma$, so by hypothesis $\alpha\h\gamma\sim_{u}\alpha$.\\
(4)$\Rightarrow$(5): As a consequence of the hypothesis, $\alpha\sim_{u}\alpha\h\alpha=\alpha+S_{n}(\alpha)$, and $\alpha\sim_{u}S_{n}(\alpha)$ (as proved in Proposition 2.5.11). Let $\beta=S_{n}(\alpha)$: then $\alpha\sim_{u}\alpha+\beta$, where $(\alpha,\beta)$ is a tensor pair, as we proved in Theorem 2.5.16.\\
(5)$\Rightarrow$(1): Let $\beta$ be an hypernatural number as in the hypothesis. Observe that, since $(\alpha,\beta)$ is a tensor pair, then $\alpha+\beta\in G_{\mathfrak{U}_{\alpha}\oplus\mathfrak{U}_{\beta}}$. So, as $\alpha\sim_{u}\beta\sim_{u}\alpha+\beta$, $\mathfrak{U}_{\alpha}\oplus\mathfrak{U}_{\alpha}=\mathfrak{U}_{\alpha}\oplus\mathfrak{U}_{\beta}=\mathfrak{U}_{\alpha}$: this proves that $\mathfrak{U}_{\alpha}$ is additively idempotent.\\
(2)$\Rightarrow$(6): Let $A$ be any subset of $\N$, and suppose that $\alpha\in S_{n}(A)$. By hypothesis, since $A\in\U_{\alpha}$, $\alpha+S_{n}(\alpha)\in S_{2n}(A)$. In particular, 

\begin{center}$\alpha\in\{\gamma\in S_{n}(\N)\mid \gamma+S_{n}(\alpha)\in S_{2n}(A)\}=$$S_{n}(\{a\in\N\mid a+\alpha\in S_{n}(A)\})$. \end{center}

Let $B=\{a\in\N\mid a+\alpha\in S_{n}(A)\}$. By construction, $\alpha\in S_{n}(B)$ and $\alpha+B\subseteq S_{n}(A)$. We claim that, if $a$ is any element in $B$, then $a+\alpha\in S_{n}(B)$. In fact, $a+\alpha\in S_{n}(B)\Leftrightarrow a+\alpha+S_{n}(\alpha)\in S_{2n}(A)\Leftrightarrow\alpha+S_{n}(\alpha)\in S_{2n}(A-a)\Leftrightarrow\alpha+S_{n}(\alpha)\in S_{2n}(A-a)$ and, as $\alpha\sim_{u}\alpha+S_{n}(\alpha)$, this is equivalent to $\alpha\in S_{2n}(A-a)$ which is equivalent to $\alpha+a\in S_{2n}(A)$, and this is true since $a\in B$.\\
So $a+\alpha\in S_{n}(B)$ for every element $a$ in $B$; in particular, $\alpha+B\subseteq S_{n}(B)$.\\
(6)$\Rightarrow$(7): Consider the set $B$ given in the hypothesis. If $a$ is any of its elements, then $a+\alpha\in S_{n}(B)\subseteq S_{n}(A)$.\\
(7)$\Rightarrow$(1): To every set $A$ in $\mathfrak{U}_{\alpha}$ we associate the set 

\begin{center} $A_{\alpha}=\{a\in A\mid a+\alpha\in S_{n}(A)\}$. \end{center}

By hypothesis, every $A_{\alpha}$ is nonempty, so the family $\mathcal{F}=\{A_{\alpha}\}_{A\in\mathfrak{U}_{\alpha}}$ has the finite intersection property, as $(A\cap B)_{\alpha}=A_{\alpha}\cap B_{\alpha}$ for every $A,B$ in $\mathfrak{U}_{\alpha}$. By $\mathfrak{c}^{+}$-enlarging property, the set 

\begin{center} $S=\bigcap_{A\in\mathfrak{U}_{\alpha}} S_{n}(A_{\alpha})$ \end{center}

is nonempty. If $\beta$ is any element of height $n$ in $S$, then $\beta$ is an element of height $n$ in $G_{\mathfrak{U}_{\alpha}}$ (since, as $A_{\alpha}\subseteq A$, $S\subseteq G_{\mathfrak{U}_{\alpha}}$) and $\beta+S_{n}(\alpha)\in G_{\mathfrak{U}_{\alpha}}$ since, by transfer, $\beta\in S_{n}(A_{\alpha})$ entails that $\beta+S_{n}(\alpha)\in S_{2n}(A)$ for every $A\in\mathfrak{U}_{\alpha}$.\\
But $\beta+S_{n}(\alpha)=\beta\h\alpha\in G_{\mathfrak{U}_{\beta}\oplus\mathfrak{U}_{\alpha}}=G_{\mathfrak{U}_{\alpha}\oplus\mathfrak{U}_{\alpha}}$, so $\mathfrak{U}_{\alpha}$ is additively idempotent.\\\end{proof}

With the same sort of considerations, we obtain a characterization of the multiplicatively idempotent ultrafilters:

\begin{prop} Let $\alpha$ be an hypernatural number in $^{\bullet}\N$ with height $n$. The following properties are equivalent:
\begin{enumerate}
	\item $\mathfrak{U}_{\alpha}$ is multiplicatively idempotent;
	\item $\alpha\sim_{u}\alpha\dia\alpha$;
	\item There is an element $\beta$ in $^{\bullet}\N$ with $\alpha\sim_{u}\beta\sim_{u}\alpha\dia\beta$;
	\item For every $\beta\sim_{u}\alpha$, $\alpha\dia\beta\sim_{u}\alpha$;	
  \item There is an element $\beta$ in $^{\bullet}\N$ such that $(\alpha,\beta)$ is a tensor pair and $\mathfrak{U}_{\alpha}=\mathfrak{U}_{\beta}=\mathfrak{U}_{\alpha\cdot\beta}$;
	\item For every subset $A$ of $\N$, if $\alpha\in S_{n}(A)$ then there exists a subset $B$ of $A$ such that $\alpha\in S_{n}(B)$ and $\alpha\cdot B\subseteq S_{n}(B)$;
	\item For every subset $A$ of $\N$, if $\alpha\in S_{n}(A)$, then there is an element $a$ in $A$ such that $\alpha\cdot a\in S_{n}(A)$.
\end{enumerate}
\end{prop}

The proof can be deduced from that of Proposition 3.2.2.\\
The results in this chapter involve also Schur, Folkman and Van der Waerden ultrafilters:

\begin{defn} An ultrafilter $\U$ in $\bN$ is 
\begin{enumerate}
	\item {\bfseries Schur} if every element $A\in\U$ satisfies Schur's property, i.e. if there are mutually distinct elements $a,b,c\in A$ such that $a+b=c$;
	\item {\bfseries Folkman} if every element $A\in\U$ satisfies Folkman's property, i.e. if for every natural number $k$ there is a subset $S_{k}=\{s_{1},...,s_{k}\}\subseteq A$ with $k$ elements such that \begin{center}$FS(S_{k})=\{\sum_{i\in I} s_{i}\mid I\neq\emptyset, I\subseteq \{1,...,k\}\}\subseteq A$;\end{center}
	\item {\bfseries Van der Waerden} if every element $A\in\U$ satisfies Van der Waerden's property, i.e. if there are arbitrarily long arithmetic progressions in $A$.
\end{enumerate}

\end{defn}

Observe that Schur's property is existential, and that Folkman's and Van der Waerden's properties are infinite conjunctions of existential properties.\\
The first result we present involves Schur ultrafilters:

\begin{prop} Every nonprincipal additively idempotent ultrafilter $\U$ is a Schur ultrafilter.\end{prop}

\begin{proof} Since Schur's property is existential, as a consequence of the Bridge Theorem, in order to prove that $\U$ is a Schur ultrafilter it is sufficient to show that there are three mutually different elements $\alpha,\beta,\gamma\in G_{\U}$ with $\alpha+\beta=\gamma$.\\
Let $\alpha\in$$^{*}\N$ be any generator of $\U$; for every element $\xi\in G_{\U}$, by Proposition 2.5.11 we deduce that $^{*}\xi\in G_{\U}$ and, as $\mathfrak{U}_{\alpha}$ is idempotent, by point four of Proposition 3.2.2 we deduce that $\alpha\h\xi$ is in $G_{\U}$. In particular, if $\xi=\alpha$, by letting $\beta=$$^{*}\alpha$ and $\gamma=\alpha\h\alpha=\alpha+$$^{*}\alpha$, the three elements $\alpha,\beta,\gamma$ are in $G_{\U}$ and $\alpha+\beta=\gamma$, so $\U$ is a Schur ultrafilter.\\\end{proof}

As a corollary, since we know that in $\bN$ there are additively idempotent ultrafilters (see Chapter One, Section 1.3), it follows that the family $\mathcal{F}_{S}$ of subsets of $\N$ satisfying the Schur's property is weakly partition regular, and this is the content of Schur's Theorem.\\
Schur's property has a multiplicative analogue, that we call multiplicative Schur's property:

\begin{defn} An element $\U$ of $\bN$ is a {\bfseries multiplicative Schur ultrafilter} if every element $A$ of $\U$ satisfies the multiplicative Schur's property, i.e. if there are three mutually distinct elements $a,b,c\in A$ such that $a\cdot b=c$. \end{defn}

With a proof which is very similar to that of Proposition 3.2.5 one can prove that every multiplicatively idempotent ultrafilter $\U$ is a multiplicative Schur ultrafilter. Here we present a different proof of the existence of this kind of ultrafilters. We recall that, given an ultrafilter $\U$ in $\bN$, $2^{\U}$ is the image of $\U$ respect the continuous extension $\overline{exp}\in\mathtt{Fun}(\bN,\bN)$ of the function $exp\in\mathtt{Fun}(\N,\N)$ such that, for every natural number $n$, $exp(n)=2^{n}$.

\begin{prop} If $\U$ is an additively idempotent ultrafilter then $\V=2^{\U}$ is a multiplicative Schur ultrafilter. \end{prop}

\begin{proof} The multiplicative Schur's property is expressed by an existential sentence, so it is sufficient to show that in $G_{\V}$ there are three mutually different elements $\alpha,\beta.\gamma$ with $\alpha\cdot\beta=\gamma$.\\
Since $\U$ is additively idempotent, as a result of Proposition 3.2.5 there are three mutually distinct elements $\eta,\mu,\xi$ in $G_{\U}$ with $\eta+\mu=\xi$. As a consequence of Theorem 2.3.5, the elements $2^{\eta},2^{\mu},2^{\xi}$ are three mutually distinct elements in $G_{\V}$. Observe that if $\alpha=2^{\eta}$, $\beta=2^{\mu}$, $\gamma=2^{\xi}$ then $\alpha\cdot\beta=\gamma$, so $\V$ is a multiplicative Schur ultrafilter.\\\end{proof}

\begin{defn} Given any natural number $n\geq 3$, let $AP_{n}$ be the existential formula 

\begin{center} $AP_{n}:$ $\exists x_{1},...,x_{n} ((x_{2}-x_{1}\neq 0)\wedge\bigwedge_{i=1}^{n-2} (x_{i+1}-x_{i}=x_{i+2}-x_{i+1}))$. \end{center}

\end{defn}

A subset $A$ of $\N$ satisfies $AP_{n}$ if and only if it contains an arithmetic progression of lenght $n$. In particular, Van der Waerden's property is the infinite conjunction $\bigwedge_{n=1}^{\infty}AP_{n}$.

\begin{prop} If $\U$ is a non principal additively idempotent ultrafilter then $2\U\oplus\U$ and $\U\oplus 2\U$ are $AP_{3}$-ultrafilters.\end{prop}

\begin{proof} Since $AP_{3}$ is an existential sentence, in order to prove the thesis is enough to find three mutually different elements in $G_{2\U\oplus\U}$ (resp. in $G_{\U\oplus 2\U}$) that are in arithmetic progression.\\
Observe that, as $\U$ is idempotent, also $2\U$ is idempotent, since $2\U\oplus 2\U=2(\U\oplus\U)=2\U$. Let $\xi\in$$^{*}\N$ be any element in $G_{\U}$; by idempotency, $\xi,$$^{*}\xi,\xi+$$^{*}\xi$ are in $G_{\U}$ and $2\xi,2$$^{*}\xi,2\xi+2$$^{*}\xi$ are in $G_{2\U}$.\\
In particular, in $G_{2\U\oplus\U}$ one finds 
\begin{enumerate}
	\item $2\xi+$$^{**}\xi=2\xi\h$$^{*}\xi$
	\item $2\xi+$$^{*}\xi+$$^{**}\xi=2\xi\h(\xi+$$^{*}\xi)$
	\item $2\xi+2$$^{*}\xi+$$^{**}\xi=(2\xi+2$$^{*}\xi)\h\xi$
\end{enumerate}
These three elements form an arithmetic progression of length 3 in $G_{2\U\oplus\U}$ (with rate $^{*}\xi$), so $2\U\oplus\U$ is a $AP_{3}$-ultrafilter.\\
Similarly, in $G_{\U\oplus 2\U}$ there is the arithmetic progression $\xi+2$$^{**}\xi$, $\xi+$$^{*}\xi+2$$^{**}\xi$, $\xi+2$$^{*}\xi+2$$^{**}\xi$ of length three, so also $\U\oplus 2\U$ is an $AP_{3}$-ultrafilter.\\\end{proof}

From this proposition it follows that the family $\mathcal{F}_{AP_{3}}$ of subsets of $\N$ containing an arithmetic progression of length three is weakly partition regular; that is, every finite coloration of $\N$ has a monochromatic three terms arithmetic progression. We think that the above nonstandard proof of this fact is an example of the advantages of the star iteration technique: both the combinatorial proof and the proof with ultrafilters given in Chapter One were less intuitive and more complex (expecially the combinatorial one, that was made only for 2-colorations).\\
Also, this proof can be generalized to obtain this result:

\begin{prop} If $\U$ is a non principal additively idempotent ultrafilter and $n,m$ are different positive natural numbers, then $m\U\oplus n\U$ and $n\U\oplus m\U$ are $\varphi_{n,m}$-ultrafilters, where $\varphi_{n,m}$ is the existential sentence 

\begin{center} $\varphi_{n,m}: \exists x,y,z ((y-x\neq 0)\wedge(n(y-x)=m(z-x)))$"'. \end{center}
\end{prop} 

\begin{proof} Let $\xi\in$$^{*}\N$ be an element in $G_{\U}$. As $n\U$ and $m\U$ are additively idempotent, by construction one finds the following elements in $G_{n\U\oplus m\U}$:
\begin{enumerate}
	\item $\alpha=n\xi+m$$^{**}\xi$
	\item $\beta=n\xi+m$$^{*}\xi+n$$^{**}\xi$
	\item $\gamma=n\xi+n$$^{*}\xi+m$$^{**}\xi$.
\end{enumerate}
Notice that $m(\gamma-\alpha)=mn$$^{*}\xi=n(\beta-\alpha)$, so $n\U\oplus m\U$ is a $\varphi_{n,m}$-ultrafilter.\\A similar proof can be done for $m\U\oplus n\U$.\\\end{proof}

\begin{cor} Given two different positive natural numbers $n, m$, if $\N$ is finitely colored then there are three mutually distinct monochromatic natural numbers $a,b,c$ such that $n(c-a)=m(b-a)$. \end{cor}

\begin{proof} By Proposition 3.2.10 it follows that the family of subsets of $\N$ satisfying the existence of three mutually distinct natural numbers $a,b,c$ with $n(c-a)=m(b-a)$ is weakly partition regular, since it contains an ultrafilter.\\ \end{proof}

Actually, this result can be strenghtened:

\begin{thm} For every natural number $k\geq 1$, for every positive natural numbers $n_{1},n_{2},...,n_{k+1}$ with $n_{i}\neq n_{i+1}$ for every index $i\leq k$, there exists a $\varphi_{n_{1},...,n_{k}}$-ultrafilter $\U$, where
\begin{center} $\varphi_{n_{1},...,n_{k}}:$ $\exists x_{1},...,x_{k},y_{1},...,y_{k},z_{1},...,z_{k}$ such that for every index $i\leq k$ $x_{i},y_{i},z_{i}$ are three mutually distinct elements and, for every index $i\leq k-1$, the following two conditions holds:
\begin{enumerate}
	\item $n_{i}(z_{i}-x_{i})=n_{i+1}(y_{i}-x_{i})$;
	\item $x_{i+1}=z_{i}$ $($if $i\leq k-1)$.
\end{enumerate}
\end{center}

\end{thm}

\begin{proof} Let $\V$ be an additively idempotent ultrafilter, let $\xi\in$$^{*}\N$ be a generator of $\V$, and consider the ultrafilter

\begin{center} $\U=n_{1}\V\oplus n_{2}\V\oplus...\oplus n_{k}\V$.\end{center}

{\bfseries Claim:} $\U$ is a $\varphi_{n_{1},...,n_{k}}$-ultrafilter.\\

Since $\varphi_{n_{1},...,n_{k}}$ is an existential sentence, to prove the claim it is enough to prove that there are elements $\alpha_{1},...,\alpha_{k},\beta_{1},...,\beta_{k},\gamma_{1},...,\gamma_{k}$ in $G_{\U}$ such that, for every index $i\leq k$, $\alpha_{i},\beta_{i},\gamma_{i}$ are mutually distinct, $n_{i}(\gamma_{i}-\alpha_{i})=n_{i+1}(\beta_{i}-\alpha_{i})$ and $\alpha_{i+1}=\gamma_{i}$.\\
We construct the elements $\alpha_{i},\beta_{i},\gamma_{i}$ inductively: let

\begin{itemize}
	\item $\alpha_{1}=\sum_{i=1}^{k}(n_{i}S_{2(i-1)}(\xi))$;
	\item $\beta_{1}=\alpha_{1}+n_{1}$$^{*}\xi$;
	\item $\gamma_{1}=\alpha_{1}+n_{2}$$^{*}\xi$.
\end{itemize}

Observe that, by construction, $n_{2}(\beta_{1}-\alpha_{1})=n_{2}\cdot n_{1}$$^{*}\xi=n_{1}(\gamma_{1}-\alpha_{1})$ and $\alpha_{1},\beta_{1},\gamma_{1}$ are generators of $\U$.\\
Now, if $\alpha_{h},\beta_{h},\gamma_{h}$ have been constructed, pose

\begin{itemize}
	\item $\alpha_{h+1}=\gamma_{h}$;
	\item $\beta_{h+1}=\alpha_{h+1}+n_{i+1}S_{2h-1}(\xi)$;
	\item $\gamma_{h+1}=\alpha_{h+1}+n_{i}S_{2h-1}(\xi)$.
\end{itemize}

Observe that $\alpha_{h+1}=\gamma_{h}$, $n_{h+2}(\beta_{h+1}-\alpha_{h+1})=n_{h+1}\cdot n_{h+1}S_{2h-1}(\xi)=n_{h+1}(\gamma_{h+1}-\alpha_{h+1})$ and that $\alpha_{h+1},\beta_{h+1},\gamma_{h+1}$ are generators of $\U$.\\
With this procedure we constuct elements $\alpha_{1},...,\alpha_{k},\beta_{1},...\beta_{k},\gamma_{1},...,\gamma_{k}$ in $G_{\U}$ with the desired properties, so $\U$ is a $\varphi_{n_{1},...,n_{k}}$-ultrafilter.\\\end{proof}

\begin{cor}  For every natural number $k\geq 1$, for every positive natural numbers $n_{1},n_{2},...,n_{k+1}$ with $n_{i}\neq n_{i+1}$ for every index $i\leq k$, if $\N$ is finitely colored then there are monochromatic elements $a_{1},...,a_{k},b_{1},...,b_{k},c_{1},...,c_{k}$ such that
\begin{enumerate}
  \item for every index $i\leq k$, $a_{i},b_{i},c_{i}$ are mutually distinct;
	\item for every index $i\leq k$, $n_{i}(z_{i}-x_{i})=n_{i+1}(y_{i}-x_{i})$;
	\item for every index $i\leq k-1$, $x_{i+1}=z_{i}$.
\end{enumerate}
\end{cor}

Now we focus on a very important result in Ramsey Theory that we discussed also in Chapter One: Folkman's Theorem. In Chapter One we proved that Folkman's Theorem can be deduced as a particular case of a more general result, Hindman's Theorem, which ultrafilter proof is based on the existence of additively idempotent ultrafilters.\\
We present a nonstandard proof of Folkman's Theorem, to show one other example of star iterations:

\begin{thm} Every nonprincipal additively idempotent ultrafilter $\U$ is a Folkman ultrafilter. \end{thm}

\begin{proof} For every natural number $k$, let 

\begin{center} $f_{k}:\{I\in\wp(\{1,...,k\})\mid I\neq\emptyset\}\rightarrow\{1,...,2^{k}-1\}$ \end{center} 

be a bijection, and let $E(\phi_{k}(x_{1},....,x_{k}, y_{1},...,y_{2^{k}-1}))$ the existential sentence 

\begin{center} $E(\phi_{k}(x_{1},....,x_{k}, y_{1},...,y_{2^{k}-1})): \exists x_{1},....,x_{k}, y_{1},...,y_{2^{k}-1}$ such that, for every $I\neq\emptyset\subseteq\{1,...,k\}$, $\sum_{i\in I}x_{i}=y_{f(I)}$. \end{center}

We observe that a subset $A$ of $\N$ contains a subset $T_{k}$ of cardinality $k$ with $FS(T_{k})\subseteq A$ if and only if $A$ satisfies $E(\phi_{k}(x_{1},....,x_{k}, y_{1},...,y_{2^{k}-1}))$.\\

{\bfseries Claim:} $\U$ is an $E(\phi_{k}(x_{1},....,x_{k}, y_{1},...,y_{2^{k}-1}))$-ultrafilter for every $k>0$ in $\N$.\\

Take any $k>0$ in $\N$ and, if $\U=\mathfrak{U}_{\alpha}$, consider

\begin{center}$\alpha,$$^{*}\alpha,$$...,$$S_{k-1}(\alpha)$. \end{center}

Pose $T_{k}=\{\alpha,$$^{*}\alpha,$$...,$$S_{k-1}(\alpha)\}$ and observe that $FS(T_{k})\subseteq G_{\U}$ since, for every $I\neq\emptyset\subseteq\{1,...,k\}$, $\sum_{i\in I}S_{i-1}(\alpha)$ in in $G_{\U}$.\\
In particular, $\U$ is an $E(\phi_{k}(x_{1},....,x_{k}, y_{1},...,y_{2^{k}-1}))$-ultrafilter for every $k>0$, and this entails that it is a Folkman ultrafilter.\\ \end{proof}

From this result it follows that the family of sets satisfying the Folkman's property is weakly partition regular, and this is the content of Folkman's Theorem.\\
We end this section with a result that is not "Ramsey-style": we show, as promised in Chapter One, that the center of $(\bN,\oplus)$ is $\N$. This proof can be found in $\cite{rif15}$.\\
Let $(a_{k})_{k\in \N}$ be a sequence of natural numbers such that $\lim_{k} (a_{k+1}-a_{k})= +\infty$, and consider the set $A=\bigcup_{k\in \N} [a_{2k},a_{2k+1})$.

\begin{prop} For every non principal ultrafilter $\U$ there exists an ultrafilter $\V$ such that $A\in \U\oplus\V\Leftrightarrow A^{c}\in\V\oplus\U$. \end{prop}

\begin{proof} Observe that for every natural number $n$ the set $A$ contains many intervals of length greater than $n$, since $\lim_{k} (a_{k+1}-a_{k})= +\infty$. By transfer, for every hypernatural number $\mu$ the hyperextension $^{*}A$ contains many intervals of length greater than $\mu$, and the same also holds for $^{*}A^{c}$. This is a key property in the proof.\\
Consider $^{*}A$, and pose $\U=\mathfrak{U}_{\alpha}$. Suppose that $\alpha\in$$^{*}A$ (the case $\alpha\in$$^{*}A^{c}$ is similar).\\
There are two possibilities:\\
1) For every natural number $n$, $\alpha+n\in$$^{*}A$. By transfer it follows that, for every hypernatural number $\mu\in$$^{*}\N$, $^{*}\alpha+\mu\in$$^{**}A$, which entails that $A\in \mathfrak{U}_{\mu}\oplus\mathfrak{U}_{\alpha}$ for every $\mu\in$$^{*}\N$. If there is an hypernatural number $\mu$ in $^{*}\N$ with $\alpha+$$^{*}\mu\in$$^{**}A^{c}$ (i.e. $A^{c}\in \mathfrak{U}_{\alpha}\oplus\mathfrak{U}_{\mu}$) we can conclude just choosing $\V=\mathfrak{U}_{\mu}$.\\
We know that $^{*}A^{c}$ contains arbitrarily long intervals, so there exists an hypernatural number $\eta$ such that the interval $I=[a_{\eta},a_{\eta+1})$ has length greater than $\alpha$ and is included in $^{*}A^{c}$. In particular, by letting $\mu=a_{\eta}$, we have that $\mu+n\in I$ for every natural number $n$ and so, by transfer, for every hypernatural number $\xi$ we have that $^{*}\mu+\xi\in$$^{*}I\subseteq$$^{**}A^{c}$. Posing $\xi=\alpha$ we get the thesis.\\
2) There exists $n\in\N$ such that $\alpha+n\in$$^{*}A^{c}$. This entails, since the intervals $[a_{2\eta},a_{2\eta+1})$ are infinite for $\eta\in$$^{*}\N\setminus\N$, that for every natural number $m\geq n$, $\alpha+m\in$$^{*}A^{c}$. By transfer it follows that for every hypernatural number $\mu$ in $^{*}\N$ with $\mu\geq n$, $^{*}\alpha+\mu\in$$^{**}A^{c}$; in particular, $A^{c}\in \mathfrak{U}_{\mu}\oplus\mathfrak{U}_{\alpha}$ for every $\mu\in$$^{*}\N\setminus\N$. If there is an hypernatural number $\mu$ in $^{*}\N\setminus\N$ such that $\alpha+$$^{*}\mu\in$$^{**}A$, we can conclude.\\
The proof follows the same ideas as the second part of (1): this time we find $\eta$ such that $I=[a_{2\eta},a_{2\eta+1})$ is included in $^{*}A$ and has length greater than $\alpha$. So, again, if $\mu=a_{2\eta}$ then by transfer we get that $\alpha+$$^{*}\mu\in$$^{**}I\subseteq$$^{**}A$, and this entails the thesis. \\ \end{proof} 

\begin{cor} The center of $(\bN,\oplus)$ is $\N$. \end{cor}

\begin{proof} By the previous proposition, for every non principal ultrafilter $\U$ there exists an ultrafilter $\V$ such that $\U\oplus\V\neq\V\oplus\U$, so $\U$ is not in the center of $(\bN,\oplus)$.\\
Conversely, if $\U=\mathfrak{U}_{n}$ for some $n\in\N$ then, for every ultrafilter $\V=\mathfrak{U}_{\alpha}$, $\mathfrak{U}_{n}\oplus\mathfrak{U}_{\alpha}=\mathfrak{U}_{n+\alpha}=\mathfrak{U}_{\alpha}\oplus\mathfrak{U}_{n}$, which is the thesis. \\ \end{proof}

\section{Invariant Formulas and Ultrafilters}

In this section we talk about the relationships between the operation $\oplus$ of sum of ultrafilters (resp. the operation $\odot$ of product of ultrafilters) and particular formulas. The following result is an example:

\begin{prop} If $\U,\V$ are Schur ultrafilters, then $\U\oplus\V$ is a Schur ultrafilter.\end{prop}

\begin{proof} Let $\alpha,\beta\in$$^{*}\N$ be generators of $\U$ such that $\alpha+\beta\in G_{\U}$, and let $\gamma,\delta\in$$^{*}\N$ be generators of $\V$ such that $\gamma+\delta\in G_{\V}$. By point 15 of Proposition 2.5.28 it follows that $(\alpha+\beta)\h(\gamma+\delta)=(\alpha\h\gamma)+(\beta\h\delta)$; as in $G_{\U\oplus\V}$ there are $\eta=\alpha\h\gamma, \mu=\beta\h\delta$ and $\xi=(\alpha+\beta)\h(\gamma+\delta)$, and $\eta+\mu=\xi$, it is proved that $\U\oplus\V$ is a Schur ultrafilter.\\\end{proof}

The previous result is a particular case of a general important phenomenon:

\begin{defn} Let $\phi(x_{1},...,x_{n})$ be a first order open formula. $\phi(x_{1},...,x_{n})$ is {\bfseries additive} if, for every $a_{1},...,a_{n},b_{1},...,b_{n}\in\N$, if $\phi(a_{1},...,a_{n})$ and $\phi(b_{1},...,b_{n})$ hold, then $\phi(a_{1}+b_{1},....,a_{n}+b_{n})$ holds.\\
Similarly, $\phi(x_{1},...,x_{n})$ is {\bfseries multiplicative} if, for every $a_{1},...,a_{n},b_{1},...,b_{n}\in\N$, if $\phi(a_{1},...,a_{n})$ and $\phi(b_{1},...,b_{n})$ holds then $\phi(a_{1}\cdot b_{1},...,a_{n}\cdot b_{n})$ holds.\end{defn}

\begin{defn} Let $\phi(x_{1},...,x_{n})$ be a first order open formula. $\phi(x_{1},...,x_{n})$ is {\bfseries additively invariant} if, for every $a_{1},...,a_{n},b$ in $\N$,

\begin{center} $\phi(a_{1},...,a_{n})$ holds if and only if $\phi(a_{1}+b,...,a_{n}+b)$ holds. \end{center}

Similarly, $\phi(x_{1},...,x_{n})$ is {\bfseries multiplicatively invariant} if, for every $a_{1},...,a_{n},b$ in $\N$ with $b\neq 0$,

\begin{center} $\phi(a_{1},...,a_{n})$ holds if and only if $\phi(a_{1}\cdot x,...,a_{n}\cdot x)$ holds. \end{center}

\end{defn}

\begin{defn} A first order existential sentence $E(\phi(x_{1},...,x_{n}))$ is {\bfseries additive} $($resp. {\bfseries multiplicative}$)$ if $\phi(x_{1},...,x_{n})$ is additive $($resp. multiplicative$)$.\\
$E(\phi(x_{1},...,x_{n}))$ is {\bfseries additively invariant} $($resp. {\bfseries multiplicatively invariant}$)$ if $\phi(x_{1},...,x_{n})$ is additively invariant $($resp. multiplicatively invariant$)$.\end{defn}

E.g. Schur's property is additive, and it is multiplicatively invariant, and the same holds for the properties $AP_{n}$, that are also multiplicative and additively invariant.\\
In this section we consider only elementary formulas, because they satisfy an important property: if $\varphi(x_{1},...,x_{n})$ is an elementary first order formula, and $\alpha_{1},...,\alpha_{n}$ are hypernatural numbers in $^{*}\N$, then

\begin{center} $^{*}\varphi(\alpha_{1},...,\alpha_{n})$ holds if and only if $^{**}\varphi(\alpha_{1},...,\alpha_{n})$ holds. \end{center}

\begin{thm} Let $\varphi=E(\phi(x_{1},...,x_{n}))$ be an elementary first order existential sentence. Then the following conditions hold:
\begin{enumerate}
	\item if $\varphi$ is additive, and $\U,\V$ are $\varphi$-ultrafilters, then $\U\oplus\V$ is a $\varphi$-ultrafilter as well;
	\item if $\varphi$ is multiplicative, and $\U,\V$ are $\varphi$-ultrafilters, then $\U\odot\V$ is a $\varphi$-ultrafilter as well;
	\item if $\varphi$ is additively invariant, $\U$ is a $\varphi$-ultrafilter and $\V$ is any ultrafilter, then $\U\oplus\V$ and $\V\oplus\U$ are $\varphi$-ultrafilters as well;
	\item if $\varphi$ is multiplicatively invariant, $\U$ is a $\varphi$-ultrafilter and $\V\neq\mathfrak{U}_{0}$ is any ultrafilter, then $\U\odot\V$ and $\V\odot\U$ are $\varphi$-ultrafilters as well.
\end{enumerate}
\end{thm}

\begin{proof} 1) Let $\alpha_{1},...,\alpha_{n}\in$$^{*}\N$ be generators of $\U$ such that $^{*}\phi(\alpha_{1},...,\alpha_{n})$ holds, and $\beta_{1},...,\beta_{n}\in$$^{*}\N$ generators of $\V$ such that $^{*}\phi(\beta_{1},...,\beta_{n})$ holds. By transfer, since $\varphi$ is elementary, $^{**}\phi(\alpha_{1},...,\alpha_{n})$ holds. By transfer, since $^{*}\phi(\beta_{1},...,\beta_{n})$ holds then $^{**}\phi($$^{*}\beta_{1},...,$$^{*}\beta_{n})$ holds. By additivity of $\phi$, and by transfer, it follows that 

\begin{center} $^{**}\phi(\alpha_{1}+$$^{*}\beta_{1},...,\alpha_{n}+$$^{*}\beta_{n})$\end{center}

holds. Since $\alpha_{1}+$$^{*}\beta_{1}$, ..., $\alpha_{n}+$$^{*}\beta_{n}$ are generators of $\U\oplus\V$ it follows that $\U\oplus\V$ is a $\varphi$-ultrafilter.\\
2) This is similar to (1), with the multiplication in replacement of the addition.\\
3) Let $\alpha_{1},...,\alpha_{n}\in$$^{*}\N$ be generators of $\U$ such that $^{*}\phi(\alpha_{1},...,\alpha_{n})$ holds, and let $\beta\in$$^{*}\N$ be any generator of $\V$. Then, since $\varphi$ is elementary, by additive invariance and transfer it follows that $^{*}\phi(\alpha_{1},...,\alpha_{n})$ is equivalent to $^{**}\phi(\alpha_{1}+$$^{*}\beta,...,\alpha_{n}+$$^{*}\beta)$, and $\alpha_{1}+$$^{*}\beta,...,\alpha_{n}+$$^{*}\beta$ are generators of $\U\oplus\V$, so $\U\oplus\V$ is a $\varphi$-ultrafilter.\\
To prove that $\V\oplus\U$ is a $\varphi$-ultrafilter, we observe that, by transfer, since $^{*}\phi(\alpha_{1},...,\alpha_{n})$ holds also $^{**}\phi($$^{*}\alpha_{1},...,$$^{*}\alpha_{n})$ holds, and by transfer and additive invariance $^{**}\phi(\beta+$$^{*}\alpha_{1},...,\beta+$$^{*}\alpha_{n})$ holds for every $\beta$ in $^{**}\N$; as $\beta+$$^{*}\alpha_{1},...,\beta+$$^{*}\alpha_{n}$ are generators of $\V\oplus\U$, $\V\oplus\U$ is a $\varphi$-ultrafilter.\\
4) The proof for the multiplicatively invariant existential formulas is analogue to (3).\\\end{proof}

There are some interesting corollaries of the above result:

\begin{cor} For every natural number $k$, for every index $i\leq k$, let $\varphi_{i}$ be an elementary first order existential sentence and $\U_{i}$ a $\varphi_{i}$-ultrafilter. Then 
\begin{enumerate}
	\item if, for every $i\leq k$, $\varphi_{i}$ is additive then $\U_{1}\oplus....\oplus\U_{k}$ is a $(\bigwedge_{i=1}^{k}\varphi_{i})$-ultrafilter;
	\item if, for every $i\leq k$, $\varphi_{i}$ is additively invariant then $\U_{1}\oplus....\oplus\U_{k}$ is a $(\bigwedge_{i=1}^{k}\varphi_{i})$-ultrafilter;
	\item if, for every $i\leq k$, $\varphi_{i}$ is multiplicative then $\U_{1}\odot....\odot\U_{k}$ is a $(\bigwedge_{i=1}^{k}\varphi_{i})$-ultrafilter;
	\item if, for every $i\leq k$, $\varphi_{i}$ is multiplicatively invariant then $\U_{1}\odot....\odot\U_{k}$ is a $(\bigwedge_{i=1}^{k}\varphi_{i})$-ultrafilter.
\end{enumerate}
\end{cor}

\begin{proof} This follows trivially by Theorem 3.3.5. \\\end{proof}

E.g., if $\U$ is a Schur ultrafilter and $\V$ is an $AP_{3}$-ultrafilter, since these two properties are multiplicatively invariant the ultrafilter $\U\odot\V$ is both a Schur and an $AP_{3}$-ultrafilter.

\begin{cor} Let $\varphi$ be an elementary first order existential sentence, and \begin{center} $\mathcal{S}_{\varphi}=\{\U\in\bN\mid\U$ is a $\varphi$-ultrafilter$\}$. \end{center}
Then
\begin{enumerate}
	\item if $\varphi$ is additive then $\mathcal{S}_{\varphi}$ is closed under $\oplus$;
	\item if $\varphi$ is multiplicative then $\mathcal{S}_{\varphi}$ is closed under $\odot$;
	\item if $\varphi$ is additively invariant then $\mathcal{S}_{\varphi}$ is a bilateral ideal in $(\bN,\oplus)$;
	\item if $\varphi$ is multiplicatively invariant then $\mathcal{S}_{\varphi}$ is a bilateral ideal in \\$(\bN\setminus\{\mathfrak{U}_{0}\},\odot)$.
\end{enumerate}
\end{cor}

\begin{proof} These conditions are straightforward consequences of Theorem 3.3.5. \\ \end{proof}

\begin{cor} If $\mathcal{S},\mathcal{F},\mathcal{VDW}$ are the sets
\begin{enumerate}
	\item $\mathcal{S}=\{\U\in\bN\mid \U$ is a Schur ultrafilter$\}$;
	\item $\mathcal{F}=\{\U\in\bN\mid \U$ is a Folkman ultrafilter$\}$;
	\item $\mathcal{VDW}=\{\U\in\bN\mid \U$ is a Van der Waerden ultrafilter$\}$;
\end{enumerate}
then
\begin{enumerate}
	\item $\mathcal{S}$ is a bilateral ideal in $(\bN\setminus\mathfrak{U}_{0},\odot)$, and it is closed under $\oplus$;
	\item $\mathcal{F}$ and $\mathcal{VDW}$ are bilateral ideals both in $(\bN,\oplus)$ and in $(\bN\setminus\mathfrak{U}_{0},\odot)$.
\end{enumerate}
\end{cor}

\begin{proof} That $\mathcal{S}$ is a bilateral ideal in $(\bN\setminus\mathfrak{U}_{0},\odot)$ follows by Corollary 3.3.7, since Schur's property is multiplicatively invariant.\\
Consider the set $\mathcal{VDW}$. Observe that an ultrafilter is in $\mathcal{VDW}$ if and only if it is an $AP_{n}$-ultrafilter for every natural number $n$; so, if for every natural number $n$ we consider the set

\begin{center} $\mathcal{AP}_{n}=\{\U\in\bN\mid \U$ is an $AP_{n}$-ultrafilter$\}$, \end{center}

by construction $\mathcal{VDW}=\bigcap_{n\in\N} \mathcal{AP}_{n}$. As, for every natural number $n$, the formula $AP_{n}$ is both additively and multiplicatively invariant, the sets $\mathcal{AP}_{n}$ are bilateral ideals both in $(\bN,\oplus)$ and in $(\bN\setminus\mathfrak{U}_{0},\odot)$. So $\mathcal{VDW}$, being an intersection of bilateral ideals, is a bilateral ideal both in $(\bN,\oplus)$ and in $(\bN\setminus\mathfrak{U}_{0},\odot)$.\\
The proof for the set $\mathcal{F}$ is similar, and can be done replacing the formulas $AP_{n}$ with the formulas $E(\phi_{k}(x_{1},...,x_{k},y_{1},...,y_{2^{k}-1})$ introduced in the proof of Theorem 3.2.14, that are both additively and multiplicatively invariant.\\ \end{proof}

\section{Idempotent $\varphi$-ultrafilters}

In this section we investigate the relations between limits of sequences of ultrafilters and $\varphi$-ultrafilters for generic weakly partition regular sentences $\varphi$. In particular, we shall focus our attention to additively (and multiplicatively) invariant existential sentences. We recall that, given a nonempty set $I$, a sequence $\mathcal{F}=\langle \U_{i}\mid i\in I\rangle$ of elements in $\bN$ and an ultrafilter $\V$ on $I$, the $\V-\lim_{I} $ of the sequence $\mathcal{F}$ is the ultrafilter

\begin{center} $\V-\lim_{I}\U_{i}=\{A\subseteq\N\mid\{i\in I\mid A\in\U_{i}\}\in\V\}$. \end{center}

Where not explicitly stated otherwise, in the following we suppose that the set of indexes $I$ is $\N$; in this case, the $\V-\lim_{\N}\U_{n}$ is simply denoted by $\V-\lim\U_{n}$. The results we prove hold also in the general case $I\neq\N$.\\
An important characteristic of limit ultrafilters is the following:

\begin{thm} Let $\V$ be an ultrafilter and, for every natural number $n$, let $\U_{n}$ be an $E(\phi_{n}(x_{1},...,x_{k_{n}}))$-ultrafilter. Then, for every element $A$ in $\U=\V-\lim\U_{n}$, the set $\{n\in\N\mid \exists x_{1},...,x_{n}\in A$ $\phi_{n}(x_{1},...,x_{k_{n}})\}\in\V$, and hence it is infinite whenever $\V$ is nonprincipal.\end{thm}

\begin{proof} For every set $A$ in $\U$, the set 

\begin{center} $I_{A}=\{n\in\N\mid A\in \U_{n}\}\in\V$.\end{center}

Now, for every $n\in I_{A}$, as $\U_{n}$ is an $E(\phi_{n}(x_{1},...,x_{k_{n}}))$-ultrafilter, $A$ satisfies $E(\phi_{n}(x_{1},...,x_{k_{n}}))$.\\\end{proof}

Observe that, as a consequence of the above proposition, whenever $m$ is a natural number there is a natural number $n$ greater than $m$ such that $A$ satisfies $E(\phi_{n}(x_{1},...,x_{k_{n}}))$.\\It is not difficul to show that, in general, $\U$ is not an $E(\phi_{n}(x_{1},...,x_{k_{n}}))$-ultrafilter for every natural number $n$: e.g., suppose that each $E(\phi_{n}(x))$ is the existential sentence "$\exists x (x=n)$": the only subset of $\N$ that satisfies all the sentences $\varphi_{n}$ is $\N$, so there is no ultrafilter which is a $\varphi_{n}(x)$-ultrafilter for all natural numbers $n$.\\
This leads to a question: under what conditions, given a family of weakly partition properties $\Phi=\{\varphi_{n}\}_{n\in\N}$, there is an ultrafilter $\U$ which is a $\varphi_{n}$-ultrafilter for every natural number $n$?\\
We recall that a sentence $\varphi$ is weakly partition regular if and only if there is a $\varphi$-ultrafilter on $\N$ (as we proved in Theorem 1.2.6).

\begin{thm} Let $\V$ be a nonprincipal ultrafilter on $\N$ and, for every natural number $n$, let $E(\phi_{n}(x_{1},...,x_{k_{n}}))$ be a weakly partition regular existential sentence, let $\U_{n}$ be an $E(\phi_{n}(x_{1},...,x_{k_{n}}))$-ultrafilter, and let $\U=\V-\lim \U_{n}$.\\
If the set $B=\{n\in\N\mid E(\phi_{n}(x_{1},...,x_{k_{n}}))\Rightarrow\bigwedge_{m\leq n}E(\phi_{m}(x_{1},...,x_{k_{m}}))\}$ is in $\V$ then $\U$ is an $E(\phi_{n}(x_{1},...,x_{k_{n}}))$-ultrafilter for every natural number $n$.\end{thm}

\begin{proof} Let $A$ be an element of $\U$, and consider the set $I_{A}=\{n\in\N\mid A\in\U_{n}\}$. Then $I_{A}\cap B\in\V$ is infinite since $\V$ is not principal. In particular, for every natural number $n$ in $I_{A}\cap B$, $A$ satisfies $E(\phi_{n}(x_{1},...,x_{k_{n}}))$ and, as $n$ is in $B$, $A$ satisfies also $\bigwedge_{m\leq n}E(\phi_{m}(x_{1},...,x_{k_{m}}))$. Since this happens for arbitrarily large natural numbers $n$, $A$ satisfies all the formulas in $\{E(\phi_{n}(x_{1},...,x_{k_{n}}))\mid n\in\N\}$, so $\U$ is an $E(\phi_{n}(x_{1},...,x_{k_{n}}))$-ultrafilter for every natural number $n$.\\ \end{proof}

\begin{cor} Let $\Phi=\langle E(\phi_{n}(x_{1},...,x_{k_{n}}))\mid n\in\N\rangle$ be a countable set of weakly partition regular existential sentences. If  

\begin{center} $S=\{n\in\N\mid E(\phi_{n}(x_{1},...,x_{k_{n}}))\Rightarrow\bigwedge_{m\leq n}E(\phi_{m}(x_{1},...,x_{k_{m}}))\}$ \end{center}

is infinite then there exists an ultrafilter $\U$ that is an $E(\phi_{n}(x_{1},...,x_{k_{n}}))$-ultrafilter for every $n\in\N$. In particular, the property $\bigwedge_{n\in\N}E(\phi_{n}(x_{1},...,x_{k_{n}}))$ is weakly partition regular. \end{cor}

\begin{proof} Let $\V$ be a nonprincipal ultrafilter on $\N$ such that $S\in\V$ and, for every natural number $n$, let $\U_{n}$ be an $E(\phi_{n}(x_{1},...,x_{k_{n}}))$-ultrafilter. Then by Theorem 3.4.2 it follows that $\U=\V-\lim\U_{n}$ is an $E(\phi_{n}(x_{1},...,x_{k_{n}}))$-ultrafilter for every natural number $n$ in $\N$.\\ \end{proof}

{\bfseries Example:} If $\V$ is a non principal ultrafilter, and $\langle\U_{n}\mid n\in\N\rangle$ is a sequence of of ultrafilters such that the set $\{n\in\N \mid \U_{n}$ is a $AP_{n}$-ultrafilter$\}$ is in $\V$, then $\U=\V-\lim\U_{n}$ is a Van der Waerden ultrafilter.\\

\begin{cor} Let $\varphi$ be a sentence. Then

\begin{center} $X_{\varphi}=\{\U\in\bN\mid\U$ is a $\varphi$-ultrafilter$\}$ \end{center}

is closed.\end{cor}

\begin{proof} Let $\langle\U_{i}\mid i\in I\rangle$ be a sequence of $\varphi$-ultrafilters, and let $\V$ be an ultrafilter on $I$. Then, if $A\in \V-\lim_{I}\U_{i}$, then $A\in\U_{i}$ for some index $i\in I$, so $A$ satisfies $\varphi$; since this is true for every $A\in\V-\lim\U_{i}$, then $\V-\lim\U_{i}$ is a $\varphi$-ultrafilter, so it is in $X$. We proved that $X$ is closed under the operation of limit ultrafilter, so by Theorem 1.1.34 it follows that $X$ is closed.\\ \end{proof}

In the previous result it has not been assumed that $\varphi$ is a first order sentence since it holds for any sentence $\varphi$.\\
We list three important examples of closed subsets of $\bN$:
\begin{enumerate}
	\item The set $\mathcal{S}=\{\U\in\bN\mid \U$ is a Schur ultrafilter$\}$;
	\item The set $\mathcal{F}=\{\U\in\bN\mid \U$ is a Folkman ultrafilter$\}$;
	\item The set $\mathcal{VDW}=\{\U\in\bN\mid \U$ is a Van der Waerden ultrafilter$\}$.
\end{enumerate}

That $\mathcal{S}$ is closed follows from Corollary 3.4.4; that $\mathcal{VDW}$ is closed can be proved observing that an ultrafilter $\U$ is a Van der Waerden ultrafilter if, for every $n\in\N$, it satisfies $AP_{n}$, and by last corollary, for every natural number $n$ the set 

\begin{center} $\mathcal{AP}_{n}=\{\U\in\bN\mid \U$ is an $AP_{n}$-ultrafilter$\}$ \end{center}

is closed. Then $\mathcal{VDW}$, which is the intersection $\bigcap_{n}\mathcal{AP}_{n}$, is closed as well.\\ 
The proof for the set $\mathcal{F}$ is similar, and can be done by replacing the sentences $AP_{n}$ with the sentences $E(\phi_{k}(x_{1},...,x_{k},y_{1},...,y_{2^{k}-1}))$ introduced in the proof of Theorem 3.2.14.\\
Putting together the results of Corollary 3.4.4 and of Theorem 3.3.5 we prove the following important theorem:

\begin{thm} Let $\varphi: \phi(x_{1},...,x_{n})$ be an existential elementary weakly partition regular formula, and suppose that $\varphi$ is additive or additively invariant. Then there is an additively idempotent $\varphi$-ultrafilter.\\
Similarly, if $\varphi$ is multiplicative or multiplicatively invariant, then there is a multiplicatively idempotent $\varphi$-ultrafilter. \end{thm}

\begin{proof} By the hypothesis of weakly partition regularity, the set 

\begin{center} $S_{\varphi}=\{\U\in\bN\mid \U$ is a $\varphi$-ultrafilter$\}$ \end{center}

is not empty. Corollary 3.4.4 ensures that $S_{\varphi}$ is closed (so, in particular, compact), while Theorem 3.3.5 ensures, when $\varphi$ is additive or additively invariant, that $S_{\varphi}$ is closed under the sum $\oplus$. Then Ellis's Theorem entails that there exists an additive idempotent in $S_{\varphi}$.\\
The proof for a multiplicative or multiplicatively invariant formula is analogue.\\ \end{proof}

\begin{cor} There is a multiplicatively idempotent Van der Waerden ultrafilter.\end{cor}

Since it is known that no ultrafilter is both multiplicatively and additively idempotent, by Corollary 3.4.6 it follows that there exists a Van der Waerden ultrafilter that is not additively idempotent. This may not seem strange, but we recall that the "canonical" proof of Van der Waerden's Theorem with ultrafilters consists in showing that special additively idempotent ultrafilters are Van der Waerden ultrafilters; this result shows that the converse is false: Van der Waerden ultrafilters are not necessarily additively idempotents.

\section{Partition Regularity of Polynomials}

In this section we face the problem of the partition regularity on $\N$ for polynomials with coefficients in $\Z$. First of all, we recall some important definitions about polynomials.

\begin{defn} Let $\mathbf{X}=\{x_{n}\}$ be a countable set of variables. The {\bfseries set of polynomials with variables in $\mathbf{X}$ and coefficients in $\Z$} $($notation $\Z[\mathbf{X}])$ is 

\begin{center} $\Z[\mathbf{X}]=\bigcup_{F\subseteq\wp_{fin}(\mathbf{X})} \Z[F]$,\end{center}

i.e. a polynomial $P$ is in $\Z[\mathbf{X}]$ if and only if there are variables $x_{1},...,x_{n}$ in $\mathbf{X}$ such that $P\in \Z[x_{1},...,x_{n}]$.\\
A {\bfseries monomial} is a polynomial in the form $M(x_{1},...,x_{k})=ax_{1}^{n_{1}}\cdot...\cdot x_{k}^{n_{k}}$, where $a\in \Z$, $x_{1},...,x_{n}\in X$ and $n_{1},...,n_{k}$ are natural numbers; the {\bfseries degree} of a monomial $M(x_{1},...,x_{k})$ is $d=\sum_{i=1}^{k} n_{i}$; the {\bfseries degree of $P(x_{1},...,x_{n})$} is the maximum degree of its monomials; given a variable $x_{i}$, the {\bfseries partial degree of $x_{i}$} in $P(x_{1},...,x_{n})$ is the degree of the polynomial obtained considering the variables distinct from $x_{i}$ as constants, i.e. considering $P(x_{1},...,x_{n})$ as a polynomial in $\Z[x_{1},...,x_{i-1},x_{i+1},...,x_{n}]$; the {\bfseries partial degree of $P(x_{1},...,x_{k})$} is the maximum partial degree of its variables; a polynomial is {\bfseries linear} if it has degree 1; it is {\bfseries homogeneous} if all its monomials have the same degree.\end{defn}

\begin{defn} Let $P(x_{1},...,x_{n})$ be a polynomial in $\Z[\mathbf{X}]$. $P(x_{1},...,x_{k})$ is 
\begin{enumerate}
	\item {\bfseries partition regular} $($on $\N)$ if the existential sentence 
	
	\begin{center} $\sigma_{P}$: $\exists a_{1},...,a_{n}\in\N\setminus\{0\}$ $P(a_{1},..,a_{n})=0$ \end{center} is weakly partition regular $($on $\N)$;
	\item {\bfseries injectively partition regular} $($on $\N)$ if the existential sentence \begin{center} $\iota_{P}: \exists a_{1},...,a_{n}\in\N\setminus\{0\}$  $(P(a_{1},...,a_{n})=0\wedge (\bigwedge_{i\neq j}a_{i}\neq a_{j}))$ \end{center} is weakly partition regular $($on $\N)$.
\end{enumerate}
\end{defn}

{\bfseries Convention:} Whenever $P(x_{1},...,x_{n})$ is a polynomial in $\Z[\mathbf{X}]$, we use the symbols $\sigma_{P}$ and $\iota_{P}$ to denote the existential sentences introduced in Definition 3.5.2.\\

Trivially, every injectively partition regular polynomial is also partition regular.\\
We make three assumptions:

\begin{enumerate}
	\item all polynomials are given in normal reduced form, i.e. for every polynomial $P(x_{1},...,x_{n})$ in $\Z[\mathbf{X}]$, if $M(x_{1},...,x_{n})=ax_{1}^{k_{1}}\cdot...\cdot x_{n}^{k_{n}}$ is one of its monomials then $a\neq 0$ and if $M_{1}(x_{1},...,x_{n})=a_{1}x_{1}^{k_{1}}\cdot...\cdot x_{n}^{k_{n}}$ and $M_{2}(x_{1},...,x_{n})=a_{2}x_{1}^{h_{1}}\cdot...\cdot x_{n}^{h_{2}}$ are two distinct monomials in $P(x_{1},...,x_{n})$ then there is an index $i\leq n$ such that $k_{i}\neq h_{i}$;
	\item whenever we denote a polynomial by $P(x_{1},...,x_{n})$, we mean that $\{x_{1},...,x_{n}\}$ are all and only the variables that appear in $P(x_{1},...,x_{n})$; i.e. for every index $i\leq n$, the partial degree of $x_{i}$ in $P(x_{1},...,x_{n})$ is at least 1, and for every variable $y$ in $X\setminus \{x_{1},...,x_{n}\}$, the partial degree of $y$ in $P(x_{1},...,x_{n})$ is 0:
	\item all the polynomials we consider have constant term 0.
\end{enumerate}
The first and second assumptions need no explanations; as for the third one, we remark that polynomials with constant term zero and polynomials with non-zero constant term have different behaviours with respect to the problem of partition regularity. E.g., Rado proved that

\begin{thm}[Rado] Suppose that $(\sum_{i=1}^{n} a_{i}x_{i})+c\in\Z[x_{1},...,x_{n}]$ is a polynomial with non-zero constant term $c$. Then $P$ is partition regular on $\N$ if and only if either 
\begin{enumerate}
	\item there exists a natural number $k$ such that $(\sum_{i=1}^{n}a_{i} k)+c=0$ or
	\item there exists an integer $z$ such that $(\sum_{i=1}^{n}a_{i} z)+c=0$ and there is a subset $J$ of $\{1,...,n\}$ such that $\sum_{j\in J}a_{j}=0$.
\end{enumerate}
\end{thm}

The above theorem is proved in $\cite{rif30}$. It shows the contrast between the case $c\neq 0$ and the case $c=0$: while the first one is related with the existence of a constant solution, the latter is not.\\
A study of the problem of the injectivity of solutions for partition regular polynomials with non zero constant term has been done by Neil Hindman and Imre Leader in $\cite{rif51}$ (actually, to be precise, their study involves the general case of partition linear homogeneous systems and their interested is in noncostant solutions, that is a more general notion than that of injective solutions).\\
We start our study of partition regularity of polynomials with the linear case, proving that the linear polynomials in $\Z[\mathbf{X}]$ with sum of coefficients zero are injectively partition regular. The interest in this result is that, given the polynomial $P$, we exibit an $\iota_{P}$-ultrafilter constructed as a linear combination of additively idempotent ultrafilters.\\
Then we deal with the nonlinear case which, in our opinion, is the most interesting: in fact, while Rado's Theorem settles the linear case, very little is known for the nonlinear one.

\subsection{An ultrafilter proof of Rado's Theorem for linear polynomials with sum of coefficients zero}

In this section we prove this result:

\begin{thm} Let $P(x_{1},...,x_{n},y_{1},...,y_{m}): (\sum_{i=1}^{n} c_{i}x_{i})-(\sum_{j=1}^{m}d_{j}y_{j})\in\Z[\mathbf{X}]$ be a polynomial such that $n+m\geq 3$, the coefficients $c_{1},...,c_{n},d_{1},...,d_{m}$ are positive, and $\sum_{i=1}^{n}c_{i}=\sum_{j=1}^{m}d_{j}$, and let $\U$ be an additively idempotent ultrafilter. If 

\begin{center} $\V_{1}=c_{1}\U\oplus (c_{1}+c_{2})\U\oplus c_{2}\U\oplus(c_{2}+c_{3})\U\oplus....\oplus (c_{n-1}+c_{n})\U$ \end{center}

and

\begin{center} $\V_{2}=d_{1}\U\oplus(d_{1}+d_{2})\U\oplus d_{2}\U\oplus(d_{2}+d_{3})\U\oplus...\oplus (d_{m-1}+d_{m})\U$ \end{center}

then the ultrafilter $\V=\V_{1}\oplus\V_{2}$ is a $\iota_{P}$-ultrafilter; in particular, $P(x_{1},...,x_{n},y_{1},...,y_{n})$ is injectively partition regular.\end{thm}

A useful tool to simplify notations in the proof of this theorem are the "tabular scriptures" for particular elements of $^{\bullet}\N$. Let $\alpha$ be an element in $^{\bullet}\N$. By writing

\begin{center}
\begin{tabular}{|l|c|c|c|c|}\hline
 & $S_{1}(\alpha)$ & $S_{2}(\alpha)$ & $...$ & $S_{k}(\alpha)$ \\ \hline
$\beta_{1}$ & $a_{1,1}$ & $a_{1,2}$ & $...$ & $a_{1,k}$  \\\hline
$\beta_{2}$ & $a_{2,1}$ & $a_{2,2}$ & $...$ & $a_{2,k}$\\\hline
$...$ & $...$ & $...$ & $...$ & $...$\\\hline
$\beta_{k}$ & $a_{k,1}$ & $a_{k,2}$ & $...$ & $a_{k,k}$\\\hline
\end{tabular}\end{center}

where the elements $a_{i,j}$ are natural numbers, we intend that, for every index $i\leq k$, $\beta_{i}=\sum_{j=1}^{k} a_{i,j}S_{j}(\alpha)$. E.g., the tabular

\begin{center}
\begin{tabular}{|l|c|c|}\hline
 & $^{*}\alpha$ & $^{**}\alpha$\\\hline
$\beta_{1}$ & $1$ & $2$  \\\hline
$\beta_{2}$ & $0$ & $3$ \\\hline
$\beta_{3}$ & $7$ & $0$ \\\hline
\end{tabular}\end{center}

is a notation to mean that $\beta_{1}=$$^{*}\alpha+2$$^{**}\alpha$, $\beta_{2}=3$$^{**}\alpha$ and $\beta_{3}=7$$^{*}\alpha$.\\
Before proving Theorem 3.5.4 we give an example that should be useful to understand the ideas involved in its proof. Consider the polynomial 

\begin{center} $P(x_{1},x_{2},x_{3},y_{1},y_{2}): 3x_{1}+2x_{2}+4x_{3}-y_{1}-8y_{2}$. \end{center}

If $\U$ is an additively idempotent ultrafilter, consider the ultrafilter 

\begin{center} $\V=3\U\oplus 5\U\oplus 2\U\oplus 6\U\oplus \U\oplus 9\U$, \end{center}

If $\alpha\in$$^{*}\N$ is a generator of $\U$, pose

\begin{center}
\begin{tabular}{|l|c|c|c|c|c|c|c|c|c|}\hline

 & $S_{1}(\alpha)$  & $S_{2}(\alpha)$ & $S_{3}(\alpha)$ & $S_{4}(\alpha)$ & $S_{5}(\alpha)$ & $S_{6}(\alpha)$ & $S_{7}(\alpha)$ & $S_{8}(\alpha)$ & $S_{9}(\alpha)$\\\hline
$\beta_{1}$ & 3 & 5 & 5 & 2 & 2 & 6 & 1 & 1 & 9\\\hline
$\beta_{2}$ & 3 & 0 & 5 & 2 & 6 & 6 & 1 & 1 & 9\\\hline
$\beta_{3}$ & 3 & 3 & 5 & 2 & 0 & 6 & 1 & 1 & 9\\\hline
$\gamma_{1}$ & 3 & 3 & 5 & 2 & 2 & 6 & 1 & 9 & 9\\\hline
$\gamma_{2}$ & 3 & 3 & 5 & 2 & 2 & 6 & 1 & 0 & 9\\\hline

\end{tabular}
\end{center}

The numbers above have been chosen in such a way that $\beta_{1},\beta_{2},\beta_{3},\gamma_{1},\gamma_{2}$ are mutually distinct generators of $\V$ and, moreover, $P(\beta_{1},\beta_{2},\beta_{3},\gamma_{1},\gamma_{2})=0$, as can be noted computing the coefficient $c_{k}$ of each element $S_{k}(\alpha)$, $1\leq k\leq 9$, in the expression $P(\beta_{1},\beta_{2},\beta_{3},\gamma_{1},\gamma_{2})$:
\begin{enumerate}
	\item $c_{1}=9+6+12-3-24=0$;
	\item $c_{2}=15+0+12-3-24=0$;
	\item $c_{3}=15+10+20-5-40=0$;
	\item $c_{4}=6+4+8-2-16=0$;
	\item $c_{5}=6+12+0-2-16=0$;
	\item $c_{6}=18+12+24-6-48=0$:
	\item $c_{7}=3+2+4-1-8=0$;
	\item $c_{8}=3+2+4-9-0=0$;
	\item $c_{9}=27+18+36-9-72=0$.
\end{enumerate}

The proof of Theorem 3.5.4 is just a generalization of the ideas presented in the above example. 

\begin{proof} Let $\U$ be an additively idempotent ultrafilter, and let $\V_{1}$, $\V_{2}$, $\V$ be the ultrafilters introduced in the statement.  We claim that $\V$ is a $\iota_{P}$-ultrafilter, which yields the thesis. Let $\alpha\in$$^{*}\N$ be a generator of $\U$, and consider the elements in Table 1 and in Table 2:

\begin{center}
\begin{table}[h]
{\bfseries Table 1}\\\\
\begin{tabular}{|l|c|c|c|c|c|c|c|c|l|}\hline
 & $S_{1}(\alpha)$ & $S_{2}(\alpha)$ & $S_{3}(\alpha)$ & $S_{4}(\alpha)$ & $S_{5}(\alpha)$ & ...& $S_{3n-5}(\alpha)$ & $S_{3n-4}(\alpha)$ & $S_{3n-3}(\alpha)$ \\\hline
$\beta_{1}$ & $c_{1}$ & $c_{1}+c_{2}$ & $ c_{1}+c_{2} $ & $c_{2}$ & $c_{2}$ & $...$ & $c_{n-1}$ & $c_{n-1}$ & $c_{n-1}+c_{n}$\\\hline
$\beta_{2}$ & $c_{1}$ & $ 0 $ & $ c_{1}+c_{2} $ & $c_{2}$ & $c_{2}+c_{3}$ & $...$ & $c_{n-1}$ &$c_{n-1}$ & $c_{n-1}+c_{n}$\\\hline
$\beta_{3}$ & $c_{1}$ & $c_{1}$ & $ c_{1}+c_{2} $ & $c_{2}$ & $0$ & $...$ & $c_{n-1}$ &$c_{n-1}$ & $c_{n-1}+c_{n}$\\\hline
$...$ &$...$ &$...$ &$...$ &$...$ &$...$ &$...$ &$...$ &$...$ &$...$\\ \hline
$\beta_{n-2}$ & $c_{1}$ & $c_{1}$ & $ c_{1}+c_{2} $ & $c_{2}$ & $c_{2}$ & $...$ & $c_{n-1}$ &$c_{n-1}$ & $c_{n-1}+c_{n}$ \\\hline
$\beta_{n-1}$ & $c_{1}$ & $c_{1}$ & $ c_{1}+c_{2} $ & $c_{2}$ & $c_{2}$ & $...$ & $c_{n-1}$ &$c_{n-1}+c_{n}$ & $c_{n-1}+c_{n}$ \\\hline
$\beta_{n}$ & $c_{1}$ & $c_{1}$ & $ c_{1}+c_{2} $ & $c_{2}$ & $c_{2}$ & $...$ & $c_{n-1}$ &$0$ & $c_{n-1}+c_{n}$ \\\hline

\end{tabular}
\end{table}
\end{center}

\begin{center}
\begin{table}[h]
{\bfseries Table 2}\\\\
\begin{tabular}{|l|c|c|c|c|c|c|c|c|l|}\hline
 & $S_{1}(\alpha)$ & $S_{2}(\alpha)$ & $S_{3}(\alpha)$ & $S_{4}(\alpha)$ & $S_{5}(\alpha)$ & ...& $S_{3m-5}(\alpha)$ & $S_{3m-4}(\alpha)$ & $S_{3m-3}(\alpha)$ \\\hline
$\gamma_{1}$ & $d_{1}$ & $d_{1}+d_{2}$ & $ d_{1}+d_{2} $ & $d_{2}$ & $d_{2}$ & $...$ & $ d_{m-1}$ & $d_{m-1}$ & $d_{m-1}+d_{m}$\\\hline
$\gamma_{2}$ & $d_{1}$ & $ 0 $ & $ d_{1}+d_{2} $ & $d_{2}$ & $d_{2}+d_{3}$ & $...$ & $ d_{m-1}$ &$d_{m-1}$ & $d_{m-1}+d_{m}$\\\hline
$\gamma_{3}$ & $d_{1}$ & $d_{1}$ & $ d_{1}+d_{2} $ & $d_{2}$ & $0$ & $...$ & $ d_{m-1}$ &$d_{m-1}$ & $d_{m-1}+d_{m}$\\\hline
$...$ &$...$ &$...$ &$...$ &$...$ &$...$ &$...$ &$...$ &$...$&$...$\\ \hline
$\gamma_{m-2}$ & $d_{1}$ & $d_{1}$ & $ d_{1}+d_{2} $ & $d_{2}$ & $d_{2}$ & $...$ &$ d_{m-1}$ & $d_{m-1}$ & $d_{m-1}+d_{m}$ \\\hline
$\gamma_{m-1}$ & $d_{1}$ & $d_{1}$ & $ d_{1}+d_{2} $ & $d_{2}$ & $d_{2}$ & $...$ &$ d_{m-1}$ & $d_{m-1}+d_{m}$ & $d_{m-1}+d_{m}$ \\\hline
$\gamma_{m}$ & $d_{1}$ & $d_{1}$ & $ d_{1}+d_{2} $ & $d_{2}$ & $d_{2}$ & $...$ & $ d_{m-1}$ &$0$ & $d_{m-1}+d_{m}$ \\\hline

\end{tabular}
\end{table}
\end{center}
For every $i\leq n$ and $j\leq 3\cdot (n-1)$, the coefficient $a_{i,j}$ in position $(i,j)$ in Table 1 is determined as follows: let $s,t$ be such that $j=3t+s$, $0\leq t\leq (n-2)$, $s\in\{1,2,3\}$.
\begin{itemize}
	\item if $s=1$, pose $a_{i,j}=c_{t+1}$;
	\item if $s=2$ then $\begin{cases} a_{t+1,j}=c_{t+1}+c_{t+2};\\ a_{t+2,j}=0; \\ a_{i,j}=c_{t+1}; \mbox{ otherwise}
\end{cases}$ 
	\item If $s=3$ pose $a_{i,j}= c_{t+1}+c_{t+2}$;
\end{itemize}
    
The coefficients in Table 2 are constructed exactly as for Table 1, by exchanging the roles of coefficients $c_{1},...,c_{n}$ and coefficients $d_{1},...,d_{m}$.\\
Observe that, with this construction, as $\U$ is additively idempotent then $\beta_{1},...,\beta_{n}$ are generators of $\V_{1}$ and $\gamma_{1},...,\gamma_{m}$ are generators of $\V_{2}$.\\
Pose 

\begin{center} $\beta=\sum_{i=1}^{3(n-1)}p_{i}S_{i}(\alpha)$,\end{center}

where, if $0\leq t\leq (n-2)$, $s\in \{1,2,3\}$ and $i=3t+s$, 
\begin{itemize}
	\item if $s=3$ $p_{i}=c_{t}+c_{t+1}$;
	\item otherwise $p_{i}=c_{t+1}$.
\end{itemize}
Similarly, pose

\begin{center} $\gamma=\sum_{i=1}^{3(m-1)}q_{i}S_{i}(\alpha)$,\end{center}

where, if $0\leq t\leq (m-2)$, $s\in \{1,2,3\}$ and $i=3t+s$, then
\begin{itemize}
		\item if $s=3$ then $q_{i}=d_{t}+d_{t+1}$;
	\item otherwise $q_{i}=d_{t+1}$.
\end{itemize}
 
For $1\leq i\leq n$, $1\leq j\leq m$ pose
\begin{itemize}
	\item $\xi_{i}=\beta_{i}\h\gamma$;
	\item $\eta_{j}=\beta\h\gamma_{j}$.
\end{itemize}

Observe that $\beta,\beta_{1},...,\beta_{n}$ are mutually distinct generators of $\V_{1}$ and $\gamma,\gamma_{1},...,\gamma_{m}$ are mutually distinct generators of $\V_{2}$, so the elements $\xi_{1},...,\xi_{n},\eta_{1},...,\eta_{m}$ are mutually distinct generators of $\V$.\\

{\bfseries Claim:} $P(\xi_{1},...,\xi_{n},\eta_{1},...,\eta_{m})=0$.\\

To prove the claim, since $P(\xi_{1},...,\xi_{n},\eta_{1},...,\eta_{m})$ is by construction an expression in $S_{1}(\alpha)$, $S_{2}(\alpha)$, ..., $S_{3(n+m-2)}(\alpha)$, it is sufficient to show that for each index $k\leq 3(n+m-2)$ the coefficient of $S_{k}(\alpha)$ in this expression is 0.\\
Given an index $k$, let $t,s$ be such that $k=3t+s$, $t\leq n+m-2$ and $s\in \{1,2,3\}$ There are two cases to consider.\\
Case 1: $t\leq n-1$.
\begin{itemize}
	
\item If $s=1$, the coefficient is

\begin{center}$\sum_{i=1}^{n} c_{i}c_{t+2}-\sum_{j=1}^{m}d_{j}c_{t+2}=c_{t+2}(\sum_{i=1}^{n}c_{i}-\sum_{j=1}^{m} d_{j})=0$.\end{center}
\item If $s=2$, the coefficient is
\begin{center}$((\sum_{i=1, i\neq t+1,t+2}^{n} c_{i}c_{t+1})+c_{t+1}(c_{t+1}+c_{t+2})+c_{t+2}(0))-\sum_{j=1}^{m} (d_{j}(c_{t+1}))=$\\\vspace{0.3cm}$=c_{t+1}(\sum_{i=1}^{n}c_{i}-\sum_{j=1}^{m}d_{j})=0$.\end{center}
\item If $s=3$, the coefficient of $S_{k}(\alpha)$ is 

\begin{center} $\sum_{i=1}^{n} c_{i}(c_{t+1}+c_{t+2})-\sum_{j=1}^{m}d_{j}(c_{t+1}+c_{t+2})=$\\\vspace{0.3cm}$=(c_{t+1}+c_{t+2})(\sum_{i=1}^{n}c_{i}-\sum_{j=1}^{m} d_{j})=0$\end{center}

since $\sum_{i=1}^{n}c_{i}=\sum_{j=1}^{m} d_{j}$ by hypothesis.

\end{itemize}

Case 2: $n-1<t\leq n+m-2$: similar to case 1, exchanging the roles of coefficients $c_{1},...,c_{n}$ and coefficients $d_{1},...,d_{m}$.\\
This shows that the coefficient of $S_{i}(\alpha)$, for $1\leq i\leq 3(n+m-2)$, in the expression $P(\xi_{1},...,\xi_{n},\eta_{1},...,\eta_{m})$ is 0, so $(\xi_{1},...,\xi_{n},\eta_{1},...,\eta_{m})$ is a solution to $P(x_{1},...,x_{n},y_{1},...,y_{m})$ made of mutually distinct elements, and hence $\V$ is a $\iota_{P}$-ultrafilter.\\ \end{proof}

\subsection{Nonlinear polynomials}

While the linear homogeneous case is settled by Rado's Theorem, very little is known for the nonlinear case, that we consider in this section.\\
The first non-linear injective partition regular polynomial that we present is $P(x,y,z,w): x+y-zw$. Its partition regularity has been first proved by P. Csikv\'ari, K. Gyarmati and A. S\'arközy in $\cite{rif10}$; here, we present a different approach, based on the following result:

\begin{prop} There exists a nonprincipal multiplicatively idempotent Schur ultrafilter. \end{prop}

\begin{proof} Schur's property is a multiplicatively invariant property, and the set of nonprincipal Schur ultrafilters is nonempty; the thesis follows by applying Theorem 3.4.5.\\\end{proof}

\begin{cor} The polynomial 

\begin{center} $P(x,y,z,w): x+y-zw$ \end{center}

is injectively partition regular. \end{cor}

\begin{proof} Let $\U$ be a nonprincipal multiplicatively idempotent Schur ultrafilter.\\

{\bfseries Claim:} $\U$ is a $\iota_{P}$-ultrafilter.\\

In fact, let $\eta,\mu,\xi\in$$^{*}\N$ be mutually distinct generators of $\U$ with $\eta+\mu=\xi$; since $\U$ is multiplicatively idempotent, the elements $\eta\cdot$$^{*}\eta$, $\mu\cdot$$^{*}\eta$, $\xi\cdot$$^{*}\eta$ are in $G_{\U}$.\\
Consider the elements $\alpha=\eta\cdot$$^{*}\eta$, $\beta=\mu\cdot$$^{*}\eta$, $\gamma=\xi$, $\delta=$$^{*}\eta$, and observe that
\begin{center} $\alpha+\beta=\eta\cdot$$^{*}\eta+\mu\cdot$$^{*}\eta=\xi\cdot$$^{*}\eta=\gamma\cdot\delta$.\end{center}

In particular, as $\alpha,\beta,\gamma,\delta$ are mutually distinct generators of $\U$, then $\U$ is a $\iota_{P}$-ultrafilter, so $P(x,y,z,w)$ is injectively partition regular.\\\end{proof}

The importance of a multiplicative idempotent Schur ultrafilter is that it mixes an additive property (Schur's) with a multiplicative property (idempotence).\\
To generalize the result of Corollary 3.5.6, a natural idea is to search for other ultrafilters with similar features: good candidates are Folkman ultrafilters.

\begin{prop} There is a nonprincipal multiplicatively idempotent Folkman ultrafilter. \end{prop}

\begin{proof} We observe that we can not directly apply Theorem 3.4.5, as Folkman's property is not expressed by an existential sentence. Nevertheless, we proved in Corollary 3.3.8 that the set $\mathcal{F}$ of Folkman ultrafilters on $\N$ is a bilateral ideal in $(\bN\setminus\{\mathfrak{U}_{0}\},\odot)$; in particular, as every Folkman ultrafilter is nonprincipal, $\mathcal{F}$ is a bilateral ideal in $(\bN\setminus\N,\odot)$, and we showed in Section 3.4 that it is a closed subset of $\bN$. Then, by Ellis's Theorem, there is a nonprincipal multiplicatively idempotent ultrafilter in $\mathcal{F}$ .\\\end{proof}

The existence of such an ultrafilter has some interesting consequences:

\begin{cor} For every natural number $n\geq 1$, the polynomial 

\begin{center} $P(x_{1},...,x_{n},y,z): (\sum_{i=1}^{n} x_{i})-yz$ \end{center}

is injectively partition regular. \end{cor}

\begin{proof} Let $\U$ be a nonprincipal multiplicatively idempotent Folkman ultrafilter, and let $\eta_{1},...,\eta_{n},\xi\in$$^{*}\N$ be mutually distinct generators of $\U$ such that $\xi=\eta_{1}+...+\eta_{n}$. Observe that

\begin{center}  $\eta_{1}\cdot$$^{*}\xi+...+\eta_{n}\cdot$$^{*}\xi=\xi\cdot$$^{*}\xi$. \end{center}

If we consider the elements $\alpha_{1}=\eta_{1}\cdot$$^{*}\xi$,..., $\alpha_{n}=\eta_{n}\cdot$$^{*}\xi$, $\beta=\xi$, $\gamma=$$^{*}\xi$, we have that $\alpha_{1},...,\alpha_{n},\beta,\gamma$ are generators of $\U$ (as $\U$ is multiplicatively idempotent) and 

\begin{center} $(\sum_{i=1}^{n} \alpha_{i})-\beta\cdot\gamma=0$. \end{center}

So $\U$ is a $\iota_{P}$-ultrafilter, and $P(x_{1},...,x_{n},y,z)$ is an injectively partition regular polynomial.\\ \end{proof}

Actually, in the previous corollary there is no need to consider two variables $y,z$:

\begin{cor} For every natural numbers $n,m$ with $n+m\geq 3$, the polynomial 

\begin{center} $P(x_{1},...,x_{n},y_{1},...,y_{m}): (\sum_{i=1}^{n}x_{i})-\prod_{j=1}^{m}y_{j}$ \end{center}

is injectively partition regular. \end{cor}

\begin{proof} Let $\U$ be a multiplicatively idempotent Folkman ultrafilter, and let $\eta_{1},...\eta_{n},\xi\in$$^{*}\N$ be mutually distinct generators of $\U$ such that $\xi=\eta_{1}+...+\eta_{n}$, and consider $\mu=$$^{*}\xi\cdot$$^{**}\xi\cdot...\cdot$$^{*_{m-1}}\xi$.\\ 
Pose \begin{center} $\alpha_{1}=\eta_{1}\cdot\mu$,..., $\alpha_{n}=\eta_{n}\cdot\mu$,\end{center}

and \begin{center} $\beta_{1}=\xi$, $\beta_{2}=$$^{*}\xi$,..., $\beta_{m}=$$^{*_{m-1}}\xi$.\end{center}

Since $\U$ is multiplicatively idempotent, the mutually distinct elements $\alpha_{1},...,\alpha_{n},\beta_{1},...,\beta_{m}$ are in $G_{\U}$, and 

\begin{center} $\sum_{i=1}^{n} \alpha_{i}=\prod_{j=1}^{m} \beta_{j}$. \end{center}

So $\U$ is a $\iota_{P}$-ultrafilter, and $P(x_{1},...,x_{n},y_{1},...,y_{m})$ is an injectively partition regular polynomial.\\ \end{proof}

The above result is also proved, in a different way, in $\cite{rif23}$.\\
The next generalization seems less straightforward: in the polynomial $(\sum_{i=1}^{n}x_{i})-\prod_{j=1}^{m}y_{j}$ we can substitute the sum of variables with a sum of monomials and retain the injective partition regularity:

\begin{cor} For every natural number $n\geq 2\in \N$, for every positive natural numbers $l_{1},...,l_{n},m\in \N$, the polynomial

\begin{center} $P(x_{1,1},...,x_{n,l_{n}},y_{1},...,y_{m}): \prod_{i_{1}=1}^{l_{1}}x_{1,i_{1}}+...+\prod_{i_{n}=1}^{l_{n}}x_{i_{n},n}-\prod_{j=1}^{m}y_{j}$ \end{center}

is injectively partition regular. \end{cor}

We do not prove this result, as it is a straightforward corollary of Theorem 3.5.13, which is a result by far more general. The result that we need in order to prove Theorem 3.5.13 is the following:

\begin{prop} If $P(x_{1},...,x_{n})$ is an homogeneous injectively partition regular polynomial, then the set 

\begin{center} $\mathcal{I}_{P}=\{\U\in\bN\mid$ $\U$ is a $\iota_{P}$-ultrafilter$\}$ \end{center}

is a compact bilateral ideal in $(\bN\setminus\N,\odot)$; in particular, it contains a nonprincipal multiplicatively idempotent ultrafilter.\\
Similarly, if $P(x_{1},...,x_{n})$ is an homogeneous partition regular polynomial, the set

\begin{center} $\mathcal{S}_{P}=\{\U\in\bN\mid \U$ is a $\sigma_{P}$-ultrafilter$\}$\end{center}

is a compact bilateral ideal in $(\bN,\odot)$, and it contains a nonprincipal multiplicatively idempotent ultrafilter.\end{prop}

\begin{proof} The result for $\mathcal{I}_{P}$ follows as the formula $\iota_{P}$ is multiplicatively invariant, since $P$ is homogeneous. Then observe that the elements of $\mathcal{I}_{P}$ are necessarily nonprincipal, and apply Theorem 3.4.5.\\
As for $\mathcal{S}_{P}$, the property of compact bilateral ideal follows by Corollary 3.3.7 and Corollary 3.4.4, since the formula $\sigma_{P}$ is multiplicatively invariant. $\mathcal{S}_{P}$ contains a nonprincipal multiplicatively idempotent ultrafilter because $\mathcal{S}_{P}\cap (\bN\setminus\N)\neq\emptyset$: in fact, by contradition suppose that every $\sigma_{P}$-ultrafilter is principal; if $k$ is a natural number such that $\mathfrak{U}_{k}\in \mathcal{S}_{P}$, by definition $P(k,....,k)=0$ and, as $P(x_{1},...,x_{n})$ is homogeneous, if $\alpha$ is any infinite hypernatural number then $P(k\alpha,...,k\alpha)=0$, so $\mathfrak{U}_{k\alpha}\in\mathcal{S}_{P}\cap (\bN\setminus\N)$, absurd.\\
In particular, $\mathcal{S}_{P}\cap (\bN\setminus\N)$ is a nonempty compact bilateral ideal in $(\bN\setminus\N,\odot)$, so it contains a nonprincipal multiplicatively idempotent ultrafilter.\\\end{proof}

\begin{defn} Let $n$ be a positive natural number, and $\{y_{1},...,y_{n}\}$ a set of mutually distinct variables. For every finite set $F\subseteq\{1,..,n\}$, we denote by $Q_{F}(y_{1},...,y_{n})$ the monomial

\begin{center} $Q_{F}(y_{1},...,y_{n})=\begin{cases} \prod_{j\in F} y_{j}, & \mbox{if  } F\neq \emptyset;\\ 1, & \mbox{if  } F=\emptyset.\end{cases}$ \end{center}

\end{defn}

\begin{thm} Let $k\geq 3$ be a natural number, $P(x_{1},...,x_{k})=\sum_{i=1}^{k} a_{i}x_{i}$ an injectively partition regular polynomial, and $n$ a positive natural number. Then, for every $F_{1},...,F_{k}\subseteq\{1,..,n\}$, the polynomial

\begin{center} $R(x_{1},...,x_{k},y_{1},...,y_{n})=\sum_{i=1}^{k} a_{i}x_{i}Q_{F_{i}}(y_{1},...,y_{n})$ \end{center}

is injectively partition regular. \end{thm}

\begin{proof} Since $P(x_{1},...,x_{k})$ is homogeneous and partition regular, by Proposition 3.5.11 it follows that there is a nonprincipal multiplicatively idempotent $\iota_{P}$-ultrafilter $\U$. Let $\alpha_{1},...,\alpha_{k}\in$$^{*}\N$ be mutually distinct generators of $\U$ such that $P(\alpha_{1},...,\alpha_{k})=0$, and let $\beta\in$$^{*}\N$ be any generator of $\U$. For every index $j\leq n$, pose

\begin{center} $\beta_{j}=S_{j}(\beta)$. \end{center}

Observe that, for every index $j\leq n$, $\beta_{j}\in G_{\U}$. Pose, for every index $i\leq k$,

\begin{center} $\eta_{i}=\alpha_{i}\cdot (\prod_{j\notin F_{i}} \beta_{j})$. \end{center}

Since $\U$ is multiplicatively idempotent, $\eta_{i}\in G_{\U}$ for every index $i\leq k$.\\

{\bfseries Claim:} $\sum_{i=1}^{k} a_{i}\eta_{i}Q_{F_{i}}(\beta_{1},...,\beta_{n})=0$.\\

In fact, 

\begin{center} $\sum_{i=1}^{k}a_{i}\eta_{i}Q_{F_{i}}(\beta_{1},...,\beta_{n})=\sum_{i=1}^{k}a_{i}\alpha_{i}(\prod_{j\notin F_{i}}\beta_{j})(\prod_{j\in F_{i}}\beta_{j})=$\\\vspace{0.3cm}$=\sum_{i=1}^{k}a_{i}\alpha_{i}(\prod_{j=1}^{n}\beta_{j})=(\prod_{j=1}^{n}\beta_{j})\sum_{i=1}^{k}a_{i}\alpha_{i}=0$.\end{center}\end{proof}

Three observations: 
\begin{enumerate}
	\item as a consequence of the proof, $\U$ is both a $\iota_{P}$ and a $\iota_{R}$-ultrafilter;
	\item if the hypothesis on the injective partition regularity of $P(x_{1},...,x_{k})$ is replaced by the hypothesis that $P(x_{1},...,x_{k})$ is partition regular, this same proof shows that in this case $P(x_{1},...,x_{n})$ is partition regular;
	\item observe that some of the variables $y_{1},...,y_{n}$ may appear in more than a monomial: e.g., the polynomial \begin{center}$P(x_{1},x_{2},x_{3},x_{4},y_{1},y_{2},y_{3}): x_{1}y_{1}y_{2}+4x_{2}y_{1}y_{2}y_{3}-3x_{3}y_{3}-2x_{4}y_{1}$\end{center} satisfies the hypothesis of the above theorem, so it is injectively partition regular.
\end{enumerate}

In view of this result, there is a particular class of polynomials that are partition regular:

\begin{defn} Let $P(x_{1},...,x_{n}):\sum_{i=1}^{k} a_{i}M_{i}(x_{1},...,x_{n})$ be a polynomial in $\Z[\mathbf{X}]$, and let $M_{1}(x_{1},...,x_{n}),...,M_{k}(x_{1},...,x_{n})$ be its distinct monomials. $P(x_{1},...,x_{n})$ admits a {\bfseries set of exclusive variables} $\{v_{1},...,v_{k}\}$ if $v_{i}$ appears only in the monomial $M_{i}(x_{1},...,x_{n})$.
In this case we say that the variable $v_{i}$ is {\bfseries exclusive} for $P(x_{1},...,x_{n})$, and the set $\{v_{1},...,v_{k}\}$ of exclusive variables is denoted by $V_{escl}(P)$.\end{defn}

E.g., the polynomial $P(x,y,z,t,w): xyz+yt-w$ admits $\{x,t,w\}$ or $\{z,t,w\}$ as sets of exlusive variables, while the polynomial $P(x,y,z): xy+yz-xz$ does not admit a set of exclusive variables.

\begin{defn} Given a polynomial $P(x_{1},...,x_{n}): \sum_{i=1}^{k} a_{i}M_{i}(x_{1},...,x_{n})$, where $M_{1}(x_{1},...,x_{n}),...,M_{k}(x_{1},...,x_{n})$ are its distinct monomials, the {\bfseries reduct of $P$} $($notation Red$(P))$ is the polynomial:
	
	\begin{center} Red$(P)(y_{1},...,y_{k})$: $\sum_{i=1}^{k} a_{i}y_{i}$. \end{center}
	
\end{defn}

E.g., if $P(x_{1},...,x_{n})$ is an homogeneous linear polynomial then 

\begin{center}Red($P$)$(y_{1},...,y_{n})=P(y_{1},...,y_{n})$;\end{center}

if $P(x,y,z,t,w)$ is the polynomial $xy+4yz-2t+yw$, then 

\begin{center} Red($P$)$(y_{1},y_{2},y_{3},y_{4})=y_{1}+4y_{2}-2y_{3}+y_{4}$.\end{center}

As a consequence of Theorem 3.5.13 we obtain the following result:

\begin{cor} Let $P(x_{1},...,x_{n})=\sum_{i=1}^{k}a_{i}M_{i}(x_{1},...,x_{n})$ be a polynomial with partial degree one, and suppose that

\begin{enumerate}
	\item $P$ admits a set of exclusive variables;
	\item Red$(P)$ is injectively partition regular.
\end{enumerate}

Then $P(x_{1},...,x_{n})$ is an injectively partition regular polynomial.
\end{cor}

\begin{proof} By reordering, if necessary, we can suppose that the exclusive variables are $x_{1},...,x_{k}$, and that the variable $x_{j}$ is exclusive for the monomial $M_{j}(x_{1},...,x_{n})$. Then, by hypothesis,

\begin{center} $\sum_{i=1}^{k}a_{i}x_{i}$ \end{center}

is partition regular as it is, by renaming the variables, equal to Red($P$). If $F=\{1,...,n-k\}$, for $i\leq k$ we put

\begin{center} $F_{i}=\{j\in F\mid x_{j+k}$ divides $M_{i}(x_{1},...,x_{n})\}.$\end{center}
Then if we put, for $j\leq n-k$, $y_{j}=x_{i+k}$, $P(x_{1},...,x_{n})$ is, by renaming the variables, equal to

\begin{center} $\sum_{i=1}^{k}a_{i}x_{i}Q_{F_{i}}(y_{1},...,y_{n-k})$. \end{center}

The above polynomial, as a consequence of Theorem 3.5.13, is injectively partition regular, and this entails the thesis.\\ \end{proof}

A natural question is if the implication in the above theorem can be reversed. Following the ideas in the proof of Rado's Theorem for linear equations, we prove that, if $P(x_{1},...,x_{n})$ is homogeneous and partition regular, then Red($P$) is partition regular.

\begin{lem} If $p$ is a prime number in $\N$ and $\alpha,\beta$ are two $\sim_{u}$-equivalent hypernatural numbers, then smod$(p)(\alpha)$=smod$(p)(\beta)$, where

\begin{center} smod$(p)(\alpha)=i$ if and only if, if $\gamma$ is the greatest exponent such that $p^{\gamma}\mid\alpha$, and $\alpha=p^{\gamma}\delta$, then $\delta\equiv i\mod p$. \end{center}

\end{lem}

\begin{proof} We observe that, given the prime number $p$ in $\N$, then $\N=\bigcup_{i=1}^{p-1} C_{i}$, where $C_{i}=\{n\in\N\mid smod(p)(n)=i\}$. Clearly, if $i\neq j$ then $C_{i}\cap C_{j}=\emptyset$, so there exists exactly one index $i$ such that $C_{i}\in \mathfrak{U}_{\alpha}$; in particular, $\alpha, \beta\in$$^{\bullet}C_{i}$, so $smod(p)(\alpha)=smod(p)(\beta)=i$. \\ \end{proof}

\begin{thm} If $P(x_{1},....,x_{n})$ is an homogeneous partition regular polynomial then Red$(P)$ is partition regular. \end{thm}

\begin{proof} By contradiction, suppose that Red($P$) is not partition regular. Then, as Red($P$) is linear, by Rado's Theorem no subset of the set of coefficients of Red($P$) sums to 0 and, since $P(x_{1},...,x_{n})$ and Red($P$) have by construction the same coefficients, this entails that no subset of the set of coefficients of $P(x_{1},...,x_{n})$ sums to 0.\\
Let $p$ be a prime number that does not divide the sum of any subset of the set of coefficients, $\U$ a $\sigma_{P}$-ultrafilter, and $\alpha_{1},...,\alpha_{n}$ generators of $\U$ with $P(\alpha_{1},...,\alpha_{n})=0$.\\
For every index $i\leq n$ let $\beta_{i},\gamma_{i}$ be hypernatural numbers such that 

\begin{center} $\alpha_{i}=p^{\gamma_{i}}\beta_{i}$, \end{center}

where $\beta_{i}$ is not divisible by $p$. By definition of $smod(p)$, $smod(p)(\alpha_{i})=smod(p)(\beta_{i})$; as the elements $\alpha_{i}$ are all $\sim_{u}$-equivalent, it follows that, for every pair of indexes $i,j\leq n$, $smod(p)(\beta_{i})=smod(p)(\beta_{j})$ and, since the elements $\beta_{i}$ are not divisible by $p$,

\begin{center} $smod(p)(\beta_{i})=(\beta_{i})\mod(p)=b\neq 0$ for every index $i\leq n$.\end{center}

By hypothesis, the polynomial $P(x_{1}...,x_{n})$ has the form

\begin{center} $P(x_{1},...,x_{n}):\sum_{j=1}^{k} a_{j}x_{1}^{d_{j_{1}}}\cdot...\cdot x_{n}^{d_{j_{n}}}$,\end{center}

where, if $d$ is the degree of $P(x_{1},...,x_{n})$, for evey index $j\leq k$ $\sum_{i=1}^{n} d_{j_{i}}=d$.\\
Since $\alpha_{1},...,\alpha_{n}$ is a solution of $P(x_{1},...,x_{n})=0,$

\begin{center} $\sum_{j=1}^{k} a_{j}\alpha_{1}^{d_{j_{1}}}\cdot...\cdot \alpha_{n}^{d_{j_{n}}}=0$. \end{center}

As $\alpha_{i}=p^{\gamma_{i}}\beta_{i}$, the above expression can be written like this:

\begin{center} (1)$\sum_{j=1}^{k} a_{j} p^{\sum_{i=1}^{n}d_{j_{i}} \gamma_{j}}\beta_{1}^{d_{j_{1}}}\cdot...\cdot\beta_{n}^{d_{j_{n}}}=0$. \end{center}

Pose, for every index $j\leq k$, 

\begin{center} $c_{j}=\sum_{i=1}^{n} d_{j_{i}}\gamma_{i}$,\end{center}

and 
 
\begin{center} $c=\min\{c_{j}\mid 1\leq j\leq k\}$. \end{center}
 
Then (1) can be rewritten as

\begin{center} $p^{c}\sum_{j=1}^{k} a_{j} p^{c_{j}-c}\beta_{1}^{d_{j_{1}}}\cdot...\cdot\beta_{n}^{d_{j_{n}}}=0$, \end{center}

and, as $p^{c}\neq 0$, this entails that

\begin{center} (2) $\sum_{j=1}^{k} a_{j} p^{c_{j}-c}\beta_{1}^{d_{j_{1}}}\cdot...\cdot\beta_{n}^{d_{j_{n}}}=0$; \end{center}

observe that for each index $j$ such that $c_{j}=c$ (and there is at least one index with this property), $p^{c_{j}-c}=1$.\\
From (2) it follows that:

\begin{center} $\sum_{j=1}^{k} a_{j} p^{c_{j}-c}\beta_{1}^{d_{j_{1}}}\cdot...\cdot\beta_{n}^{d_{j_{n}}}\equiv 0\mod p$. \end{center}

Observe that, for all the indexes $j\leq k$ with $c_{j}\neq c$, the term $a_{j} p^{c_{j}-c}\beta_{1}^{d_{j_{1}}}\cdot...\cdot\beta_{n}^{d_{j_{n}}}\equiv 0\mod p$. So, if $\tilde{J}$ is the nonempty set of indexes $j$ with $\Gamma_{j}=\Gamma$,

\begin{center} (3)$\sum_{j\in\tilde{J}} a_{j}\beta_{1}^{d_{j_{1}}}\cdot...\cdot\beta_{n}^{d_{j_{n}}}\equiv 0\mod p$. \end{center}

We observed that, modulo $p$, the elements $\beta_{i}$, $i\leq n$, are equal to some natural number $b$ which is not 0, so the expression (3) can be rewritten as:

\begin{center} (4)$\sum_{j\in\tilde{J}}a_{j}b^{d}\equiv 0\mod p$.\end{center}

Since $b\neq 0\mod p$ then $b^{d}\neq 0\mod p$, so (4) holds if and only if $\sum_{j\in\tilde{J}}a_{j}=0\mod p$, and that is a contradiction as we have chosen $p$ in such a way that the sum of elements of every subset of coefficients is not divisible by $p$. \end{proof}

\section{Properties of the Set of Injectively Partition Regular Polynomials}

In this section we study the properties of the set of injectively partition regular polynomials, with a particular interest for the homogeneous polynomials. 

\begin{defn} The {\bfseries set of injectively partition regular polynomials on $\N$} is 

\begin{center} $\mathcal{P}=\{P\in\Z[\mathbf{X}]\mid P$ is injectively partition regular on $\N$$\}$,\end{center}

and the {\bfseries set of homogeneous injectively partition regular polynomials on $\N$} is

\begin{center} $\mathcal{H}=\{P\in \mathcal{P}\mid P$ is homogeneous$\}$. \end{center}
\end{defn}

In this section we always assume that the polynomials are given in normal reduced form, that the variables of $P(x_{1},...,x_{n})$ are exactly $x_{1},...,x_{n}$, that every considered polynomial has costant term 0 and, given a polynomial $P(x_{1},...,x_{n})$, the symbols $\sigma_{P},\iota_{P}$ denote the sentences introduced in definition 3.5.2.\\
The first interesting property that we prove is that the multiples of a polynomial $P(x_{1},...,x_{n})$ in $\mathcal{P}$ are in $\mathcal{P}$:

\begin{thm} Given a polynomial $P(x_{1},...,x_{n})$ in $\mathcal{P}$ and a generical polynomial $Q(y_{1},...,y_{m})$ in $\Z[\mathbf{X}]$, $P(x_{1},...,x_{n})\cdot Q(y_{1},...,y_{m})$ is in $\mathcal{P}$; in particular, if $P(x_{1},...,x_{n})$ is in $\mathcal{H}$ and $Q(y_{1},...,y_{m})\in \Z[\mathbf{X}]$ is homogeneous, then $P(x_{1},...,x_{n})\cdot Q(y_{1},...,y_{m})\in\mathcal{H}$. \end{thm}

\begin{proof} Pick $\U$ $\iota_{P}$-ultrafilter, and let $\alpha_{1},...,\alpha_{n}\in$$^{*}\N$ be mutually distinct generators of $\U$. Let $\beta_{1},...,\beta_{m}$ be mutually distinct elements in $G_{\U}\setminus\{\alpha_{1},...,\alpha_{n}\}$. Then $\alpha_{1},...,\alpha_{n},\beta_{1},...,\beta_{m}$ are mutually distinct elements such that $P(\alpha_{1},...,\alpha_{n})\cdot Q(\beta_{1},...,\beta_{m})=0$. So $\U$ is also a $\iota_{P\cdot Q}$-ultrafilter, and $P(x_{1},....,x_{n})\cdot Q(y_{1},...,y_{m})\in\mathcal{P}$.\\
In particular, if $P(x_{1},...,x_{n})\in\mathcal{H}$ and $Q(y_{1},...,y_{m})$ is homogeneous, the product $P(x_{1},...,x_{n})\cdot Q(y_{1},...,y_{m})$ is an homogeneous element of $\mathcal{P}$ so it is in $\mathcal{H}$.\\\end{proof}

In particular, $\mathcal{P}$ is a sub-semigroup of $(\Z[\mathbf{X}], \cdot)$ and $\mathcal{H}$ is a sub-semigroup of $(H[\mathbf{X}],\cdot)$, where $H[\mathbf{X}]$ is the sub-semigroup of $(\Z[\mathbf{X}], \cdot)$ of homogeneous polynomials.\\
Conversely, we now prove that if $P(x_{1},...,x_{n})$ is in $\mathcal{P}$, and it is not irreducible, then at least one of its factors belongs to $\mathcal{P}$. Only in this theorem, when writing $Q_{i}(x_{1},...,x_{n})$ we mean that the set of variables of $Q_{i}$ is a subset of $\{x_{1},...,x_{n}\}$; if $\alpha_{1},...,\alpha_{n}$ are hypernatural numbers in $^{\bullet}\N$, by $Q(\alpha_{1},...,\alpha_{n})$ we denote the number obtained by replacing each variable $x_{i}$ in $Q$ by $\alpha_{i}$.

\begin{thm} Suppose that the polynomial $P(x_{1},...,x_{n})$ is factorized as 

\begin{center} $P(x_{1},....,x_{n})=\prod_{i=1}^{k}Q_{i}(x_{1},...,x_{n})$.\end{center} 

Then $P(x_{1},...,x_{n})\in \mathcal{P}$ if and only if there exists an index $i$ such that $Q_{i}(x_{1},...,x_{n})\in \mathcal{P}$. \end{thm}

\begin{proof} The implication $\Leftarrow$ is an immediate consequence of Theorem 3.6.2.\\
$\Rightarrow$: Let $\U$ be a $\iota_{P}$-ultrafilter on $A$, and let $\alpha_{1},...,\alpha_{n}$ be mutually distinct generators of $\U$ such that $P(\alpha_{1},...,\alpha_{n})=0$. Then, for at least one index $i$, $Q_{i}(\alpha_{1},...,\alpha_{n})=0$, so $\U$ is a $\iota_{Q_{i}}$-ultrafilter, and hence $Q_{i}(x_{1},...,x_{n})$ is in $\mathcal{P}$.\\\end{proof}

From the previous theorem it follows this important fact:\\

{\bfseries Fact:} To study the partition regularity of polynomials it is sufficient to work with irreducible polynomials.\\

The situation for the sum of elements in $\mathcal{P}$ is less simple:

\begin{prop} If $P(x_{1},...,x_{n})$, $Q(y_{1},...,y_{m})$ are in $\mathcal{H}$, and $\{x_{1},...,x_{n}\}\cap\{y_{1},..,y_{m}\}=\emptyset$, then $P(x_{1},...,x_{n})+Q(y_{1},...,y_{m})\in\mathcal{P}$.\end{prop}

\begin{proof} Suppose that $\U$ is a $\iota_{P}$-ultrafilter and $\V$ is a $\iota_{Q}$-ultrafilter. We claim that $\U\odot\V$ is a $\iota_{P+Q}$-ultrafilter. In fact, let $\alpha_{1},...,\alpha_{n}\in$$^{*}\N$ be mutually distinct generators of $\U$, and let $\beta_{1},...,\beta_{m}\in$$^{*}\N$ be mutually distinct generators of $\V$, with $P(\alpha_{1},...,\alpha_{n})=Q(\beta_{1},...,\beta_{m})=0$. Then, by transfer, $Q($$^{*}\beta_{1},...,$$^{*}\beta_{m})=0$. Consider the elements $\eta_{1},...,\eta_{n},\xi_{1},...,\xi_{m}$ in $G_{\U\odot\V}$, where

\begin{center} $\eta_{i}=\alpha_{i}\cdot$$^{*}\beta_{1}$ for $i\leq n$ and $\xi_{j}=\alpha_{1}\cdot$$^{*}\beta_{j}$ for $j\leq m$. \end{center}

Then, as $P(x_{1},...,x_{n})$ and $Q(y_{1},...,y_{m})$ are homogeneous, if $d_{p}$ and $d_{q}$ are their respective degrees, then we have

\begin{center} $P(\eta_{1},...,\eta_{n})+Q(\xi_{1},...,\xi_{m})=$$^{*}\beta_{1}^{d_{p}}P(\alpha_{1},...,\alpha_{n})+\alpha_{1}^{d_{1}}Q($$^{*}\beta_{1},...,$$^{*}\beta_{m})=0+0=0$.\end{center}

So, as $\eta_{1},...,\eta_{n},\xi_{1},...,\xi_{m}$ are mutually distinct generators, $\U\odot\V$ is a $\iota_{P+Q}$-ultrafilter, and hence $P(x_{1},..,x_{n})+Q(y_{1},...,y_{m})\in\mathcal{P}$. \\ \end{proof}

\begin{thm} Let $P(x_{1},...,x_{n}), Q(y_{1},...,y_{m})$ be polynomials in $\mathcal{P}$ such that
\begin{enumerate}
	\item the partial degree of both $P(x_{1},...,x_{n})$ and $Q(y_{1},...,y_{m})$ is 1;
	\item $P(x_{1},...,x_{n})+Q(y_{1},...,y_{m})$ admits a sets of exclusive variables;
	\item Red$(P)$ and Red$(Q)$ are polynomials in $\mathcal{P}$.
\end{enumerate}
Then $P(x_{1},...,x_{n})+Q(y_{1},...,y_{m})\in \mathcal{P}$. \end{thm}

\begin{proof} This result follows by Corollary 3.5.16: in fact, $P(x_{1},...,x_{n})+Q(y_{1},...,y_{m})$ has partial degree 1, admits a set of exclusive variables and its reduct is weakly partition regular, as Red($P+Q$) is, renaming the variables if necessary, equal to Red($P$)+Red($Q$) and by Proposition 3.6.4 it follows that Red($P$)+Red($Q$) is partition regular, since both Red($P$) and Red($Q$) are in $\mathcal{H}$ by hypothesis.\\\end{proof}

In the theorem below, just observe that, given an homogeneous polynomial $P(x_{1},...,x_{n})\in\Z[\mathbf{X}]$ with degree $d$, $x_{1}^{d}\cdot....\cdot x_{n}^{d}\cdot P(\frac{1}{x_{1}},...,\frac{1}{x_{n}})\in \Z[\mathbf{X}]$.

\begin{thm} If $P(x_{1},...x_{n})$ is a polynomial in $\mathcal{H}$ with degree $d$ then the polynomial $Q(x_{1},...,x_{n})=x_{1}^{d}\cdot....\cdot x_{n}^{d}\cdot P(\frac{1}{x_{1}},...,\frac{1}{x_{n}})$ is in $\mathcal{P}$. \end{thm}

\begin{proof} Let $\U$ be a $\iota_{P}$-ultrafilter, and let $\alpha_{1},...,\alpha_{n}\in$$^{*}\N$ be mutually distinct generators of $\U$ with $P(\alpha_{1},...,\alpha_{n})=0$. Consider 

\begin{center} $\beta_{1}=\frac{1}{\alpha_{1}}$, ... ,$\beta_{n}=\frac{1}{\alpha_{n}}$. \end{center}

By construction, $\beta_{1},...,\beta_{n}$ are $\sim_{u}$-equivalent elements and $P(\frac{1}{\beta_{1}},...,\frac{1}{\beta_{n}})=0$, so $Q(\frac{1}{\beta_{1}},...,\frac{1}{\beta_{n}})=0$. The problem is that $\beta_{1},...,\beta_{n}$ are not hypernatural numbers, and in order to apply the Bridge Theorem we need $n$ mutually distinct hypernatural numbers $\xi_{1},...,\xi_{n}$ with $Q(\xi_{1},...,\xi_{n})=0$.\\
Since $P(x_{1},...,x_{n})$ is homogeneous, for every $\gamma\in$$^{\bullet}\N$ 

\begin{center} $Q(\gamma\cdot \beta_{1},...,\gamma\cdot\beta_{n})=0$. \end{center}

If there is an hypernatural number $\gamma$ such that:
\begin{enumerate}
	\item $\gamma\cdot\beta_{i}\in$$^{\bullet}\N$ for every $i\leq n$ and
	\item $\gamma\cdot\beta_{i}\sim_{u}\gamma\cdot\beta_{j}$ for every $i,j\leq n$

\end{enumerate}

then the elements $\gamma\cdot\beta_{1},...,\gamma\cdot\beta_{n}$ are $n$ mutually distinct $\sim_{u}$-equivalent hypernatural numbers, so $\mathfrak{U}_{\gamma\cdot\beta_{1}}$ is a $\iota_{Q}$-ultrafilter on $\N$, and $Q\in \mathcal{P}$. Consider 

\begin{center} $\eta=(\alpha_{1})!$ \end{center}

(where $(\alpha_{1})!$ denotes the factorial of $\alpha_{1}$), and 

\begin{center}$\gamma=$$^{*}\eta$.\end{center}

We claim that $\gamma$ has the properties (1) and (2): 
\begin{enumerate}
	\item $\gamma\cdot\beta_{i}$ is an hypernatural number for every index $i\leq n$, because $\gamma=$$^{*}(\alpha_{1}!)=$$(^{*}\alpha_{1})!$ that (as $^{*}\alpha_{1}>\alpha_{i}$ for every $i\leq n$) is divisible by each $\alpha_{i}$;
	\item $\gamma\cdot\beta_{i}\sim_{u}\gamma\cdot\beta_{j}$ for every $i,j\leq n$ because, for every $i\leq n$, $\gamma\cdot\beta_{i}=\beta_{i}\cdot$$^{*}\eta\in G_{\mathfrak{U}_{\beta_{i}}\odot\U_{\eta}}$ and, as $\beta_{i}\sim_{u}\beta_{j}$ for every $i,j\leq n$, $\mathfrak{U}_{\beta_{i}}\odot\mathfrak{U}_{\eta}=\mathfrak{U}_{\beta_{j}}\odot\mathfrak{U}_{\eta}$. 
\end{enumerate}
\end{proof}

In $\cite{rif6}$ the autors actually proved that the above result can be generalized. In fact, they proved that, if $G(x_{1},...,x_{n})$ is a partition regular system of homogeneous equations (this means that for every finite coloration of $\N$ there is a monochromatic solution to $G(x_{1},...,x_{n})=0$), then the system $G(\frac{1}{x_{1}},...,\frac{1}{x_{n}})$ is partition regular.\\
As a corollary of Theorem 3.6.6, we prove that there are partition regular polynomials that do not admit a set of exclusive variables:

\begin{cor} The polynomial $P(x,y,z): yz+xz-xy$ is injectively partition regular. \end{cor}

\begin{proof} The polynomial $Q(x,y,z): x+y-z$ is injectively partition regular. So the polynomial $xyz(\frac{1}{x}+\frac{1}{y}-\frac{1}{z})$ is injectively weakly partition regular, and this entrails the thesis.\\\end{proof}

A different proof of this corollary can be found in $\cite{rif10}$.

\section{Further Studies}

In this section we present three possible developments of the research on partition regular polynomials.

\subsection{Polynomials with partial degree $\geq 2$}

Corollary 3.5.16 regards polynomials with partial degree 1; a natural question is if it can be extended to polynomials with partial degree greater than one. The answer is negative:

\begin{prop} The polynomial

\begin{center} $P(x,y,z)=x+y-z^{2}$ \end{center}

is not weakly partition regular. \end{prop}

This result is proved in $\cite{rif10}$. Observe that $x+y-z^{2}$ admits a set of exclusive variables and that its reduct Red($P$)=$y_{1}+y_{2}-y_{3}$ is weakly partition regular.\\
A way to obtain some partial result regarding partition regular polynomials with partial degree greater than one is to consider the multiplicative Van der Waerden ultrafilters:

\begin{defn} An ultrafilter $\U$ is a {\bfseries multiplicative Van der Waerden ultrafilter} if every element $A$ of $\U$ satisfies the multiplicative Van der Waerden's property, i.e. if $A$ contain arbitrarily long geometric progressions.\end{defn}

It is well-known that, in $\bN$, there are multiplicative Van der Waerden ultrafilters: e.g., for every Van der Waerden ultrafilter $\U$, the ultrafilter $2^{\U}$ is a multiplicative Van der Waerden ultrafilter.

\begin{prop} Let $h,k$ be positive natural numbers. Then, for every positive natural numbers $n_{1},...,n_{h},m_{1},...,m_{k}$ such that $\sum_{i=1}^{h}n_{i}=\sum_{j=1}^{k}m_{j}$, the polynomial:

\begin{center} $P(x_{1},...,x_{h},y_{1},...,y_{k}):\prod_{i=1}^{h}x_{i}^{n_{i}}-\prod_{k=1}^{m}y_{j}^{m_{j}}$ \end{center}

is injectively partition regular. \end{prop}

\begin{proof} As $\sum_{i=1}^{h}n_{i}=\sum_{j=1}^{k}m_{j}$, the polynomial 

\begin{center} $Q(t_{1},...,t_{h},z_{1},...,z_{k}): \sum_{i=1}^{h}n_{i}t_{i}-\sum_{j=1}^{k}m_{j}z_{j}$ \end{center}

is injectively partition regular (as we proved in Theorem 3.5.4); in particular, there are mutually distinct positive natural numbers $a_{1},..,a_{h},b_{1},...,b_{k}$ with 

\begin{center} $Q(a_{1},...,a_{h},b_{1},...,b_{k})=0$. \end{center}

Pose $M=\max\{a_{1},...,a_{h},b_{1},....,b_{k}\}$. Let $\U$ be a multiplicative Van der Waerden ultrafilter. We claim that $\U$ is a $\iota_{P}$-ultrafilter.\\
In fact, let $\eta,\xi$ be hypernatural numbers in $^{\bullet}\N$ such that $\eta,\eta\xi,....,\eta\xi^{M}$ are generators of $\U$, and pose 

\begin{center} $\alpha_{i}=\eta\xi^{a_{i}}$ for $i\leq h$; $\beta_{j}=\eta\xi^{b_{j}}$ for $j\leq k$. \end{center}

By construction, $\alpha_{1},...,\alpha_{h},\beta_{1},...,\beta_{k}$ are mutually distinct generators of $\U$, and $P(\alpha_{1},...,\alpha_{k},\beta_{1},...,\beta_{h})=0$:

\begin{center} $P(\alpha_{1},...,\alpha_{k},\beta_{1},...,\beta_{h})=\prod_{i=1}^{h}(\eta\xi^{a_{i}})^{n_{i}}-\prod_{j=1}^{k}(\eta\xi^{b_{j}})^{m_{j}}=$\\\vspace{0.3cm}$\eta^{\sum_{i=1}^{k}n_{i}}\xi^{\sum_{i=1}^{k}a_{i}n_{i}}-\eta^{\sum_{j=1}^{h}m_{j}}\xi^{\sum_{j=1}^{h}b_{j}m_{j}}=0$. \end{center}

So $\U$ is a $\iota_{P}$-ultrafilter; in particular, $P(x_{1},...,x_{h},y_{1},...,y_{k})$ is an injectively partition regular polynomial. \\ \end{proof}

This result, combined with Theorem 3.6.3, could be a basic tool to prove the partition regularity for equations with partial degree greater than 1, and to try to answer to the following question:\\

{\bfseries Question:} Is there a characterization of partition regular polynomials in $\Z[\mathbf{X}]$?

\subsection{Extension of Deuber's Theorem on Rado's Conjecture}

In 1973, Walter Deuber proved, in his PhD Thesis, Rado's conjecture on large subsets of $\N$:

\begin{defn} $1)$ An $m\times n$ matrix $M$ with integer entries is {\bfseries partition regular} $($on $\N)$ if the existential formula $\exists x_{1},...,x_{n} M(x_{1},...,x_{n})=0$ is partition regular on $\N$.\\
$2)$ A subset $A$ of $\N$ is {\bfseries large} if and only if for every $m\times n$ partition regular matrix $M$, there are elements $a_{1},...,a_{n}$ in $A$ such that $A(a_{1},...,a_{n})=0$. \end{defn}

{\bfseries Rado's Conjecture:} {\itshape The family 

\begin{center}$\mathcal{L}=\{A\subseteq\N\mid A$ is large$\}$ \end{center}

of large subsets of $\N$ is strongly partition regular, i.e. if an element $A$ of $\mathcal{L}$ is partitioned into finitely many sets $A=A_{1}\cup...\cup A_{n}$ then, for at least one index $i\leq n$, $A_{i}\in\mathcal{L}$.}\\

Deuber's result settled the case for the family of large sets (his original proof is in $\cite{rif11}$, see also $\cite{rif12}$); our generalization to non-linear polynomials gives rise to two questions similar to Rado's Conjecture:

\begin{defn} A subset $A$ of $\N$ is a {\bfseries polynomial set} if for every weakly partition regular polynomial $P(x_{1},...,x_{n})$ there are elements $a_{1},...,a_{n}$ in $A$ such that $P(a_{1},...,a_{n})=0$.\\ $A$ is an {\bfseries homogeneous set} if for every homogeneous partition regular polynomial $P(x_{1},...,x_{n})$ there are elements $a_{1},...,a_{n}$ in $A$ such that $P(a_{1},...,a_{n})=0$.\end{defn}

{\bfseries Question 1:} Are the family of polynomial subsets of $\N$ or the family of homogeneous subsets of $\N$ weakly partition regular?\\

We do not know the answer for the polynomial case; as for the homogeneous case, the answer is affirmative:

\begin{thm} There is an ultrafilter $\U$ on $\N$ such that, for every homogeneous partition regular polynomial $P(x_{1},...,x_{n})$, $\U$ is a $\iota_{P}$-ultrafilter. \end{thm}

\begin{proof} For every homogeneous partition regular polynomial $P(x_{1},...,x_{n})$, consider 

\begin{center} $\mathcal{S}_{P}=\{\U\in\bN\mid \U$ is a $\iota_{P}$-ultrafilter$\}$. \end{center}

As we proved in Section 3.4, every set $\mathcal{S}_{P}$ is a closed subset of $\bN$, and a bilateral ideal in $(\bN,\odot)$. \\

{\bfseries Claim:} The family $\{\mathcal{S}_{P}\}_{P\in\mathcal{H}}$ has the finite intersection property.\\

In fact, as a consequence of Corollary 3.3.6, for every natural number $k$, for every polynomials $P_{1}(x_{1,1},...,x_{1,n_{1}}),....,P_{k}(x_{k,1},...,x_{k,n_{k}})$, if, for every index $i\leq k$, the ultrafilter $\U_{i}$ is an element of $\mathcal{S}_{P_{i}}$ then $\U_{1}\odot...\odot\U_{k}$ is a $\iota_{P_{i}}$ ultrafilter for every $i\leq k$, so

\begin{center} $\U_{1}\odot...\odot\U_{k}\in\bigcap_{i=1}^{k}\mathcal{S}_{P_{i}}$. \end{center}

As $\bN$ is compact, and the family $\{\mathcal{S}_{P}\}_{P\in\mathcal{H}}$, where $\mathcal{H}$ is the set of homogeneous partition regular polynomials, has the finite intersection property, the intersection 

\begin{center} $\bigcap_{P\in\mathcal{H}}\mathcal{S}_{P}$ \end{center}

is nonempty, and if $\U$ is an ultrafilter in this intersection, by construction it is a $\iota_{P}$ ultrafilter for every homogeneous weakly partition regular polynomial $P(x_{1},...,x_{n})$.\\\end{proof}

\begin{cor} The family of homogeneous subsets of $\N$ is weakly partition regular. \end{cor}

\begin{proof} Let $\U$ be an ultrafilter such that, for every homogeneous partition regular polynomial $P(x_{1},...,x_{n})$, $\U$ is a $\iota_{P}$-ultrafilter, and $A$ any of its elements; by construction $A$ is an homogeneous set. So $\U$ is a subset of the family of homogeneous sets, and by Theorem 1.2.3 it follows that this family is weakly partition regular.\\\end{proof}

{\bfseries Question 2:} Are the family of polynomial subsets of $\N$ or the family of homogeneous subsets of $\N$ strongly partition regular?\\

We do not know. We conjecture that this is true at least for the homogenous case, while the non-homogeneous case seems particularly challenging.

\subsection{Partition regularity on $\Z,\Q,\R$}

In this chapter we concentrated on the partition regularity on $\N$ for polynomials in $\Z[\mathbf{X}]$; a natural generalization would be to study the partition regularity of the same polynomials on $\Z,\Q,\R$.

\begin{defn} Let $A$ be a subset of $\R$, and $P(x_{1},...,x_{n})$ a polynomial in $\Z[\mathbf{X}]$. $P(x_{1},...,x_{n})$ is {\bfseries partition regular on $A$} if the existential formula:

\begin{center} $\sigma_{P}: \exists a_{1},...,a_{n}\in A\setminus\{0\}$ $P(a_{1},...,a_{n})=0$ \end{center}

is weakly partition regular on $A$.\\
$P(x_{1},...,x_{n})$ is {\bfseries injectively partition regular on $A$} if the existential formula:

\begin{center} $\iota_{P}:\exists a_{1},...,a_{n}\in A\setminus\{0\}$ $(P(a_{1},...,a_{n})=0\wedge \bigwedge_{i\neq j, i,j=1}^{n} a_{i}\neq a_{j})$ \end{center}

is weakly partition regular on $A$.\\
The {\bfseries set of partition regular polynomials on $A$} is denoted by $\mathcal{P}_{A}$. \end{defn}

As $\Z$, $\Q$ and $\N$ have the same cardinality, and since the hyperextension $^{*}\N$ has the $\mathfrak{c}^{+}-$enlarging property, the Bridge Theorem could be proved also for ultrafilters on $\Z$, $\Q$; from now on, in this section, we assume that the hyperextension $^{*}\N$ has the $|\wp(\R)|^{+}$-enlarging property, because with this hypothesis the Bridge Theorem is valid also for ultrafilters on $\R$.\\
First of all, from the definitions it follows that 

\begin{center} $\mathcal{P}_{\N}\subseteq\mathcal{P}_{\Z}\subseteq\mathcal{P}_{Q}\subseteq\mathcal{P}_{R}$ \end{center}

A known fact is that every linear homogeneous polynomial $P(x_{1},...,x_{n})\in\Z[\mathbf{X}]$ is partition regular on $\N$ if and only if it is partition regular on $\R$ (see e.g. $\cite{rif51}$). Actually, the arguments used to prove Theorem 3.6.6 could be arranged to prove the following result:

\begin{thm} Let $P(x_{1},...,x_{n})$ be an homogeneous polynomial in $\Z[\mathbf{X}]$. Then the following conditions are equivalent:

\begin{enumerate}
	\item $P(x_{1},...,x_{n})$ is injectively partition regular on $\N$;
	\item $P(x_{1},...,x_{n})$ is injectively partition regular on $\Z$;
	\item $P(x_{1},...,x_{n})$ is injectively partition regular on $\Q$.
\end{enumerate}
\end{thm}

When the polynomial is not homogeneous, the equivalence does not hold. In fact, e.g., the polynomial $x+y^{2}$ is partition regular on $\Z$ (since $x=y=-1$ is a solution) but it has no solutions on $\N$, and the polynomial $2x+1$ is partition regular on $\Q$ but it has no solutions on $\Z$.\\
A result that is useful to study the partition regularity in this more general context is the following:

\begin{thm} Let $^{*}\R$ be an hyperextension of $\R$ with the $|\R|^{+}$-saturation property. Let $A$ be a subset of $\R$, $P(x_{1},...,x_{n})$ a polynomial in $\mathcal{P}_{A}$ and $f\in\mathtt{Fun}(A,A)$ an injective function. The following conditions are equivalent:
\begin{enumerate}
	\item there is a $\iota_{P}$-ultrafilter $\U$ in $\Theta_{f[A]}$;
	\item the existential sentence $\phi(y_{1},...,y_{n})$: "there are mutually distinct $y_{1},...,y_{n}$ with $P(f(y_{1}),...,f(y_{n}))=0$" is weakly partition regular.
\end{enumerate}
\end{thm}

\begin{proof} $(1)\Rightarrow(2)$ Let $\U$ be a $\iota_{P}$-ultrafilter, and $\alpha_{1},...,\alpha_{n}\in$$^{*}\R$ mutually distinct elements in $G_{\U}$ such that $P(\alpha_{1},...,\alpha_{n})=0$. From the hypothesis it follows that there is an ultrafilter $\V$ in $\beta A$ such that $\U=\overline{f}(\V)$. Then, as proved in Theorem 2.3.6 (the result was stated and proved for $A=\N$ but, since the hyperextension $^{*}\R$ satisfies the $|\R|^{+}$-saturation property, the same proof shows that the result holds for $A=\Q,\R,\Z$ as well), there are mutually distinct elements $\beta_{1},...,\beta_{n}\in G_{\V}$ with $^{*}f(\beta_{i})=\alpha_{i}$ for every index $i\leq n$. This proves that $\V$ is a $\phi(y_{1},...,y_{n})$-ultrafilter.\\
$(2)\Rightarrow(1)$ Let $\V$ be a $\phi(y_{1},....,y_{n})$-ultrafilter, $\beta_{1},...,\beta_{n}$ mutually distinct elements in $G_{\V}$ with $P($$^{*}f(\beta_{1}),....,$$^{*}f(\beta_{n}))=0$, and consider $\U=\overline{f}(\V)$. By construction, if for every index $i\leq n$ we pose $\alpha_{i}=$$^{*}f(\beta_{i})$, then $\alpha_{1},...,\alpha_{n}$ are mutually distinct elements in $G_{\U}$ with $P(\alpha_{1},...,\alpha_{n})=0$, so $\U$ is a $\iota_{P}$-ultrafilter in $\Theta_{f[A]}$.\\\end{proof}

We present three corollaries of this result:

\begin{cor} A polynomial $P(x_{1},...,x_{n})$ is in $\mathcal{P}_{\Z}$ if and only if $P(-x_{1},...,-x_{n})$ is in $\mathcal{P}_{\Z}$. \end{cor}

\begin{proof} The function $x\rightarrow -x$ is a bijective function in $\mathtt{Fun}(\Z,\Z)$, so the result is a trivial consequence of Theorem 3.7.10.\\ \end{proof}

The previous corollary entails that there are injectively partition regular polynomials $P(x_{1},...,x_{n})$ on $\Z$, with partial degree 1, such that Red($P$) is not partition regular: consider, e.g., the polynomial

\begin{center} $P(x_{1},x_{2},x_{3},x_{4},x_{5},x_{6}): x_{1}x_{2}x_{3}+x_{4}x_{2}x_{3}-x_{5}x_{6}$;\end{center}

as a consequence of Theorem 3.5.13, this polynomial is injectively partition regular on $\N$ so, in particular, it is injectively partition regular on $\Z$. So, the polynomial

\begin{center} $Q(x_{1},x_{2},x_{3},x_{4},x_{5},x_{6})=P(-x_{1},-x_{2},-x_{3},-x_{4},-x_{5},-x_{6}): -x_{1}x_{2}x_{3}-x_{4}x_{2}x_{3}-x_{5}x_{6}$ \end{center}

is injectively weakly partition regular on $\Z$, while its reduct 

\begin{center} Red($Q$)=$-y_{1}-y_{2}-y_{3}$ \end{center}

is not.\\
The second corollary regards partition regular polynomials on $\Q$:

\begin{cor} A polynomial $P(x_{1},...,x_{n})$ is in $\mathcal{P}_{\Q}$ if and only if the polynomial $P(\frac{1}{x_{1}},...,\frac{1}{x_{n}})$ is in $\mathcal{P}_{\Q}$. \end{cor}

\begin{proof} Let $inv:\Q\setminus\{0\}\rightarrow\Q\setminus\{0\}$ be the function that maps every rational number $q$ into $q^{-1}$; $inv$ is a bijection, so the thesis is a trivial consequence of Theorem 3.7.10.\\ \end{proof}

Our last corollary involves the partition regularity on $\R$:
 
\begin{cor} Let $z\neq 0$ be an integer, and $\R_{>0}=\{r\in\R\mid r>0\}$. Then a polynomial $P(x_{1},...,x_{n})$ is in $\mathcal{P}_{\R_{>0}}$ if and only if $P(x_{1}^{z},...,x_{n}^{z})$ is in $\mathcal{P}_{\R_{>0}}$. \end{cor}

\begin{proof} For every integer $z\neq 0$, the function $f_{z}:\R_{>0}\rightarrow\R_{>0}$ such that $f_{z}(x)=x^{z}$ is a bijection, so the thesis is a straightforward consequence of Theorem 3.7.10. \\ \end{proof}

E.g., for every natural numbers $n,m$ with $n+m\geq 3$ and for every integer $z\neq 0$, the polynomial $(\sum_{i\leq n} x_{i}^{z})-(\prod_{j\leq m}y_{j}^{z})$ is in $\mathcal{P}_{\R_{>0}}$, as $(\sum_{i\leq n} x_{i})-(\prod_{j\leq m}y_{j})$ is injectively partition regular on $\N$.\\
A particular case of this result are the polynomials

\begin{center} $P_{n}(x,y,z): x^{n}+y^{n}-z^{n}$ \end{center}

that are injectively weakly partition regular on $\R$ while, for $n\geq 3$, as a consequence of Fermat's Last Theorem, they are not weakly partition regular on $\N$ (as they do not admit any solution in $\N$). The problem of the partition regularity of $x^{2}+y^{2}-z^{2}$ on $\N$ is still open; we conclude this chapter observing that this polynomial is in $\mathcal{P}_{\N}$ if and only if there is a Schur ultrafilter in $\Theta_{Sq}$, where 

\begin{center} $Sq=\{n\in\N\mid \exists m\in\N$ $n=m^{2}\}$. \end{center}

\chapter{Finite Embeddability and Functional Relations}

In this chapter we study the notion of finite embeddability and its generalizations, that have some interesting combinatorial properties. The finite embeddability is defined as follows: given subsets $A,B$ of $\N$, $A$ is finitely embeddable in $B$ if for every finite subset $F$ of $A$ there is a natural number $n$ such that $n+F\subseteq B$. This notion can be generalized to ultrafilters: given ultrafilters $\U,\V$ on $\N$, $\U$ is finitely embeddable in $\V$ if for every set $B$ in $\V$ there is a set $A$ in $\U$ such that $A$ is finitely embeddable in $B$.\\
In Section One we present the idea that generated the study of the finite embeddability.\\
Section Two is dedicated to recalling some basic facts about pre-orders, that will be used throughout the chapter. \\
Sections Three, Four and Five are dedicated to the study of the finite embeddability defined on subsets of $\N$ and on $\bN$. In particular, we prove that the finite embeddability is a pre-order with maximal elements, and that these maximal elements have interesting combinatorial properties.\\
The definition of finite embeddability can be reformulated in terms of translations, and it can be generalized by replacing the translations with arbitrary sets of functions. In Sections Six, Seven and Eight we study this generalized notion, called finite mappability, with a particular attention to its combinatorial properties.\\
In Section Nine we consider the particular case of the finite mappability under affinities, proving that it is related with Van der Waerder ultrafilters.\\
Finally, in Section Ten, we present two possible future directions of our research.

\section{The Initial Idea}

In $\cite{rif2}$ Mathias Beiglböck introduced the concept of "ultrafilter-shift":

\begin{defn} Let $\U$ be an ultrafilter on $\Z$ and let $A$ be a subset of $\Z$. The {\bfseries $\U$-shift of $A$} is the set

\begin{center} $A-\U=\{z\in\Z\mid (A-z)\in\U\}$.\end{center}\end{defn}

This concept was used as a tool to give a short proof to the following theorem of Renling Jin (see $\cite{rif25}$):

\begin{thm}[Jin] Let $A, B$ be subsets of $\N$ with positive Banach density. Then $A+B=\{a+b\mid a\in A, b\in B\}$ is piecewise syndetic. \end{thm}

Inspired by Beiglböck's proof, in $\cite{rif14}$ Mauro Di Nasso introduced the notion of finite embeddability for subsets of $\N$, and used it to improve on Jin's Theorem.

\begin{defn} Given subsets $A, B$ of $\N$, $A$ is {\bfseries finitely embeddable in $B$} $($notation $A\leq_{fe} B)$ if and only if for every finite subset $F$ of $A$ there is a natural number $n$ such that $n+A\subseteq B$.\end{defn}

Finite embeddability has a nice nonstandard characterization and turns out to be closely related to the ultrafilter-shifts of Beiglböck on $\N$. In the result below, as in the rest of this chapter, we suppose that the hyperextension $^{*}\N$ satisfies the $\mathfrak{c}^{+}$-enlarging property and we recall that, given an hypernatural number $\alpha$, $\mathfrak{U}_{\alpha}$ is the ultrafilter on $\N$ such that, for every subset $A$ of $\N$, $A\in\mathfrak{U}_{\alpha}\Leftrightarrow \alpha\in$$^{*}A$.

\begin{prop} Let $A,B$ be subsets of $\N$. The following three conditions are equivalent:
\begin{enumerate}
	\item $A$ is finitely embeddable in $B$;
	\item there is an hypernatural number $\alpha$ in $^{*}\N$ such that $\alpha+A\subseteq$$^{*}B$
	\item $A$ is included in some ultrafilter shift of $B$ on $\N$.
\end{enumerate}
\end{prop}

\begin{proof} $(1)\Rightarrow(2)$ Suppose that $A$ is finitely embeddable in $B$, and let $F$ be a finite subset of $A$. Let $S_{F}$ 

\begin{center} $S_{F}=\{n\in\N\mid n+F\subseteq B\}$. \end{center}

The family $\{S_{F}\}_{F\in\wp_{fin}(A)}$ (where $\wp_{fin}(A)$ denotes the set of finite subsets of $A$) has the finite intersection property, since $S_{F_{1}}\cap S_{F_{2}}=S_{F_{1}\cup F_{2}}$ for every finite subsets $F_{1},F_{2}$ of $A$.\\
By $\mathfrak{c}^{+}-$enlarging property, it follows that

\begin{center} $T=\bigcap_{F\in\wp_{fin}(A)}$$^{*}S_{F}\neq\emptyset$. \end{center}

Observe that, for every finite set $F\subseteq\N$, 

\begin{center} $^{*}S_{F}=\{\alpha\in$$^{*}\N\mid \alpha+$$^{*}F\subseteq$$^{*}B\}=\{\alpha\in$$^{*}\N\mid \alpha+F\subseteq$$^{*}B\}$, \end{center} 

since $F=$$^{*}F$. If $\alpha$ is any hypernatural number in $T$, then $\alpha+A\subseteq$$^{*}B$ as, by construction, $\alpha+F\subseteq$$^{*}B$ for every finite subset $F$ of $A$.\\
$(2)\Rightarrow(1)$ Let $\alpha$ be an hypernatural number in $^{*}\N$ with $\alpha+A\subseteq$$^{*}B$, and suppose that $A$ is not finitely embeddable in $B$. Take a finite subset $F$ of $A$ such that, for every natural number $n$, the translation $n+F$ is not included in $B$. By transfer it follows that for every hypernatural number $\eta$ the set $\eta+F$ is not included in $^{*}B$, and this is absurd since $\alpha+F\subseteq\alpha+A\subseteq$$^{*}B$. So $A$ is finitely embeddable in $B$. \\
$(2)\Rightarrow (3)$ Consider the ultrafilter $\U=\mathfrak{U}_{\alpha}$. As $\alpha+A\subseteq$$^{*}B$, for every $a\in A$ $\alpha\in($$^{*}B-a)=$$^{*}(B-a)$, so

\begin{center} $B-a\in\mathfrak{U}_{\alpha}$ for every $a\in A$. \end{center}

This entails that $A\subseteq B-\mathfrak{U}_{\alpha}$.\\
$(3)\Rightarrow (2)$ Let $\U=\mathfrak{U}_{\alpha}$ be an ultrafilter such that $A\subseteq B-\U$. Observe that $B-\mathfrak{U}_{\alpha}=\{n\in\N\mid B-n\in\mathfrak{U}_{\alpha}\}=\{n\in\N\mid\alpha\in$$^{*}B-n\}=\{n\in\N\mid \alpha+n\in$$^{*}B\}$. Since $A\subseteq \{n\in\N\mid \alpha+n\in$$^{*}B\}$, it follows that $\alpha+A\subseteq$$^{*}B$. \\\end{proof}

Another interesting property of finite embeddability involves the Banach density of subsets of $\N$ (that has been introduced in Section 3.1):

\begin{prop} Let $A, B$ be subsets of $\N$ with $A\leq_{fe} B$. Then $BD(A)\leq_{fe} BD(B)$. \end{prop}

\begin{proof} We recall the nonstandard characterization of Banach density introduced in Section 3.1:

\begin{center} $BD(A)\geq a$ if and only if there are $\alpha\in$$^{*}\N$, $\beta\in$$^{*}\N\setminus\N$ such that $st(\frac{|^{*}A\cap [\alpha,\alpha+\beta)|}{\beta})\geq a$. \end{center}
As $A\leq_{fe} B$, by Proposition 4.1 there is an hypernatural number $\gamma$ with $\gamma+A\subseteq$$^{*}B$; by transfer it follows that

\begin{center} $^{*}\gamma+$$^{*}A\subseteq$$^{**}B$. \end{center}

Observe that 

\begin{center} $|$$^{*}A\cap [\alpha,\alpha+\beta)|=|($$^{*}\gamma+$$^{*}A)\cap [$$^{*}\gamma+\alpha,$$^{*}\gamma+\alpha+\beta)|\leq |$$^{**}B\cap[$$^{*}\gamma+\alpha,$$^{*}\gamma+\alpha+\beta)|$ \end{center}

so 

\begin{center} $BD(A)\leq st(\frac{|^{*}A\cap[\alpha,\alpha+\beta)|}{\beta})\leq st(\frac{|^{**}B\cap[^{*}\gamma+\alpha,^{*}\gamma+\alpha+\beta)|}{\beta})\leq BD(B)$. \end{center}
\end{proof}

As suggested by the above results, the combinatorial properties of the relation of finite embeddability deserve to be studied in more detail. This chapter is dedicated to the study of this relation and of its generalizations.\\

{\bfseries Convention:} Throughout this chapter we assume that the hyperextension $^{*}\N$ satisfies the $\mathfrak{c}^{+}$-enlarging property, and we work in $^{\bullet}\N$.\\

This is done since, as usual, we want to have nonempty sets of generators for all ultrafilters $\U$ in $\bN$.

\section{Partial Pre--Orders}

In this chapter we use some general features of partial pre-orders:

\begin{defn} Let $S$ be a set, and $\leq$ a binary relation on $S$. $(S,\leq)$ is a {\bfseries partial pre-ordered set} if the relation $\leq$ is transitive and reflexive on $S$, i.e. if for every $x,y,z$ is $S$ the following two conditions hold:
\begin{enumerate}
	\item $x\leq y$, $y\leq z\Rightarrow x\leq z$;
	\item $x\leq x$.
\end{enumerate}
In this case we also say that $\leq$ is a {\bfseries partial pre-order on $S$}.\\
The pre-order is {\bfseries total} if, for every $x\neq y$ in $S$, $x\leq y$ or $y\leq x$.\\
A pre-order $\leq$ is an {\bfseries order} if $\leq$ is antysimmetric, i.e. if for every $x,y$ in $S$, if $x\leq y$ and $y\leq x$ then $x=y$. In this case, we say that $(S,\leq)$ is a {\bfseries partially ordered set} if $\leq$ is a partial order, and that $(S,\leq)$ is a {\bfseries totally ordered set} if $\leq$ is a total order.\end{defn}

When dealing with partially pre-ordered sets, we are usually interested in particular elements:

\begin{defn} Let $(S,\leq)$ be a partially pre-ordered set, $A$ a subset of $S$, and let $x$ be an element in $S$. Then
\begin{itemize}
	\item $x$ is the {\bfseries greatest element} of $(S,\leq)$ if for every element $y$ in $S$, $y\leq x$;
	\item $x$ is {\bfseries maximal} if for every element $y$ in $S$, if $x\leq y$ then $y\leq x$;
	\item $x$ is an {\bfseries upper bound} of $A$ is $a\leq x$ for every element $a\in A$.
\end{itemize}

The {\bfseries upper cone generated by $x$} is the set $\mathcal{C}(x)=\{y\in S\mid x\leq y\}$. \end{defn}

An important sort of substructures of pre-ordered sets are the chains:

\begin{defn} Let $(S,\leq)$ be a partially pre-ordered set. A sequence \\$\langle x_{i}\mid i\in I\rangle$ of elements in $S$ is a {\bfseries chain} if
\begin{enumerate}
	\item $(I,\preceq)$ is a totally ordered set;
	\item $x_{i}\leq x_{j}$ whenever $i\prec j$ in $(I,\preceq)$.
\end{enumerate}
\end{defn}

Chains are fundamental in relation with Zorn's Lemma:

\begin{lem}[Zorn] Let $(S,\leq)$ be a partially ordered set. Suppose that every chain in $(S,\leq)$ has an upper bound in $S$. Then the set $S$ contains at least one maximal element.\end{lem}

Let $(S,\leq)$ be a partially pre-ordered set. There is a general technique to construct a partially ordered set associated with this pre-order. First of all, we consider the $\leq$-equivalence classes:

\begin{defn} Two elements $x,y\in S$ are {\bfseries $\leq$-equivalent} $($notation $x\equiv y)$ if $x\leq y$ and $y\leq x$. \end{defn}

Observe that $\equiv$ is actually an equivalence relation: it is symmetrical by definition, and it is transitive and reflexive as $\leq$ is a pre-order.\\
The pre-order $\leq$ can be induced on the set of $\leq$-equivalence classes:

\begin{defn} Let $(S,\leq)$ be a partially ordered set. Given two $\leq$-equivalence classes $[x],[y]$ in $S_{/_{\equiv}}$, we pose $[x]\leq [y]$ if and only if $x\leq y$. \end{defn}

We observe that $(S_{/_{\equiv}},\leq)$ is a partially ordered set: the relation $\leq$ is trivially transitive and reflexive on $S_{/_{\equiv}}$. It is also antisymmetrical: suppose that $[x]\leq [y]$ and $[y]\leq [x]$. Then it follows that $x\leq y$ and $y\leq x$, so $[x]=[y]$.\\

{\bfseries Convention:} Throughout this chapter, whenever we deal with a partial pre-ordered set $(S,\leq)$, we reserve the notation $(S_{/_{\equiv}},\leq)$ to denote the partial ordered set constructed as exposed above.\\

Observe that every chain in $(S,\leq)$ as an upper bound if and only if every chain in $(S_{/_{\equiv}},\leq)$ as an upper bound.\\
In this chapter we deal with pre-orders that are filtered: 

\begin{defn} Let $(S,\leq)$ be a partially pre-ordered set. The pre-order $\leq$ is {\bfseries filtered} if for every elements $a,b$ in $S$ there is an element $c$ in $S$ such that $a\leq c$ and $b\leq c$. \end{defn}

Observe that the pre-order $\leq$ is filtered on $S$ if and only if the order $\leq$ is filtered on $S_{/_{\equiv}}$.\\
The property of filtration is particularly important when dealing with partially ordered sets that satisfy the hypothesis of Zorn's Lemma:

\begin{thm} Let $(S,\leq)$ be a partially ordered set such that every chain in $S$ has an upper bound. The following two conditions are equivalent:
\begin{enumerate}
	\item there is a greatest element in $(S,\leq)$;
	\item the order $\leq$ is filtered.
\end{enumerate}
\end{thm}

\begin{proof} $(1)\Rightarrow (2)$ Let $m$ be the greatest element in $(S,\leq)$. Then $a\leq m$ and $b\leq m$ for every $a,b$ in $S$, so the order $\leq$ is filtered.\\
$(2)\Rightarrow (1)$ Suppose that $\leq$ is filtered. By Zorn's Lemma, since every chain in $S$ has an upper bound, there are maximal elements in $(S,\leq)$. Suppose that $m_{1}$, $m_{2}$ are two maximal elements. Then, as $\leq$ is filtered, there is an element $m$ such that $m_{1}\leq m$ and $m_{2}\leq m$. Since $m_{1}$, $m_{2}$ are maximal, it follows that $m_{1}=m$ and $m_{2}=m$, hence $m_{1}=m_{2}$: there is only one maximal element, that is the greatest element in $(S,\leq)$.\\ \end{proof}

We also observe that, if $(S,\leq)$ is a filtered partially pre-ordered set, then an element $s\in S$ is maximal if and only if its equivalence class $[s]$ is the greatest element in $(S_{/_{\equiv}},\leq)$. The proof of this fact is trivial.

\section{Finite Embeddability for Subsets of $\N$}

In this section we study the relation of finite embeddability.\\

{\bfseries Note:} We assume that $0\in\N$ so, e.g., if $A,B$ are subsets of $\N$ and $A\subseteq B$, then $A\leq_{fe} B.$

\begin{prop} The relation $\leq_{fe}$ is a partial pre-order on $\wp(\N)$.\end{prop}

\begin{proof} We have to prove that $\leq_{fe}$ is transitive and reflexive.\\
Transitive: Let $A,B,C$ be subsets of $\N$ with $A\leq_{fe} B$ and $B\leq_{fe} C$, and $F$ a finite subset of $A$. As $A\leq_{fe} B$, there is a natural number $a$ with $a+F\subseteq B$; since $a+F$ is a finite subset of $B$ and $B\leq_{fe} C$, there is a natural number $b$ with $b+a+F\subseteq C$. If $c=a+b$, by construction $c+F\subseteq C$. This proves that $A$ is finitely embeddable in $C$.\\
Reflexive: for every finite subset $F\subseteq A$, $0+F=F\subseteq A$ so $A$ is trivially finitely embeddable in $A$.\\\end{proof}

The relation $\leq_{fe}$ is not antisymmetric, so it is not a partial order. E.g, if $O$ denote the set of odd natural numbers and $E$ the set of even natural numbers then $O\leq_{fe} E$ and $E\leq_{fe} O$, as $1+O\subseteq$$^{*}E$ and $1+E\subseteq$$^{*}O$. Notice also that the relation of finite embeddability is not a total pre-order: e.g., if $A=\{1,3\}$ and $B=\{2,5\}$ then $A$ and $B$ are incomparable. 

\begin{defn} Given $A, B$ subsets of $\N$, $A$ is {\bfseries $fe$-equivalent} to $B$ $($notation $A\equiv_{fe} B)$ if $A\leq_{fe} B$ and $B\leq_{fe} A$. For every set $A$ in $\wp(\N)$, we denote by $[A]$ the equivalence class of $A$ respet to $\equiv_{fe}$:

\begin{center} $[A]=\{B\in\wp(\N)\mid A\equiv_{fe} B\}$. \end{center}
\end{defn}

As a consquence of the general results about partial pre-orders exposed in Section 4.2, the relation $\equiv_{fe}$ is an equivalence relation, and $(\wp(\N)_{/_{\equiv_{fe}}},\leq_{fe})$ is a partially ordered set.\\

{\bfseries Fact:} $[\N]$ is the greatest element in $(\wp(\N)_{/_{\equiv_{fe}}},\leq_{fe})$. \\

This is evident, as every subset $A$ of $\N$ is finitely embeddable in $\N$.

\begin{defn} A subset $A$ of $\N$ is {\bfseries $fe$-maximal} if $B\leq_{fe} A$ for every set $B$ in $\wp(\N)$. \end{defn}

Observe that by the definition it follows that a set $A$ is $fe$-maximal if and only if $[A]$ is the greatest element in $(\wp(\N)_{/_{\equiv_{fe}}},\leq_{fe})$ if and only if $\N\leq_{fe} A$.\\
We recall that a subset $A$ of $\N$ is thick if it contains arbitrarily long intervals.

\begin{prop} Let $A$ be a subset of $\N$. The following two conditions are equivalent:
\begin{enumerate}
	\item $A$ is $fe$-maximal;
	\item $A$ is thick.
\end{enumerate}
\end{prop}

\begin{proof} $(1)\Rightarrow(2)$ Suppose that $\N\leq_{fe} A$. Let $n$ be a natural number, and consider the finite subset $\{0,...,n\}$ of $\N$. As $\N\leq_{fe} A$, there is a natural number $m$ such that $\{m,...,m+n\}\subseteq A$, so $A$ contains an interval of length $n$. Since this is true for every natural number $n$, it follows that $A$ contains arbitrarily long intervals, so $A$ is thick.\\
$(2)\Rightarrow(1)$ Suppose that $A$ contains arbitrarily long intervals: let $F$ be a finite subset of $\N$, and pose $n=\max F$. As $A$ contains an interval of length $n$, there is a natural number $m$ such that $m+\{0,...,n\}\subseteq A$; in particular, $m+F\subseteq m+\{0,...,n\}\subseteq A$, so $\N$ is finitely embeddable in $A$. \\ \end{proof}

An important feature of finite embeddability is the fact that it is "upward closed under additively invariant formulas" (that have been introduced in Section 3.3). We recall that, given a first order formula $\phi(x_{1},...,x_{n})$, $E(\phi(x_{1},...,x_{n}))$ denotes the existential closure of $\phi(x_{1},...,x_{n})$, which is the sentence $\exists x_{1},...,x_{n} (x_{1},...,x_{n})$.

\begin{prop} Let $\phi(x_{1},...,x_{n})$ be an additively invariant formula, and let $A,B$ be subsets of $\N$. If $A$ satisfies $E(\phi(x_{1},...,x_{n}))$ and $A\leq_{fe} B$ then also $B$ satisfies $E(\phi(x_{1},...,x_{n}))$. \end{prop}

\begin{proof} If $A$ satisfies $E(\phi(x_{1},...,x_{n}))$ then there are $a_{1},...,a_{n}\in A$ such that $\phi(a_{1},...,a_{n})$ holds. Let $F=\{a_{1},...,a_{n}\}$. As $A\leq_{fe} B$ there is a natural number $m$ such that $m+F\subseteq B$. Since $\phi(x_{1},...,x_{n})$ is additively invariant, $\phi(a_{1},...,a_{n})$ implies $\phi(a_{1}+m,...,a_{n}+m)$, where $a_{1}+m,...,a_{n}+m$ are elements in $B$. So $B$ satisfies $E(\phi(x_{1},...,x_{n}))$. \\ \end{proof}

As a corollary, e.g., if $A$ contains a lenght 5 arithmetic progression, and $A\leq_{fe} B$, then $B$ contains a lenght 5 arithmetic progression as well. So, if $A$ satisfies Van der Waerden's property, i.e. if $A$ contains arbitrarily long arithmetic progressions, then $B$ contains arbitrarily long arithmetic progressions as well.

\begin{prop} Let $\Phi$ be the set of additively invariant existential sentences satisfied by $\N$. Then every thick subset $A$ of $\N$ satisfies all the sentences $E(\phi(x_{1},...,x_{n}))\in \Phi$. \end{prop}

\begin{proof} If $E(\phi(x_{1},...,x_{n}))$ is a sentence in $\Phi$, as $\N\leq_{fe} A$, by Proposition 4.3.5 it follows that $A$ satisfies $E(\phi(x_{1},...,x_{n})).$ \\ \end{proof}

Another important feature of finite embeddability is that every set that is $\leq_{fe}$-above a piecewise syndetic set is piecewise syndetic as well (the notion of piecewise syndetic set has been introduced in Section 1.1.5):

\begin{prop} Let $A$ be a piecewise syndetic subset of $\N$, and let $A\leq_{fe} B$. Then $B$ is piecewise syndetic.\end{prop}

\begin{proof} Since $A$ is piecewise syndetic, there is a natural number $n$ such that the set $A-[0,n]=\{m\in\N\mid\exists i\leq n$ with $i+m\in A\}$ is thick.\\

{\bfseries Claim:} $B-[0,n]$ is thick.\\

The claim trivially entails that $B$ is piecewise syndetic.\\
Let $F=\{a_{1},...,a_{k}\}$ be a finite subset of $A-[0,n]$. By hypothesis, there are natural numbers $i_{1},...,i_{k}\leq n$ such that $F^{\prime}=\{a_{1}+i_{1},...,a_{k}+i_{k}\}\subseteq A$. Since $A\leq_{fe} B$, there is a natural number $n_{F}$ such that $n_{F}+F^{\prime}\subseteq B$. So, for every index $j\leq k$, 

\begin{center} $n_{F}+a_{j}+i_{j}\in B$. \end{center}

In particular, $n_{F}+a_{j}\in B-[0,n]$ for every index $j\leq k$, hence $n_{F}+F\subseteq B-[0,n]$. This proves that $A-[0,n]\leq_{fe} B-[0,n]$ and, as $A$ is thick, this entails by maximality that $B-[0,n]$ is thick.\\\end{proof}

The finite embeddability for subsets of $\N$ can also be characterized in terms of ultrafilters:

\begin{prop} Given subsets $A, B$ of $\N$, the following two conditions are equivalent:
\begin{enumerate}
	\item $A$ is finitely embeddable in $B$;
	\item there is an ultrafilter $\V$ on $\N$ such that, for every ultrafilter $\U$ on $\N$, if $A\in \U$ then $B\in \U\oplus\V$.
\end{enumerate}\end{prop}

\begin{proof} $(1)\Rightarrow(2):$ Suppose that $A$ is finitely embeddable in $B$, and let $\alpha$ be an hypernatural number in $^{*}\N$ with $\alpha+A\subseteq$$^{*}B$. By transfer it follows that 

\begin{center} (i) $^{*}\alpha+$$^{*}A\subseteq$$^{**}B$.\end{center}
Consider the ultrafilter $\V$ generated by $\alpha$, let $\U$ be any ultrafilter in $\bN$ with $A\in\U$, and let $\beta\in$$^{*}\N$ be a generator of $\U$. By (i) it follows that \begin{center} $^{*}\alpha+\beta\in$$^{**}B$, \end{center} so $B\in \mathfrak{U}_{^{*}\alpha+\beta}$, and the conclusion follows because $\mathfrak{U}_{\beta+ ^{*}\alpha}=\mathfrak{U}_{\beta}\oplus\mathfrak{U}_{\alpha}=\U\oplus\V$. \\ 
$(2)\Rightarrow(1):$ Let $\V$ be an ultrafilter as in the hypothesis, and $\alpha\in$$^{*}\N$ a generator of $\V$. For every element $a\in A$, $A$ is an element of the principal ultrafilter $\mathfrak{U}_{a}$ so, by hypothesis, $B\in\mathfrak{U}_{a}\oplus\mathfrak{U}_{\alpha}=\mathfrak{U}_{a+\alpha}$. This entails that $a+\alpha\in$$^{*}B$ for every $a\in A$, so $\alpha+A\subseteq$$^{*}B$, and by Proposition 4.1.4 this entails that $A\leq_{fe} B$. \\ \end{proof}

This result originates a question: can the relation of finite embeddability be extended to ultrafilters? We address this question in next section.

\section{Finite Embeddability for Ultrafilters}

\subsection{Basic properties of the relation of finite embeddability for ultrafilters}

In this section we study the properties of the extension of the finite embeddability to ultrafilters (introduced in $\cite{rif5}$). This extension is denoted by $\trianglelefteq_{fe}$. There are at least three possible ways to extend finite embeddability of sets to ultrafilters. Namely, given ultrafilters $\U,\V$ in $\bN$, we could define

\begin{enumerate}
	\item $\U\trianglelefteq_{fe}\V$ if for every set $A$ in $\U$, for every set $B$ in $\V$, $A\leq_{fe} B$;
	\item $\U\trianglelefteq_{fe}\V$ if for every set $A$ in $\U$ there is a set $B$ in $\V$ with $A\leq_{fe} B$
	\item $\U\trianglelefteq_{fe}\V$ if for every set $B$ in $\V$ there is a set $A$ in $\U$ with $A\leq_{fe} B$. 
\end{enumerate}

Definitions (1) and (2) have undesired properties: since $A=\N$ is an element of every ultrafilter, a consequence of definition (1) is that for every ultrafilters $\U$ and $\V$, $\U\trianglelefteq_{fe}\V$ if and only if every element $B$ in $\V$ is maximal respect to $\leq_{fe}$, i.e. if $B$ is thick. But there are no ultrafilters that only contain thick sets. A consequence of definition (2) is that for every $\U,\V$ ultrafilters, $\U\trianglelefteq_{fe}\V$, as $A\leq_{fe}\N$ for every subset $A$ of $\N$.\\We choose definition number three, which has nice properties.

\begin{defn} Given ultrafilters $\U,\V$ in $\bN$, we say that $\U$ is {\bfseries finitely embeddable} in $\V$ $($notation $\U\trianglelefteq_{fe}\V)$ if for every set $B$ in $\V$ there is a set $A$ in $\U$ with $A\leq_{fe} B$. \end{defn}

\begin{prop} Let $\U,\V$ be ultrafilters on $\N$. Then:
\begin{enumerate}
	\item if $\U$ is principal and $\V$ is nonprincipal, $\U$ is finitely embeddable in $\V$ while $\V$ is not finitely embeddable in $\U$;
	\item if $\U$ is the principal ultrafilter generated by $n$ and $\V$ is the principal ultrafilter generated by $m$ then $\U\trianglelefteq_{fe}\V$ if and only if $n\leq m$.
\end{enumerate}
 \end{prop}

\begin{proof} (1) Suppose that $\U$ is generated by the natural number $n$, and let $A$ be an element of $\V$. Since $\V$ is nonprincipal, $A$ is infinite, so in $A$ there is an element $m$ greater than $n$; in particular, $\{n\}\leq_{fe} A$, as $(m-n)+\{n\}\subseteq A$. So $\U\trianglelefteq_{fe}\V$.\\
Conversely, suppose that $\V\trianglelefteq\U$, where $\U$ is the principal ultrafilter generated by $n$. The singleton $\{n\}$ is an element of $\U$ so, by definition, there is an element $A$ in $\V$ with $A\leq\{n\}$ but, as $\V$ is nonprincipal, $A$ is infinite, and an infinite set is not finitely embeddable in a finite set. So we get an absurd, and $\V$ is not finitely embeddable in $\U$.\\
(2) Suppose that $\U\trianglelefteq_{fe}\V$, and consider the element $\{m\}$ in $\V$. By definition, there is an element $A$ in $\U$ and a natural number $a$ with $a+A\subseteq\{m\}$. And this is true if and only if $A$ is a singleton, and the only singleton in $\U$ is $\{n\}$. So $a+n=m$, in particular $n\leq m$.\\
Conversely, suppose that $n\leq m$, and take an element $A$ in $\V$. As $n\leq m$, $m-n$ is a natural number, and $(m-n)+\{n\}=\{m\}$; in particular, $(m-n)+\{n\}\subseteq A$ and, as $\{n\}$ is an element of $\U$, this proves that $\U\trianglelefteq_{fe}\V$. \\ \end{proof}

If we identify as usual each natural number $n$ with the corresponding principal ultrafilter $\mathfrak{U}_{n}$, the above proposition states that:\\

{\bfseries Fact:} $\N$ forms an initial segment of $(\bN, \trianglelefteq_{fe})$.\\

Similarly to the relation of finite embeddability for subsets of $\N$, finite embeddability for ultrafilters is "upward closed with respect to additively invariant existential sentences":

\begin{prop} Let $\phi(x_{1},...,x_{n})$ be an additively invariant formula. If $\U$ is an $E(\phi(x_{1},...,x_{n}))$-ultrafilter, and $\U\trianglelefteq_{fe}\V$, then also $\V$ is an $E(\phi(x_{1},...,x_{n}))$-ultrafilter. \end{prop}

\begin{proof} Let $\U$ be an $E(\phi(x_{1},...,x_{n}))$-ultrafilter and $\V$ an ultrafilter such that $\U\trianglelefteq_{fe}\V$. Let $A$ be any element of $\V$. Since $\U\trianglelefteq_{fe}\V$, there is a set $B$ in $\U$ with $B\leq_{fe} A$. As $\U$ is an $E(\phi(x_{1},...,x_{n}))$-ultrafilter, $B$ satisfies $E(\phi(x_{1},...,x_{n}))$ and, by Proposition 4.3.5, it follows that $A$ satisfies $E(\phi(x_{1},...,x_{n}))$ as well. So every set $A$ in $\V$ satisfies $E(\phi(x_{1},...,x_{n}))$, and $\V$ is an $E(\phi(x_{1},...,x_{n}))$-ultrafilter. \\ \end{proof}

E.g.: if $\U$ is a Van der Waerden ultrafilter, and $\V$ an ultrafilter such that $\U\trianglelefteq_{fe}\V$, then also $\V$ is a Van der Waerden ultrafilter; in fact, for every natural number $n$, $\U$ is an $AP_{n}$-ultrafilter (the sentences $AP_{n}$ have been introduced in Section 3.4, and are additively invariant and existential), so $\V$ is an $AP_{n}$-ultrafilter for every natural number $n$, in particular it is a Van der Waerden's ultrafilter.

\begin{defn} Let $\Delta$ denote the set of ultrafilters $\U$ in $\bN$ such that, for every set $A$ in $\U$, $A$ has positive Banach density:

\begin{center} $\Delta=\{\U\in\bN\mid\forall A\in\U$, $BD(A)>0\}$. \end{center}
\end{defn}

The notation has been chosen keeping the one introduced in $\cite{rif21}$.

\begin{prop} Let $\U,\V$ be ultrafilters in $\bN$ with $\U\in \Delta$. If $\U$ is finitely embeddable in $\V$ then $\V\in \Delta$. \end{prop}

\begin{proof} Let $\U$ be an element of $\Delta$, $\V$ an ultrafilter such that $\U\trianglelefteq_{fe}\V$, and $A$ a set in $\V$. Since $\U\trianglelefteq_{fe}\V$, there is a set $B$ in $\U$ with $B\trianglelefteq_{fe} A$ and, since $BD(B)>0$, by Proposition 4.1.5 it follows that $BD(A)>0$. Since this holds for every set $A$ in $\V$, $\V\in \Delta$. \\ \end{proof}

\begin{cor} There are nonprincipal ultrafilters $\U, \V$ with $\neg(\U\trianglelefteq_{fe}\V)$.\end{cor}

\begin{proof} Let $\U$ be an ultrafilter in $BD_{>0}$ and $\V$ a nonprincipal ultrafilter in $\Delta^{c}$. Then by Proposition 4.4.5 it follows that $\U$ is not finitely embeddable in $\V$.\\\end{proof}

In Section 4.3 we proved that $(\wp(\N),\leq_{fe})$ is a partially pre-ordered set. The question is whether this property can be generalized to $(\bN,\trianglelefteq_{fe})$. The answer is affirmative:

\begin{prop} $(\bN,\trianglelefteq_{fe})$ is a partially pre-ordered set. \end{prop}

\begin{proof} We have to prove that $\trianglelefteq_{fe}$ is transitive and reflexive.\\
Transitive: Suppose that $\U\trianglelefteq_{fe}\V$ and $\V\trianglelefteq_{fe}\W$. Let $A$ be a set in $\W$. Since $\V\trianglelefteq_{fe}\W$, there exists $B$ in $\V$ with $B\leq_{fe} A$ and, since $B\in\V$, there exists $C$ in $\U$ with $C\leq_{fe} B$. By transitivity of the relation $\leq_{fe}$ on $\wp(\N)$ we get that $C\leq_{fe} A$, so $\U\trianglelefteq_{fe}\W$.\\
Reflexive: For every ultrafilter $\U$ and for every set $A$ in $\U$, $A\leq_{fe} A$, so $\U\trianglelefteq_{fe} \U$.\\\end{proof}

Similarly to $\leq_{fe}$, $\trianglelefteq_{fe}$ is not antisymmetric (a simple proof of this fact is given in Corollary 4.4.30). Following the procedure introduced in Section 4.2, we introduce the following definition:

\begin{defn} Two ultrafilters $\U,\V$ in $\bN$ are {\bfseries equivalent respect to finite embeddability} $($notation $\U\equiv_{fe}\V)$ if $\U\trianglelefteq_{fe}\V$ and $\V\trianglelefteq_{fe}\U$. For every ultrafilter $\U$ we denote its equivalence class respect to finite embeddability by $[\U]$:

\begin{center} $[\U]=\{\V\in\bN\mid \U\equiv_{fe}\V\}$.\end{center} \end{defn}

By the general theory of pre-orders it follows that $\equiv_{fe}$ is an equivalence relation on $\bN$ and that $(\bN_{/_{\equiv_{fe}}},\trianglelefteq_{fe})$ is a partial ordered set.\\
Next subsection is dedicated to the study of this partial ordered set: the most important result is that in $(\bN_{/_{\equiv_{fe}}},\trianglelefteq_{fe})$ there is a greatest element.

\subsection{The partial ordered set $(\bN_{/_{\equiv_{fe}}},\trianglelefteq_{fe})$.}

In our opinion, the property of $\trianglelefteq_{fe}$ with the most important consequences is the following:

\begin{prop} If $\U,\V$ are ultrafilters in $\bN$, then both $\U$ and $\V$ are finitely embeddable in $\U\oplus\V$. \end{prop}

\begin{proof} Let $A$ be an set in $\U\oplus\V$; by definition,

\begin{center} $A\in\U\oplus\V\Leftrightarrow\{n\in\N\mid\{m\in\N\mid n+m\in A\}\in\V\}\in\U$. \end{center}

Consider the set

\begin{center} $B=\{n\in\N\mid\{m\in\N\mid n+m\in A\}\in\V\}$ \end{center}

and, for every $n\in B$, consider the set

\begin{center} $C_{n}=\{m\in\N\mid n+m\in A\}$. \end{center}

{\bfseries Claim:} $B\leq_{fe} A$ and $C_{n}\leq_{fe} A$ for every $n\in B$.\\

From the claim it follows that $\U\trianglelefteq_{fe}\U\oplus\V$ and $\V\trianglelefteq_{fe}\U\oplus\V$, and this proves the thesis. To prove the first part of the claim, let $F=\{n_{1},...,n_{k}\}$ be a finite subset of $B$; consider

\begin{center} $C_{F}=\bigcap_{i=1}^{k} C_{n_{i}}=\bigcap_{i=1}^{k} \{m\in\N\mid m+n_{i}\in A\}=\{m\in\N\mid m+F\subseteq A\}$. \end{center}

Since $B$ is in $\U$, $C_{n_{i}}\in\V$ for every index $i\leq k$, so $C_{F}$ is in $\V$. In particular, $C_{F}$ is not empty; if $m$ is a natural number in $C_{F}$, by construction

\begin{center} $m+F\subseteq A$; \end{center}

this proves that for every finite subset $F$ of $B$ there is a natural number (in $C_{F}$) such that $n+F\subseteq A$ so, by definition, $B\leq_{fe} A$.\\
To prove the second statement in the claim we observe that, if $n$ is an element of $B$, by definition $n+C_{n}\subseteq A$; in particular, for every finite subset $F$ of $C_{n}$, $n+F\subseteq A$, and this entails that $C_{n}\leq_{fe} A$ for every natural number $n$ in $B$.\\ \end{proof}

Two important consequences of this result in the study of $(\bN,\trianglelefteq_{fe})$ and $(\bN_{/_{\equiv_{fe}}},\trianglelefteq_{fe})$ are:

\begin{cor} For every ultrafilters $\U,\V$ in $\bN$ there is an ultrafilter $\W\in\bN$ such that $\U$ and $\V$ are finitely embeddable in $\W$. \end{cor}

\begin{proof} This is a straightforward consequence of Proposition 4.4.9: just consider $\W=\U\oplus\V$. \\ \end{proof}

\begin{cor} For every equivalence classes $[\U], [\V]$ in $\bN_{/_{\equiv_{fe}}}$ there is an equivalence class $[\W]$ such that $[\U]$ and $[\V]$ are finitely embeddable in $[\W]$. \end{cor}

\begin{proof} If $\W$ is any ultrafilter in $\bN$ such that $\U,\V\trianglelefteq_{fe}\W$, then $[\U].[\V]\trianglelefteq_{fe}[\W]$. \\ \end{proof}

One other important feature of the finite embeddability is that every chain in $(\bN,\trianglelefteq_{fe})$ has an upper bound. In this context, the chains in $(\bN,\trianglelefteq_{fe})$ are called $fe$-chains and the upper bounds are called $fe$-upper bounds.

\begin{thm} Every $fe$-chain $\langle \U_{i}\mid i\in I\rangle$ of elements in $\bN$ has an $fe$-upper bound $\U$. \end{thm}

\begin{proof} For every element $i\in I$ consider the set 

\begin{center} $G_{i}=\{j\in I\mid j\geq i\}$.\end{center}

The family $\{G_{i}\}_{i\in I}$ has the finite intersection property, so there is an ultrafilter $\V$ on $I$ that extends this family. Observe that the ultrafilter $\V$ is principal and generated by an element $i\in I$ if and only if $i$ is the greatest element of $I$. In this case, the thesis is trivial, since $\U_{i}$ is an $fe$-upper bound for the chain. So we assume that $\V$ is non principal.\\
In this case we observe that for every element $A$ in $\V$ and for every element $i\in I$ there is an element $j$ in $A$ with $j>i$.\\
In fact, suppose that there exists a set $A$ in $\V$ and an element $i$ in $I$ such that $A$ does not contain elements greater than $i$. Then $A\cap G_{i}$ contains at most $i$ and, since $A$ and $G_{i}$ are in $\V$, the intersection $A\cap G_{i}$ is nonempty, so $A\cap G_{i}=\{i\}$, and $\V$ is the principal ultrafilter generated by $i$, while we supposed $\V$ nonprincipal.\\

{\bfseries Claim:} The ultrafilter 

\begin{center} $\U=\V-\lim_{I} \U_{i}\in\bN$ \end{center}

is an $fe$-upper bound for the $fe$-chain $\langle \U_{i}\mid i\in I\rangle$.\\

To prove that $\U_{i}\trianglelefteq_{fe} \U$ for every index $i$, let $A$ be an element of $\U$. By definition,

\begin{center} $A\in \U\Leftrightarrow I_{A}=\{i\in I\mid A\in \U_{i}\} \in \V$. \end{center}

$I_{A}$ is a set in $\V$ so, as we observed, there is an element $j>i$ in $I_{A}$. But 

\begin{center} $j\in I_{A}$ if and only if $A\in \U_{j}$, \end{center}

so, in particular, $A\in \U_{j}$. Now, by definition of $fe$-chain, $\U_{i}\trianglelefteq_{fe} \U_{j}$, and since $A\in \U_{j}$, there exists an element $B$ in $\U_{i}$ with $B\leq_{fe} A$. This proves that, for every index $i\in I$, $\U_{i}\trianglelefteq_{fe}\U$, so $\U$ is an upper bound for the $fe$-chain $\langle \U_{i}\mid i\in I\rangle$.\\ \end{proof}

\begin{cor} For every $fe$-chain $\langle [\U_{i}]\mid i\in I\rangle$ of equivalence classes in $\bN_{/_{\equiv_{fe}}}$ there is an upper bound $[\U]$. \end{cor}

\begin{proof} This corollary is a particular case of a general fact, observed in Section 4.2: if in a partially ordered set $(S,\leq)$ every chain has an upper bound, the same property holds for the quotient set $(S_{/_{\equiv}},\leq)$.\\ \end{proof}

\begin{thm} In $(\bN_{/_{\equiv_{fe}}} ,\trianglelefteq_{fe})$ there is a greatest element. \end{thm}

\begin{proof} $(\bN_{/_{\equiv_{fe}}} ,\trianglelefteq_{fe})$ is a filtered ordered set (by Corollary 4.4.10), and every $fe$-chain admits an upper bound (by Corolary 4.4.13). By Theorem 4.2.8 it follows that in $(\bN_{/_{\equiv_{fe}}} ,\trianglelefteq_{fe})$ there is a greatest element.\\ \end{proof}

\begin{defn} An ultrafilter $\U\in\bN$ is $fe$-maximal if $[\U]$ is the greatest element in $(\bN_{/_{\equiv_{fe}}} ,\trianglelefteq_{fe})$, i.e. if for every ultrafilter $\V$, $\V\trianglelefteq_{fe}\U$. \end{defn}

\begin{thm} Let $\U$ be an ultrafilter on $\N$. The following conditions are equivalents:

\begin{enumerate}
	\item $\U$ is $fe$-maximal;
	\item $\forall \V\in\bN$ $\V\oplus\U\trianglelefteq_{fe}\U$ or $\V\oplus\U\trianglelefteq_{fe}\U$;
	\item $\forall \V\in\bN$ $\U\oplus\V\trianglelefteq_{fe}\U$;
	\item $\forall\V\in\bN$ $\V\oplus\U\trianglelefteq_{fe}\U$;
	\item $\forall\V\in\bN$ $\U\oplus\V\trianglelefteq_{fe}\U$ and $\V\oplus\U\trianglelefteq_{fe}\U$.
\end{enumerate}

\end{thm}

\begin{proof} Observe that, trivially $(5)\Rightarrow (4) \Rightarrow (2)$ and $(5)\Rightarrow (3)\Rightarrow (2)$.\\
$(1)\Rightarrow (5)$ Since $\U$ is $fe$-maximal, for every ultrafilter $\W\in\bN$ we have $\W\trianglelefteq_{fe}\U$. In particular, taking $W=\U\oplus\V$ and $\W=\V\oplus\U$, we get the thesis.\\
$(2)\Rightarrow (1)$ Suppose that $\U$ is not maximal. Then there is an ultrafilter $\V$ such that $\neg(\V\trianglelefteq\U)$. As $\V\trianglelefteq_{fe}\U\oplus\V$ and $\V\trianglelefteq_{fe}\V\oplus\U$, since the finite embeddability is transitive then $\neg(\U\oplus\V\trianglelefteq_{fe}\U)$ and $\neg(\V\oplus\U\trianglelefteq_{fe}\U)$, and this is a contradiction.\\\end{proof}

Important properties of maximal ultrafilters are exposed in Section 4.4.4: crucial concepts in the study of these ultrafilters are the properties of the $fe$-cones

\begin{center} $\mathcal{C}_{fe}(\U)=\{\V\in\bN\mid \U\trianglelefteq_{fe}\U\}$. \end{center}

Next section is dedicated to the study of these sets.

\subsection{The cones $\mathcal{C}_{fe}(\U)$}

\begin{defn} Given an ultrafilter $\U$ in $\bN$, with $\mathcal{C}_{fe}(\U)$ we denote the upper cone of $\U$ in $(\bN,\trianglelefteq_{fe})$, i.e. the set of ultrafilters in $\bN$ in which $\U$ is finitely embeddable:

\begin{center} $\mathcal{C}_{fe}(\U)=\{\V\in \bN\mid \U\trianglelefteq_{fe}\V\}$. \end{center}

\end{defn}

{\bfseries Fact:} For every ultrafilters $\U,\V$ in $\bN$, if $\U\equiv_{fe}\V$ then $\mathcal{C}_{fe}(\U)=\mathcal{C}_{fe}(\V)$.\\

For every ultrafilter $\U$ in $\bN$, the set $\mathcal{C}_{fe}(\U)$ has interesting topological and algebraical properties:

\begin{prop} For every ultrafilter $\U$ in $\bN$, $\mathcal{C}_{fe}(\U)$ is closed in the Stone topology. \end{prop}

\begin{proof} To prove that $\mathcal{C}_{fe}(\U)$ is closed we show that, given a sequence $\langle \U_{i}\mid i\in\ I\rangle$ of elements in $\mathcal{C}_{fe}(\U)$ and an ultrafilter $\V$ on $I$, the limit

\begin{center} $\W=\V-\lim_{I}\U_{i}$ \end{center}

belongs to $\mathcal{C}_{fe}(\U)$. This fact, as proved in Section 1.1.6, entails that $\mathcal{C}_{fe}(\U)$ is closed.\\
Let $A$ be an element of $\W$. By definition of limit of ultrafilters, the set of indexes $i\in I$ such that $A\in \U_{i}$ is in $\V$; in particular, it is nonempty. Let $i$ be an index such that $A\in \U_{i}$. Since $\U\trianglelefteq_{fe}\U_{i}$ (because $\U_{i}\in\mathcal{C}_{fe}(\U)$), there is a set $B$ in $\U$ with $B\leq_{fe} A$; this proves that, for every set $A$ in $\W$, there is a set $B$ in $\U$ such that $B\leq_{fe} A$, so $\U\trianglelefteq_{fe}\W$, and this entails that $\W\in\mathcal{C}_{fe}(\U)$.\\ \end{proof}

The above theorem characterizes topologically the set $\mathcal{C}_{fe}(\U)$. The next proposition shows an important algebraical feature of $\mathcal{C}_{fe}(\U)$:

\begin{prop} For every ultrafilter $\U$ in $\bN$, the set $\mathcal{C}_{fe}(\U)$ is a two-sided ideal of $(\bN,\oplus)$. \end{prop}

\begin{proof} Let $\V$ be an ultrafilter in $\mathcal{C}_{fe}(\U)$, and $\W$ an ultrafilter in $\bN$. By Proposition 4.4.9 it follows that $\V\trianglelefteq_{fe}\W\oplus\V$ and $\V\trianglelefteq_{fe}\V\oplus\W$. Then, as $\U\trianglelefteq_{fe}\V$, by transitivity it follows that $\U\trianglelefteq_{fe} \W\oplus \V$ and $\U\trianglelefteq_{fe}\V\oplus\W$; in particular, $\W\oplus\V, \V\oplus\W\in \mathcal{C}_{fe}(\U)$, so $\mathcal{C}_{fe}(\U)$ is a bilateral ideal in $(\bN,\oplus)$.\\ \end{proof}

\begin{cor} For every ultrafilter $\U$, the set $\mathcal{C}_{fe}(\U)$ contains the close sets $\overline{\{\U\oplus\V\mid \V\in\bN\}}$ and $\{\V\oplus\U\mid \V\in\bN\}$. \end{cor}

\begin{proof} From Proposition 4.4.19, since $\mathcal{C}_{fe}(\U)$ is a bilateral ideal in $(\bN,\oplus)$ that contains $\U$, it follows that 

\begin{center} $\{\U\oplus\V\mid \V\in\bN\}\subseteq \mathcal{C}_{fe}(\U)$ and $\{\V\oplus\U\mid \V\in\bN\}\subseteq\mathcal{C}_{fe}(\U)$.\end{center}

As $\mathcal{C}_{fe}(\U)$ is closed, the thesis follows taking the closures on both sides of the inclusions, and observing that $\{\V\oplus\U\mid \V\in\bN\}$ is closed. \\ \end{proof}

The result exposed in Corollary 4.4.20 can be improved to characterize $\mathcal{C}_{fe}(\U)$. To this end we need the following two lemmas:

\begin{lem} Let $B$ be a subset of $\N$, and $\U$ an element of $\bN$. The following two conditions are equivalent:
\begin{enumerate}
	\item there is an element $A$ in $\U$ such that $A$ is finitely embeddable in $B$;
	\item there is an ultrafilter $\V$ on $\N$ such that $B$ in $\U\oplus\V$
\end{enumerate}
\end{lem}

\begin{proof} $(1)\Rightarrow (2)$ This follows by Proposition 4.3.8, that states that if $A\leq_{fe} B$ then there is an ultrafilter $\V$ such that, for every ultrafilter $\W$, if $A\in \W$ then $B\in\W\oplus\V$, and the conclusion follows by choosing $\W=\U$.\\
$(2)\Rightarrow(1)$ This is a consequence of Proposition 4.4.9, since if $B\in\U\oplus\V$, as $\U\trianglelefteq_{fe}\U\oplus\V$ there is an element $A$ in $\U$ such that $A\leq_{fe} B$. \\ \end{proof}

\begin{lem} Let $\U,\V$ be ultrafilters in $\bN$. The following two conditions are equivalent:
\begin{enumerate}
	\item $\U$ is finitely embeddable in $\V$;
	\item for every element $B$ of $\V$ there is an ultrafilter $\W$ in $\bN$ such that $B\in\U\oplus\W$.
\end{enumerate}
\end{lem}

\begin{proof} $(1)\Rightarrow(2)$ Suppose that $\U\trianglelefteq_{fe}\V$, and take any element $B$ of $\V$. By definition of $\trianglelefteq_{fe}$, there is an element $A$ of $\U$ with $A\leq_{fe} B$; by Lemma 4.4.21, this entails that there is an ultrafilter $\W$ in $\bN$ with $B\in\U\oplus\W$.\\
$(1)\Rightarrow(2)$ Let $B$ be an element of $\V$. By hypothesis, there is an ultrafilter $\W$ such that $B\in\U\oplus\W$ and, as by Proposition 4.4.9 $\U\trianglelefteq_{fe}\U\oplus\V$, there exists $A$ in $\U$ with $A\leq_{fe} B$. \\ \end{proof}

Given the above two lemmas, we can characterize the cones $\mathcal{C}_{fe}(\U)$:

\begin{thm} For every ultrafilter $\U$ in $\bN$, $\mathcal{C}_{fe}(\U)$ is the closure in the Stone topology of the set $\{\U\oplus\V\mid \V\in\bN\}$:

\begin{center} $\mathcal{C}_{fe}(\U)=\overline{\{\U\oplus\V\mid \V\in\bN\}}$. \end{center} \end{thm}

\begin{proof} The inclusion $\overline{\{\U\oplus\V\mid \V\in\bN\}}\subseteq \mathcal{C}_{fe}(\U)$ has been proved in Corollary 4.4.20.\\
For the reverse inclusion, in Lemma 4.4.22 it has been proved that, if an ultrafilter $\W$ is in $\mathcal{C}_{fe}(\U)$, and $A$ is an element of $\W$, then there is an ultrafilter $\V$ with $A\in\U\oplus\V$. Reformulating this observation from a topological point of view, this proves that for every element $A$ in $\V$ there is an ultrafilter $\uZ$ in $\{\U\oplus\V\mid \V\in\bN\}$ with $A\in\uZ$. This property holds, in the Stone topology, if and only if $\W\in\overline{\{\U\oplus\V\mid \V\in\bN\}}$, so 
$\mathcal{C}_{fe}(\U)\subseteq\overline{\{\U\oplus\V\mid \V\in\bN\}}$. \\ \end{proof}

Recall that $\mathcal{C}_{fe}={\{\U\oplus\V\mid\V\in\bN\}}$ is a bilateral ideal in $(\bN,\oplus)$. So it is readily seen that\\

{\bfseries Fact:} $\mathcal{C}_{fe}(\U)$ is the minimal closed right ideal in $(\bN,\oplus)$ containing $\U$.\\

In next section we show how the sets $\mathcal{C}_{fe}(\U)$ can be used to characterize the set $\mathcal{M}_{fe}$ of ultrafilters such that their class of $\equiv_{fe}$-equivalence is the greatest element of $(\bN_{/_{\equiv_{fe}}},\trianglelefteq_{fe})$. Surprisingly, this set is related to the minimal bilateral ideal $K(\bN,\oplus)$ of $(\bN,\oplus)$.

\subsection{$\mathcal{M}_{fe}$ is the closure of the minimal bilater ideal of $(\bN,\oplus)$}

\begin{defn} If $M$ is the greatest element in $(\bN_{/_{\equiv_{fe}}},\trianglelefteq_{fe})$, with $\mathcal{M}_{fe}$ we denote its equivalence class

\begin{center} $\mathcal{M}_{fe}=\{\U\in\bN\mid [\U]=M\}$. \end{center}

The ultrafilters in $\mathcal{M}_{fe}$ are called {\bfseries maximal}.

\end{defn}

Observe that an ultrafilter $\U$ in $\bN$ is maximal if and only if, for every ultrafilter $\V$ in $\bN$, $\V\trianglelefteq_{fe}\U$.\\
In this section we prove that 

\begin{center} $\mathcal{M}_{fe}=\overline{K(\bN,\oplus)}$. \end{center}

The results needed to prove this equality have been already exposed in this chapter, except for these two lemmas:

\begin{lem} For every right ideal $R$ of $(\bN,\oplus)$, $\overline{R}$ is a right ideal of $(\bN,\oplus)$. \end{lem}

\begin{proof} Let $\U$ be an element of $\overline{R}$, and $\V$ an ultrafilter in $\bN$. As a consequence of a well-known characterization of closed sets in the Stone topology, proving that $\U\oplus\V$ is in $\overline{R}$ is equivalent to prove that, for every element $A\in\U\oplus\V$, there is an ultrafilter $\uZ\in R$ with $A\in\uZ$.\\
Let $A$ be an element of $\U\oplus\V$. By definition, the set 

\begin{center} $B=\{n\in\N\mid\{m\in\N\mid n+m\in A\}\in\V\}$ \end{center}

is an element of $\U$. Since $\U$ is in $\overline{R}$ and $B$ is in $\U$, there is an ultrafilter $\W$ in $R$ such that $B\in\W$. In particular, since $B\in\W$, $A\in\W\oplus\V$, and $\W\oplus\V$ is an element of $R$ since $R$ is a right ideal and $\W$ is in $R$. This proves that $\overline{R}$ is a right ideal.\\ \end{proof}

\begin{lem} For every maximal ultrafilter $\U$ in $\bN$, $\mathcal{C}_{fe}(\U)=\mathcal{M}_{fe}$. \end{lem}

\begin{proof} The inclusion $\mathcal{M}_{fe}\subseteq \mathcal{C}_{fe}(\U)$ holds for every ultrafilter $\U$ in $\bN$.\\ Conversely, if $\U$ is maximal, the reverse inclusion $\mathcal{C}_{fe}(\U)\subseteq\mathcal{M}_{fe}$ holds as well because, if $\V$ is an element in $\mathcal{C}_{fe}(\U)$ and $\W$ is any ultrafilter then, as $\W\trianglelefteq_{fe}\U$ (by maximality of $\U$) and $\U\trianglelefteq_{fe}\V$ (by definition of $\mathcal{C}_{fe}(\U)$), $\W\trianglelefteq_{fe}\V$, so $\V$ is maximal, and $\mathcal{C}_{fe}(\U)\subseteq\mathcal{M}_{fe}$.\\ \end{proof}

\begin{thm} Let $K(\bN,\oplus)$ denote the smallest bilateral ideal of $(\bN,\oplus)$. Then, for every minimal right ideal $R$,

\begin{center} $\mathcal{M}_{fe}=\overline{R}$;\end{center}

in particular,

\begin{center} $\mathcal{M}_{fe}=\overline{K(\bN,\oplus)}.$ \end{center} \end{thm}

\begin{proof} Let $R$ be a right ideal included in $K(\bN,\oplus$). $K(\bN,\oplus)$ is the smallest bilateral ideal of $(\bN,\oplus)$ and, as for every ultrafilter $\U$ in $\bN$, $\mathcal{C}_{fe}(\U)=\mathcal{M}_{fe}$ is a bilateral ideal in $(\bN,\oplus)$, it follows that

\begin{center} $R\subset K(\bN,\oplus)\subseteq\mathcal{C}_{fe}(\U)$ \end{center}

for every ultrafilter $\U\in\bN$. In particular, if $\U$ is maximal, since by Lemma 4.4.26 $\mathcal{M}_{fe}=\mathcal{C}_{fe}(\U)$ it follows that

\begin{center} $(\dagger)$ $R\subseteq K(\bN,\oplus)\subseteq \mathcal{M}_{fe}$. \end{center}

As $\mathcal{M}_{fe}$ is closed (since $\mathcal{M}_{fe}=\mathcal{C}_{fe}(\U)$ for every maximal ultrafilter $\U$, and $\mathcal{C}_{fe}(\U)$ is closed), if we take the closures in $(\dagger)$ it follows that

\begin{center} $(1)$ $\overline{R}\subseteq\overline{K(\bN,\oplus)}\subseteq \mathcal{M}_{fe}$. \end{center}

For the reverse inclusions, let $\U$ be an ultrafilter in $R$. Since $R$ is a right ideal, by Lemma 4.4.25 it follows that $\overline{R}$ is a closed right ideal containing $\U$.\\
Since we already proved that $\overline{R}$ is included in $\mathcal{M}_{fe}$, it follows that $\U$ is in $\mathcal{M}_{fe}$ and, as we proved in Lemma 4.4.26, $\mathcal{M}_{fe}=\mathcal{C}_{fe}(\U)$, which is the minimal closed right ideal containing $\U$. So

\begin{center} $(2)$ $\mathcal{M}_{fe}\subseteq \overline{R}$. \end{center}

Considering (1) and (2), it follows that

\begin{center} $\mathcal{M}_{fe}=\overline{R}=\overline{K(\bN,\oplus)}$ \end{center}

and this concludes the proof. \\ \end{proof}

This result has three interesting corollaries. The first one is a known fact in the topology of $\bN$:

\begin{cor} If $R$ is a right minimal ideal in $(\bN,\oplus)$ then \begin{center}$\overline{R}=\overline{K(\bN,\oplus)}.$\end{center} \end{cor} 

\begin{proof} Trivial, since by Theorem 4.4.27 it follows that $\overline{R}=\mathcal{M}_{fe}$ and $\overline{K(\bN,\oplus)}=\mathcal{M}_{fe}$.\\ \end{proof}

\begin{cor} An ultrafilter $\U$ is maximal if and only if every element $A$ of $\U$ is piecewise syndetic. \end{cor}

\begin{proof} This follows from this well-known characterization of $K(\bN,\oplus)$ (that we exposed in Chapter One): an ultrafilter $\U$ is in $K(\bN,\oplus)$ if and only if every element $A$ of $\U$ is piecewise syndetic  . \\ \end{proof}

\begin{cor} The relation $\trianglelefteq_{fe}$ is not antysimmetric on $\bN$. \end{cor}

\begin{proof} We just need to observe that, whenever $\U,\V$ are two different ultrafilters in $K(\bN,\oplus)$, by Theorem 4.4.27 it follows that $\U\trianglelefteq_{fe}\V$ and $\V\trianglelefteq_{fe}\U$. \\ \end{proof}

\section{A Direct Nonstandard Proof of $\mathcal{M}_{fe}=\overline{K(\bN,\oplus)}$}

In this section we give a nonstandard proof of the equality $\mathcal{M}_{fe}=\overline{K(\bN,\oplus)}$ by showing, in an alternative way, that there is an intimate connection between maximal ultrafilters and piecewise syndetic sets:

\begin{thm} Let $\U$ be an ultrafilter in $\bN$. The following conditions are equivalent:

\begin{enumerate}
	\item $\U$ is maximal in $(\bN,\trianglelefteq_{fe})$;
	\item $\forall A\in\U$, $A$ is piecewise syndetic.
\end{enumerate}
\end{thm}

\begin{proof} $(1)\Rightarrow (2)$ Suppose that $\U$ is maximal, and let $\V$ be an ultrafilter such that every set $B$ in $\V$ is piecewise syndetic (the ultrafilter $\V$ exists since the family of piecewise syndetic sets is partition regular). Let $A$ be any set in $\U$. Since $\U$ is maximal, $\V\trianglelefteq_{fe}\U$, so there is a set $B$ in $\V$ with $B\leq_{fe} A$. As $B$ is piecewise syndetic, by Proposition 4.3.7 it follows that $A$ is piecewise syndetic. Since this holds for every set $A$ in $\U$, we have the thesis.\\
$(2)\Rightarrow (1)$ Let $A$ be a set in $\U$, and $\V$ an ultrafilter in $\bN$. Since $A$ is piecewise syndetic, there is a natural number $n$ such that 

\begin{center} $T=\bigcup_{i=1}^{n} A+i$ \end{center} 

is thick. By transfer it follows that there are hypernatural numbers $\alpha\in$$^{*}\N$ and $\eta\in$$^{*}\N\setminus\N$ such that the interval $[\alpha,\alpha+\eta]$ is included in $^{*}T$. In particular, since $\eta$ is infinite, $\alpha+\N\subseteq$$^{*}T$.\\
For every $i\leq n$ consider 

\begin{center} $B_{i}=\{n\in\N\mid \alpha+n\in$$^{*}A+i\}$. \end{center}

Since $\bigcup_{i=1}^{n} B_{i}=\N$, there is an index $i$ such that $B_{i}\in\V$.\\

{\bfseries Claim:} $B_{i}\leq_{fe} A$.\\

In fact, by construction $\alpha+B_{i}\subseteq$$^{*}A+i$, so

\begin{center} $(\alpha-i)+B_{i}\subseteq$$^{*}A$. \end{center} 
By Proposition 4.1.4, this entails that $B_{i}\leq_{fe} A$, and this proves that $\V\trianglelefteq_{fe}\U$ for every ultrafilter $\V$. Hence $\U$ is maximal.\\\end{proof}

\section{Finite Mappability of Subsets of $\N$}

\subsection{The generalization of finite embeddability}

Our aim in this section is to generalize the relation of finite embeddability for subsets of $\N$. Our idea of generalization is grounded on the following observation: the core in the definition of finite embeddability are the translations. In fact, if $\mathbb{T}$ denotes the set of translations in $\N$:

\begin{center} $\mathbb{T}=\{f_{n}:\N\rightarrow\N\mid n\in\N$ and $f_{n}(m)=n+m$ for every $m$ in $\N\}$ \end{center}

then the definition of finite embeddability can be reformulated in this way:\\

{\bfseries Fact:} Given subsets $A,B$ of $\N$, $A$ is finitely embeddable in $B$ if and only if for every finite subset $F$ of $A$ there is a function $f_{n}$ in $\mathbb{T}$ such that $f_{n}(F)\subseteq B$.\\

We generalize finite embeddability by substituting the set of translations $\mathbb{T}$ with other sets of functions $\mathcal{F}\in$ $\mathtt{Fun}$($\N,\N$) (we recall that, whenever $A,B$ are sets, with $\mathtt{Fun}$($A,B$) we denote the set of functions with domain $A$ and range included in $B$):

\begin{defn} Let $\mathcal{F}$ be a subset of $\mathtt{Fun}$$(\N,\N)$, and $A, B$ subsets of $\N$. $A$ is {\bfseries $\mathcal{F}$-finitely mappable in $B$} $($notation $A\leq_{\mathcal{F}} B)$ if and only if for every finite subset $F$ of $A$ there is a function $f$ in $\mathcal{F}$ such that $f(F)\subseteq B$. \end{defn}

From now on, to be coherent with the notation just introduced, we denote the relation of finite embeddability as $\leq_{\mathbb{T}}$, where $\mathbb{T}$ is the above set of translations.\\
We summarize some basic observations in the following proposition:

\begin{prop} Let $A,A_{1},A_{2},B,B_{1},B_{2}$ be subsets of $\N$ and $\mathcal{F}$, $\mathcal{F}_{1}$, $\mathcal{F}_{2}$,...,$\mathcal{F}_{k}$ subsets of $\mathtt{Fun}$$(\N,\N)$. The following properties hold:

\begin{enumerate}
	\item If $\mathcal{F}=\{f\}$ then $A\leq_{\{f\}} B$ if and only if $f(A)\subseteq B$;
	\item If $A\leq_{\mathcal{F}_{1}\cup \mathcal{F}_{2}} B$ then $A\leq_{\mathcal{F}_{1}} B$ or $A\leq_{\mathcal{F}_{2}} B$;
	\item If $A\leq_{\mathcal{F}_{1}\cup....\cup\mathcal{F}_{k}} B$ then there is an index $i\leq k$ such that $A\leq_{\mathcal{F}_{i}} B$;
	\item If $\mathcal{F}=\{f_{1},...,f_{k}\}$ then $A\leq_{\{f_{1},...,f_{k}\}} B$ if and only if there is an index $i\leq k$ such that $f_{i}(A)=B$;
	\item If $\mathcal{F}_{1}\subseteq \mathcal{F}_{2}$ and $A\leq_{\mathcal{F}_{1}} B$ then $A\leq_{\mathcal{F}_{2}} B$;
	\item If $A_{1}\subseteq A_{2}$ and $A_{2}\leq_{\mathcal{F}} B$ then $A_{1}\leq_{\mathcal{F}} B$;
	\item If $B_{1}\subseteq B_{2}$ and $A\leq_{\mathcal{F}} B_{1}$ then $A\leq_{\mathcal{F}} B_{2}$;
	\item If $A_{1}\leq_{\mathcal{F}} B$ and $A_{2}\leq_{\mathcal{F}} B$ then $A_{1}\cap A_{2}\leq_{\mathcal{F}} B$;
	\item If $A\leq_{\mathcal{F}} B_{1}$ and $A\leq_{\mathcal{F}} B_{2}$ then $A\leq_{\mathcal{F}} B_{1}\cup B_{2}$.
\end{enumerate}

\end{prop}

\begin{proof} 

1) Suppose that $A\leq_{\{f\}} B$; by definition, whenever $F$ is a finite subset of $A$ the set $f(F)$ is included in $B$. This entails that $f(A)\subseteq B$.\\
Conversely, if $f(A)\subseteq B$ then, for every finite subset $F$ of $A$, $f(F)\subseteq B$; in particular, $A\leq_{\{f\}} B$.\\
2) Suppose that $A\leq_{\mathcal{F}_{1}\cup\mathcal{F}_{2}} B$. For every natural number $n$ let $A_{n}$ be the set

\begin{center} $A_{n}=A\cap [0,n]$. \end{center}

There are two possibilities: 
\begin{enumerate}
	\item for arbitrarily large natural numbers $n$ there is a function $f_{n}\in \mathcal{F}_{1}$ such that $f_{n}(A_{n})\subseteq B$;
	\item there is a natural number $n$ such that, for every $m\geq n$, for every $f\in \mathcal{F}_{1}$, $f(A_{m})$ is not included in $B$.
\end{enumerate}

In case (1), we claim that $A\leq_{\mathcal{F}_{1}} B$: in fact, for every finite subset $F$ of $A$ there is a function $f\in \mathcal{F}_{1}$ with $f(F)\subseteq B$ since, if $m=\max F$, and $n$ is a natural number greater than $m$ such that there is a function $f_{n}$ in $\mathcal{F}_{1}$ with $f_{n}(A_{n})\subseteq B$, since $F\subseteq A_{n}$ then $f_{n}(F)\subseteq B$; in particular, $A\leq_{\mathcal{F}_{1}} B$.\\
In case (2), we claim that $A\leq_{\mathcal{F}_{2}} B$: let $N$ be the natural number such that, for every $m\geq N$, for every function $f\in \mathcal{F}_{1}$, $f(A_{m})$ is not included in $B$. Since, by hypothesis, $A\leq_{\mathcal{F}_{1}\cup \mathcal{F}_{2}} B$, for every natural number $m\geq N$ there is a function $g_{m}$ in $\mathcal{F}_{2}$ such that $g_{m}(A_{m})\subseteq B$. Then, for every finite subset $F$ of $A$, if $M=\max\{\max(F),N\}$, as $N\leq M$ there is a function $g$ in $\mathcal{F}_{2}$ with $g(A_{M})\subseteq B$; in particular, as $F\subseteq A_{M}$, $g(F)\subseteq B$, and $A\leq_{\mathcal{F}_{2}} B$.\\
3) It follows by induction from (2).\\
4) By (3), there is an index $i\leq k$ such that $A\leq_{f_{i}} B$ and by (1) it follows that $f_{i}(A)\subseteq B$.\\
5) Suppose that $A\leq_{\mathcal{F}_{1}} B$. By definition of $\mathcal{F}_{1}$-finite mappability, for every finite subset $F$ of $A$ there is a function $f$ in $\mathcal{F}_{1}$ such that $f(F)\subseteq B$. In particular, as $\mathcal{F}_{1}\subseteq \mathcal{F}_{2}$, for every finite subset $F$ of $A$ there is a function $f$ in $\mathcal{F}_{2}$ such that $f(F)\subseteq B$, so $A\leq_{\mathcal{F}_{2}} B$.\\
6) We have only to observe that, as $A_{1}\subseteq A_{2}$, every finite subset of $A_{1}$ is also a finite subset of $A_{2}$.\\
7) We have only to observe that, for every finite subset $F$ of $A$, for every function $f\in \mathcal{F}$, if $f(F)\subseteq B_{1}$ then $f(F)\subseteq B_{2}$.\\
8) This is a particular case of the result (6).\\
9) This is a particular case of the result (7). \\\end{proof}

Similarly to the relation of finite embeddability, the relation of finite mappability can be reformulated in a nonstandard fashion:

\begin{prop} Let $A,B$ be subsets of $\N$, and $\mathcal{F}$ a subset of $\mathtt{Fun}$$(\N,\N)$. The following two conditions are equivalent:
\begin{enumerate}
	\item $A\leq_{\mathcal{F}}B$;
	\item there is a function $\varphi$ in $^{*}\mathcal{F}$ such that $\varphi(A)\subseteq$$^{*}B$.
\end{enumerate}
\end{prop}

\begin{proof} $(1)\Rightarrow(2)$ Suppose that $A\leq_{\mathcal{F}} B$ and, for every finite subset $F$ of $A$, consider the set

\begin{center} $S_{F}=\{f\in \mathcal{F}\mid f(F)\subseteq B\}$ \end{center}

As $A\leq_{\mathcal{F}} B$, for every finite subset $F$ of $A$ $S_{F}\neq\emptyset$, and the family $\{S_{F}\}_{F\in\wp_{fin}(A)}$ has the finite intersection property, since 

\begin{center} $S_{F_{1}}\cap S_{F_{2}}=S_{F_{1}\cup F_{2}}$. \end{center}

By $\mathfrak{c}^{+}$-enlarging property, $\bigcap_{F\in\wp_{fin}(A)}$$^{*}S_{F}\neq\emptyset$; let $\varphi$ be a function in this intersection. By construction and transfer, $\varphi$ has the following two properties:
\begin{enumerate}
	\item $\varphi\in$$^{*}\mathcal{F}$;
	\item $\varphi($$^{*}F)\subseteq$$^{*}B$ for every finite subset $F$ of $A$.
\end{enumerate}
As $^{*}F=F$ for every finite subset of $\N$, by condition (2) it follows that $\varphi(A)\subseteq$$^{*}B$.\\
$(2)\Rightarrow(1)$ Let $\varphi$ be a function in $^{*}\mathcal{F}$ such that $\varphi(A)\subseteq$$^{*}B$, and suppose that $A$ is not $\mathcal{F}$-finitely mappable in $B$. As a consequence, there is a finite subset $F$ of $A$ such that, for every function $g\in\mathcal{F}$, $g(F)$ is not included in $B$. By transfer it follows that, for every function $g\in$$^{*}\mathcal{F}$, $g($$^{*}F)$ is not included in $^{*}B$, and this is absurd since, as we observed before, $^{*}F=F$, and $\varphi(F)\subseteq \varphi(A)\subseteq $$^{*}B$. By assuming that $A$ is not $\mathcal{F}$-finitely mappable in $B$ it follows an absurd, so $A\leq_{\mathcal{F}} B$. \\ \end{proof}

From now on, a central tool in the study of the relations $\leq_{\mathcal{F}}$ are the hyperextensions $^{*}\varphi$ of internal functions $\varphi$ in $\mathtt{Fun}$($^{*}\N,$$^{*}\N$); in particular, in this context we need this particular property:

\begin{prop} Let $\varphi$ be an internal function in $\mathtt{Fun}$$($$^{*}\N,$$^{*}\N)$, and $\alpha,\beta$ be hypernatural numbers in $^{*}\N$. If $\alpha\sim_{u}\beta$ then $($$^{*}\varphi)(\alpha)\sim_{u}($$^{*}\varphi)(\beta)$. \end{prop}

\begin{proof} To prove the thesis we have to show that, if $\alpha\sim_{u}\beta$, then for every subset $A$ of $\N$ $($$^{*}\varphi)(\alpha)\in$$^{**}A$ if and only if $($$^{*}\varphi)(\beta)\in$$^{**}A$.\\
Observe that, by transfer, $^{*}\{n\in\N\mid \varphi(n)\in$$^{*}A\}=\{\eta\in$$^{*}\N\mid($$^{*}\varphi)(\eta)\in$$^{**}A\}$. As $\{n\in\N\mid \varphi(n)\in$$^{*}A\}$ is a subset of $\N$, and $\alpha\sim_{u}\beta$, the following equivalence holds:

\begin{center} $\alpha\in$$^{*}\{n\in\N\mid \varphi(n)\in$$^{*}A\}$ if and only if $\beta\in$$^{*}\{n\in\N\mid \varphi(n)\in$$^{*}A\}$. \end{center}

From this observation, we obtain the chain of equivalences:

\begin{center} ($^{*}\varphi)(\alpha)\in$$^{**}A\Leftrightarrow\alpha\in\{\eta\in$$^{*}\N\mid($$^{*}\varphi)(\eta)\in$$^{**}A\}=$$^{*}\{n\in\N\mid \varphi(n)\in$$^{*}A\}\Leftrightarrow$\\\vspace{0.3cm}$\Leftrightarrow\beta\in$$^{*}\{n\in\N\mid \varphi(n)\in$$^{*}A\}=\{\eta\in$$^{*}\N\mid($$^{*}\varphi)(\eta)\in$$^{**}A\}\Leftrightarrow($$^{*}\varphi)(\beta)\in$$^{**}A$; \end{center}

and hence the thesis. \\ \end{proof}

The above proposition can be restated in this way: if $\alpha\sim_{u}\beta$ are hypernatural numbers in $^{*}\N$, and $\varphi:$$^{*}\N\rightarrow$$^{*}\N$ is an internal function, then

\begin{center} ($\dagger$) if $\mathfrak{U}_{\alpha}=\mathfrak{U}_{\beta}$ then $\mathfrak{U}_{(^{*}\varphi)(\alpha)}=\mathfrak{U}_{(^{*}\varphi)(\beta)}$. \end{center}

With this observation, we can prove the following important result:

\begin{prop} Given subsets $A,B$ of $\N$ and a subset $\mathcal{F}$ of $\mathtt{Fun}$$(\N,\N)$, the following two conditions are equivalent:
\begin{enumerate}
	\item $A\leq_{\mathcal{F}}B$;
	\item there is a function $\varphi\in$$^{*}\mathcal{F}$ such that, for every ultrafilter $\U$ in $\bN$ with $A\in\U$, for every generator $\alpha\in$$^{*}\N$ of $\U$, $B\in\mathfrak{U}_{( ^{*}\varphi)(\alpha)}$.
\end{enumerate}
\end{prop}

\begin{proof} We already proved that $A\leq_{\mathcal{F}} B$ if and only if there is a function $\varphi\in$$^{*}\mathcal{F}$ such that $\varphi(A)\subseteq$$^{*}B$. By transfer, $\varphi(A)\subseteq$$^{*}B$ if and only if $^{*}\varphi($$^{*}A)\subseteq$$^{**}B$. So we have the following chain of equivalences:

\begin{center} $A\leq_{\mathcal{F}}B\Leftrightarrow(\exists\varphi\in$$^{*}\mathcal{F})($$^{*}\varphi($$^{*}A)\subseteq$$^{**}B)$\\\vspace{0.3cm}$\Leftrightarrow(\exists\varphi\in$$^{*}\mathcal{F})(\forall\alpha\in$$^{*}\N)(\alpha\in$$^{*}A\Rightarrow$$^{*}\varphi(\alpha)\in$$^{**}B)\Leftrightarrow$\\\vspace{0.3cm}$\Leftrightarrow (\exists \varphi\in$$^{*}\mathcal{F})(\forall\alpha\in$$^{*}\N)$$(A\in\mathfrak{U}_{\alpha}\Rightarrow B\in\mathfrak{U}_{^{*}\varphi(\alpha)})$,\end{center}

and this proves the thesis.\\ \end{proof}

\subsection{Well-structured sets of functions} 

The question that we want to answer is wheter the relation $\leq_{\mathcal{F}}$ is an order or not. We already know that, in general, $\leq_{\mathcal{F}}$ is not an order, as we have proved in Section 4.3 that $\leq_{\mathbb{T}}$ is not antisymmetric.\\
The problem is that, in general, $\leq_{\mathcal{F}}$ is not even transitive nor reflexive. In fact:

\begin{prop} For every subset $\mathcal{F}$ of $\mathtt{Fun}$$(\N,\N)$, the relation $\leq_{\mathcal{F}}$ is 
\begin{enumerate}
	\item transitive if and only if for every finite subset $F$ of $\N$, for every functions $f,g$ in $\mathcal{F}$ there is a function $h$ in $\mathcal{F}$ such that $h(F)\subseteq (g\circ f) (F)$;
	\item reflexive if and only if for every finite subset $F$ of $\N$ there is a function $f$ in $\mathcal{F}$ such that $f(F)\subseteq F$. 
\end{enumerate}
 \end{prop}

\begin{proof} (1) Suppose that $\leq_{\mathcal{F}}$ is transitive, and let $F$ be a finite subset of $\N$ and $f,g$ functions in $\mathcal{F}$. By definition of $\mathcal{F}$-finite mappability: 

\begin{center} $F\leq_{\mathcal{F}} f(F)$ and $f(F)\leq_{\mathcal{F}} g(f(F))$ \end{center}

and, since $\leq_{\mathcal{F}}$ is transitive by hypothesis, $F\leq_{\mathcal{F}} (g\circ f) (F)$.  As $F$ is finite, this happens if and only if there is a function $h$ in $\mathcal{F}$ with $h(F)\subseteq (g\circ f)(F)$.\\
Conversely, let $A,B,C$ be subsets of $\N$ with $A\leq_{\mathcal{F}} B$ and $B\leq_{\mathcal{F}} C$, and let $F$ be a finite subset of $A$. Since $A\leq_{\mathcal{F}} B$, there is a function $f$ in $\mathcal{F}$ such that $f(F)\subseteq B$ and, since $B\leq_{\mathcal{F}} C$, there is a function $g$ in $\mathcal{F}$ with $g(f(F))\subseteq C$. By hypothesis, there is a function $h$ in $\mathcal{F}$ such that $h(F)\subseteq g(f(F))\subseteq C$. This proves that, for every finite subset $F$ of $A$, there is a function $h$ in $\mathcal{F}$ such that $h(F)\subseteq C$ and so, by definition, $A\leq_{\mathcal{F}} C$.\\
(2) Suppose that $\leq_{\mathcal{F}}$ is reflexive; for every $F$ finite subset of $\N$, as $F\leq_{\mathcal{F}} F$, there is a function $f$ in $\mathcal{F}$ with $f(F)\subseteq F$.\\
Conversely, let $A$ be a subset of $\N$, and let $F$ be a finite subset of $A$. By hypothesis, there is a function $f$ in $\mathcal{F}$ such that $f(F)\subseteq F\subseteq A$: by definition of $\mathcal{F}$-finite mappability, this proves that $A\leq_{\mathcal{F}} A$.\\ \end{proof}

We present two trivial corollaries of Proposition 4.6.6:

\begin{cor} If $\mathcal{F}$ is a set of functions closed under composition then $\leq_{\mathcal{F}}$ is transitive; and if the identity map $i$ is in $\mathcal{F}$, then $\leq_{\mathcal{F}}$ is reflexive. \end{cor}

\begin{cor} If $\leq_{\mathcal{F}_{1}}$ is reflexive, and $\mathcal{F}_{1}\subseteq\mathcal{F}_{2}$, then also $\leq_{\mathcal{F}_{2}}$ is reflexive. \end{cor}

Since we are mainly interested in relations of $\mathcal{F}$-finite mappability that are pre-orders, i.e. satisfying transitivity and reflexivity, we introduce the following notion:

\begin{defn} A set of functions $\mathcal{F}\subseteq \mathtt{Fun}(\N,\N)$ is {\bfseries well-stuctured} if the relation of $\mathcal{F}$-finite mappability is a pre-order, i.e. if it satisfies both the transitivity and reflexivity properties. \end{defn}

Observe that, by definition, when $\mathcal{F}$ is well-structured the pair $(\mathcal{F},\leq_{\mathcal{F}})$ is a partially pre-ordered set.

\begin{defn} Given $A, B$ subsets of $\N$, and a set of functions $\mathcal{F}\subseteq$$\mathtt{Fun}$$(\N,\N)$, $A$ is {\bfseries $\mathcal{F}-$equivalent} to $B$ $($notation $A\equiv_{\mathcal{F}} B)$ if and only if  $A\leq_{\mathcal{F}} B$ and $B\leq_{\mathcal{F}} A$. \end{defn}

By the general properties of pre-orders it follows that, when $\mathcal{F}$ is well-structured, the relation $\equiv_{\mathcal{F}}$ is an equivalence relation and that $(\wp(\N)_{/_{\equiv_{\mathcal{F}}}},\leq_{\mathcal{F}})$ is a partial ordered set.\\
Observe that, for every nonempty subset $\mathcal{F}$ of $\mathtt{Fun}$($\N,\N$) and for every subset $A$ of $\N$, $A$ is $\mathcal{F}$-finitely mappable in $\N$. When $\mathcal{F}$ is well structured, the $\mathcal{F}$-equivalence class of $\N$ is the maximal element in $\wp(\N)_{/_{\equiv_{\mathcal{F}}}}$.

\begin{defn} Given a nonempty subset $\mathcal{F}$ of $\mathtt{Fun}$$(\N,\N)$, and a subset $A$ on $\N$, $A$ is {\bfseries $\mathcal{F}$-maximal} if, for every subset $B$ of $\N$, $B\leq_{\mathcal{F}} A$. The set of $\mathcal{F}$-maximal subsets of $\N$ is denoted by $\mathcal{S}_{\mathcal{F}}$:

\begin{center} $\mathcal{S}_{\mathcal{F}}=\{A\subseteq\N\mid$ $A$ is $\mathcal{F}-$maximal$\}$. \end{center}

\end{defn}

{\bfseries Fact:} If $\leq_{\mathcal{F}}$ is transitive, $\mathcal{S}_{\mathcal{F}}$ consists of the subsets $A$ of $\N$ that are $\mathcal{F}$-equivalent to $\N$:

\begin{center} $\mathcal{S}_{\mathcal{F}}=\{A\in\N\mid A\equiv_{\mathcal{F}}\N\}=\{A\in\N\mid \N\leq_{\mathcal{F}} A\}$. \end{center}

In fact, if $A$ is $\mathcal{F}$-maximal, in particular $\N\leq_{\mathcal{F}} A$, so 

\begin{center}$\mathcal{S}_{\mathcal{F}}\subseteq\{A\in\N\mid \N\leq_{\mathcal{F}} A\}$;\end{center} conversely, if $\leq_{\mathcal{F}}$ is transitive, and $\N\leq_{\mathcal{F}} A$, since for every subset $B$ of $\N$ $B\leq_{\mathcal{F}}\N$, by transitivity $B\leq_{\mathcal{F}}A$, so 

\begin{center} $\{A\in\N\mid \N\leq_{\mathcal{F}} A\}\subseteq \mathcal{S}_{\mathcal{F}}$, \end{center}

in particular $\mathcal{S}_{\mathcal{F}}=\{A\in\N\mid A\equiv_{\mathcal{F}}\N\}=\{A\in\N\mid \N\leq_{\mathcal{F}} A\}$

\begin{prop} Let $\mathcal{F}$ be a well-structured subset of $\mathtt{Fun}$$(\N,\N)$, and $A$ a subset of $\N$. Then $A$ is $\mathcal{F}$-maximal if and only if for every natural number $n$ there is a function $f_{n}$ in $\mathcal{F}$ such that $f_{n}(m)\in A$ for every $m\leq n$. \end{prop}

\begin{proof} Just observe that $A\in\mathcal{S}_{\mathcal{F}}$ if and only if $\N\leq_{\mathcal{F}} A$ if and only if, for every natural number $n$, there is a function $f_{n}$ such that $f_{n}([0,n])\subseteq A$. \\\end{proof}

Similarly to the relation of finite embeddability, the relations of $\mathcal{F}$-finite mappability can be generalized to ultrafilters. In next section we expose some important properties of these generalized relations.

\section{Finite Mappability of Ultrafilters on $\N$}

In this section, we define the relations of $\mathcal{F}$-finite mappability for ultrafilters, and we study their properties; the idea is to follow the same steps as we made with the relation of finite embeddability for ultrafilters. In particular, we shall study under what hypotheses the results obtained for the finite embeddability can be generalized in this more general context.

\begin{defn} Given ultrafilters $\U,\V$ on $\N$ and a set of functions $\mathcal{F}\subseteq$$\mathtt{Fun}$$(\N,\N)$, we say that $\U$ is {\bfseries $\mathcal{F}$-finitely mappable} in $\V$ $($notation $\U\trianglelefteq_{\mathcal{F}}\V)$ if and only if for every set $B$ in $\V$ there is a set $A$ in $\U$ such that $A\leq_{\mathcal{F}} B$. \end{defn}

In this proposition we present some basic properties of the relation $\trianglelefteq_{\mathcal{F}}$:

\begin{prop} If $\U,\V$ are ultrafilters on $\N$ and $\mathcal{F}, \mathcal{F}_{1}, \mathcal{F}_{2},...,\mathcal{F}_{k}$ are subsets of $\mathtt{Fun}$$(\N,\N)$, then the following properties hold:

\begin{enumerate}
\item if $\mathcal{F}=\{f\}$ then $\U\trianglelefteq_{\mathcal{F}}\V$ if and only if $\V=\overline{f}(\U)$;
\item if $\mathcal{F}=\mathcal{F}_{1}\cup \mathcal{F}_{2}$ then $\U\trianglelefteq_{\mathcal{F}}\V$ if and only if $\U\trianglelefteq_{\mathcal{F}_{1}}\V$ or $\U\trianglelefteq_{\mathcal{F}_{2}}\V$;
\item if $\mathcal{F}=\mathcal{F}_{1}\cup...\cup\mathcal{F}_{k}$ then $\U\trianglelefteq_{\mathcal{F}}\V$ if and only if there is an index $i\leq k$ such that $\U\trianglelefteq_{\mathcal{F}_{i}}\V$;
\item if $\mathcal{F}=\{f_{1},...,f_{k}\}$ then $\U\trianglelefteq_{\mathcal{F}}\V$ if and only if there is an index $i\leq k$ such that $\V=f_{i}(\U)$;
\item if $\mathcal{F}_{1}\subseteq \mathcal{F}_{2}$ and $\U\trianglelefteq_{\mathcal{F}_{1}}\V$, then $\U\trianglelefteq_{\mathcal{F}_{2}}\V$.
\end{enumerate}

\end{prop}

\begin{proof} (1) Suppose that $\U\trianglelefteq_{\{f\}}\V$. By definition, for every set $B$ in $\V$ there is a set $A$ in $\U$ such that $A\leq_{\{f\}}B$; as we proved in Proposition 4.6.2, this entails that $f(A)\subseteq B$ so, in particular, $f^{-1}(B)\in\U$ for every set $B\in\V$: by definition, this entails that $\V=\overline{f}(\U)$.\\
Conversely, if $\V=\overline{f}(\U)$, then for every set $B$ in $\V$ the set $f^{-1}(B)$ is in $\U$, and $f^{-1}(B)\leq_{f} B$. This proves that $\U\trianglelefteq_{\{f\}} \overline{f}(\U)=\V$.\\ 
2) Suppose that $\U\trianglelefteq_{\mathcal{F}_{1}\cup \mathcal{F}_{2}}\V$. There are only two possibilities:
\begin{enumerate}
	\item For every set $B$ in $\V$ there is a set $A$ in $\U$ such that $A\leq_{\mathcal{F}_{1}} B$;
	\item There is a set $B$ in $\V$ such that, for every set $A$ in $\U$, $A$ is not $\mathcal{F}_{1}$-finitely mappable in $B$.
\end{enumerate}

In case (1), by definition $\U\trianglelefteq_{\mathcal{F}_{1}}\V$.\\
In case (2), every subset of $B$, and in particular every subset $S\subseteq B$ with $S\in\V$, satisfies this property: 

\begin{center} For every set $A$ in $\U$, $A$ is not $\mathcal{F}_{1}$-finitely mappable in $S$. \end{center}

{\bfseries Claim: $\U\trianglelefteq_{\mathcal{F}_{2}}\V$}.\\

Let $Y$ be a set in $\V$. Then the intersection $Y\cap B$ has the following two properties:
\begin{enumerate}
	\item $Y\cap B$ is in $\V$;
	\item $Y\cap B$ is a subset of $B$.
\end{enumerate}

By property $(1)$, since by hypothesis $\U\trianglelefteq_{\mathcal{F}_{1}\cup\mathcal{F}_{2}}\V$, it follows that there is an element $A$ in $\U$ such that $A$ is $\mathcal{F}_{1}\cup \mathcal{F}_{2}$-finitely mappable in $Y\cap B$. By property (2) it follows that $A$ is not $\mathcal{F}_{1}$-finitely mappable in $Y\cap B$; and by Proposition 4.6.2, these two conditions together entails that $A\leq_{\mathcal{F}_{2}} Y\cap B$.\\
Since $Y\cap B\subseteq Y$, and $A\leq_{\mathcal{F}_{2}}Y\cap B$, it follows that $A\leq_{\mathcal{F}_{2}} Y$; this proves that $\U\leq_{\mathcal{F}_{2}}\V$.\\
3) This follows, by induction, from point (2).\\
4) This is an immediate consequence of points (3) and (1).\\
5) By hypothesis, for every set $B$ in $\V$ there is a set $A$ in $\U$ such that $A\leq_{\mathcal{F}_{1}} B$. As $\mathcal{F}_{1}\subseteq\mathcal{F}_{2}$, by Proposition 4.6.2 it follows that $A\leq_{\mathcal{F}_{2}} B$; in particular, $\U\trianglelefteq_{\mathcal{F}_{2}}\V$.\\ \end{proof}

A question that arises naturally is how principal and nonprincipal ultrafilters are related with respect to $\mathcal{F}$-finite mappability.

\begin{prop} Let $\U,\V$ be the principal ultrafilters generated by $n$ and $m$ respectively, and let $\mathcal{F}$ be a subset of $\mathtt{Fun}$$(\N,\N)$. The following two conditions are equivalent:
\begin{enumerate}
	\item $\U\trianglelefteq_{\mathcal{F}} \V$;
	\item there is a function $f$ in $\mathcal{F}$ such that $f(n)=m$.
\end{enumerate}
\end{prop}

\begin{proof} $(1)\Rightarrow(2)$ Consider the set $\{m\}$ in $\V$. Since $\U\trianglelefteq_{\mathcal{F}}\V$, there is a set $A\in\U_{n}$ such that $A\leq_{\mathcal{F}}\{m\}$. Since $\U$ is the principal ultrafilter generated by $n$, then $n\in A$ and by definition there is a function $f$ in $\mathcal{F}$ such that $f(\{n\})\subseteq \{m\}$, and this happens if and only if $f(n)=m$\\
$(2)\Rightarrow(1)$ Let $B$ be a set in $\V$, and $f$ a function in $\mathcal{F}$ such that $f(n)=m$. In particular, as $m\in B$ (since $\V$ is the principal ultrafilter generated by $m$), $f(\{n\})\subseteq B$. So $\{n\}\leq_{\mathcal{F}} B$ for every set $B$ in $\V$, and $\U\trianglelefteq_{\mathcal{F}}\V$.\\\end{proof}

\begin{prop} Given a subset $\mathcal{F}$ of $\mathtt{Fun}$$(\N,\N)$, the following two conditions are equivalent:
\begin{enumerate}
	\item Every principal ultrafilter $\U$ is $\mathcal{F}$-finitely embeddable in every nonprincipal ultrafilter $\V$;
	\item For every natural number $n$, for every infinite subset $A$ of $\N$ there is a function $f$ in $\mathcal{F}$ such that $f(n)\in A$.
\end{enumerate} \end{prop}

\begin{proof} $(1)\Rightarrow (2)$: Let $A$ be an infinite subset of $\N$, $n$ a natural number and $\V$ a nonprincipal ultrafilter such that $A\in\V$. By hypothesis, as $\U_{n}\trianglelefteq_{\mathcal{F}}\V$, there is a set $B$ in the principal ultrafilter $\U_{n}$ and a function $f$ in $\mathcal{F}$ such that $f(B)\subseteq A$. In particular, since $n\in B$, this proves that $f(n)\in A$.\\
$(2)\Rightarrow (1)$: Let $\V$ be a nonprincipal ultrafilter and $\U_{n}$ the principal ultrafilter generated by the natural number $n$. For every set $A$ in $\V$, by hypothesis $\{n\}\leq_{\mathcal{F}} A$, as $A$ is infinite; so $\U_{n}\trianglelefteq_{\mathcal{F}}\V$ for every natural number $n$, and this proves the thesis.\\ \end{proof}

Note that, in general, it is possible that a nonprincipal ultrafilter $\U$ is $\mathcal{F}$-finitely mappable in a principal ultrafilter $\mathfrak{U}_{n}$. For example, let 

\begin{center} $\mathcal{F}=\{ f_{n} \}$,\end{center} 

where $f_{n}$ is the constant function with value $n$. Then every ultrafilter $\U$ in $\bN$ is $\mathcal{F}$-finitely mappable in $\mathfrak{U}_{n}$.\\
The question that arises is if $\trianglelefteq_{\mathcal{F}}$ is, or is not, an order. We know that this in general is false as $\trianglelefteq_{\mathbb{T}}$ is not an order on $\bN$. There is also one other problem: $\trianglelefteq_{\mathcal{F}}$, similarly to $\leq_{\mathcal{F}}$, is not, in general, reflexive or transitive.

\begin{prop} If $\leq_{\mathcal{F}}$ is a pre-order then $\trianglelefteq_{\mathcal{F}}$ is a pre-order.\end{prop}

\begin{proof} We have to prove that $\trianglelefteq_{\mathcal{F}}$ is transitive and reflexive.\\
Transitive: suppose that $\U,\V,\W$ are ultrafilters in $\bN$ such that $\U\trianglelefteq_{\mathcal{F}} \V$ and $\V\trianglelefteq_{\mathcal{F}} \W$, and let $C$ be a set in $\W$. Since $\V\trianglelefteq_{\mathcal{F}}\W$, there is a set $B$ in $\V$ such that $B\leq_{\mathcal{F}} C$ and, since $\U\trianglelefteq_{\mathcal{F}}\V$, there is a set $A$ in $\U$ such that $A\leq_{\mathcal{F}}B$. Since $\leq_{\mathcal{F}}$ is transitive (as we supposed $\mathcal{F}$ well-structured), $A\leq_{\mathcal{F}} C$, so $\U\trianglelefteq_{\mathcal{F}} \W$.\\
Reflexive: for every ultrafilter $\U$ in $\bN$, for every set $A$ in $\U$, $A\leq_{\mathcal{F}} A$ (as $\mathcal{F}$ is well-structured), so $\U\trianglelefteq_{\mathcal{F}}\U$.\\ \end{proof}

To obtain an antysimmetric relation, we follow the general procedure for pre-orders exposed in Section 4.2:

\begin{defn} Let $\mathcal{F}$ be a well-structured subset of $\mathtt{Fun}$$(\N,\N)$. Given ultrafilters $\U, \V$ on $\N$, $\U$ is {\bfseries $\mathcal{F}$-equivalent} to $\V$ $($notation $\U\equiv_{\mathcal{F}} \V)$ if and only if $\U\trianglelefteq_{\mathcal{F}} \V$ and $\V\trianglelefteq_{\mathcal{F}} \U$. \end{defn}

Observe that $\equiv_{\mathcal{F}}$ is an equivalence relation on $\bN$.

\begin{defn} If $\mathcal{F}$ is a well-structured subset of $\mathtt{Fun}$$(\N,\N)$, for every ultrafilter $\U$ we denote by $[\U]_{\mathcal{F}}$ the $\mathcal{F}$-equivalence class of $\U$:

\begin{center} $[\U]_{\mathcal{F}}=\{\V\in\bN\mid \U\equiv_{\mathcal{F}}\V\}$. \end{center}

When there is no danger of confusion, we simply denote $[\U]_{\mathcal{F}}$ as $[\U]$. \end{defn}

By the general facts about pre-orders it follows that:

\begin{thm} $(\bN_{/_{\equiv_{\mathcal{F}}}},\trianglelefteq_{\mathcal{F}})$ is a partially ordered set whenever $\mathcal{F}$ is a well-structured subset of $\mathtt{Fun}$$(\N,\N)$. \end{thm}

This theorem shows that, at least when $\mathcal{F}$ is well-structured, by considering the quotient space we obtain a partial ordered set, similarly to the case of finite embeddability. In next section we study more closely the structure of ($\bN_{/_{\equiv_{\mathcal{F}}}},\trianglelefteq_{\mathcal{F}}$).

\subsection{The partially ordered set $(\bN_{/_{\equiv_{\mathcal{F}}}},\trianglelefteq_{\mathcal{F}})$}

In this section we study the partially ordered set $(\bN_{/_{\equiv_{\mathcal{F}}}},\trianglelefteq_{\mathcal{F}})$; in particular, the question we want to answer is the following: is there a greatest element in $(\bN_{/_{\equiv_{\mathcal{F}}}},\trianglelefteq_{\mathcal{F}})$?

\begin{defn} The chains in $(\bN_{/_{\equiv_{\mathcal{F}}}},\trianglelefteq_{\mathcal{F}})$ and in $(\bN,\trianglelefteq_{\mathcal{F}})$ are called {\bfseries $\mathcal{F}$-chains}. Similarly, the upper bounds of subsets in $(\bN_{/_{\equiv_{\mathcal{F}}}},\trianglelefteq_{\mathcal{F}})$ and in $(\bN,\trianglelefteq_{\mathcal{F}})$ are called {\bfseries $\mathcal{F}$-upper bounds}.\end{defn}

\begin{prop} Every $\mathcal{F}$-chain $\langle \U_{i}\mid i\in I \rangle$ of ultrafilters has an $\mathcal{F}$-upper bound $\U$. \end{prop}

\begin{proof} The proof is analogue to that of Theorem 4.4.12.\\
If $I$ has a greatest element $i$, then the ultrafilter $\U_{i}$ is trivially an upper bound for the chain.\\
Otherwise, suppose that $I$ has not a greatest element, an let $\V$ be an ultrafilter on $I$ such that, for every $i\in I$, the set 

\begin{center} $G_{i}=\{j\in I\mid j\geq i\}$ \end{center}

is in $\V$ (as already proved in Theorem 4.4.12, the family $G_{i}$ has the finite intersection property, so such an ultrafilter $\V$ exists, and it is nonprincipal since $I$ has not a greatest element).\\

{\bfseries Claim:} The ultrafilter $\U=\V-\lim_{i\in I}\U_{i}$ is an upper bound for the chain.\\

In fact, let $B$ be an element of $\U$, and consider the ultrafilter $\U_{i}$. Since $B\in\U$, by definition of limit of ultrafilters the set 

\begin{center} $I_{B}=\{i\in I\mid B\in\U_{i}\}$ \end{center}

is in $\V$; as $\V$ contains the family $\{G_{i}\}_{i\in I}$, in $I_{B}$ there is an element $j$ with $i\leq j$. In particular, $B\in\U_{j}$. As $\langle\U_{i}\mid i\in I\rangle$ is an $\mathcal{F}-$chain, $\U_{i}\trianglelefteq_{\mathcal{F}}\U_{j}$, so in $\U_{i}$ there is a set $A$ such that $A\leq_{\mathcal{F}} B$; this proves that $\U_{i}\trianglelefteq_{\mathcal{F}}\U$ for every index $i\in I$, so $\U$ is an upper bound for the chain.\\ \end{proof}

Observe that, in the previous proposition, we did not assume that $\mathcal{F}$ is well-structured. When $\mathcal{F}$ is well-structured the above result can be extended to the partially ordered set $(\bN_{/_{\equiv_{\mathcal{F}}}},\trianglelefteq_{\mathcal{F}})$:

\begin{cor} Let $\mathcal{F}$ be a well-structured subset of $\mathtt{Fun}$$(\N,\N)$, and let $\langle [\U_{i}]\mid i\in I \rangle$ be an $\mathcal{F}$-chain in $\bN_{/_{\equiv_{\mathcal{F}}}}$. Then there is an upper bound $[\U]$ for the chain. \end{cor}

\begin{proof} When $\mathcal{F}$ is well-structured, $(\bN,\trianglelefteq_{\mathcal{F}})$ is a partially pre-ordered set, so this result follows by the analogue general property of pre-orders.\\ \end{proof}

\begin{cor} If $\mathcal{F}$ is well-structured there are maximal elements in $(\bN_{/_{\equiv_{\mathcal{F}}}},\trianglelefteq_{\mathcal{F}})$. \end{cor}

\begin{proof} This follows by Zorn's Lemma because every $\mathcal{F}$-chain in $\bN_{/_{\equiv_{\mathcal{F}}}}$ has an upper bound. \\ \end{proof}

In Section 4.4.2 we used a result analogous to Corollary 4.7.12 to prove the existence of a greatest element in $\bN_{/_{\equiv_{\mathbb{T}}}}$. The proof used an important property of $\trianglelefteq_{\mathbb{T}}$, namely the fact that $\trianglelefteq_{\mathbb{T}}$ is filtered (see definition 4.2.7).

\begin{prop} If $\mathcal{F}$ is a well-structured subset of $\mathtt{Fun}$$(\N,\N)$, the following two conditions are equivalent:
\begin{enumerate}
	\item the relation $\trianglelefteq_{\mathcal{F}}$ is filtered on $\bN$;
	\item the relation $\trianglelefteq_{\mathcal{F}}$ is filtered on $\bN/_{\equiv_{\mathcal{F}}}$.
\end{enumerate}\end{prop}

As observed in Section 4.2, the above property holds for all pre-orders.

\begin{prop} If $\mathcal{F}_{1}, \mathcal{F}_{2}$ are subsets of $\mathtt{Fun}$$(\N,\N)$ such that $\mathcal{F}_{1}\subseteq \mathcal{F}_{2}$ and $\trianglelefteq_{\mathcal{F}_{1}}$ is filtered, then also $\trianglelefteq_{\mathcal{F}_{2}}$ is filtered. \end{prop}

\begin{proof} Observe that, for every ultrafilters $\U,\V$ in $\bN$, if $\U\trianglelefteq_{\mathcal{F}_{1}}\V$ then, as $\mathcal{F}_{1}\subseteq\mathcal{F}_{2}$, $\U\trianglelefteq_{\mathcal{F}_{2}}\V$.\\
In particular, let $\W$ be an ultrafilter such that $\U\trianglelefteq_{\mathcal{F}_{1}}\W$ and $\V\trianglelefteq_{\mathcal{F}_{1}}\W$. From the observation it follows that $\U\trianglelefteq_{\mathcal{F}_{2}}\W$ and $\V\trianglelefteq_{\mathcal{F}_{2}}\W$. This proves that for every ultrafilters $\U,\V$ there is an ultrafilter $\W$ with $\U\trianglelefteq_{\mathcal{F}_{2}}\W$ and $\V\trianglelefteq_{\mathcal{F}_{2}}\W$, so $\trianglelefteq_{\mathcal{F}_{2}}$ is filtered.\\ \end{proof}

\begin{defn} A subset $\mathcal{F}$ of $\mathtt{Fun}$$(\N,\N)$ is {\bfseries filtered} if $\trianglelefteq_{\mathcal{F}}$ is filtered on $\bN$ $($equivalently, if $\trianglelefteq_{\mathcal{F}}$ is filtered on $\bN_{/_{\equiv_{\mathcal{F}}}})$. \end{defn}

\begin{thm} Let $\mathcal{F}$ be a well-structured subset of $\mathtt{Fun}$$(\N,\N)$. The following two conditions are equivalent:
\begin{enumerate}
	\item there is a greatest element in $(\bN_{/_{\equiv_{\mathcal{F}}}},\trianglelefteq_{\mathcal{F}})$;
	\item the order $\trianglelefteq_{\mathcal{F}}$ is filtered.
\end{enumerate}
\end{thm}

The above result is a consequence of the analogue general property of pre-orders proved in Section 4.2 (Theorem 4.2.8).\\
Filtered sets of functions can be characterize in nonstandard terms. This characterization follows by the properties of the cones $\mathcal{C}_{\mathcal{F}}(\U)$ in $(\bN,\trianglelefteq_{\mathcal{F}})$, and these are the structure that we study in next section.

\subsection{The cones $\mathcal{C}_{\mathcal{F}}(\U)$}

\begin{defn} Given an ultrafilter $\U$ in $\bN$, we denote by $\mathcal{C}_{\mathcal{F}}(\U)$ the upper cone of $\U$ in $(\bN,\trianglelefteq_{\mathcal{F}})$, i.e. the set of ultrafilters in $\bN$ in which $\U$ is $\mathcal{F}$-finitely mappable:

\begin{center} $\mathcal{C}_{\mathcal{F}}(\U)=\{\V\in\bN\mid \U\trianglelefteq_{\mathcal{F}}\V\}$. \end{center} \end{defn}

In Section 4.4.3 we have studied and characterized the sets $\mathcal{C}_{fe}(\U)$ that, following the definition 4.7.17, we denote from now on as $\mathcal{C}_{\mathbb{T}}(\U)$. In this section we try to generalize the results obtained for the sets $\mathcal{C}_{\mathbb{T}}(\U)$ in this more general context:\\

{\bfseries Fact:} If $\mathcal{F}$ is well-structured, and $\U$, $\V$ are $\mathcal{F}$-equivalent ultrafilters, then $\mathcal{C}_{\mathcal{F}}(\U)=\mathcal{C}_{\mathcal{F}}(\V)$.\\

This fact follows by the transitivity of $\trianglelefteq_{\mathcal{F}}$.

\begin{prop} For every ultrafilter $\U$ in $\bN$ and for every nonempty subset $\mathcal{F}$ of $\mathtt{Fun}$$(\N,\N)$ the set $\mathcal{C}_{\mathcal{F}}(\U)$ is closed in the Stone topology. \end{prop}

\begin{proof} We use the characterization of closed subsets of $\bN$ given in terms of limit ultrafilters: to prove that $\mathcal{C}_{\mathcal{F}}(\U)$ is closed we show that, given a family $\langle \U_{i}\mid i\in\ I\rangle$ of elements in $\mathcal{C}_{\mathcal{F}}(\U)$ and an ultrafilter $\V$ on $I$, the ultrafilter 

\begin{center} $\W=\V-\lim_{I}\U_{i}$ \end{center}

is in $\mathcal{C}_{\mathcal{F}}(\U)$.\\
To prove this, let $A$ be an element of $\W$. By definition, this means that the set 

\begin{center} $A_{I}=\{i\in I\mid A\in\U_{i}\}$ \end{center}

is nonempty, as it is in $\V$. Let $i$ be an element of $A_{I}$; this entails that $A\in \U_{i}$. Since $\U_{i}$ is an element of $\mathcal{C}_{\mathcal{F}}(\U)$, $\U\trianglelefteq_{\mathcal{F}}\U_{i}$, so there is a set $B$ in $\U$ with $B\leq_{\mathcal{F}} A$. This proves that for every set $A$ in $\W$ there is a set $B$ in $\U$ with $B\leq_{\mathcal{F}} A$, so $\U\trianglelefteq_{\mathcal{F}}\W$ and $\W\in\mathcal{C}_{\mathcal{F}}(\U)$.\\ \end{proof}

When $\mathcal{F}=\mathbb{T}$ we proved that $\mathcal{C}_{\mathbb{T}}(\U)=\overline{\{\U\oplus\V\mid\V\in\bN\}}$; is there a similar characterization for a generic cone $\mathcal{C}_{\mathcal{F}}(\U)$?

\begin{lem} Let $\U$ be an ultrafilter in $\bN$, $\alpha\in$$^{*}\N$ a generator of $\U$, $\mathcal{F}$ a subset of $\mathtt{Fun}$$(\N,\N)$ and $B$ a subset of $\N$. The following two conditions are equivalent:
\begin{enumerate}
	\item there is a set $A$ in $\U$ such that $A\leq_{\mathcal{F}} B$;
	\item there is a function $\varphi$ in $^{*}\mathcal{F}$ such that $B\in \mathfrak{U}_{( ^{*}\varphi)(\alpha)}$.
\end{enumerate}
\end{lem}

\begin{proof} $(1)\Rightarrow (2)$: Since $A\leq_{\mathcal{F}} B$, $A\in \U$, and $\alpha\in$$^{*}\N$ is a generator of $\U$, then by Proposition 4.6.5 it follows that there is a function $\varphi$ in $^{*}\mathcal{F}$ such that $B\in\mathfrak{U}_{( ^{*}\varphi)(\alpha)}$.\\
$(2)\Rightarrow (1)$: Let $\varphi$ be a function in $^{*}\mathcal{F}$ such that $($$^{*}\varphi)(\alpha)\in$$^{**}B$. By transfer it follows that 

\begin{center} $\{\mu\in$$^{*}\N\mid ($$^{*}\varphi)(\mu)\in$$^{**}B\}=$$^{*}\{n\in\N\mid \varphi(n)\in$$^{*}B\}$, \end{center}

and, as $\alpha\in\{\mu\in$$^{*}\N\mid ($$^{*}\varphi)(\mu)\in$$^{**}B\}$, it follows that $A=\{n\in\N\mid \varphi(n)\in$$^{*}B\}$ is in $\U$, since $\U=\mathfrak{U}_{\alpha}$ and $\alpha\in$$^{*}A$. By construction, $\varphi(A)\subseteq$$^{*}B$ and, as $\varphi$ is a function in $\mathcal{F}$, by Proposition 4.6.3 this entails that $A\leq_{\mathcal{F}} B$. As $A\in\U$, this proves the thesis.\\ \end{proof}

{\bfseries Note 1:} As a consequence of Proposition 4.6.4, the above lemma do not depend on the choice of $\alpha$ in $G_{\U}$, since whenever $\alpha,\beta\in$$^{*}\N$ are generators of $\U$ then $\mathfrak{U}_{(^{*}\varphi)(\alpha)}=\mathfrak{U}_{(^{*}\varphi)(\beta)}$.\\ 

{\bfseries Note 2:} Given a function $\varphi$ in $\mathtt{Fun}$($^{*}\N,$$^{*}\N$), let $\overline{\varphi}$ denote this function in $\mathtt{Fun}$($\bN,\bN$): for every ultrafilter $\U$, if $\alpha$ is any generator of $\U$ with $h(\alpha)\leq 1$, define

\begin{center} $\overline{\varphi}(\mathfrak{U}_{\alpha})=\mathfrak{U}_{(^{*}\varphi)(\alpha)}$. \end{center}

This definition not only is similar to the definition of $\overline{f}$ for a function $f\in\mathtt{Fun}(\N,\N)$, but it can be seen as its extention to nonstandard functions. In fact we have the following property:

\begin{prop} Let $g$ be a function in $\mathtt{Fun}$$(\N,\N)$, and $\varphi$ the function $\varphi=$$^{*}g$ in $\mathtt{Fun}($$^{*}\N,$$^{*}\N)$. Then

\begin{center} $\overline{\varphi}=\overline{g}$. \end{center}

\end{prop}

\begin{proof} For every ultrafilter $\U$ in $\bN$, for every generator $\alpha$ of $\U$ with $h(\alpha)\leq 1$, as we proved in Chapter Two we have 

\begin{center} $\overline{g}(\mathfrak{U}_{\alpha})=\mathfrak{U}_{(^{*}g)(\alpha)}$, \end{center}

and $($$^{*}g)(\alpha)=\varphi(\alpha)=($$^{*}\varphi)(\alpha)$ since we have this property

\begin{center} For every natural number $n\in\N$, $\varphi(n)=($$^{*}g)(n)=g(n)$, \end{center}

so, by transfer, we have

\begin{center} For every hypernatural number $\eta\in$$^{*}\N$, $($$^{*}\varphi)(\eta)=($$^{**}g)(\eta)=($$^{*}g)(\eta)$, \end{center}

and this shows that $\overline{^{*}g}(\mathfrak{U}_{\alpha})=\overline{g}(\mathfrak{U}_{\alpha})$ for every function $g$ in $\mathtt{Fun}$($\N,\N$). \\\end{proof}

\begin{defn} For every function $\varphi$ in $\mathtt{Fun}$$(^{*}\N,^{*}\N)$, we denote by $\overline{\varphi}$ the function in $\mathtt{Fun}$$(\bN,\bN)$ such that, for every ultrafilter $\U$ in $\bN$, if $\alpha\in$$^{*}\N$ is a generator of $\U$ then $\overline{\varphi}(\U)=\mathfrak{U}_{( ^{*}\varphi)(\alpha)}$. \end{defn}

We can now characterize the sets $\mathcal{C}_{\mathcal{F}}(\U)$:

\begin{thm} Let $\mathcal{F}$ be a subset of $\mathtt{Fun}$$(\N,\N)$, $\U$ an ultrafilter in $\bN$, and $\alpha\in$$^{*}\N$ a generator of $\U$. Then

\begin{center} $\mathcal{C}_{\mathcal{F}}(\U)=\overline{\{\mathfrak{U}_{( ^{*}\varphi)(\alpha)}\mid \varphi\in}$$\overline{^{*}\mathcal{F}\}}=\overline{\{\overline{\varphi}(\U)\mid \varphi\in ^{*}\mathcal{F}\}}$. \end{center} 
\end{thm}

\begin{proof} We repeatedly use the result of Lemma 4.7.19.\\
Let $\V$ be an element in $\mathcal{C}_{\mathcal{F}}(\U)$; by definition, for every set $B$ in $\V$ there is a set $A$ in $\U$ such that $A\leq_{\mathcal{F}} B$. As $\alpha\in$$^{*}\N$ is a generator of $\U$, Lemma 4.7.19 entails that there is a function $\varphi$ in $ ^{*}\mathcal{F}$ such that $B\in\mathfrak{U}_{( ^{*}\varphi)(\alpha)}$; in the Stone topology, this is equivalent to say that $\V$ is in the closure of $\{\mathfrak{U}_{( ^{*}\varphi)(\alpha)}\mid \varphi\in$$^{*}\mathcal{F}\}$, so

\begin{center} $\mathcal{C}_{\mathcal{F}}(\U)\subseteq\overline{\{\mathfrak{U}_{( ^{*}\varphi)(\alpha)}\mid \varphi\in}$$\overline{^{*}\mathcal{F}\}}$ \end{center}

Conversely, let $\V$ be an element in $\overline{\{\mathfrak{U}_{( ^{*}\varphi)(\alpha)}\mid \varphi\in}$$\overline{^{*}\mathcal{F}\}}$. In the Stone topology, this is equivalent to say that for every set $B$ in $\V$ there is a function $\varphi$ in $^{*}\mathcal{F}$ such that $B\in\mathfrak{U}_{( ^{*}\varphi)(\alpha)}$; by Lemma 4.7.19, this entails that there is a set $A$ in $\U$ such that $A\leq_{\mathcal{F}} B$; in particular, $\U\trianglelefteq_{\mathcal{F}} \V$, so $\V\in\mathcal{C}_{\mathcal{F}}(\U)$ and

\begin{center} $\overline{\{\mathfrak{U}_{( ^{*}\varphi)(\alpha)}\mid \varphi\in}$$\overline{^{*}\mathcal{F}\}}\subseteq \mathcal{C}_{\mathcal{F}}(\U)$. \end{center}
Since we proved both inclusions, the two sets are equal, and this proves the thesis.\\ \end{proof}

To give an example of application of this theorem, we consider the case $\mathcal{F}=\mathbb{T}$.\\
First of all, since

\begin{center} $\mathbb{T}=\{t_{n}\in$ $\mathtt{Fun}$($\N,\N$)$\mid (n\in\N) \wedge(\forall m\in\N$ $t_{n}(m)=m+n)\}$,\end{center}

then, by transfer,

\begin{center} $^{*}\mathbb{T}=\{t_{\mu}\in$ $\mathtt{Fun}$($^{*}\N,$$^{*}\N)\mid$ ($\mu\in$$^{*}\N)\wedge(\forall\eta\in$$^{*}\N$ $t_{\mu}(\eta)=\mu+\eta)\}$. \end{center}

Observe that, for every function $t_{\mu}$ in $^{*}\mathbb{T}$, for every hypernatural number $\alpha$ in $^{*}\N$,

\begin{center} $($$^{*}t_{\mu})(\alpha)=$$^{*}\mu+\alpha$. \end{center}

So, by Theorem 4.7.22, for every hypernatural number $\alpha\in$$^{*}\N$, if $\U=\mathfrak{U}_{\alpha}$ then

\begin{center} $\mathcal{C}_{\mathbb{T}}(\U)=\overline{\{\mathfrak{U}_{^{*}\mu+\alpha}\mid \mu\in}$$\overline{^{*}\N\}}$. \end{center}

As we proved in Chapter Two, $\mathfrak{U}_{^{*}\mu+\alpha}=\mathfrak{U}_{\alpha}\oplus\mathfrak{U}_{\mu}$ for every hypernatural number $\mu$; so

\begin{center} $\mathcal{C}_{\mathbb{T}}(\U)=\overline{\{\mathfrak{U}_{\alpha}\oplus\mathfrak{U}_{\mu}\mid \mathfrak{U}_{\mu}\in\bN\}}=\overline{\{\U\oplus\V\mid\V\in\bN\}}$, \end{center}

as expected.

\subsection{Characterizations of filtered functional pre--orders}

In this section we give a nonstandard and a standard characterization of filtered sets of functions. By Theorem 4.7.2 we know that, given an ultrafilter $\U=\mathfrak{U}_{\alpha}$ and a generical ultrafilter $\V$, $\U\trianglelefteq_{\mathcal{F}}\V$ if and only if $\V\in\overline{\{\mathfrak{U}_{^{*}\varphi(\alpha)}\mid\varphi\in}$$\overline{^{*}\mathcal{F}\}}$.\\
This can be equivalently restated in this way:

\begin{center} $(\dagger) \mathfrak{U}_{\alpha}\trianglelefteq_{\mathcal{F}}\V\Leftrightarrow \bigcap_{\varphi\in^{*}\mathcal{F}}\mathfrak{U}_{^{*}\varphi(\alpha)}\subseteq\V$. \end{center}

This is a particular case of the following general property of the Stone Topology:

\begin{prop} Let $S$ be a subset of $\bN$, and $\U$ an ultrafilter on $\N$. The following two conditions are equivalent;

\begin{enumerate}
	\item $\U\in\overline{S}$;
	\item $\bigcap_{\V\in S}\V \subseteq\U$.
\end{enumerate}
\end{prop}

\begin{proof} In the Stone-Topology, $\U\in\overline{S}$ if and only for every set $A$ in $\U$ there is an ultrafilter $\V$ in $S$ such that $A\in S$. We use this property to prove the equivalence of (1) and (2).\\
$(1)\Rightarrow(2)$ Suppose that there is a set $A$ in $\bigcap_{\V\in S}\V$ with $A\notin\U$. Then $A^{c}\in\U$ and, since $\U\in\overline{S}$, this entails that $A^{c}\in\V$ for some ultrafilter $\V\in S$, and this is absurd since $A\in\V$ for every ultrafilter $\V\in S$. So $\bigcap_{\V\in S}\V\subseteq\U$.\\
$(2)\Rightarrow(1)$ Suppose that $\U\notin\overline{S}$. This entails that there is a set $A$ in $\U$ such that, for every ultrafilter $\V$ in $S$, $A\notin\V$. So $A^{c}\in\bigcap_{\V\in S}\V$, hence $A^{c}\in\U$, which provides a contradiction. So $\U\in\overline{S}$.\\\end{proof}

\begin{defn} For every hypernatural number $\alpha$, let $\mathfrak{F}_{\alpha}$ denote the filter

\begin{center}$\mathfrak{F}_{\alpha}=\bigcap_{\varphi\in^{*}\mathcal{F}}\mathfrak{U}_{^{*}\varphi(\alpha)}$. \end{center}

\end{defn}

By definition it follows that $\mathfrak{F}_{\alpha}$ is closed under superset, and that a set $A$ is in the filter $\mathfrak{F}_{\alpha}$ if and only if, for every function $\varphi$ in $^{*}\mathcal{F}$, $^{*}\varphi(\alpha)\in$$^{**}A$.\\
Also, as a consequence of $(\dagger)$, it follows that an ultrafilter $\V$ is in the cone $\mathcal{C}_{\mathcal{F}}(\mathfrak{U}_{\alpha})$ if and only if $\mathfrak{F}_{\alpha}\subseteq\V$.\\
We are now ready to characterize the filtered sets of functions:

\begin{defn} We say that a set $\mathcal{F}$ of functions satisfies the condition $(F)$ if the following condition holds:

\begin{center} $(F)$: $\neg(\exists \alpha,\beta\in$$^{*}\N,\exists A\subseteq\N$ such that $A\in \mathfrak{F}_{\alpha}, A^{c}\in \mathfrak{F}_{\beta})$. \end{center}

\end{defn}

\begin{thm} Let $\mathcal{F}\subseteq\mathtt{Fun}(\N,\N)$ be a set of functions. The two following conditions are equivalent:
\begin{enumerate}
	\item $\mathcal{F}$ satisfies condition $(F)$;
	\item $\mathcal{F}$ is filtered.
\end{enumerate}
\end{thm}

\begin{proof} $(1)\Rightarrow (2)$ Let $\mathfrak{U}_{\alpha},\mathfrak{U}_{\beta}$ be ultrafilters on $\N$. By condition (F) it follows that for all sets $A\in \mathfrak{F}_{\alpha}$ and $B\in \mathfrak{F}_{\beta}$ the intersection $A\cap B$ is non empty. In fact, suppose by contrast that there are sets $A\in\mathfrak{F}_{\alpha}$, $B\in\mathfrak{F}_{\beta}$ such that $A\cap B=\emptyset$. Then consider the set $A^{c}$. $A^{c}$ is a superset of $B$, so $A^{c}\in\mathfrak{F}_{\beta}$, while $A\in\mathfrak{F}_{\alpha}$, and this is absurd.\\
So $\mathfrak{F}_{\alpha}\cup \mathfrak{F}_{\beta}$ is a filter, and every ultrafilter $\W$ that extends $\mathfrak{F}_{\alpha}\cup \mathfrak{F}_{\beta}$ is, by construction, greater (respect to $\trianglelefteq_{\mathcal{F}})$ than $\mathfrak{U}_{\alpha}$ and $\mathfrak{U}_{\beta}$. So $\mathcal{F}$ is filtered.\\ 
$(2)\Rightarrow (1)$ Let $\trianglelefteq_{\mathcal{F}}$ be filtered, and suppose by contrast that $\mathcal{F}$ does not satisfy condition (F). Then there are $\alpha,\beta\in$$^{*}\N$, and a subset $A$ of $\N$ such that $A\in \mathfrak{F}_{\alpha}$ and $A^{c}\in \mathfrak{F}_{\beta}$. Consider $\mathfrak{U}_{\alpha}$ and $\mathfrak{U}_{\beta}$, and let $\V$ be an ultrafilter such that $\mathfrak{U}_{\alpha}\trianglelefteq_{\mathcal{F}}\V$, $\mathfrak{U}_{\beta}\trianglelefteq_{\mathcal{F}}\V$. Then $\mathfrak{F}_{\alpha}\subseteq\V$ and $\mathfrak{F}_{\beta}\subseteq\V$, and this is absurd, because it follows that $A\in\V$ and $A^{c}\in\V$.\\\end{proof}

With the above characterization we can reprove that $\trianglelefteq_{\mathbb{T}}$ is filtered. In fact, $\mathbb{T}$ satisfies condition $(F)$: suppose, by contrast, that there are $\alpha,\beta\in$$^{*}\N$ and a subsets $A$ of $\N$ such that, for every function $\varphi\in$$^{*}\mathbb{T}$, $^{*}\varphi(\alpha)\in$$^{**}A$ and $^{*}\varphi(\beta)\in$$^{**}A^{c}$.\\
As $\mathbb{T}\subseteq$$^{*}\mathbb{T}$, for every natural number $n$ the translation by $n$ is in $^{*}\mathbb{T}$. Since $A\in \mathfrak{F}_{\alpha}$, this entails that:

\begin{center} $\forall n\in \N$, $n+\alpha\in$$^{**}A$. \end{center}

Since $n,\alpha\in$$^{*}\N$, the above property can be reformulate as follows:

\begin{center} $\forall n\in\N$, $n+\alpha\in$$^{*}A$. \end{center}

By transfer we have:

\begin{center} $\forall \eta\in$$^{*}\N$, $\eta+$$^{*}\alpha\in$$^{**}A$. \end{center}

Hence, choosing $\eta=\beta$, it follows that $\beta+$$^{*}\alpha\in$$^{**}A$. Observe that $\beta+$$^{*}\alpha=$$^{*}t_{\alpha}(\beta)$, where $t_{\alpha}$ is the translation such that, for every hypernatural number $\eta$, $t_{\alpha}(\eta)=\eta+\alpha$. Since $t_{\alpha}\in\mathbb{T}$, as $A^{c}\in \mathfrak{F}_{\beta}$ it follows that $^{*}t_{\alpha}(\beta)\in$$^{**}A^{c}$, so 

\begin{center} $^{*}t_{\alpha}(\beta)\in$$^{**}A^{c}$ and $^{*}t_{\alpha}(\beta)\in$$^{**}A$, \end{center}

and this is absurd.\\
This proves that $\mathbb{T}$ satisfies the condition $(F)$, and this is one other proof of the filtration of $\mathbb{T}$.\\

We end this section by translating Theorem 4.7.26 in standard terms. First of all, we introduce the following definition:

\begin{defn} A subset $A$ of $\N$ is {\bfseries $\mathcal{F}$-uniformly maximal} if there is an ultrafilter $\U$ such that, for every ultrafilter $\V\in\mathcal{C}_{\mathcal{F}}(\U)$, $A\in\V$. \end{defn}

Observe that a set $A$ is $\mathcal{F}$-uniformaly maximal if and only if there is an hypernatural number $\alpha$ such that $A\in\mathfrak{F}_{\alpha}$, and that if $A$ is $\mathcal{F}$-uniformly maximal and $A\subseteq B$ then $B$ is $\mathcal{F}$-uniformly maximal as well. 

\begin{prop} Given a set $\mathcal{F}\subseteq\mathtt{Fun}(\N,\N)$ of functions, the following two conditions are equivalent:
\begin{enumerate}
	\item $\mathcal{F}$ satisfies condition (F);
	\item for every $\mathcal{F}$-uniformly maximal set $A$ the set $A^{c}$ is not $\mathcal{F}$-uniformly maximal. 
\end{enumerate}
\end{prop}

\begin{proof} We have just to observe that condition (F) does not hold if and only if there are hypernatural numbers $\alpha,\beta$ and a subset $A$ of $\N$ such that $A\in\mathfrak{F}_{\alpha}$, $A^{c}\in\mathfrak{F}_{\beta}$ if and only if there is an $\mathcal{F}$-uniformly maximal set $A$ such that $A^{c}$ is $\mathcal{F}$-uniformly maximal.\end{proof}

By considering Theorem 4.7.26 and Proposition 4.7.28, we get this standard characterization of filtered families of functions:

\begin{thm} Let $\mathcal{F}\subseteq\mathtt{Fun}(\N,\N)$ be a set of functions. The following two conditions are equivalent:
\begin{enumerate}
	\item $\mathcal{F}$ is filtered;
	\item for every subset $A$ of $\N$, if $A$ is $\mathcal{F}$-uniformly maximal then $A^{c}$ is not $\mathcal{F}$-uniformly maximal.
\end{enumerate}

\end{thm}

\subsection{Generating functions}

For particular subsets $\mathcal{F}$ of $\mathtt{Fun}$($\N,\N$) the cones $\mathcal{C}_{\mathcal{F}}(\U)$ have simple algebraical characterizations:

\begin{defn} Let $G$ be a function in $\mathtt{Fun}$$(\N\times\N^{k},\N)$, and let $S$ be a subset of $\N^{k}$. The {\bfseries set of functions generated by $(G,S)$} is the set

\begin{center} $\mathcal{F}(G,S)=\{f_{a_{1},...,a_{k}}(n)\in$ $\mathtt{Fun}$$(\N,\N)\mid$ $(a_{1},...,a_{k}\in S)\wedge(\forall n\in\N$ $f_{a_{1},...,a_{k}}(n)=G(n,(a_{1},...,a_{k})))\}$. \end{center}

The function $G$ is called {\bfseries generating function} of $\mathcal{F}(G,S)$, and $S$ is called {\bfseries set of parameters} of $\mathcal{F}(G,S)$.

\end{defn}

When $S=\N^{k}$, the family $\mathcal{F}(G,\N^{k})$ is simply denoted by $\mathcal{F}(G)$. In the following table is shown that some important sets of functions $\mathcal{F}$ are generated by an appropriate pair ($G,S$):\\

\begin{tabular}{l|l}

{\bfseries Sets of Functions} & {\bfseries Generating Functions, Sets of Parameters}\\
 & \\
Translations & $G(n,m)=n+m$, $S=\N$ \\
 & \\
Proper Translations & $G(n,m)=n+m$, $S=\N\setminus\{0\}$ \\
& \\
			Non-zero Homoteties & $G(n,m)=n\cdot m$, $S=\N\setminus\{0\}$\\
 & \\			
			$\{f_{m}(n)=n^{m}\mid m>0\}$ & $G(n,m)=n^{m}$, $S=\N\setminus\{0\}$\\
			 & \\
			$\{f_{m}(n)=m^{n}\mid m>1\}$ & $G(n,m)=m^{n}$, $S=\N\setminus\{0,1\}$\\
			 & \\
			Non-constant Affinities & $G(n,(a,b))=an+b$, $S=\N^{2}\setminus \{(0,b)\mid b\in\N\}$\\ 
			 & \\
			Non-constant Polynomials & $G(n,(a_{0},...,a_{m}))=\sum_{i=0}^{m}a_{i}n^{i}$, \\
			with degree $m$ & $S=\N^{m+1}\setminus\{(a_{0},0,0,...,0)\mid a_{0}\in\N\}$

\end{tabular}
\\ \\

When the set $\mathcal{F}$ of functions is generated by a pair $(G,S)$, we can algebraically characterize the sets $\mathcal{C}_{\mathcal{F}}(\U)$:

\begin{thm} Let $G$ be a function in $\mathtt{Fun}$$(\N\times\N^{k},\N)$, $S$ a nonempty subset of $\N^{k}$, $\U$ an ultrafilter on $\N$ and consider $\mathcal{F}=\mathcal{F}(G,S)$. Then

\begin{center} $\mathcal{C}_{\mathcal{F}}(\U)=\overline{\{\overline{G}(\U\otimes\V)\mid \V\in\beta S\}}$. \end{center}

\end{thm}

\begin{proof} As a consequence of Theorem 4.7.22, we know that  

\begin{center} $\mathcal{C}_{\mathcal{F}}(\U)=\overline{\{\mathfrak{U}_{(^{*}\varphi)(\alpha)}\mid \varphi\in ^{*}\mathcal{F}\}}$, \end{center}

where $\alpha$ is a generator of $\U$ with $h(\alpha)=1$. By definitions and transfer,

\begin{center} $^{*}\mathcal{F}=$$^{*}\mathcal{F}(G,S)=$$^{*}\{f_{a_{1},...,a_{k}}\in$ $\mathtt{Fun}$($\N,\N$)$\mid$ $a_{1},...,a_{k}\in S$ and, for every natural number $n$, $f_{a_{1},...,a_{k}}(n)=G(n,(a_{1},...,a_{k}))\}=$\\\vspace{0.3cm}=$\{\varphi_{\alpha_{1},...,\alpha_{k}}\in$ $\mathtt{Fun}$$($$^{*}\N,$$^{*}\N)\mid$ $\alpha_{1},...,\alpha_{k}\in$$^{*}S$ and, for every hypernatural number $\alpha\in$$^{*}\N$, $\varphi_{\alpha_{1},...,\alpha_{k}}(\alpha)=$$^{*}G(\alpha,(\alpha_{1},...\alpha_{k}))\}$. \end{center}

Observe that, for every ultrafilter $\V$ in $\beta (\N^{k})$, we have

\begin{center} $\V\in\beta S$ if and only if there is a $k$-tuple $(\alpha_{1},...,\alpha_{k})$ in $^{*}S^{k}$ such that $\V=\mathfrak{U}_{(\alpha_{1},...,\alpha_{k})}$. \end{center}

{\bfseries Claim:} For every $k$-tuple $(\alpha_{1},...,\alpha_{k})$ in $^{*}S$

\begin{center} $\mathfrak{U}_{(^{*}\varphi_{\alpha_{1},...,\alpha_{k}})(\alpha)}=\overline{G}(\mathfrak{U}_{\alpha}\otimes\mathfrak{U}_{(\alpha_{1},...,\alpha_{k})})$. \end{center}

Suppose that the claim has been proved. Then 

\begin{center} $\{\mathfrak{U}_{(^{*}\varphi)(\alpha)}\mid \varphi\in$$^{*}\mathcal{F}\}=\{\overline{G}(\U\otimes\V)\mid \V\in\beta S\}$,\end{center}

and, since $\mathcal{C}_{\mathcal{F}}(\U)=\overline{\{\mathfrak{U}_{(^{*}\varphi)(\alpha)}\mid \varphi\in ^{*}\mathcal{F}\}}$, the thesis follows.\\
To prove the claim, let $A$ be a subset of $\N$. We have this chain of equivalences:

\begin{center} $A\in\overline{G}(\mathfrak{U}_{\alpha}\otimes\mathfrak{U}_{(\alpha_{1},...,\alpha_{k})})$ if and only if \\\vspace{0.3cm}

$\{n\in\N\mid \{(a_{1},...,a_{k})\in \N^{k}\mid G(n,(a_{1},...,a_{k}))\in A\}\in\mathfrak{U}_{(\alpha_{1},...,\alpha_{k})}\}\in\mathfrak{U}_{\alpha}$ if and only if\\
\vspace{0.3cm}
$\{n\in\N\mid (\alpha_{1},...,\alpha_{k})\in$$^{*}\{(a_{1},...,a_{k})\in \N^{k}\mid G(n,(a_{1},...,a_{k}))\in A\}\}\in\mathfrak{U}_{\alpha}$ if and only if\\
\vspace{0.3cm}
$\{n\in\N\mid $$^{*}G(n,(\alpha_{1},...,\alpha_{k}))\in$$^{*}A\}\in\mathfrak{U}_{\alpha}$ if and only if\\
\vspace{0.3cm}
$\alpha\in$$^{*}\{n\in\N\mid$$^{*}G(n,(\alpha_{1},...,\alpha_{k}))\in$$^{*}A\}$ if and only if\\
\vspace{0.3cm}
$^{**}G(\alpha,$$^{*}(\alpha_{1},...,\alpha_{k}))\in$$^{**}A$ if and only if\\
\vspace{0.3cm}
$($$^{*}\varphi_{\alpha_{1},...,\alpha_{k}})(\alpha)\in$$^{**}A$ if and only if\\
\vspace{0.3cm}
$A\in\mathfrak{U}_{(^{*}\varphi_{\alpha_{1},...,\alpha_{k}})(\alpha)}$. \end{center} \end{proof}

\section{Maximal Ultrafilters in $(\bN,\trianglelefteq_{\mathcal{F}})$}

In this section we study the set $\mathcal{M}_{\mathcal{F}}$ of maximal ultrafilters in ($\bN,\trianglelefteq_{\mathcal{F}})$, with particular attention to the connections between $\mathcal{M}_{\mathcal{F}}$ and the set $\mathcal{S}_{\mathcal{F}}$ of $\mathcal{F}$-maximal subsets of $\N$.

\begin{defn} Given a subset $\mathcal{F}$ of $\mathtt{Fun}$$(\N,\N)$ and an ultrafilter $\U$ in $\bN$, $\U$ is {\bfseries $\mathcal{F}$-maximal} if, for every ultrafilter $\V$ in $\bN$, $\V$ is $\mathcal{F}$-finitely mappable in $\U$. The set of $\mathcal{F}$-maximal ultrafilters is denoted by $\mathcal{M}_{\mathcal{F}}$:

\begin{center} $\mathcal{M}_{\mathcal{F}}=\{\U\in\bN\mid \U$ is $\mathcal{F}$-maximal$\}$. \end{center} 

\end{defn}

\begin{prop} If $\mathcal{F}$ is a filtered and well-ordered subset of $\mathtt{Fun}$$(\N,\N)$, an ultrafilter $\U$ is $\mathcal{F}$-maximal if and only if $[\U]$ is the greatest element in $(\bN_{/_{\equiv_{\mathcal{F}}}},\trianglelefteq_{\mathcal{F}})$, and

\begin{center} $\mathcal{M}_{\mathcal{F}}=\{\U\in\bN\mid [\U]$ is the greatest element $(\bN_{/_{\equiv_{\mathcal{F}}}},\trianglelefteq_{\mathcal{F}})\}$. \end{center}
\end{prop}

\begin{proof} The hypotheses on $\mathcal{F}$ ensures that $\trianglelefteq$ is a filtered pre-order. Then the result follows since it is a particular case of a property that holds for every pre-order.\\\end{proof}

\begin{prop} If $\mathcal{F}_{1}$, $\mathcal{F}_{2}$ are subsets of $\mathtt{Fun}$$(\N,\N)$ and $\mathcal{F}_{1}\subseteq \mathcal{F}_{2}$, then $\mathcal{M}_{\mathcal{F}_{1}}\subseteq \mathcal{M}_{\mathcal{F}_{2}}$. \end{prop}

\begin{proof} For every ultrafilters $\U,\V$ in $\bN$, if $\U\trianglelefteq_{\mathcal{F}_{1}}\V$ then $\U\trianglelefteq_{\mathcal{F}_{2}}\V$, so every $\mathcal{F}_{1}$-maximal ultrafilter is also a $\mathcal{F}_{2}$-maximal ultrafilter.\\\end{proof}

\begin{prop} Let $S$ be a weakly partition regular family of subsets of $\N$, and suppose that $S$ is $\leq_{\mathcal{F}}$-upward closed $($i.e., whenever $A\in S$ and $A\leq_{\mathcal{F}} B$, $B\in S)$. Then $\mathcal{S}_{\mathcal{F}}\subseteq S$. \end{prop}

\begin{proof} We just have to observe that, if $A$ is a set in $S$ and $M$ is a set in $\mathcal{S}_{\mathcal{F}}$, since $A\leq_{\mathcal{F}} M$ then $M\in S$. So $\mathcal{S}_{\mathcal{F}}\subseteq S$.\\\end{proof}

A first correlation between $\mathcal{S}_{\mathcal{F}}$ and $\mathcal{M}_{\mathcal{F}}$ is the following:

\begin{prop} Let $\mathcal{F}$ be a subset of $\mathtt{Fun}$$(\N,\N)$ such that $\leq_{\mathcal{F}}$ is transitive, and suppose that the family $\mathcal{S}_{\mathcal{F}}$ of $\mathcal{F}$-maximal elements in $(\wp(\N),\leq_{\mathcal{F}})$ is weakly partition regular. Then $\mathcal{M}_{\mathcal{F}}\neq\emptyset$, and 

\begin{center} $\mathcal{M}_{\mathcal{F}}=\{\U\in\bN\mid \U\subseteq\mathcal{S}_{\mathcal{F}}\}$. \end{center}

\end{prop}

\begin{proof} In Section 1.2 we proved that every weakly partition regular family contains an ultrafilter. Since $\mathcal{S}_{\mathcal{F}}$ is weakly partition regular by hypothesis, let $\U$ be an ultrafilter contained in $\mathcal{S}_{\mathcal{F}}$; $\U$ is $\mathcal{F}$-maximal since, for every set $A\in \U$, as $\U\subseteq\mathcal{S}_{\mathcal{F}}$ $A\in\mathcal{S}_{\mathcal{F}}$, so $\N\leq_{\mathcal{F}} A$; in particular, if $\V$ is an ultrafilter in $\bN$, this proves that $\V\trianglelefteq_{\mathcal{F}}\U$. This shows that 

\begin{center} (1) $\{\U\in\bN\mid \U\subseteq\mathcal{S}_{\mathcal{F}}\}\subseteq \mathcal{M}_{\mathcal{F}}$.\end{center} 

To prove the reverse inclusion, let $\U$ be a maximal ultrafilter, and $\V$ an ultrafilter included in $\mathcal{S}_{\mathcal{F}}$. Since, my maximality, $\U\trianglelefteq_{\mathcal{F}} \V$, for every set $A$ in $\V$ there is a set $B$ in $\U$ such that $B\leq_{\mathcal{F}}A$. But, as $B$ is maximal and $B\leq_{\mathcal{F}} A$, by transitivity it follows that $A$ is in $\mathcal{S}_{\mathcal{F}}$: this proves that every set $A$ in $\V$ is included in $\mathcal{S}_{\mathcal{F}}$, so: 

\begin{center} (2) $\mathcal{M}_{\mathcal{F}}\subseteq\{\U\in\bN\mid \U\subseteq\mathcal{S}_{\mathcal{F}}\}$.\end{center} 

Putting togheter (1) and (2), we obtain $\mathcal{M}_{\mathcal{F}}=\{\U\in\bN\mid \U\subseteq\mathcal{S}_{\mathcal{F}}\}$.\\ \end{proof}

\begin{cor} If $\mathcal{F}$ is a subset of $\mathtt{Fun}$$(\N,\N)$ such that $\leq_{\mathcal{F}}$ is transitive and $\mathcal{S}_{\mathcal{F}}$ is strongly partition regular then for every maximal set $A$ in $\mathcal{S}_{\mathcal{F}}$ there is a $\mathcal{F}$-maximal ultrafilter $\U$ such that $A\in\U$. In particular, \begin{center}$\mathcal{S}_{\mathcal{F}}=\bigcup \mathcal{M}_{\mathcal{F}}$.\end{center} \end{cor}

\begin{proof} By Theorem 1.2.3 it follows that, since $\mathcal{S}_{\mathcal{F}}$ is strongly partition regular, then it is an union of ultrafilters, so

\begin{center} (1) $\mathcal{S}_{\mathcal{F}}= \bigcup \{\U\in \bN\mid \U\subseteq\mathcal{S}_{\mathcal{F}}\}$. \end{center}

Since every strong partition regular family of subsets of $\N$ is, in particular, weakly partition regular, by Proposition 4.8.5 it follows that

\begin{center} (2) $\{\U\in\bN\mid \U\subseteq\mathcal{S}_{\mathcal{F}}\}= \mathcal{M}_{\mathcal{F}}$.\end{center}

As a consequence, $\mathcal{S}_{\mathcal{F}}=\bigcup \mathcal{M}_{\mathcal{F}}$.\\ \end{proof}

\begin{cor} If the family $\mathcal{S}_{\mathcal{F}}$ is weakly partition regular and $\mathcal{F}$ is well-structured then the order $\trianglelefteq_{\mathcal{F}}$ is filtered. \end{cor}

\begin{proof} If $\mathcal{S}_{\mathcal{F}}$ is weakly partition regular, then there is some maximal ultrafilter $\U$ in $\bN$, so there is a greatest element in $(\bN/_{\equiv_{\mathcal{F}}},\trianglelefteq_{\mathcal{F}})$; in particular, as a consequence of Proposition 4.2.8, the relation $\trianglelefteq_{\mathcal{F}}$ is filtered. \\ \end{proof}

\begin{thm} If $\mathcal{F}$ is filtered and well-structured then $\mathcal{S}_{\mathcal{F}}\subseteq\bigcup\mathcal{M}_{\mathcal{F}}$. \end{thm}

\begin{proof} First of all we observe that, by hypothesis, there are $\mathcal{F}$-maximal ultrafilters.\\
Let $A$ be a set in $\mathcal{S}_{\mathcal{F}}$, and suppose by contradiction that for every maximal ultrafilter $\U$ the set $A$ is not in $\U$, i.e. that the complement $A^{c}$ is in $\U$. Let $\U$ be a maximal ultrafilter, and let $\alpha\in$$^{*}\N$ be a generator of $\U$. In particular, $\alpha\in$$^{*}(A^{c})$.\\
Since $A$ is in $\mathcal{S}_{\mathcal{F}}$, $A^{c}\leq_{\mathcal{F}} A$ so, by Proposition 4.6.5, there is a function $\varphi$ in $^{*}\mathcal{F}$ with $\varphi(A^{c})\subseteq $$^{*}A$. By transfer, this implies that 

\begin{center} $( ^{*}\varphi)($$^{*}A^{c})\subseteq ^{**}A$. \end{center}

As $\alpha\in$$^{*}A^{c}$, this entails that $($$^{*}\varphi)(\alpha)\in$$^{**}A$, so $A\in\mathfrak{U}_{( ^{*}\varphi)(\alpha)}$. But, by Theorem 4.7.22, $\mathfrak{U}_{\alpha}\trianglelefteq_{\mathcal{F}} \mathfrak{U}_{( ^{*}\varphi)(\alpha)}$ for every function $\varphi$ in $^{*}\mathcal{F}$ and, since $\mathfrak{U}_{\alpha}$ is maximal, this entails that $\mathfrak{U}_{( ^{*}\varphi)(\alpha)}$ is maximal. This is absurd, since $A\in\mathfrak{U}_{( ^{*}\varphi)(\alpha)}$.\\ \end{proof}

By combining the results proved in this section, we obtain the following theoremt:

\begin{thm} For every well-structured and filtered set of functions $\mathcal{F}\subseteq\fN$ the following two conditions are equivalent:
\begin{enumerate}
	\item $\mathcal{S}_{\mathcal{F}}$ is weakly partition regular;
	\item $\mathcal{S}_{\mathcal{F}}$ is strongly partition regular.
\end{enumerate}
\end{thm}

\begin{proof} $(1)\Rightarrow (2)$ Suppose that $\mathcal{S}_{\mathcal{F}}$ is weakly partition regular. Since $\trianglelefteq_{\mathcal{F}}$ is filtered and well-structured, the previous proposition ensures that $\mathcal{S}_{\mathcal{F}}\subseteq \bigcup \mathcal{M}_{\mathcal{F}}$, while Proposition 4.8.5 ensures that $\bigcup \mathcal{M}_{\mathcal{F}}\subseteq\mathcal{S}_{\mathcal{F}}$. So 

\begin{center} $\mathcal{S}_{\mathcal{F}}=\bigcup \mathcal{M}_{\mathcal{F}}$\end{center}

and this shows that the family $\mathcal{S}_{\mathcal{F}}$ is an union of ultrafilters, and this is a condition equivalent to state that $\mathcal{S}_{\mathcal{F}}$ is strongly partition regular. \\
$(2)\Rightarrow (1)$ Every strongly partition regular family is, in particular, weakly partition regular.\\ \end{proof}

A first corollary is that:

\begin{cor} If $\mathcal{F}\subseteq\fN$ is a well-structured and filtered set of functions and $\mathcal{S}_{\mathcal{F}}$ is weakly partition regular, then an ultrafilter $\U$ is in $\mathcal{M}_{\mathcal{F}}$ if and only if $\U\subseteq\mathcal{S}_{\mathcal{F}}$. \end{cor}

\begin{proof} Simply observe that the previous theorem tells that $\mathcal{S}_{\mathcal{F}}$ is strongly partition regular, and apply Theorem 1.2.3.\\  \end{proof}

One other way to formulate the previous corollary is to say that whenever the class $\mathcal{S}_{\mathcal{F}}$ is weakly partition regular, with $\mathcal{F}$ well-structured and filtered, then an ultrafilter $\U$ is $\mathcal{F}$-maximal if and only if it is made of maximal sets with respect $\leq_{\mathcal{F}}$ or, equivalently, that a set is maximal in $(\wp(\N),\leq_{\mathcal{F}})$ if and only if it is contained in some maximal ultrafilter $\U$ in $(\bN,\trianglelefteq_{\mathcal{F}})$.\\
Since we gave in Proposition 4.6.12 a condition, for a set $A$, to be in $\mathcal{S}_{\mathcal{F}}$, we get:

\begin{prop} Let $\mathcal{F}\subseteq\fN$ be a well-structured and filtered set of functions and let $\mathcal{S}_{\mathcal{F}}$ be weakly partition regular. Then an ultrafilter $\U$ is maximal if and only if for every $A\in \U$, for every natural number $n\in\N$, there is a function $f_{n}$ in $\mathcal{F}$ with $f_{n}([0,n])\subseteq A$. \end{prop}

In next section we will give an example of application of Theorem 4.8.9 by proving that the class of subsets of $\N$ that contains arbitrarily long arithmetical progressions in strongly partition regular.

\section{Finite Mappability Under Affinities}

In this section we consider the set of affinities $\mathcal{A}$:

\begin{center} $\mathcal{A}=\{f_{(a,b)}\in \mathtt{Fun}(\N,\N)\mid ((a,b)\in \N^{2})\wedge (a\neq 0)\wedge(\forall n\in\N$ $f_{(a,b)}(n)=an+b)\}$. \end{center}

We show that, as consequences of results already proved in this chapter, the relation $\trianglelefteq_{\mathcal{A}}$ has some important properties.

\begin{prop} The set of affinities is well-structured and filtered. \end{prop}

\begin{proof} $\mathcal{A}$ is closed under composition since, for every $(a,b), (c,d)$ in $\N^{2}$, 

\begin{center} $f_{(a,b)}\circ f_{(c,d)}=f_{(ac,ad+b)}$.\end{center}

Moreover, the function $f_{(1,0)}\in\mathcal{A}$ is the identity function. So by Corollary 4.6.7 $\mathcal{A}$ is well-structured.\\
$\mathcal{A}$ is filtered since it contains the family $\mathbb{T}$ of traslations, which is filtered, and in 4.6.8 we proved that if a set of functions $\mathcal{F}$ is a superset of a filtered set of functions then $\mathcal{F}$ is filtered as well.\\\end{proof}

In consequence there are maximal $\trianglelefteq_{\mathcal{A}}$-ultrafilters.\\
As a corollary of Proposition 4.6.12, we get that:

\begin{prop} A subset $A$ of $\N$ is $\mathcal{A}$-maximal if and only if it contains arbitrarily long arithmetic progressions. \end{prop}

\begin{proof} As a consequence of Proposition 4.6.12, a set $A$ is $\mathcal{A}$-maximal if and only if for every natural number $n$ there are $a,b$ in $\N$ such that $f_{(a,b)}([1,...,n])\subseteq A$. As $f_{(a,b)}([1,...,n])$ is an arithmetic progression of lenght $n$, it follows that a subset $A$ of $\N$ is $\mathcal{A}$-maximal if and only if it contains arbitrarily long arithmetical progressions.\\\end{proof}

Van der Werden's Theorem states that the family $\mathcal{S}_{\mathcal{A}}$ of subsets of $\N$ that contains arbitrarily long arithmetical progressions is weakly partition regular; by Theorem 4.8.9 it follows that:

\begin{prop} The family of subsets of $\N$ that contains arbitrarily long arithmetical progressions is strongly partition regular. \end{prop}

Also, from Proposition 4.8.5 it follows that 

\begin{prop} An ultrafilter $\U$ is $\mathcal{A}$-maximal if and only if it is a Van der Waerden's ultrafilter. \end{prop}

Moreover, since $\mathbb{T}$ is a subset of $\mathcal{A}$, every $\trianglelefteq_{\mathbb{T}}$-maximal ultrafilter is also $\mathcal{A}$-maximal: since the set of $\trianglelefteq_{\mathbb{T}}$-maximal ultrafilters is $\overline{K(\bN,\oplus)}$, from this it follows that:

\begin{prop} Every ultrafilter $\U$ in $\overline{K(\bN,\oplus)}$ is a Van der Waerden's ultrafilter. \end{prop}

Observe that the multiplicative analogue of $\mathbb{T}$ is the set of homoteties $\mathbb{H}$:

\begin{center}$\mathbb{H}=\{f_{a}\in \mathtt{Fun}(\N,\N)\mid (a\in\N\setminus\{0\})\wedge (\forall n\in\N$ $f_{a}(n)=an)\}$, \end{center}

and, by essentially the same argumetns we used for the relation $\trianglelefteq_{\mathbb{T}}$, we could prove that 

\begin{center} $\mathcal{M}_{\mathbb{H}}=\overline{K(\bN,\odot)}$. \end{center}

So an ultrafilter is $\mathbb{H}$-maximal if and only if it is in the closure of the minimal bilater ideal of $(\bN,\odot)$. Since $\mathbb{H}\subseteq\mathcal{A}$, from this it follows that

\begin{prop} Every ultrafilter $\U$ in $\overline{K(\bN,\odot)}$ is a Van der Waerden's ultrafilter. \end{prop}

Now we consider the cones $\mathcal{C}_{\mathcal{A}}(\U)$. Recall that $\mathcal{A}$ is the set of functions generated by $G:\N\times\N^{2}\rightarrow\N$, where

\begin{center} $G(n,(a,b))=an+b$ for every natural number $n$ in $\N$, \end{center}

and with set of parameters $S=\N^{2}\setminus \{(0,n)\mid n\in\N\}$. By Theorem 4.7.28 it follows that, for every ultrafilter $\U$ in $\bN$, 

\begin{center} $\mathcal{M}(\U,\mathcal{A})=\overline{\{\overline{G}(\U\otimes\V)\mid \V\in\beta S\}}$. \end{center}

Observe that, for every ultrafilter $\V$ in $\beta S$, $\overline{G}(\U\otimes\V)$ is the ultrafilter on $\N$ defined by the following condition: if $A$ is a subset of $\N$, then

\begin{center} $A\in\overline{G}(\U\otimes\V)$ if and only if $\{n\in\N\mid \{(a,b)\in\N^{2}\mid an+b\in A\}\in\V\}\in\U$. \end{center}

We conclude by observing that, if $\U$ is any ultrafilter in $K(\bN,\oplus)$ (e.g., if $\U$ is a minimal additive idempotent) then, as $\mathcal{C_{\mathcal{A}}}(\U)=\mathcal{M}_{\mathcal{A}}$ by $\mathcal{A}$-maximality of $\U$, it follows that

\begin{center} $\mathcal{M}_{\mathcal{A}}=\overline{\{\overline{G}(\U\otimes\V)\mid\V\in\bN\}}$. \end{center}

\section{Further Studies}

In this section we indicate two possible directions for further studies on finite mappability.

\subsection{Relations between different sets of maximal ultrafilters}

In Section 4.9 we proved that the Van der Waerden's ultrafilters are exactly the maximal ultrafilters respect the relation of $\mathcal{A}$-finite mappability, where $\mathcal{A}$ denotes the set of affinities, and we showed that as a consequence the minimal bilateral ideal of both $(\bN,\oplus)$ and $(\bN,\odot)$ consist of Van der Waerden's ultrafilters. \\

{\bfseries Question:} Is it possible to characterize other important subsets of $\bN$ as sets of $\mathcal{F}$-maximal ultrafilters for appropriate sets of functions $\mathcal{F}$?\\

A particularly interesting case, in our opinion, is that of polynomials: if

\begin{center} $\mathcal{P}_{d}=\{f_{(a_{0},...,a_{d})}\in \mathtt{Fun}(\N,\N)\mid (a_{0},...,a_{d}\in \N)\wedge(a_{d}\neq 0)\wedge(\forall n\in\N$ $f_{(a_{0},...,a_{d})}(n)=a_{d}n^{d}+a_{d-1}n^{d-1}+...+a_{1}n+a_{0})\}$,\end{center}

then $\mathcal{P}=\bigcup_{d\geq 1}\mathcal{P}_{d}$

Observe that $\mathcal{A}=\mathcal{P}_{1}$.\\

{\bfseries Question:} Is there a correlation between $\mathcal{M}_{\mathcal{P}}$ and the partition regularity of equations?

\subsection{Orderings on the hyperextension $^{*}\N$}

Two important and well-known relations on a generical hyperextension $^{*}\N$ are the Puritz and the Rudin-Keisler pre-orders, that we shortly recall.

\begin{defn} Let $\U,\V$ be two ultrafilters on $\N$. $\U$ is {\bfseries Rudin-Keisler above $\V$} $($notation $\V\leq_{RK}\U)$ if there is a function $f$ in $\mathtt{Fun}(\N,\N)$ such that $\overline{f}(\U)=\V$ $($i.e. if $A\in\V\Leftrightarrow f^{-1}(A)\in\U)$.\\
The relation $\leq_{RK}$ is called {\bfseries Rudin-Keisler preorder}.\end{defn}

Observe that, since for every functions $f,g$ in $\mathtt{Fun}(\N,\N)$, $\overline{f}\circ\overline{g}=\overline{f\circ g}$, the Rudin-Keisler preorder is transitive (and the reflexivity is immediate), so it is actually a preorder. A well-known fact about this relation is:

\begin{prop} For every ultrafilters $\U,\V$ on $\N$, the following two properties are equivalent:
\begin{enumerate}
	\item $\U\leq_{RK}\V$ and $\V\leq_{RK}\U$;
	\item there is a bijection $f$ in $\mathtt{Fun}(\N,\N)$ with $\overline{f}(\U)=\V$.
\end{enumerate}\end{prop}

If we denote with $\equiv_{RK}$ the equivalence relation such that $\U\equiv_{RK}\V$ if and only if $\U\leq_{RK}\V$ and $\V\leq_{RK}\U$, the induced order $\leq_{RK}$ on the quotient space is the Rudin-Keisler order. This has been extensively studied in literature (see e.g. $\cite[\mbox{Cap 11}]{rif21}$). Here we are not interested in the specifical properties of this order; we just want to show that, if we translate it in the context of sets of generators of ultrafilters, one obtains a well-known and studied relation that refines it, namely the Puritz pre-order.

\begin{defn} Let $^{*}\N$ be a hyperextension of $\N$, and $\alpha,\beta$ two hypernatural numbers in $^{*}\N$. We say that $\alpha$ is {\bfseries Puritz above} $\beta$ $($notation $\beta\leq_{P}\alpha)$ if there is a function $f$ in $\mathtt{Fun}(\N,\N)$ with $^{*}f(\alpha)=\beta$. \end{defn}

Since $^{*}f\circ$$^{*}g=$$^{*}(f\circ g)$ for every functions $f,g\in\mathtt{Fun}(\N,\N)$, the relation $\beta\leq_{P}\alpha$ is a pre-order; similarly to the Rudin-Keisler case, it can be proved that

\begin{prop} For every $\alpha,\beta$ in $^{*}\N$, the following two conditions are equivalent:
\begin{enumerate}
	\item $\alpha\leq_{P}\beta$ and $\beta\leq_{P}\alpha$;
	\item there is a bijection $f$ in $\mathtt{Fun}(\N,\N)$ with $^{*}f(\alpha)=\beta$.
	
\end{enumerate}\end{prop}

Denote by $\sim_{P}$ (and call $P$-equivalence) the equivalence relation such that, for every $\alpha,\beta\in$$^{*}\N$, $\alpha\sim_{P}\beta$ if and only if $\alpha\leq_{P}\beta$ and $\beta\leq_{P}\alpha$. The induced relation $\leq_{P}$ on the set of equivalence classes of $\sim_{P}$ (that are called constellations in the literature) is an order, which is called the Puritz order for constellations. This order, and its relation with the Rudin-Keisler order, are studied, e.g., in $\cite{rif38}$.\\
In a precise sense, the Puritz pre-ordering is a refinement of the Rudin-Keisler pre-ordering: in fact, $\mathfrak{U}_{\alpha}\leq_{RK}\mathfrak{U}_{\beta}$ whenever $\alpha\leq_{P}\beta$, while the converse is false, since $\mathfrak{U}_{\alpha}\leq_{RK}\mathfrak{U}_{\beta}$ implies only $\alpha\leq_{P}\beta^{\prime}$ for some $\beta^{\prime}\sim_{u}\beta$.\\
From the characterization of $\mathcal{C}_{\mathcal{F}}(\U)$ it follows that there is a similarity between the relations $\trianglelefteq_{\mathcal{F}}$ and the Rudin-Keisler pre-order. This similarity seem to be particularly interesting from the point of view of the hyperextension $^{\bullet}\N$ (more generally, from the point of any hyperextension with at least the $\mathfrak{c}^{+}$-enlarging property). As for the $\mathcal{F}$-finite embeddability, when $\mathcal{F}$ is well-structured, by the characterization of $\mathcal{C}_{\mathcal{F}}(\U)$ for a generical ultrafilter $\U$ in $\bN$ it follows that the relation $\trianglelefteq_{\mathcal{F}}$ can be translated in nonstandard terms by posing, for generical $\alpha,\beta$ in $^{*}\N$,

\begin{center}$\alpha\trianglelefteq_{\mathcal{F}}\beta\Leftrightarrow\beta\in\overline{\{}$$\overline{^{*}\varphi(\alpha)\mid \varphi\in}$$\overline{^{*}\mathcal{F}\}}\cap$$^{*}\N$, \end{center}

where the closure in $^{**}\N$ is taken in the S-Topology.\\

{\bfseries Question:} Are there connections between the Rudin-Keisler pre-order, the Puritz pre-order and the $\mathcal{F}$-finite embeddability in the nonstandard context for suitable families $\mathcal{F}$?

\end{document}